\documentclass{article}

\usepackage{bm}
\usepackage{amsmath,amssymb}
\usepackage{amsthm}
\usepackage{graphicx,psfrag,epsf}
\usepackage{enumerate}
\usepackage{enumitem}
\usepackage{url} 
\usepackage{caption}
\usepackage{subcaption}
\usepackage[colorlinks=true,linkcolor=blue]{hyperref}
\usepackage{algorithm}
\usepackage{algpseudocode}
\usepackage{multirow}
\usepackage{mathtools}
\usepackage{pgfplots}
\usepackage{comment}
\usepackage{authblk}

\usepackage[margin=1in]{geometry}

\usepackage{booktabs}
\usepackage{manyfoot}
\DeclareNewFootnote{A}[arabic]
\newtheorem{theorem}{Theorem}[section]
\newtheorem{lemma}[theorem]{Lemma}
\newtheorem{proposition}[theorem]{Proposition}
\newtheorem{corollary}[theorem]{Corollary}
\newtheorem{exmp}{Example}[section]
\newtheorem{definition}{Definition}
\newtheorem{assumption}{Assumption}
\theoremstyle{remark}
\newtheorem{remark}{Remark}

\newcommand{\R}{\mathbb{R}} 
\newcommand{\C}{\mathbb{C}} 
 
\newcommand{\E}{\mathbb{E}}
\newcommand{\calN}{\mathcal{N}}

\newcommand{\calM}{{\cal M}}
\newcommand{\calL}{{\cal L}}
\newcommand{\calX}{{\cal X}}

\newcommand{\calE}{{\cal E}}
\newcommand{\calK}{{\cal K}}
\newcommand{\calJ}{{\cal J}}

\renewcommand{\P}{ {P}}

\newcommand{\Second}{\textup{I}\!\textup{I}}

\usepackage{xcolor}

\newcommand{\rev}[1]{\textcolor{cyan}{ #1}}

\allowdisplaybreaks

\begin{document}

\title{Adaptive Bayesian Regression on Data with Low Intrinsic Dimensionality}

\author[1]{Tao Tang}
\author[3]{Nan Wu}
\author[1]{Xiuyuan Cheng}
\author[1,2]{David Dunson}

\affil[1]{{\small Department of Mathematics, Duke University}}
\affil[2]{{\small 
Department of Statistical Science, Duke University}}
\affil[3]{\small Department of Mathematical Sciences, The University of Texas at Dallas
}

\date{\vspace{-30pt}}

\maketitle

\begin{abstract}
We study how the posterior contraction rate under a Gaussian process (GP) prior depends on the intrinsic dimension of the predictors and the smoothness of the regression function. An open question is whether a generic GP prior that does not incorporate knowledge of the intrinsic lower-dimensional structure of the predictors can attain an adaptive rate for a broad class of such structures. We show that this is indeed the case, establishing conditions under which the posterior contraction rates become adaptive to the intrinsic dimension in terms of the covering number of the data domain (the Minkowski dimension) and prove the nonparametric posterior contraction rate, up to a logarithmic factor. When the domain is a compact manifold, we prove the RKHS approximation to intrinsically defined H\"older functions on the manifold of any order of smoothness by a novel analysis, leading to the optimal adaptive posterior contraction rate. We propose an empirical Bayes prior on the kernel bandwidth using kernel affinity and $k$-nearest neighbor statistics, bypassing explicit estimation of the intrinsic dimension. The efficiency of the proposed Bayesian regression approach is demonstrated in various numerical experiments.
\end{abstract}

\

\noindent
\rev{
[Comments in October 2026:\\
This version corrects several proof-level issues in the published paper.
The main correction concerns higher covariant derivatives: 
$\nabla^k f$ is generally not a symmetric tensor for $k \ge 3$,
requiring revisions to the discussion in Section 2.4 and to related arguments in the appendices.
The H\"older class is unchanged, and the affected estimates in the proof are recovered by controlling the symmetry defect through curvature terms.
We also revised the proofs of several other lemmas in the appendices and corrected smaller technical and typographical issues. 
All the corrections do not change the definitions or statements of the main results.]}

\section{Introduction}
\label{sec:intro}

Our interest is in nonparametric regression methodology that can adapt to the intrinsic lower-dimensional structure in the predictors to address the curse of dimensionality. For concreteness, we focus on Bayesian Gaussian process (GP) regression, though our theoretical developments have broader ramifications. GP regression is extremely popular in many application areas due to the combination of simplicity, computational tractability, ease of incorporating prior information and flexibility. We consider the following model: 
 \begin{align}
 \label{eq:model}
    Y_i = f^*(X_i)+w_i,\quad 
    w_i \sim \mathcal{N}(0, \sigma^2),\quad i = 1, \cdots,n,  
 \end{align}
where  $X_i \in  \mathcal{X} \subset \mathbb{R}^D$,
$\calX$ is the data domain, $f^*: \mathcal{X} \to \mathbb{R}$ is the true regression function that generated the data,
and $w_i$ is a residual error. We introduce the notation $f$ to denote the inferred regression function. We choose a GP prior for $f$ and ideally would like the resulting posterior for $f$ to concentrate near $f^*$.
We assume $\sigma^2$ is known and fixed for simplicity of exposition, and we discuss possible extensions in Section \ref{sec:discuss}.

It is well known that nonparametric regression is subject to a curse-of-dimensionality problem depending on the number of predictors $D$. Given $n$ independent observations on an $s$ times differentiable $f^*$ on $\R^D$, the minimax 
nonparametric estimation rate of $f^*$ is $n^{-{s}/{(2s+D)}}$ \cite{stone1982optimal}. 
As $D$ is commonly large in modern applications, huge sample sizes may be needed to obtain sufficiently accurate estimates of $f^*$ unless some other structure can be imposed to reduce dimensionality. For example, suppose the predictor domain $\mathcal{X}$ has intrinsic dimensionality $\varrho \ll D$ in a sense we will clarify later. 
A natural question then arises whether nonparametric estimators of $f^*$ can adapt to the (typically unknown) intrinsic structure of the data and address the curse-of-dimensionality by achieving 
the estimation 
rate of $n^{-{s}/{(2s+\varrho)}}$. The focus of this paper is to develop a Bayesian nonparametric regression estimator that is adaptive to the intrinsic structure in $\mathcal{X}$ without requiring prior knowledge of the structure or its dimensionality.

Intrinsic dimensionality in data analysis has been extensively studied in various contexts. For nonparametric regression and classification, a common notion of lower-dimensional structure is sparsity, assuming that a small subset of the features impact the response \cite{ lafferty2008rodeo, jiang2021variable, yang2015minimax}. 
Instead, our focus is on the dimensionality of the feature space $\mathcal{X}$ itself. 
In this context, a popular assumption is that $\mathcal{X}$ corresponds to a smooth manifold $\mathcal{M} \subset \mathbb{R}^D$  \cite{bickel2007local, scott2006minimax, cheng2013local, ye2008learning, ye2009svm}. 
In this work, we consider a broader concept of low-dimensionality defined by the covering number, which includes the Riemannian manifold as a special case. While also obtaining general results on posterior contraction rates depending on the covering number, we show a minimax-optimal adaptive rate $n^{-s/(2s + d)}$ (up to a log factor) in the case of a $d$-dimensional manifold for an intrinsic class of H\"older functions where the smoothness $s$ can be arbitrarily high. 

There is an existing literature using the Minkowski dimension, also known as the box-counting dimension, of the data domain as a notion of intrinsic dimensionality. 
The definition of Minkowski dimension is based on the covering number, and manifolds provide an example of a subset having a low Minkowski dimension. In this context, a variety of nonparametric regression and classification algorithms have been studied, including local polynomial regression, $k$ nearest neighbors, Nadaraya-Watson kernel regression, decision trees and least squares kernel ridge regression \cite{bickel2007local, kpotufe2011k, kpotufe2013adaptivity, kulkarni1995rates, scott2006minimax, hamm2021adaptive}. 
In comparison, low intrinsic dimensionality beyond the manifold setting in Bayesian nonparametric regression has been less developed (except for \cite{castillo2024posterior} which uses deep neural networks, see more below). 
In this work, we derive general conditions to prove posterior contraction rates adaptive to intrinsic data dimensionality measured by the covering number (Minkowski dimension) and thus generalize beyond the manifold assumption. 

Posterior contraction rates for Bayesian nonparametric regression have had considerable development since the seminal work of \cite{ghosal2000convergence,shen2001rates}. 
Adaptive rates of GP regression for $f^*$ defined on $[0,1]^D$ were proved in \cite{van2008rates, van2009adaptive}. A series of subsequent papers analyzed 
 the performance of Bayesian regression under the assumption of low intrinsic data dimensionality. 
\cite{castillo2014thomas} used a heat kernel on a {\em known} manifold within a GP prior and 
proved the  minimax adaptive posterior contraction rate.
\cite{yang2016bayesian} established a minimax optimal adaptive rate in estimating $f^*$ on an unknown manifold. 
The rate adapts to the manifold dimension and smoothness of the regression function, but the function smoothness is restricted to H\"older class with $s \le 2$ 
and the prior for the kernel bandwidth parameter requires knowing or estimating the true manifold dimension $d$.
\cite{dunson2022graph} used graph Laplacians to estimate a GP covariance function incorporating the intrinsic geometry of the manifold
and proved posterior contraction rates for such GPs,
where the regression function lies in a subspace of a Besov space 
(linearly spanned by eigenfunctions of the manifold Laplace-Beltrami operator).

In addition, \cite{rosa2024posterior} considered GP priors having a Mat\'ern kernel on a known compact Riemannian manifold. 
They proved  posterior contraction rates  adaptive to the manifold dimension,
and when the kernel regularity parameter matches the smoothness of $f^*$  the  optimal  rate can be achieved.
Recently, \cite{rosa2024nonparametric} considered Bayesian nonparametric regression on an embedded data manifold based on a graph Laplacian eigen-basis, 
proving an optimal minimax rate adaptive to the manifold dimensionality $d$ and the smoothness $s$ of the regression function $f^*$ for arbitrarily high $s$.
Their regression function is in an extrinsic H\"older class,
and the theory requires high regularity of the data density (of H\"older order $s-1$) when the smoothness $s$ of $f^*$ is high;
the methodology involves eigen-computation of graph Laplacians
and their adaptive prior may need the knowledge of $d$ in practice.

Deep GP priors, which improve flexibility over traditional GP priors through several nested layers of GPs, have seen increasing focus in recent years. \cite{Finocchio23deep} studied posterior contraction rates for a class of deep GP models, but with a gap between their practical methodology and the theoretical model. \cite{castillo2025deep} introduced a deep horseshoe GP prior for data-driven selection of multiscale bandwidths for H\"older functions on Euclidean domains, while proving posterior contraction rates adaptive to the smoothness of $f^*$ and effective dimensionality of the data. 
For deep GPs with heavy-tailed priors, \cite{castillo2024posterior} proved posterior contraction rates adaptive to smoothness of $f^*$ and the Minkowski dimension of the data domain. Our results show that similar automatic adaptivity to low-dimensional data structures and smoothness of regression functions can also be achieved by more traditional kernel methods, which may have advantages in terms of simplicity.

In this work, we focus on GP priors in which the covariance function,
which is in a regularity class that includes the squared exponential, 
is directly computed from the Euclidean coordinates of $X_i$'s. We propose an empirical Bayesian prior that adapts to the intrinsic dimensionality of data and the smoothness of $f^*$, without requiring prior knowledge of either. 
Our main contributions include the following.  
\begin{itemize}  
 \item  
 We consider GP regression on a data domain embedded in high dimensional Euclidean space, 
and establish conditions for 
nonparametric
posterior contraction rates (up to a log factor)
adaptive to low intrinsic data dimensionality, 
measured by the covering number (the Minkowski dimension).
We propose a GP methodology that satisfies these conditions. 
The covariance function of the GP is a kernel defined on the ambient Euclidean space
belonging to a class of functions that satisfy certain technical conditions, 
with the squared exponential kernel being a representative example.

\item 
When the data domain is an (unknown) low-dimensional Riemannian manifold, 
we prove optimal adaptive posterior contraction rates, up to a logarithmic factor, 
for any order of function smoothness ($0 < s < \infty$). The function regularity order $s$ is measured by a H\"older class on the manifold which is intrinsically defined. 
The key element of our analysis is an on-manifold RKHS approximation result of intrinsic manifold H\"older functions that can go to an arbitrarily high order of $s$.

\item 
To avoid including knowledge of the intrinsic dimension in the prior, we propose an empirical Bayes approach using kernel affinity and $k$-nearest neighbor ($k$NN) statistics. 
This approach achieves our theoretical adaptive rates
without prior knowledge or estimation of the intrinsic dimensionality of data. The experimental performance of the proposed method is shown on simulated manifold data and image data. 

\end{itemize}

We start by developing a theoretical framework for general low-dimensional data domains having Minkowski dimension $\varrho$, and then consider the case of manifold data. In the general case, the posterior contraction rate is $n^{- s/(2s + \varrho)}$, and for $d$-dimensional submanifolds, the rate becomes $n^{- s/(2s + d)}$, both up to a log factor. The rate is considered minimax-optimal for the manifold case, and for the general case our result provides an upper bound of the error. In addition, we extend our theory to cover the case of a union of submanifolds having possibly different intrinsic dimensionalities,  as an example of simple stratified spaces (of intrinsically low dimension) 
beyond the setting of a single connected manifold.

The remaining sections of this paper are organized as follows: In Section \ref{sec:pre}, we review the necessary background and notation. Section \ref{sec:pos_rate} introduces conditions governing the posterior contraction rate of the GP on a general set $\mathcal{X}$. In Section \ref{sec:rate_manifold}, we prove the optimal contraction rate for H\"older functions on an unknown manifold
and introduce the proposed empirical Bayes prior. 
Section \ref{sec:num} experimentally evaluates the proposed GP method in comparison to other methods. 
Finally, Section \ref{sec:discuss} discusses future extensions.
The proofs are deferred to Appendix \ref{sec:proof-lemma-4.1} and Appendices \ref{sec:proof-sec3}-\ref{app:proof-sec4},
and the technical lemmas are in Appendix \ref{ap:A}.

\paragraph*{Notations.}
The notation in this work is standard.
$a \vee b = \max\{ a, b \}$  and $a \wedge b = \min \{ a, b \}$.
$\| \cdot \|_\infty$ stands for the $\infty$-norm in Euclidean space, or the functional $\infty$-norm on $C(\calX)$ where $\calX$ is the data domain, depending on the context.
For the asymptotic notations, 
$f = O(g)$ indicates that there exists a constant $C > 0$ such that $|f| \leq C |g|$ in the limit.
For non-negative $f$ and $g$,
$f \sim g$ if there exist $C_1 > C_2 >0 $ s.t. $ C_2 g \le f \le C_1 g$ in the limit;
$f \lesssim g$ means that there exists a constant $C > 0$ such that $f \leq C g$ in the limit.

\section{Preliminaries}
\label{sec:pre}

In this section, we review background information 
on Gaussian processes (GP), reproducing kernel Hilbert spaces (RKHS), 
RKHS on a measurable subset $\mathcal{X} \subset \R^D$,
and concepts of Riemannian geometry. 
Throughout the work, we focus on the case of compact $\calX$, and assume that $\mathcal{X}$ is a bounded set inside $[0,1]^D$ without loss of generality.

\subsection{Reproducing kernel Hilbert Space}

Reproducing kernel Hilbert spaces (RKHS) are commonly used in studying the theoretical properties of GPs.
See e.g. \cite{aronszajn1950theory,berlinet2011reproducing} for a general introduction 
and \cite{ghosal2017fundamentals} in the context of nonparametric Bayes. 
In this work, we will rely on some RKHS lemmas in characterizing properties of the posterior. 
Here, we provide a brief overview of key definitions and concepts.

A symmetric function $k:\mathbb{X} \times \mathbb{X} \to \mathbb{R}$
is called a positive definite kernel on a non-empty set $\mathbb{X}$ if 
for all $n \in \mathbb{N}$, $x_1, ..., x_n \in \mathbb{X}$ and $\alpha_1, ..., \alpha_n \in \mathbb{R}$, the inequality $\sum_{i} \sum_{j} \alpha_i \alpha_j k(x_i,x_j) \ge 0$ holds. 
Each RKHS on $\mathbb{X}$, denoted by $\mathbb{H}$, corresponds to a unique reproducing kernel $k:\mathbb{X} \times \mathbb{X} \to \mathbb{R}$ satisfying the property 
\begin{equation*} 
   f(x) = \langle f,k(\cdot, x) \rangle_{\mathbb{H}}, \ \  \forall f \in  \mathbb{H}, \quad x \in \mathbb{X},
\end{equation*}
where $\langle \cdot, \cdot \rangle_{\mathbb{H}}$ is the innerproduct of  ${\mathbb{H}}$, and $k$ is positive definite.
Conversely, for any given positive definite kernel $k$, there exists a unique reproducing kernel Hilbert space $\mathbb{H}$ in which $k$ serves as the reproducing kernel. 
The reproducing kernel $k$ can also be represented using the feature map $\Phi : \mathbb{X} \to \mathbb{H}$ as $k(x,y) = \langle \Phi(x), \Phi(y) \rangle_{\mathbb{H}}$, 
and a canonical feature map is  $\Phi(x) := k(x, \cdot)$. 
On a set $\mathbb{X}$, a positive definite kernel $k$ uniquely defines an RKHS associated with $k$. 

Functions belonging to an RKHS can be well approximated by linear combinations of functions of the form 
$k(x_i,x)$.
More specifically, the set $\{f: \mathbb{X} \to \mathbb{R}|\hspace{1.5mm} f(x) = \sum_{i=1}^{m}a_ik(x_i,x), 
\ a_1,\ldots, a_m \in \mathbb{R}, \, x_1, ..., x_m \in \mathbb{X}, \, m \in \mathbb{N} \}$ is dense in $\mathbb{H}$.
When $\mathbb{X}$ is equipped with a measure $dx$,
$k$ and a function $g : \mathbb{X} \to \R$ satisfying integrability conditions,
we also have that 
$f(x) = \int_{\mathbb{X}} k( x, y) g(y)  dy $ is in $\mathbb{H}$
and 
$\| f \|_{\mathbb{H}}^2 = \int_\mathbb{X} \int_\mathbb{X} k(x,y) g(x) g(y) dx dy$.  
A formal statement of this property is given in Lemma \ref{lemma:computation-Hnorm} 
where we consider the RKHS on a subset of $\R^D$, to be introduced below.

\subsection{RKHS on a set and subsets}

In this work, we consider $\mathbb{X} = [0,1]^D$, and data samples lie on a subset $\calX \subset \mathbb{X}$.
We focus on the squared exponential kernel for $\epsilon > 0$ defined as
 \begin{align}\label{eq:def-kernel}
     h_{\epsilon}(x,x') 
     = h\bigg(\frac{\|x-x'\|^2}{\epsilon}\bigg)
     = \exp\bigg(-\frac{\|x-x'
     \|^2}{2\epsilon}\bigg),
 \end{align}
  where $h(r) = e^{-r/2}$,
  and 
  $\|\cdot\|$ is the Euclidean distance in the ambient space $\R^D$. 
  The kernel \eqref{eq:def-kernel} can be defined for all pairs of $x$ and $x'$ in $\R^D$.
  Our theory applies to a general class of $h$ satisfying technical conditions (Assumption \ref{assump:general-h}),
  of which the squared exponential kernel is a representative example.

For any subset $S \subset [0,1]^D$, 
by restricting to when $x, x' \in S$, 
the kernel $h_\epsilon$ induces an RKHS on $S$, which we denote as $\mathbb{H}_{\epsilon}(S)$. 
This allows us to consider $\mathbb{H}_{\epsilon}(\calX)$, where $\calX$ is the data domain.
We provide properties of $\mathbb{H}_{\epsilon}(\calX)$ 
and the connections between 
$\mathbb{H}_{\epsilon}([0,1]^D)$ and $\mathbb{H}_{\epsilon}(\calX)$ 
in Appendix \ref{ap:A}, which will be used in our analysis. 
This work mainly concerns Gaussian processes and RKHS on the data domain $\mathcal{X}$.

\subsection{Gaussian process on a general subset $\mathcal{X}$}

GPs are widely used as priors for unknown functions.
We consider $F^t_x$ as a centered GP indexed by $x \in \mathcal{X}$, where $t>0$ is a kernel bandwidth. 
$F^t_x$ is determined by the covariance function which is assumed to take the form $h_t$ introduced in
\eqref{eq:def-kernel}, that is, 
$\E [F^t_{x} F^t_{x'}]= h_t(x,x')$. Along with a prior $p(t)$ on the bandwidth $t$, the law of the GP provides a prior $\Pi$  for the unknown regression function $f$. Using $f^t$ to denote the value of $f$ for a specific bandwidth $t$, we have  
 \begin{align}\label{eq:prior}
        f^t | t 
        \sim {\rm GP}(0, h_{t}(x, x')),\quad t \sim p(t).
    \end{align}  
The prior $p(t)$ will be carefully constructed to obtain adaptive posterior concentration.
Our prior $\Pi$ and the density $p(t)$ of the prior on $t$
potentially all depend on $n$, and we omit the dependence in the notation.

Suppose data consist of $n$ observations $\{ X_i, Y_i\}_{i=1}^n$.
Let $\mathbf{f} \in \mathbb{R}^n$ denote the values of $f$ at the $X_i$'s, namely $\mathbf{f}_i = f(X_i)$.
A GP prior for $f$ implies that the conditional distribution of $\mathbf{f}$ given $X_1, X_2, \ldots, X_n$,
denoted as $p(\mathbf{f}|X_1, X_2, \ldots, X_n)$, follows a Gaussian distribution $\mathcal{N}(0, \Sigma_{\mathbf{f}\mathbf{f}})$. 
Here, $\Sigma_{\mathbf{f}\mathbf{f}} \in \mathbb{R}^{n \times n}$ represents the covariance matrix induced from the kernel $k$ of the GP,
that is, the $(i, j)$ element of $\Sigma_{\mathbf{f}\mathbf{f}}$ equals $\text{Cov}(f(X_i), f(X_j)) = k(X_i, X_j)$, $1 \leq i, j \leq n$.
By combining the prior distribution $\mathcal{N}(0, \Sigma_{\mathbf{f}\mathbf{f}})$ with the likelihood function in equation \eqref{eq:model}, we can obtain the posterior distribution given the observed data $\{ X_i, Y_i\}_{i=1}^n$,
denoted as $\Pi(\cdot | \{ X_i, Y_i\}_{i=1}^n)$.
This posterior distribution serves as the foundation for conducting inference and making predictions. 
Theoretically, we will analyze the posterior contraction rate as well as the convergence of the posterior mean estimator for $f^*$ defined as
$ \hat{f}(x) = \int f(x) d\Pi(f|\{X_i,Y_i\}_{i=1}^{n})$.

\subsection{Riemannian manifold and intrinsic derivatives}
\label{sec:rem_manifold}

We introduce some notations of differential geometry that are used in our analysis.
All the concepts of Riemannian geometry are standard and can be found in textbooks, e.g., \cite{do1992riemannian,petersen2006riemannian}.
Suppose $(\mathcal{M},g)$ is a $d$-dimensional connected smooth closed (compact and without boundary)  Riemannian manifold isometrically embedded in $ \mathbb{R}^D$ through $\iota: \mathcal{M} \rightarrow \mathbb{R}^D$. 
When there is no confusion, we also denote $\iota(\calM) \subset \R^D$ as the manifold $\calM$.
The metric tensor $g$ is central to the (intrinsic) geometry of $\calM$,
where   we say that a construction is {\it intrinsic} if it is fully determined by $g$
(and not by e.g. the embedding mapping $\iota$).
Otherwise, we say an object is {\it extrinsic}. 
For example, the geodesic distance $d_\calM(x,y)$, the Riemannian volume $dV$, 
the injectivity radius $\xi$,
 normal coordinates,
the Riemannian connection $\nabla$ and covariant derivatives 
are all intrinsic; 
In contrast,  the second fundamental form  $\Second$
and the manifold reach $\tau$
are both associated with  $\iota (\calM)$ and 
are extrinsic.
A detailed review of the notations of 
$d_\calM(x,y)$, $dV$, $\xi$,
the exponential map $\exp_x$, 
normal coordinates, 
geodesic curve,
and covariant derivatives can be found in Appendix \ref{app:diff-geo-lemmas}.
Below we elaborate more on the intrinsic derivatives of a differentiable function on $\calM$.

For $f \in C^k(\calM)$, there are different ways to consider the derivatives of $f$ on $\calM$. A common way is to parametrize $f$ on a geodesic-ball neighborhood of $x$ in normal coordinates, that is,
to consider the composed function $\tilde f := f\circ \exp_x$ as a $C^k$ function on  $B_\xi^{\R^d}(0)$
and then use the standard derivatives of $\tilde f$ in $\R^d$.
In this work, we heavily use the {\it covariant derivative} induced by the Riemannian (Levi-Civita) connection $\nabla$.
The $k$-th covariant derivative of $f$, denoted as $\nabla^k f$, is an order-$k$ tensor field on $\calM$.
See Appendix \ref{app:diff-geo-lemmas} for the formal definition and the concepts of vector/tensor fields on $\calM$.

The covariant derivative $\nabla^k f(x)$ is an order-$k$ tensor on $T_x \calM \times \cdots \times T_x\calM$,
\rev{and is closely related to the $\R^d$-derivative of $\tilde{f}(u)$ at $u=0$ in that 
$
\nabla^k f(x)(v,\ldots,v)
=
D^k\tilde f(0)(v,\ldots,v)$,
for any $ v\in T_x\mathcal M$.}
We use the notation
\[
\|\nabla^k f(x)\|_{op}
:=
\sup_{v\in S^{d-1}\subset T_x\mathcal M}
|\nabla^k f(x)(v,\ldots,v)|.
\]
This quantity is intrinsic and independent of the choice
of local coordinates or orthonormal basis at $x$.
The covariant derivative is also closely related to the parallel transport (induced by $\nabla$). 
In this work, we will use the parallel transport along the unique minimizing geodesic curve. 
Specifically, for $y \in B_{\xi}(x)$, the {\it parallel transport} along the unique minimizing geodesic from $x$ to $y$ is a mapping $P_{x, y}: T_x \mathcal{M} \rightarrow T_y \mathcal{M}$.
Recall that  the (unique) minimizing geodesic from $x$ to $y$ can be written as $\gamma(t) = \exp_x(tv)$, 
where $\gamma(0) = x$, $\gamma(d_\calM(x,y)) = y$,
and $\dot \gamma(0) = v$ is a unit vector in $T_x \calM$, 
We say that a vector field $U$ is parallel along $\gamma$ if $\nabla_{\dot \gamma} U = 0$ along $\gamma$,
and this is equivalent to that $U(\gamma(t)) = P_{x, \gamma(t)} U(0)$. 
In other words, for any $w \in T_x \calM$, 
define $W$ by $W(\gamma(t)) := P_{x, \gamma(t)} w$, then $W$ is parallel along $\gamma$ 
(the vector field $W$ is only defined on $\gamma$ but this suffices here).
In particular, $\dot \gamma(t)$ is parallel along $\gamma$, namely, $\nabla_{\dot \gamma} \dot \gamma = 0$ along $\gamma$.

We will use the above concepts of $\|\nabla^k f(x)\|_{op}$
and the parallel transport $P_{x, y}$, both of which are intrinsic, 
to define our H\"older class on $\calM$ in Section \ref{sec:rate_manifold}.
A similar construction of a manifold H\"older class was introduced in \cite{ki2021intrinsic}.
Before the end of this subsection, 
we further comment on the Euclidean derivatives $D^k \tilde f $ and compare with the covariant derivatives $\nabla^k f$.
Because the exponential map and normal coordinates are intrinsic, 
the derivatives $D^k \tilde f$ on $B_\xi^{\R^d}(0)$ are also intrinsically defined and are hence a type of intrinsic derivatives. However, when $u \neq 0$, $D^k \tilde f(u)$ generally does not equal $\nabla^k f( \exp_x(u))$. This means that quantities like 
$\sup_{ u \in B_\xi^{\R^d}(0)}  \| D^k \tilde f(u) \|_{op}$, which is defined on a geodesic-ball neighborhood of $x$, 
differ from (is larger than) our notion of $\| \nabla^k f(x) \|_{op}$. A manifold H\"older class was previously defined using the (Euclidean) H\"older norm of $\tilde f$, see e.g. \cite{rosa2024nonparametric}.
Such a definition utilizes an atlas cover of the manifold and considers the (Euclidean) H\"older norm of $\tilde f$ on each atlas.
The resulting H\"older norm thus depends on the chosen atlas cover
(though the norms from different atlas covers are equivalent so the the resulting H\"older class is the same). In contrast, our definition of H\"older norm does not involve any choice of atlas. 
While our H\"older norm differs from those defined from $\tilde f$, the H\"older norms should be equivalent, i.e. bounded by a constant from each other.
In this work, our notion of H\"older norm facilitates the parallel transport techniques used in the quantitative analysis of the H\"older norms.

\section{Posterior contraction rates:
general result}
\label{sec:pos_rate}

In this section, we prove the general result of posterior contraction rates adaptive to the intrinsic low dimensionality of the observed data. 
We will show that the posterior contraction rate 
is at least $ n^{{-s}/{(2s+\varrho)}} $ up to a logarithmic factor, where $s$ depends on the approximation property of the true function $f^*$, and $\varrho$ is the intrinsic low dimensionality of data.
The formal characterization will be detailed in Assumption \ref{assump:A1-A2-prime}. 
These are the most general conditions to show the adaptive posterior contraction rate in this paper, and we will focus on the manifold case in Section \ref{sec:rate_manifold}.

\subsection{Definitions and general assumptions}

We introduce our assumptions on the data distribution and true function $f^*$ 
in the regression model \eqref{eq:model}. We first define the posterior contraction rate.
Let  $\{X_i,Y_i\}_{i=1}^{n}$ denote the observed data, 
and recall that  $\Pi(A|\{X_i,Y_i\}_{i=1}^{n} )$ denotes the posterior of an event $A$ under the prior $\Pi$ as in \eqref{eq:prior}.
Let $d_n$ be a semi-metric measuring the discrepancy between $f$ and $f^*$. 
Following  \cite{ghosal2000convergence,ghosal2007convergence,van2009adaptive},
we say that the posterior contraction rate of the GP prior
with respect to $d_n$
is at least $\varepsilon_n$ if
\begin{align}
    \Pi( d_n(f, f^*) > \varepsilon_n| \{X_i,Y_i\}_{i=1}^{n} ) \to 0 \ \ \textit{in probability as} \ n \to \infty.  \nonumber
\end{align}    

We recall the notation of covering numbers.
Generally, suppose $(E, \| \cdot \|)$ is a normed space and $S \subset E$.
Given $\varepsilon>0$, $ N \subset E$ is called an $\varepsilon$-net of $S$
 if $\forall x \in S$, $\exists s \in N$, s.t. $\| x-s \| \le \varepsilon$. 
The covering number $\calN ( \varepsilon, S, \|\cdot\|)$ of $S$ (under norm $\|\cdot\|$)
is defined to be the minimum cardinality of an $\varepsilon$-net of $S$. 
Namely,
$\calN ( \varepsilon, S, \|\cdot\|) = \min\{ m \in \mathbb{N}: \exists s_1, \cdots, s_m \in E, \, {\rm s.t.} \,  
S \subset \cup_{i=1}^m \bar B_{\| \cdot \|}( s_i, \varepsilon )\} $,
where $\bar B_{\| \cdot \|}(s, \varepsilon):= \{ x \in E, \, \| x- s\| \le \varepsilon \}$ is the $\varepsilon$-ball centered at $s$ under norm $\| \cdot \|$.
We are ready to introduce the assumptions on the data domain $\calX$ and $f^*$.

\begin{assumption}
\label{assump:A1-A2-prime}
For positive constants $\varrho$ and $s$, 

(A1)  Intrinsic low-dimensionality of $\calX$:
 $\calX  \subset [0,1]^D$ and 
there exist positive constants $C_{\mathcal{X}}$
and $r_0 \in (0,1)$ that may depend on $\mathcal{X}$,
s.t.  
\begin{equation}\label{eq:cond-A1-prime}
\calN( r, \calX, \| \cdot \|_\infty) 
\le C_{\mathcal{X}} r^{-\varrho},
\quad \forall r \in (0, r_0].
  \end{equation}
  
(A2) Approximation of $f^*$ by RKHS:
There exist positive constants
 $\epsilon_0$,  
$ \nu_1$, $\nu_2$  
that may depend on $\mathcal{X}$ and $f^*$,
s.t. 
for all $ \epsilon < \epsilon_0$, there is a function $F^\epsilon \in {\mathbb{H}}_{\epsilon}( \calX )$ 
  such that 
  \begin{align}
        \sup_{x \in \calX }| F^\epsilon(x) - f^*(x)|  \le \nu_1  \epsilon^{s/2},
        \quad  
\|F^\epsilon \|^2_{\mathbb{H}_{\epsilon}( \calX )} \le \nu_2  \epsilon^{-\varrho/2}. \label{eq:cond-A2-prime}
    \end{align}
\end{assumption}

In (A1), $\|\cdot \|_\infty$ denotes the $\infty$-norm in $\R^D$. 
One can define the covering number of $\calX$ using another norm $d_D$ in $\R^D$, 
however, as long as the $d_D$-unit ball is contained in the $\|\cdot\|_\infty$-unit ball, 
we have $\calN(r, \calX, \|\cdot\|_\infty) \le \calN(r, \calX, d_D)$.
This includes the case when $d_D$ is the $p$-norm for any $p \ge 1$ in $\R^D$, and our assumption \eqref{eq:cond-A1-prime} is weaker.
The factor $\varrho$ in (A1) corresponds to the intrinsic dimensionality of data,
though technically it is an upper bound on dimensionality in the small $r$ limit.
The limit of ${\log  \calN( r, \mathcal{X}, \|\cdot \|_{\infty})}/{\log({1}/{r})}$ as $r \to 0+$ 
is the Minkowski or box-counting dimension  \cite{falconer2004fractal,hamm2021adaptive}.
If the sup-limit
exists and equals $\varrho_0$, 
$\varrho_0 =  \inf\{  \rho \ge  0, \, \limsup_{r \to 0+} 
\calN( r, \calX, \| \cdot \|_\infty)  r^{\rho} = 0\}$
is called the upper Minkowski dimension of $\calX$, 
then  (A1) holds with 
$\varrho = \varrho_0 + \varepsilon $ for any $\varepsilon  >0$.
The condition (A1) also holds with $\varrho = \varrho_0$ if $  \limsup_{r \to 0+} 
\calN( r, \calX, \| \cdot \|_\infty)  r^{\varrho_0} = 0$.
Our proved rate involves a dimension factor $\varrho$ that either equals to or can be arbitrarily close to the (upper)  Minkowski dimension of $\calX$.  
The factor $s$ in (A2) corresponds to the smoothness of $f^*$. 
We will show in Section \ref{sec:rate_manifold} that (A2) holds when $\calX$ is a smooth manifold and $f^*$ is a $C^s$ H\"older function on $\calX$. Generally, the condition (A2) requires certain regularity of $f^*$ on $\calX$.

Intuitively, the assumption (A1) asks the subset $\calX$ to occupy a small portion of the ambient space $\R^D$
such that the intrinsic complexity of $\calX$ is lower than $D$.
This holds when $\calX$ is restricted to a subspace, which is equivalent to some notion of sparsity. Subspaces are linear subsets of low dimensionality, and assumption (A1) also covers non-linear cases, such as the important case when $\calX$ is a sub-manifold or stratified space; see the examples below.

\begin{exmp}[Low-dimensional manifold]
\label{eg:manifold}
   Let $\mathcal{X} = \mathcal{M} \subset [0,1]^D$ be a $d$-dimensional compact 
   smooth Riemannian manifold 
   isometrically embedded in $\R^{D}$ with $d \in \mathbb{N}$, $d \le D$. 
   One can construct an $r$-covering of $\calM$ 
   which satisfies  
   $\calN( r, \calM, \|\cdot \|_\infty) \le C_{\calM} r ^{-d}$ for all small enough $r$,
   see e.g. \cite{hamm2021adaptive}. 
   (A1) also holds for a submanifold which has less regularity or with boundary, 
   e.g.
   the unit cube  $ [0,1]^d$ satisfies (A1) with $\varrho = d$.
   Similarly, it holds 
    when (the bounded set) $\calX$ is restricted to a low-dimensional subspace in $\R^D$.
\end{exmp}

Because our notion of low dimensionality in (A1) is Minkowski dimension-like and characterized by the covering number,
it is more general than the manifold assumption used in the previous manifold regression literature, such as \cite{bickel2007local, yang2016bayesian}. 

\begin{exmp}[Stratified space]\label{eg:stratified}
   A stratified space \cite{weinberger1994topological} generalizes the concept of a manifold by allowing for more complicated geometric structures, 
   in particular, consisting of different ``strata'' or ``layers,'' each having its own well-behaved geometric properties. The covering dimension of a stratified space measures its topological complexity. It is the maximum dimension among the strata. A stratified space has a finite covering dimension if and only if all of its strata have finite dimensions. A simple example is $\mathcal{X} = \mathcal{M}_1 \bigcup \mathcal{M}_2$
   where $\mathcal{M}_1, \mathcal{M}_2 \subset [0,1]^D$ are two compact connected smooth Riemannian manifolds having dimensions $d_1, d_2 \in \mathbb{Z}^{+}$, respectively. One can verify that in this case $\mathcal{X}$ satisfies assumption (A1) with $\varrho = \max\{ d_1, d_2 \}$, see Lemma \ref{Lemma: dimension of the union}.
\end{exmp}

\subsection{Dimension-adaptive prior of kernel bandwidth}

The prior on the kernel bandwidth $\epsilon>0$ in \eqref{eq:def-kernel} has a key impact on the posterior contraction rate. We denote the prior as $p(t)$, for $t = \epsilon > 0$. For the posterior contraction rate to be adaptive to the intrinsic dimensionality $\varrho$, the 
prior $p(t)$ needs to be carefully constructed.

We introduce a dimension-adaptive prior condition in Assumption \ref{assump:A3}, which allows us to bypass the need to condition on $\varrho$ in the prior for the bandwidth; such conditioning is common practice in the literature with the rescaled Gamma prior providing a notable case, see Example \ref{exmp:gm}. 
In Section \ref{sec:knn_prior} we propose an empirical Bayes prior that satisfies
Assumption \ref{assump:A3} when data lie on a submanifold.
The prior $p(t)$ can depend on sample size $n$ and data $\{ X_i \}_{i=1}^n$,
and we omit such dependence for the brevity of notation.

\begin{assumption}[Condition on the prior of bandwidth $\epsilon$]
\label{assump:A3}
(A3)
Given positive constants $\varrho$ and $s$, 
there exist 
 $c_2 > c_1 > 0$
 and
 $ c_3, a_1, a_2, K_1, K_2, C_1, C_2 > 0$, such that 
\begin{align}
p(t) & \ge C_1t^{-a_1} \exp\Big(-\frac{K_1}{t^{
    \varrho/2}}\Big),
    \quad  \forall t \text{ s.t. } 
    c_1  \le t/ \left( 
    n^{-{2}/{(2s+\varrho)}} (\log n)^{\frac{2(1+D)}{2s +\varrho}} \right)
    \le c_2, \label{eq:prior-lower} \\
p(t) & \le C_2t^{-a_2} \exp\Big(-\frac{K_2}{t^{\varrho/2}}\Big),
      \quad \forall t \text{ s.t. } 
      0 <
      t / \left( n^{-{2}/{(2s+\varrho)}} (\log n )^{\frac{-4(1+D)}{(2+\varrho/s)\varrho}} \right)
      \le c_3. \label{eq:prior-upper}
\end{align}
\end{assumption}

The inequality \eqref{eq:prior-lower} ensures that $p(t)$ is sufficiently large for $t \sim n^{-{2}/{(2s+\varrho)}}$,
and \eqref{eq:prior-upper} ensures that $p(t)$ is close to zero when $t$ is smaller than the order of $n^{-{2}/{(2s+\varrho)}}$. To provide a prior $p(t)$ that satisfies (A3),
one method is to use a rescaled Gamma prior when the intrinsic dimensionality $\varrho$ is known, as shown in the next example. 

\begin{exmp}[Rescaled Gamma prior \cite{van2009adaptive}]
\label{exmp:gm} 
Let ${\rm Ga}(a_0, b_0)$ denote the gamma distribution with 
probability density function proportional to $ t^{a_0-1}e^{-b_{0}t}$, where $a_0, b_0 > 0$ are two constants. 
The rescaled Gamma prior of $\epsilon = t$ is such that 
$t^{-\varrho/2}$ follows the distribution of ${\rm Ga}(a_0, b_0)$.
As has been shown in \cite{van2009adaptive}, this prior $p(t)$ satisfies both of the inequalities in (A3) for all $t > 0$, 
and as a result, this prior satisfies (A3) for any $s \in (0,\infty)$.
\end{exmp}

For Bayesian manifold regression, the rescaled Gamma prior was adopted in  \cite{yang2016bayesian}; they estimate the manifold dimension when not known. While methods are available for manifold dimension estimation \cite{levina2004maximum,farahmand2007manifold}, 
they may encounter difficulties in practice, especially when the sample size is small.
We discuss the manifold data case in more detail in Section \ref{sec:rate_manifold}.
In Section \ref{sec:knn_prior}, we will propose an empirical Bayes prior computed using kernel affinities and $k$-nearest neighbor techniques 
and theoretically show that our prior can satisfy (A3) with high probability (Proposition \ref{prop:v_hat}).
As a result, our prior can achieve the adaptive contraction rate without knowledge of the intrinsic manifold dimension. 
In practice, our prior can perform more stably than the rescaled Gamma prior with estimated manifold dimension.

\subsection{The general result of adaptive rates}\label{subsec:adapt-rates-general}

We are ready to prove the dimension-adaptive posterior contraction rate.
The proofs are provided in Appendix \ref{sec:proof-sec3}.

We consider two scenarios, the {\it fixed design} where the predictors $\{ X_i \}_{i=1}^n$ are given and fixed,
and the {\it random design}
where the marginal distribution of $X_i$ is $P_X$ on $\calX$.
Under the fixed design, the residual $w_i | X \sim \calN( 0, \sigma^2)$ are independent across $i$, and as a result the variables $Y_i | X$ are also independent (but not i.i.d.).
To measure the discrepancy between $f$ and the ground truth $f^*$,  
 the in-sample mean squared error is defined as
 $\| {f} - f^*  \|^2_{n}:= \frac{1}{n}\sum_{i=1}^{n}({f}(X_i) - f^*(X_i) )^2$,
 which is well-defined under the fixed design.
Under the random design, we also  define the population error as
$\| {f} - f^* \|^2_{2}:= \int_{\mathcal{X}} ( {f}(x) - f^*(x))^2 P_X(dx)$.

\begin{theorem}[Fixed design posterior contraction rate]
\label{thm:contraction_n_norm}
   Suppose Assumptions \ref{assump:A1-A2-prime}-\ref{assump:A3} are satisfied with the same positive factors $\varrho$ and $s$.
   Then, there exists a positive constant $C$ s.t.
the posterior contraction rate 
with respect to $\|\cdot \|_n$ 
   is at least 
$\bar \varepsilon_n =  C  n^{-\frac{s}{2s+\varrho}} (\log n )^{\frac{ D+1}{2 + \varrho/s} + \frac{D+1}{2}}
 \lesssim n^{- {s}/{ (2s+\varrho) }}(\log n)^{D+1}$.
\end{theorem}

Recall that the posterior mean of $f$ is defined as  $\hat{f} (x) = \int f^t (x) d\Pi(f^t |\{X_i,Y_i\}_{i=1}^{n})$.
The next theorem shows that the estimator $\hat f$ achieves the same adaptive rate of convergence under the fixed design.

\begin{theorem}[Fixed design estimator convergence rate]
\label{thm:fix-design-estimator}
   Under the assumption of Theorem \ref{thm:contraction_n_norm} with $\bar \varepsilon_n$ as therein,
 suppose $f^*$ is bounded on $\calX$.
  Then, 
   with probability tending to one,
   $  \| \hat f - f^* \|_n  \le 3 \bar \varepsilon_n$.
\end{theorem}

To extend the theory to random design, we will adopt a truncation of $f$ as originally proposed in \cite{yang2016bayesian}. 
We will assume a known upper bound of $f^*$, that is, $\| f^*\|_\infty \le M$, and this $M$ can be any upper bound of $\| f^*\|_\infty$.
For a function $f$ and $A > 0$, define $f_A := (f \vee (-A)) \wedge A$.
We will consider the posterior of $f_M$, and the corresponding posterior mean  estimator is defined as $\hat{f}_M(x) := \int f_M(x) d\Pi(f|\{X_i,Y_i\}_{i=1}^{n})$. 
The theoretical necessity of the truncation lies in that we will use empirical process techniques to bound $\| f-f^* \|_2$ by comparing to $\| f-f^* \|_n$, which would require the function class to be bounded to begin with.

With the truncation, the next theorem proves the adaptive posterior contraction rate 
and the posterior-mean estimator convergence rate under the random design. 

\begin{theorem}[Random design]
\label{thm:contraction_2_norm}
   Suppose Assumptions \ref{assump:A1-A2-prime}-\ref{assump:A3} are satisfied with the same positive factors $\varrho$ and $s$,
and  $X_i$ are i.i.d. samples drawn from some distribution $P_X$ on $\calX$.
   In addition, for some positive constant $M$, $\|f^*\|_{\infty} \le M$. 
   Let  $\bar \varepsilon_n$ be as in Theorem \ref{thm:contraction_n_norm}.
    Then, 
    there exists an absolute constant $c$ s.t. 
\[
    \Pi( \| f_M - f^*\|_2 >   c(M+1) \bar \varepsilon_n  
    | \{X_i,Y_i\}_{i=1}^{n} ) 
    \to 0 \ \ \textit{in probability as} \ n \to \infty.  
\]
Moreover, with probability tending to one, 
    $||\hat{f}_M-f^*||_{2} \le   c(M+1) \bar \varepsilon_n 
    $. 
\end{theorem} 

In Theorems \ref{thm:contraction_n_norm}-\ref{thm:contraction_2_norm}, the rate $n^{-{s}/{(2s+\varrho)}}$ only depends on $s$ and $\varrho$,
and the ambient dimensionality $D$ appears in the logarithmic terms. 
The constant $C$ depends on $\varrho$, $s$, $\calX$, $D$, the kernel function $h$,
and also the constants in (A1)(A2)(A3), and the dependence can be tracked in the proofs.
In Appendix \ref{sec:sub_rate}, we also extend our theory through 
a relaxation of Assumption \ref{assump:A3}(A3)
that replaces $\varrho$ with an upper bound $\varrho_+$ in \eqref{eq:prior-lower}
and a lower bound $\varrho_-$ in \eqref{eq:prior-upper} respectively,
see Assumption \ref{assump:A3-mis}. 
This can be intuitively understood as a ``misspecification'' of $\varrho$ between $\varrho_\pm$. 
When $\varrho_\pm$ does not equal to $\varrho$, the contraction rate will degenerate into $\varepsilon_n \sim n^{-r(\varrho, \varrho_\pm,s)}$
(up to a logarithmic factor) where $r(\varrho, \varrho_\pm,s)$ is worse than ${s}/{(2s+\varrho)}$,
and we recover the ${s}/{(2s+\varrho)}$ rate when $\varrho_\pm = \varrho$,
see Theorem \ref{thm:contraction-mis}. 
This theoretical extension allows to establish adaptive posterior contraction rates when using a prior $p(t)$ not exactly but close to satisfying (A3), e.g., the rescaled Gamma prior with a misspecified $\varrho$.

\section{Adaptive rates for data on manifolds}
\label{sec:rate_manifold}

In this section, we focus on the special case where data samples lie on a smooth closed Riemannian manifold of intrinsic dimensionality $d$. 
We will obtain in Section \ref{sec:pos_holder} the optimal posterior contraction rate
$O(n^{-s/(2s+\varrho)})$, up to a logarithmic factor,
when $f^*$ is a $C^{k,\beta}$-H{\"o}lder function on the manifold,
$\varrho = d$ and the smoothness order $s = k+\beta >0$ can be any positive number. 
This relies on a key manifold RKHS approximation result proved in Section \ref{sec:rkhs_appr}.

In Section \ref{sec:knn_prior}, we propose a new empirical Bayes prior based on kernel affinity and $k$NN statistics, 
which enables us to achieve the optimal rate without knowledge of the manifold intrinsic dimension $d$.  
The proofs are postponed to Appendix \ref{sec:proof-lemma-4.1} and Appendix \ref{app:proof-sec4}.

\subsection{RKHS approximation of H{\"o}lder functions on manifold}
\label{sec:rkhs_appr}

We provide a manifold RKHS approximation result to ensure Assumption \ref{assump:A1-A2-prime}(A2). 
We show RKHS approximation of a class of manifold H{\"o}lder functions that are intrinsically defined; this result can be of independent interest. 

\begin{assumption}[Data manifold]\label{assump:A4} 
The data domain $\calX = \calM $ is a $d$-dimensional smooth
connected
closed Riemannian manifold isometrically embedded in $[0,1]^D \subset \R^D$.
\end{assumption}

As will be shown in the proof of Lemma \ref{lemma:22_holder}, it suffices to have
$C^{  \max\{  2k ,3 \} }$
regularity of $\calM$ instead of $C^\infty$
when approximating a target function $f \in C^{k,\beta}(\calM) $. We say that a manifold $\calM$ is $C^r$, $r$ being a positive integer, when both the Riemannian metric $g$ and the embedding map $\iota$ are at least $C^r$. 
In Assumption \ref{assump:A4}, the connectedness of $\calM$ can be removed, which is equivalent to when $\calX$ is a disjoint union of connected manifolds.
Our theory can extend to such cases where each manifold can have distinct dimensionalities, see Appendix \ref{app:extension-theory-manifold-mixed-d}.
In this section, we assume smoothness and connectedness of $\calM$ for simplicity.

Recall the differential geometry notations in Section \ref{sec:rem_manifold}, and in particular, 
the definition of covariant derivatives (with respect to the Riemannian connection $\nabla$) and the parallel transport $P_{x,y}$ along unique minimizing geodesics within the geodesic ball $B_\xi(x)$,
$\xi > 0$ being the injectivity radius of $\calM $.
To simplify notation, when the $k$-th covariant derivative at $x$ is applied to the same vector $v \in T_x \calM$ for $k$ many times, we introduce the notation 
$
\nabla^k_v f(x) : = \nabla^k f(x) (v, \cdots, v).
$
Given a non-negative integer $k $, suppose $f \in C^k(\mathcal{M})$, we define 
$
\| \nabla^k f(x) \|_{op}: = 
 \sup_{v \in S^{d-1} \subset T_x  \mathcal{M}} |\nabla^k_v f(x)|$,
and 
$ \|\nabla^k f\|_{\infty} : = 
\sup_{x \in \mathcal{M}} \| \nabla^k f(x)\|_{op}$.
For $0 <\beta \leq 1$, we define 
$$
L_{k, \beta}(f, x):=  
\sup_{y \in B_{\xi}(x)} 
\sup_{v \in S^{d-1} \subset T_x \mathcal{M} } {|\nabla^k_v f(x) -\nabla^k_{P_{x , y}v} f(y)|}/{d^\beta_{\mathcal{M}}(x,y)},
$$
and we further define $L_{k, \beta}(f) :=\sup_{x \in \mathcal{M}}L_{k, \beta}(f, x)$. 

\begin{definition}[H\"older class on manifold]\label{def:Holder-manifold}
For $k = 0,1,\cdots$,  
$ 0 <\beta \le 1$,
and $f \in C^k(\calM)$,
the H\"older norm of $f$ is defined as
$
\|f\|_{k, \beta} :=\sum_{l=0}^{k} \|\nabla^l f\|_{\infty} + L_{k, \beta}(f).
$
We say $f \in C^{k, \beta}(\mathcal{M})$  whenever $\|f\|_{k, \beta} <\infty$. 
\end{definition}

Recall the kernel function defined in \eqref{eq:def-kernel}, where $h( r )=e^{- r/2}$;  $\calM $ is embedded in $\R^D$ through $\iota:\calM \rightarrow [0,1]^D$.
For $\epsilon>0$ and $f \in L^1(\mathcal{M})$, we define the on-manifold integral operator $G_\epsilon:  L^1(\mathcal{M}) \rightarrow  L^1(\mathcal{M})$ as
\begin{align}\label{eq:def-Geps-integral-operator}
G_\epsilon (f)(x)
: = \frac{1}{(2\pi \epsilon)^{d/2}} \int_{\mathcal{M}} h\Big(\frac{\|\iota(x)-\iota(y)\|^2_{\mathbb{R}^D}}{\epsilon}\Big)f(y)dV(y).
\end{align}

\begin{lemma}
\label{lemma:22_holder}
Under Assumption \ref{assump:A4},
given non-negative integer $k$ and $ 0 < \beta \le 1$,
there exists a constant $\epsilon_1( \calM, d, k )$ such that when $\epsilon < \epsilon_1$,
for any $f \in  C^{k,\beta}(\mathcal{M})$,
there exist $f_j \in C^{k-2j, \beta}(\calM)$, $j=1,\cdots, \lfloor k/2 \rfloor $,
and $R_{f ,\epsilon} \in C(\calM)$ s.t.
\begin{align}\label{G epsilon expansion with remainder}
G_\epsilon (f)(x)
= f(x)
   +\sum_{j=1}^{\lfloor k/2 \rfloor} \epsilon^j f_j(x)  
   + R_{f,\epsilon}(x), 
\end{align}
 
\begin{enumerate}
\item[(i)] 
The remainder term $R_{f ,\epsilon} $ satisfies $\|R_{f ,\epsilon}\|_{\infty} \le \tilde{C}_1(\mathcal{M},d, k) \|f\|_{k, \beta} \epsilon^{(k+\beta)/2}$;

\item[(ii)]
For all $0 \leq j \leq \lfloor k/2 \rfloor$, $\|f_j\|_{k-2j,\beta} \leq \tilde{C}_2(\mathcal{M}, d, k) \|f\|_{k, \beta}$ (when $j=0$, we set $f_0 =f$);
\end{enumerate}

\noindent
where the constants $\tilde{C}_1(\mathcal{M}, d, k)$ and $\tilde{C}_2(\mathcal{M}, d, k)$ are determined by $d$,  $k$,  and the manifold geometry.
Specifically, both constants depend on 
the bounds of
 the curvature tensor of $\calM$,
the second fundamental form of $\iota(\mathcal{M})$ 
and their covariant derivatives;
$\tilde{C}_1(\mathcal{M}, d, k)$ also depends on the volume of $\calM$
and bounds of intrinsic derivatives of the Riemann metric tensor;
$\tilde{C}_2(\mathcal{M}, d, k)$ also depends on the diameter of $\mathcal{M}$.
The small $\epsilon$ 
threshold  $\epsilon_1$
 depends on $d$, $k$, 
 the injectivity
radius of $\mathcal{M}$ and the reach of $\iota(\calM)$.
\end{lemma}

Our analysis characterizes the specific dependence of the constants on the manifold geometry, including both intrinsic and extrinsic quantities. 
Lemma \ref{lemma:22_holder} can be proved for a broader class of $h$
that satisfies differentiability and subexponential decay on $[0,\infty)$,
namely Assumption \ref{assump:general-h}(i), 
see Remark \ref{rk:diff-decay-h-manifold-integral}. 
This will suffice for the RKHS approximation results in Proposition \ref{prop:appr_s_l_inf} which are based on Lemma \ref{lemma:22_holder}.
For the RKHS covering estimates to hold, one will need additional conditions (sub-exponential decay and radial monotonicity) on the spectral density of the kernel.
The technical conditions on the kernel function are summarized in Assumption \ref{assump:general-h}.
The square exponential kernel corresponding to $h(r)=e^{-r/2}$ is a typical example satisfying these assumptions
and also widely used in practice.

The next proposition constructs an RKHS approximation of a H\"older function $f$ on $\calM$.

\begin{proposition}
\label{prop:appr_s_l_inf}
Under Assumption \ref{assump:A4}, 
given non-negative integer $k$ and $ 0 < \beta \le 1$,
there exists a constant $\epsilon_2(\calM, d, k)$
such that when $\epsilon < \epsilon_2$,
for any $f \in  C^{k ,\beta}(\mathcal{M})$, 
we can find $F = \sum_{i=0}^{\lfloor k/2 \rfloor} \epsilon^i 
 F_i  $ with $F_i  \in C^{k - 2i ,\beta}(\mathcal{M})$, and
\begin{equation}   
\|G_\epsilon (F)- f\|_{\infty} \le \gamma_1(\mathcal{M}, d, k) \|f\|_{k, \beta}  \epsilon^{(k+\beta)/2}, \label{eq:cond-A2_holder_1}
\end{equation}
\begin{equation}
\|G_\epsilon (F) \|^2_{{\mathbb{H}}_{\epsilon}(\calM) } \le \gamma_2(\mathcal{M}, d, k) \|f\|^2_{k, \beta} \epsilon^{-d/2},
\label{eq:cond-A2_holder_2}
\end{equation}
where both constants  $\gamma_1(\mathcal{M}, d, k)$ and $\gamma_2(\mathcal{M}, d, k)$ 
depend on $d$, $k$,  and the manifold geometry, inheriting the $\calM$-dependence from the constants 
$\tilde{C}_1(\mathcal{M}, d, k)$ and $\tilde{C}_2(\mathcal{M}, d, k)$ in Lemma \ref{lemma:22_holder}.
The threshold  $\epsilon_2$
depends on $d$, $k$
and inherits the $\calM$-dependence from the thresholds $\epsilon_1(\calM, d, k)$ 
in Lemma \ref{lemma:22_holder}.
\end{proposition}

The proof of Proposition \ref{prop:appr_s_l_inf} is based on the expansion \eqref{G epsilon expansion with remainder} provided by Lemma \ref{lemma:22_holder},
and it uses a high order correction scheme.
Proposition \ref{prop:appr_s_l_inf} provides the RKHS approximation and Hilbert norm control needed in Assumption \ref{assump:A1-A2-prime}(A2) when data domain $\calX$ is a manifold and $f^*$ belongs to the manifold H\"older class.
Specifically, 
the equations \eqref{eq:cond-A2_holder_1} and \eqref{eq:cond-A2_holder_2} correspond to the condition \eqref{eq:cond-A2-prime} in (A2),
where $s = k+\beta > 0$ and $\varrho = d$. 
We will use Proposition \ref{prop:appr_s_l_inf} in Section \ref{sec:pos_holder} in obtaining the adaptive posterior contraction rate.

Our RKHS approximation result is stronger than needed by the posterior contraction analysis. 
This firstly lies in the requirement on the smallness of $\epsilon$.
In (A2), the needed small bandwidth threshold $\epsilon_0$ is allowed to depend on $\calX = \calM$ and $f^*$,
while in Proposition \ref{prop:appr_s_l_inf} the threshold $\epsilon_2(\calM)$ only depends on $\calM$. 
In other words, we have shown that once $\epsilon$ is less than a threshold that only depends on $\calM$, 
the approximation bounds in Proposition \ref{prop:appr_s_l_inf} hold uniformly for all $f$ in the H\"older class instead of a specific target function to approximate. 
In addition, while it suffices to show the existence of constants $\nu_1$, $\nu_2$ in (A2), 
in Proposition \ref{prop:appr_s_l_inf},
we separate the constants dependence on $f$ and $\calM$ in  the two bounds \eqref{eq:cond-A2_holder_1} and \eqref{eq:cond-A2_holder_2}.
In each bound,
the constant has a factor proportional to the H\"older norm of $f$ multiplied by a factor that only depends on manifold geometric quantities.  

\subsection{Posterior contraction for H\"older functions}
\label{sec:pos_holder}

Combining Proposition \ref{prop:appr_s_l_inf} with  Theorems \ref{thm:contraction_n_norm} and \ref{thm:contraction_2_norm}
allows us to prove the optimal contraction rate for any $f^* \in C^{k,\beta}(\mathcal{M})$.

\begin{assumption}[H\"older regression function]\label{assump:A5}
 The true function $f^* \in C^{k, \beta}(\calM)$ for some $k = 0, 1, \cdots$ and $0< \beta \le 1$.
\end{assumption}

\begin{corollary}
\label{cor:contraction_n_norm_manifold}
     Under Assumptions 
     \ref{assump:A4}-\ref{assump:A5},    
     and suppose that the prior on the kernel bandwidth $\epsilon$ satisfies
     Assumption \ref{assump:A3} with $\varrho = d$ and  $s = k+\beta$.

    (i) Fixed design.
    There exists a positive constant $C$ s.t.
   the posterior contraction rate 
   for $\| f - f^*\|_n$
   is at least 
        $\bar \varepsilon_n =  C  n^{-\frac{s}{2s+ d}} (\log n )^{\frac{ D+1}{2 + d/s} + \frac{D+1}{2}}
 \lesssim n^{- {s}/{ (2s+ d) }}(\log n)^{D+1}$.
 Furthermore, if $f^*$ is bounded on $\calX$, 
 then with probability tending to one, 
 $\| \hat f - f^* \|_n \le  3 \bar \varepsilon_n$.

(ii) Random design.
 Suppose $X_i$ are i.i.d. samples drawn from some distribution $P_X$ on $\calM $, 
 and for some positive constant $M$, $\|f^*\|_{\infty} \le M$.
  Then, there exists an absolute constant $c$ s.t. 
 the posterior contraction rate 
 for $\| f_M - f^*\|_2$
   is at least $c(M+1) \bar \varepsilon_n$, 
   and with probability tending to one, 
    $||\hat{f}_M-f^*||_{2} \le c(M+1) \bar \varepsilon_n$.
\end{corollary}

\begin{remark}[Extension to stratified space]
    We focus on posterior contraction rates for H\"older functions on a single manifold, 
    while our analysis can be extended to stratified spaces (Example \ref{eg:stratified}).
    In Appendix \ref{app:extension-theory-manifold-mixed-d},
we extend the theory to when $\calX$ is a finite union of disjoint manifolds (of possibly different dimensions) 
and there is a constant separation between the strata.
The extension is by firstly extending Lemma \ref{lemma:22_holder} with a modified definition of $G_\epsilon(f)$,
and then Proposition \ref{prop:appr_s_l_inf} and Corollary \ref{cor:contraction_n_norm_manifold} follow with essentially the same proofs.
In the convergence rates, the intrinsic dimension $\varrho$ equals the maximum dimension of the manifolds.
  Extension to more complex stratified spaces is left for future work.
\end{remark}

\subsection{Empirical Bayes prior on kernel bandwidth}
\label{sec:knn_prior}

In this subsection, 
we propose a bandwidth prior $p(t)$ that satisfies Assumption \ref{assump:A3}(A3) and does not 
require knowledge of the intrinsic dimension $d$ of the manifold. 
In practice, our prior can give more stable performance than the previous approach based on estimating the manifold dimension \cite{yang2016bayesian}, see the experimental comparison in Appendix \ref{app:subsec-additional-EB}.

Our empirical Bayes prior on the bandwidth $t$ takes the form 
\begin{equation}\label{eq:EBprior}
p(t) \propto t^{-a_0}\exp\Big( - \frac{b_0}{\hat{v}_n(t)}\Big) \ \textit{when } \ \gamma_1 T_n^2 < t \le  1  ;
\quad 
p(t) = 0 \ \textit{otherwise},
\end{equation}
where 
    $a_0,b_0 >0$ are arbitrary hyperparameters, 
    $\gamma_1$  is a positive constant (in our experiments set to be $1/4$),
    $T_n$ is computed from  averaged $k$-nearest neighbor ($k$NN) distances,
    and $\hat v_n(t)$ is the averaged kernel affinity to be specified below.
    
    Specifically, 
    for some subset $S$ of $[n]:= \{1, \cdots, n \}$ and $k \le n$,
    we let $k = \lceil \gamma_2 \log^2(n) \rceil$
    where $\gamma_2$ is a positive constant (set to be 1/4 in our experiments), 
    and define
    \begin{equation}\label{eq:def-Tn-knn}
    T_n : = \frac{1}{|S|}\sum_{i \in S}  \hat{R}_k (X_i),
    \quad 
    \hat{R}_k (x) := \inf_{r} \Big\{ r > 0,\, \text{ s.t. } \sum_{j=1}^{n} {\bf 1 }_{\{ \| X_j - x\| < r \}} \ge k \Big\},
    \end{equation}
    namely $\hat{R}_k(X_i)$ is the distance from $X_i$ to its $k$-th nearest neighbor in the $n$ samples $\{X_j\}_{j=1}^n$ ($X_i$ is its own 1st nearest neighbor).
    Our theory permits $S$ to be an arbitrary subset, and we explain the choice in practice at the end of this subsection.
    In our experiments, when $n$ is small (less than 200) we let $k=2$.

    The quantity $\hat v_n(t)$ is an averaged kernel affinity defined as
    \begin{equation}\label{eq:def_vn}
    \hat{v}_n(t) 
    : = \frac{1}{n(n-1)} \sum_{i=1}^n  \sum_{j \ne i} h_{t}(X_i, X_j),
    \end{equation}
    where $h_t(x,x')$ is defined in \eqref{eq:def-kernel}.
    $\hat v_n(t)$ is a function of the bandwidth $t > 0$.
We omit the normalizing constant 
$\hat Z_n: =\int_{ \gamma_1 T_n^2}^1  t^{-a_0}\exp\Big( - \frac{b_0}{\hat{v}_n(t)}\Big) dt $ from \eqref{eq:EBprior}, since it is not needed in computational implementations based on Markov Chain Monte Carlo (MCMC) sampling.
$\hat Z_n$ will be analyzed in our theoretical analysis.

The proposed prior $p(t)$ does not require knowledge of either the intrinsic dimension 
$\varrho=d$ or the regularity order $s$. The empirical statistics $T_n$ and $\hat v_n(t)$
used in constructing $p(t)$ contain information on $d$ implicitly. 
Our analysis will show that under our setting, 
$\hat{v}_n(t) \sim t^{d/2}$ when $t > C n^{-2/d}$ (up to a logarithmic factor),
and $T^2_n \sim n^{-2/d}$ (up to a logarithmic factor).
The validity of the proposed prior $p(t)$ will be proved in Proposition \ref{prop:v_hat}, and we need certain assumptions on the data density.
We denote by $p_X$  the probability density function with respect to the base measure $dV$ on $\calM$, where $dV$ is the intrinsic Riemannian volume. 

\begin{assumption}[Boundedness of data density]
\label{assump:A6}
$p_X$ is uniformly bounded from below and above, that is, for positive constants $p_{\min}, p_{\max}$, 
    $0 < p_{\min} \le p_X(x) \le p_{\max}, 
     \forall x \in \mathcal{M}$.
\end{assumption}

\begin{proposition}[Validity of the empirical prior]
\label{prop:v_hat}
Under Assumption \ref{assump:A4}, suppose $X_i$ are i.i.d. samples drawn from density $p_X$ on $\calM $ where  $p_X \in C^2(\mathcal{M})$ and satisfies Assumption \ref{assump:A6}.
Let $\varrho = d$, 
given any $s > 0$ and $a_0 $, $b_0$, $\gamma_1$, $\gamma_2$  some fixed positive constants,
there exists $n_{0}(\mathcal{M},p_X,s)$ such that  when $n > n_0$, 
let $p(t)$ be as in \eqref{eq:EBprior}\eqref{eq:def-Tn-knn}\eqref{eq:def_vn}
with $S$ being an arbitrary subset of $[n]$ in the definition of $T_n$
and $k = \lceil \gamma_2 \log^2(n) \rceil$,
then, with probability $\ge 1-2n^{-10}$,
the prior $p(t)$ satisfies Assumption \ref{assump:A3} with $\varrho $ and  $s$,
where the constants $c_1$, $K_1$, etc. can be properly specified. 
\end{proposition}

The proposition shows that the proposed $p(t)$ satisfies the needed condition in Corollary \ref{cor:contraction_n_norm_manifold}. 
 As a result, this prior leads to the optimal posterior contraction rate 
 that is automatically adaptive to $s$ and $d$.  
 When $\calX$ is a disjoint union of manifolds, we can extend Proposition \ref{prop:v_hat} after a modification of the definition of $\hat v_n(t)$, 
see Appendix \ref{app:extension-theory-manifold-mixed-d}.

The prior $p(t)$ can be computed without incurring more expensive computation than other steps in the Bayesian regression. 
Specifically, the kernel affinity statistic $\hat v_n(t)$ sums the (off-diagonal) entries of the kernel matrix, which is of lower computational complexity than 
constructing the kernel and computing the posterior mean of $f$. 
The $k$NN statistic $\hat R_k(X_i)$ can be computed by standard subroutines and the computational complexity is less than other kernel operations,
and we  compute 
the $k$NN distance for $|S|$ points as in \eqref{eq:def-Tn-knn}. 
While our theory allows $|S|$ to be arbitrarily small,
in practice, a larger $|S|$ can potentially help reduce variance and improve the algorithm's stability at finite sample sizes. 
The primary limitation of using a larger subset $S$ is increased computational cost.
More algorithmic details of the Bayesian regression are provided in Appendix \ref{ap:algo}.

\section{Numerical experiments}
\label{sec:num}

We numerically implement Bayesian inference on various datasets and compare the proposed method, namely the empirical Bayes prior on the kernel bandwidth in Section \ref{sec:knn_prior}, with other Bayesian and non-Bayesian baselines. 
Code implementation is available at {\small \url{https://github.com/taotangtt/gp-manifold-regression}}.

\subsection{Algorithm summary}

For Bayesian inference, we marginalize out the unknown regression function $f$ using conjugacy of the GP prior. 
This produces a posterior for the bandwidth parameter $t$, which we sample from using Metropolis-Hastings. 
With a GP prior, for a given bandwidth $t$, the conditional posterior of $f$  is available analytically.
We average over the samples from the marginal posterior of $t$ in estimating the posterior mean $\hat{f}$.
Details of the algorithm are given in Appendix \ref{ap:algo}.

\subsection{Swiss Roll data}\label{sec:2d_swiss}

\begin{figure}
\hspace{-15pt}
   \begin{minipage}{0.32\textwidth}
 \includegraphics[height=0.96\linewidth]{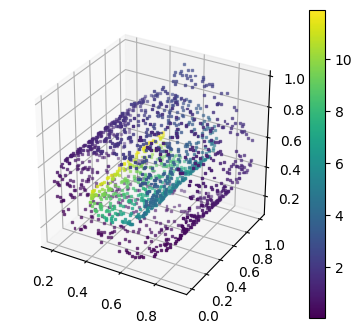}
       \subcaption{}
    \end{minipage}
    \begin{minipage}{0.35\textwidth}
        \includegraphics[height=0.95\linewidth]{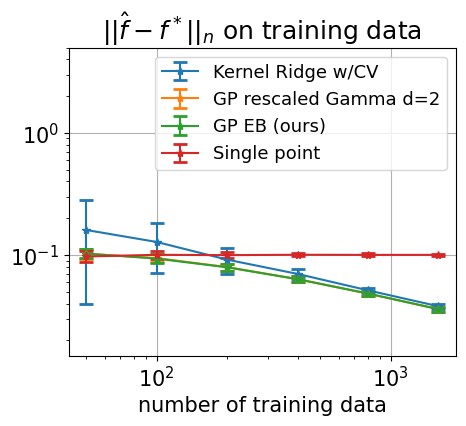}
        \subcaption{}
    \end{minipage}
    \begin{minipage}{0.35\textwidth}
        \includegraphics[height=0.95\linewidth]{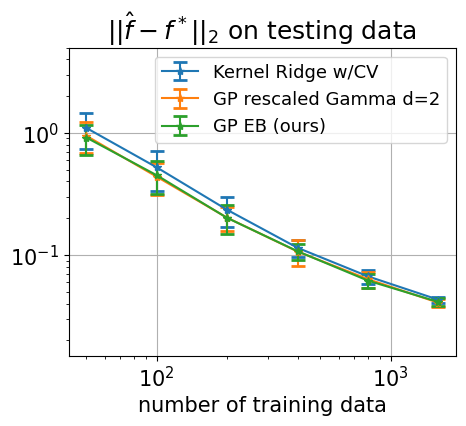}
        \subcaption{}
    \end{minipage}
            \vspace{-5pt}
    \caption{
    Swiss roll data.
    (a) Sample $X_i$ and response $Y_i$ plotted as color field on $X_i$, showing 1600 samples. 
    (b) 
    The empirical error $||\hat{f} - f^*||_n$ on training samples, plotted as the training size increases from 50 to 1600.
    The error bars indicate the standard deviation.
    (c) $||\hat{f} - f^*||_2$ on testing samples, as introduced in Section \ref{sec:2d_swiss}.
In both (b)(c), the GP rescaled Gamma curve (orange) is almost not visible because the values are close to those on the proposed GP EB curve (green).
}
\label{fig:swiss_roll}
\end{figure}

\paragraph*{Data}
The samples $X_i$ lie on a 2D manifold $\calM$ embedded in $\R^3$. 
The response $Y_i$ follows  \eqref{eq:model}
where $f^*$ is a smooth function on $\calM$
and the noise level $\sigma = 0.1$.
The dataset is illustrated in Figure \ref{fig:swiss_roll}(a). 
Details of data simulation can be found in Appendix \ref{app:exp_detail}.

\paragraph*{Method}
We generate $n$ training data samples $\{X_i, Y_i\}_{i=1}^n$, and compute out-of-sample error on a separate test data set.
We use $n = 50,100,200,400,800, 1600$ training samples, and compare performances of the following approaches: 

\vspace{5pt}
(i) Kernel ridge w/CV: Kernel ridge regression, where the kernel bandwidth is selected using a validation set consisting of 10\% training data.
    
(ii) GP rescaled Gamma: 
Bayesian regression with GP prior, 
where the kernel bandwidth $t$ is sampled from the posterior of the rescaled Gamma prior by MCMC.
Here we use the true manifold dimension $d=2$ in the implementation. 

(iii) GP EB (ours):
Bayesian regression with GP prior and the empirical Bayes (EB) prior $p(t)$ in Section \ref{sec:knn_prior}. 
\vspace{5pt}

On training samples (the in-sample case), we also implement another baseline where one uses the observed value $Y_i$ as the estimate of $f(X_i)$.
Because this only uses information on one data sample,
we call this baseline ``single point''.

To further investigate alternative approaches, we also implemented 
(ii') GP estimated $d$: a variant of (ii), where the manifold dimension is estimated from data as proposed in \cite{yang2016bayesian},
(iv) GP max-likelihood: selecting the kernel bandwidth $t$ based on maximizing the marginal likelihood; 
(v) GP median heuristic: setting $t$ to be the median of the distances between samples. 
We apply to the Swiss Roll data at sample size $n=50, 100, 200$, and the results are detailed in Appendix \ref{app:subsec-additional-EB}.
The proposed EB (iii) and (iv) are comparable and perform the best,
with (ii') similar on testing error and worse on training error,
and (v) giving much larger errors. 
Further comparison of the distribution of the errors reveals that (ii') can give long-tail outlier in-sample errors at small sample size,
likely due to the unstable estimation of the manifold dimension;
The proposed EB (iii) shows more stable performance in comparison.

\paragraph*{Evaluation metrics}
We compute the in-sample error $||\hat{f}-f^*||_n$ on the training set, and the out-of-sample error using $n_{\rm te} = 5000$ test samples that were not used at all in data fitting or hyperparameter choice. 
We compute the mean and standard deviation of $||\hat{f}-f^*||_n$ ($||\hat{f}-f^*||_2$) on training (testing) data, respectively, over repeated experiments,
and we repeat 200 runs when $n \le 200$,
and 100 runs when $n > 200$.

\paragraph*{Result}

The training and testing errors are plotted in Figure \ref{fig:swiss_roll}(b)(c).
The performances of the two GP methods (ii)(iii) are very close, such that the curves almost overlap in both plots. 
Recall that in (ii) we inserted the true intrinsic dimensionality (instead of estimating it from data as proposed in \cite{yang2016bayesian}),
while (iii) does not use such information. 
The result suggests that the proposed empirical Bayes prior adapts to the dimension well and obtains comparable performance in estimating $f^*$. 
The single-point method cannot generalize to test samples, and the in-sample performance is surpassed by other baselines when training size increases.

The two GP methods perform better than (i) kernel ridge with cross validation:
(i) is slightly worse on the testing error, and the gap is more visible on the training error, where (i) also gives larger variance at small training size.
Kernel ridge regression shares a form similar to the posterior mean in GP regression, and cross-validation is used to optimize its generalization performance with respect to squared error loss. Thus, (i) may have a potential advantage in terms of test error. We have observed in additional simulations (by adjusting the parameters of Swiss Roll data and $f^*$) 
that (i) can perform better than GP methods on the out-of-sample error, while the in-sample error is still worse (results not reported). The larger variance in in-sample error by (i), particularly at small sample sizes, is likely due to (i) requiring a validation set and hence reduces the training set. An additional advantage of GP methods in practice lies in the potential ability to characterize uncertainty in the estimation of the regression function and prediction (not shown in this work).

\begin{figure}
    \hspace{-15pt}
       \begin{minipage}{0.32\textwidth}
     \includegraphics[height=0.96\linewidth]{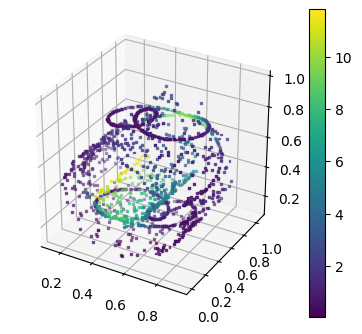}
           \subcaption{}
        \end{minipage}
        \begin{minipage}{0.35\textwidth}
            \includegraphics[height=0.95\linewidth]{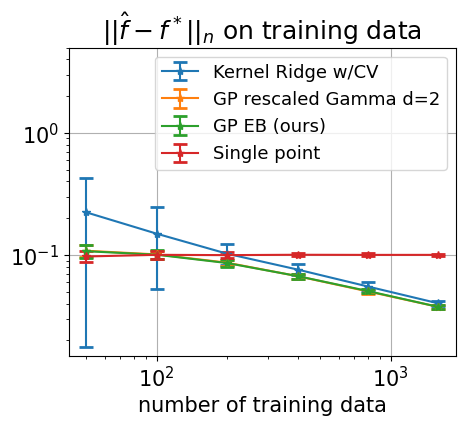}
            \subcaption{}
        \end{minipage}
        \begin{minipage}{0.35\textwidth}
            \includegraphics[height=0.95\linewidth]{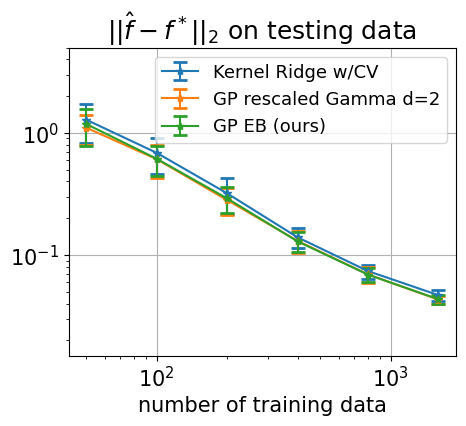}
            \subcaption{}
        \end{minipage}
        \vspace{-5pt}
        \caption{
        Same plots as in Figure \ref{fig:swiss_roll}
        for Swiss roll plus curve data.}
    \label{fig:swiss_roll_curve}
    \end{figure}

\subsection{Low-dimensional set $\calX$ of mixed local dimensions}

\paragraph*{Data}

The samples $X_i$ lie on a union of two manifolds embedded in $\R^3$:
one is the Swiss Roll which is a 2D manifold,
and the other is a curve which is a 1D manifold. 
The response $Y_i$ follows  \eqref{eq:model}
where $f^*$ is a smooth function on the two manifolds (and continuous at the intersection). 
The noise level $\sigma = 0.1$.
The dataset is illustrated in Figure \ref{fig:swiss_roll_curve}(a). 
Details of data simulation can be found in Appendix \ref{app:exp_detail}.

\paragraph*{Method and evaluation}
We follow the same procedure of creating training and testing sets, computing the baselines (i)(ii)(iii), 
and reporting training and testing errors as in Section \ref{sec:2d_swiss}. 
When computing the rescaled Gamma GP (ii), we insert the intrinsic dimensionality $d=2$. 
For (iii), we adopt the modified $\hat v_n$ as in \eqref{eq:def_vn-mixed-d}
in the proposed EB prior \eqref{eq:EBprior}.

\paragraph*{Result}
The mean and standard deviation of 
$||\hat{f}-f^*||_n$ on training data
and $||\hat{f}-f^*||_2$ on testing data
are shown in Figure  \ref{fig:swiss_roll_curve}(b)(c) respectively. 
The relative performances of the different baselines are mostly similar to the case of the Swiss roll data in Figure \ref{fig:swiss_roll}.
 The proposed GP model (iii) performs comparably to the GP baseline (ii);
 the kernel ridge (i) is comparable on the testing error and worse on the in-sample error,
 especially at small training size. 
 We emphasize that though this dataset consists of a union of two manifolds having distinct dimensions, strictly speaking, it goes beyond the theoretical assumption of our extended theory in Appendix \ref{app:extension-theory-manifold-mixed-d} 
 because the two manifolds intersect.
 The experimental result suggests that the proposed method can extend to more complex data of intrinsically low dimensionality.

\subsection{Lucky Cat data}

\paragraph*{Data}
We study a dataset of high dimensional image data with intrinsic low dimensionality induced by a one-dimensional group action, following the setup in \cite{yang2016bayesian}.
The Lucky Cat dataset \cite{nene1996columbia} contains 72 images of size $ 128 \times 128 $,
 resulting in the ambient dimensionality being $D = 16,384$. 
 Each sample (image) $X_i$ is the side view of a 3D object from a rotation angle $\theta_i \in [0, 2 \pi]$, and thus the samples lie on a 1-dimensional manifold embedded in the high dimensional Euclidean space. 
 Figure \ref{fig:lucky_cat}(a) shows two examples of the image data.
Because the image $X$ in this dataset and the rotation angle $\theta$ have a one-to-one correspondence, we set $f^*(X) = \cos(\theta)$, which is a continuous function on the one-dimensional data manifold. 
The response $Y_i$ is as in \eqref{eq:model} where $\sigma = 0.1$.

\paragraph*{Method and evaluation}

For all three baselines (i)(ii)(iii), we randomly partition $n=18$, $36$, and $54$ samples into a training set, leaving the remaining samples as a testing set.
This process is repeated $400$ times for each training size, and $||\hat{f} - f^*||_2$ on the testing set is reported. 
We inserted the true intrinsic dimensionality $d=1$ when computing the rescaled Gamma GP (ii).
Not assuming known $\sigma$, we also implement the Bayes estimation of $\sigma$ jointly with $f$ using our EB prior (iii), 
where we adopt a prior of $\sigma^2$ uniformly on $[10^{-4},1]$.

\paragraph*{Result}
As shown in the table in Figure \ref{fig:lucky_cat}(b), the proposed GP model (iii) performs comparably to the GP model (ii), and both perform better than the (i) kernel ridge in the out-of-sample error. 
Note that (ii) presumes knowledge of the true intrinsic dimensionality, which, when the sample size is small, may be difficult to estimate from data. 
In the table, we also include two additional baselines,
Lasso \cite{tibshirani1996regression}
and Elastic net (EL-net) \cite{zou2005regularization}, for reference.
The mean and standard deviation of the errors 
of EL-net and Lasso
are quoted from \cite{yang2016bayesian} (averaged over 100 repeated runs).
Note that (iii) maintained comparable performance when inferring $\sigma$ jointly with $f$.
Overall, on this dataset where the sample size is very small (only a few tens) compared to the data dimensionality, 
the GP models outperform the other methods across all training sizes.

\begin{figure}
\hspace{-5pt}
       \begin{minipage}{0.33\textwidth}
     \includegraphics[height=0.45\linewidth]{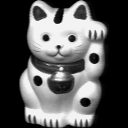}
    \includegraphics[height=0.45\linewidth]{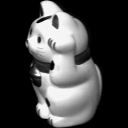}
            \subcaption{}
        \end{minipage}
        \begin{minipage}{0.65\textwidth}
            \hspace{10pt}
\small
        \centering
    \begin{tabular}{lccc}
    \hline
& $n=18$ & $n=36$ & $n=54$ \\
\hline
EL-net$^*$ & 0.416 (0.152) & 0.198 (0.042) & 0.149 (0.031) \\
Lasso$^*$ & 0.431 (0.128) & 0.232 (0.061) & 0.163 (0.038) \\
Kernel ridge w/CV & 0.226 (0.091) & 0.112 (0.038) & 0.080 (0.023) \\
GP rescaled Gamma 
        & 0.194 (0.068) & 0.096 (0.029) & 0.074 (0.019) \\
GP EB (ours) & 0.190 (0.072) & 0.099 (0.030) & 0.077 (0.019) \\
~~ + unknown $\sigma$  &0.185 (0.071)  &0.103 (0.030) &0.082 (0.019)\\
\hline
    \end{tabular}
        \subcaption{}
        \end{minipage}
                        \vspace{-5pt}
        \caption{
        Lucky cat data.
        (a) Two sample images of the object at different rotation angles.
        (b) 
        $||\hat{f} - f^*||_2$ on testing samples 
        at training sizes 18, 36, 54
        for different methods.
        The errors for kernel ridge, 
        rescaled Gamma GP, 
        and our EB GP are averaged from 400 repeated runs, with the standard deviation given in the parentheses. 
        *The EL-net and Lasso baselines are quoted from \cite{yang2016bayesian}.}
    \label{fig:lucky_cat}
    \end{figure}

\section{Discussion}\label{sec:discuss}
The work can be extended in several future directions.
It would be interesting to develop RKHS approximation analysis 
on a more general low-dimensional domain $\calX$.
To do this, one will need to define the notion of smoothness $s$ properly
when intrinsically (non-linear) low-dimensional structures are present in data. 
It would also be helpful to improve the RKHS covering number analysis when restricted to a low-dimensional $\calX$,
particularly, to improve the dependence on the ambient dimension $D$.
Meanwhile, one can try to cover more complicated stratified space beyond the case of finite union of disjoint manifolds,
e.g., an infinite union.
In addition, it would be useful to 
 extend our findings to other types of kernels, 
such as the Mat\'{e}rn kernel, 
and evaluate the theoretical and practical behaviors of these kernels.
Finally, further comparison with more Gaussian process methods, in theory and in practical applications,
would help to advance the understanding on this topic. 

\paragraph*{Empirical Bayes methods.}
It would be interesting to further compare and study the different empirical Bayes (EB) approaches.
In particular, our EB prior based on the averaged kernel affinity $\hat v_n(t)$ 
shows improved performance over the EB method based on estimating the manifold dimension at small sample size, 
yet the current asymptotic theory cannot explain the advantage. 
A fully  non-asymptotic analysis focusing on finite sample size would be helpful. 
Another potentially fruitful direction is to extend our analysis to more EB methods,
e.g., the GP MLE approach, which we numerically studied.
The theoretical analysis of the MLE approach has its own challenge \cite{karvonen2023maximum} and needs to be under a different framework.
We expect that some of our estimates, such as the manifold RKHS approximation results in Section \ref{sec:rate_manifold}, 
will be transferable to the analysis of general Bayesian and non-Bayesian kernel methods.

\paragraph*{Observation variance.}
In this work, we have assumed known $\sigma$ in our posterior contraction rate analysis.
The theory may be extended to infer unknown $\sigma$ (by choosing a prior on $\sigma$) following the arguments in \cite{van2008rates,van2009adaptive} based on the general framework in \cite{ghosal2007convergence,ghosal2000convergence}.
This corresponds to a full Bayesian approach. 
It would also be useful to extend the theory to allow estimating $\sigma$ by certain empirical Bayes methods.
Meanwhile,  our analysis \rev{under a standardized scale} 
suggests that the constant $C$ in the convergence rate
scales with the ratio ${\sigma}/{\|f^* \|}$ (Remark \ref{rk:sigma-dependence}), though the expression is only an upper bound.
One can interpret $\|f^*\|^2/\sigma^2$ as a Signal-to-Noise Ratio,
as hinted by the  information theoretical arguments in  \cite{van2011information}.
We think information theoretical techniques, possibly combined with a non-asymptotic analysis, 
can help to further elucidate the influence of $\sigma^2$ on the nonparametric Bayesian approach.

   
\section*{Acknowledgement}

The authors thank Hau-tieng Wu for helpful discussions.
TT and XC were partially supported by Simons Foundation
(grant ID: MPS-MODL-00814643).
XC was also partially supported by 
NSF DMS-2237842, DMS-2007040. 
DD was partially supported by the United States National Institutes of Health Project R01ES035625 and by the European Research Council under the European Union’s Horizon 2020 research and innovation program (grant agreement No 856506). 
NW was partially supported by the Simons Foundation (grant ID: MPS-TSM-00002707).

\bibliographystyle{alpha} 
\bibliography{sample.bib}

\begin{appendix}

\setcounter{figure}{0} \renewcommand{\thefigure}{A.\arabic{figure}}
\setcounter{table}{0} \renewcommand{\thetable}{A.\arabic{table}}

\setcounter{remark}{0} \renewcommand{\theremark}{A.\arabic{remark}}
\setcounter{equation}{0} \renewcommand{\theequation}{A.\arabic{equation}}

\setcounter{theorem}{0} \renewcommand{\thetheorem}{\Alph{section}.\arabic{theorem}}

\setcounter{assumption}{0} \renewcommand{\theassumption}{A.\arabic{assumption}}

\section{Proofs in Section \ref{sec:pos_rate} and extension}\label{sec:proof-sec3}

The proofs in Section \ref{sec:pos_rate}
follow the framework in \cite{van2009adaptive} and \cite{yang2016bayesian}.
For convergence rate of posterior mean estimator under fixed design (Theorem \ref{thm:fix-design-estimator}), we avoid the need of function truncation and improve from the result in \cite{yang2016bayesian}.

\subsection{Fixed design results}\label{subsec:proof-sec3-fixed-design}

The posterior contraction rate 
under hierarchical GP priors \eqref{eq:prior}
was studied in \cite{van2009adaptive} 
and in \cite{yang2016bayesian} for manifold data,
under the general framework of \cite{ghosal2000convergence,ghosal2007convergence}.
Following these previous works, our strategy to prove the posterior contraction rate is to verify a set of conditions on our hierarchical GP prior \eqref{eq:prior}.
Specifically, suppose we can specify two sequences $ \varepsilon_n$ and $ \bar{\varepsilon}_n$, 
which are asymptotically $o(1)$ as $n \to \infty$,
and we can show that for  some Borel measurable subsets $B_n$ of $C(\mathcal{X})$ and $n$ sufficiently large,
the following three inequalities hold
\begin{align}
        \P [ ||f^{t} - f^*||_{\infty} \le \varepsilon_n ] 
         \ge e^{-n \varepsilon_n^2}, 
            \label{eq:pf-A1A1prime-goal1} \\
        \P [ f^{t} \notin B_n ]
            \le e^{-4n \varepsilon^2_n}, 
            \label{eq:pf-A1A1prime-goal2} \\
        \log \calN(\bar{\varepsilon}_n,B_n,||\cdot||_{\infty}) 
        \le n \bar{\varepsilon}^2_n, \label{eq:pf-A1A1prime-goal3} 
\end{align}   
then 
the posterior contraction rate with respect to $||\cdot||_n$ would be at least $\varepsilon_n \vee \bar{\varepsilon}_n $, up to multiplying an absolute constant.
In \eqref{eq:pf-A1A1prime-goal1}\eqref{eq:pf-A1A1prime-goal2} and below, the notation $\P$ stands for the prior $\Pi$ and we adopt this convention in our proof following the literature.

\begin{proof}[Proof of Theorem \ref{thm:contraction_n_norm}]
We specify the two sequences $ \varepsilon_n$ and $ \bar{\varepsilon}_n$ as 
\begin{equation}\label{eq:def_ep_ep_bar-thm3.1}
\varepsilon_n = \bar{C}_1 n^{-\frac{s}{2s+\varrho}}(\log n)^{k_1},
\quad 
\bar{\varepsilon}_n = \bar{C}_2 \varepsilon_n(\log n)^{k_2},
\quad 
k_1 : = \frac{1+D}{2+\varrho/s},
\quad k_2 := \frac{1+D}{2}, 
\end{equation}
where the positive constants $\bar{C}_1$, $\bar{C}_2$ will be specified later.
Below, we prove that the three conditions \eqref{eq:pf-A1A1prime-goal1}\eqref{eq:pf-A1A1prime-goal2}\eqref{eq:pf-A1A1prime-goal3} respectively.
We will show in our proof that $\bar{\varepsilon}_n > 3 {\varepsilon}_n$ for large enough $n$.
Thus, this will imply a posterior contraction rate under $\| \cdot \|_n$ that is at least 
\begin{equation}\label{eq:order-bar-epsn}
 \bar{\varepsilon}_n  
=
 \bar{C}_1 \bar{C}_2 
 n^{-\frac{s}{2s+\varrho}} (\log n )^{\frac{ D+1}{2 + \varrho/s} + \frac{D+1}{2}}
 \lesssim n^{-\frac{s}{2s+\varrho}}(\log n)^{D+1},
\end{equation}
and prove the theorem with $C = \bar C_1 \bar C_2$.
    
\vspace{10pt}
\noindent
$\bullet$ \underline{Part I: To prove \eqref{eq:pf-A1A1prime-goal1}}.
We denote $ {\mathbb{H}}_{t}(\calX )$ as $\tilde{\mathbb{H}}_{t}$.
    To proceed, we define the centered and decentered concentration function of the Gaussian process $f^{t}$ 
    conditioning on a fixed bandwidth $t$.
    The centered concentration function is defined as 
\[
\phi_0^{t}(\varepsilon') 
       : = -\log 
       \P [ ||f^{t}||_{\infty} \le \varepsilon' | \, t]. 
\]
For any $f \in C( \calX)$, the  decentered concentration function is defined as
    \[
       \phi_{f}^{t}(\varepsilon') : = \inf_{q \in \tilde{\mathbb{H}}_{t}:||q-f||_{\infty} \le \varepsilon'} ||q||^2_{\tilde{\mathbb{H}}_{t}} 
       -\log \P [ ||f^{t}||_{\infty} \le \varepsilon' | \, t ].
\]
In both definitions,  the sup norm $||\cdot||_{\infty}$ is on $\mathcal{X}$. 
   
    By definition, we know $\P [ ||f^{t}||_{\infty} \le \varepsilon' | \, t ] = \exp(-\phi_0^{t}(\varepsilon'))$.
    Meanwhile, following \cite{kuelbs1994gaussian}, 
    the definition of $\phi_{f}^{t}$ will guarantee that, when $f = f^*$,  
    \begin{equation}\label{eq:e-phi-f*-t-upper-bound-proof-1}
    \P [ ||f^{t} -f^*||_{\infty} 
    \le 2\varepsilon' | \, t ]
    \ge e^{-\phi_{f^*}^{t}(\varepsilon')}.
    \end{equation}

We recall a few positive constants:
$r_0$ as in (A1),
and 
$\epsilon_0$, $\nu_1$, $\nu_2$ as in (A2).
Under (A1), $\calX$ satisfies the needed assumption in Lemma \ref{lemma:small_ball}.
Meanwhile, we consider $t$  and $\varepsilon'$ satisfying    

\begin{equation}\label{eq:range-t-eps'-proof-1}
    t < \min \{ \epsilon_0,1, r_0^2\},
\quad 
\nu_1 t^{s/2} < \varepsilon'  <  1/e < 1/2.
\end{equation}
With such $t$ and $\varepsilon'$, Lemma \ref{lemma:small_ball} applies to give that for some positive constant $C_4$,
\begin{equation*}
-\log
    \P [ ||f^{t}||_{\infty} \le \varepsilon' | \, t] 
    \le C_4 t^{- \varrho /2} (\log \frac{1}{\sqrt{t}\varepsilon'} )^{D+1}.
\end{equation*}
\rev{We take $C_4$ to be the explicit constant $C>1$ at the end of
 the proof of Lemma \ref{lemma:small_ball}, depending only on $\varrho$,
$D$, $C_{\calX}$, and the kernel constants $\delta_h,c_h,a_1$.
Its dependence on domain $\calX$ occurs only through $C_{\calX}$ and $K'$;
the proof chooses $K'$ using $C_{\calX}$ and the kernel constants alone.
Thus $C_4$ can be chosen uniformly over subsets of $[0,1]^D$
satisfying (A1) with the same $\varrho,C_{\calX},r_0$.}
Under (A2), for each fixed $t < \epsilon_0 $ which is satisfied for $t$ in \eqref{eq:range-t-eps'-proof-1}, there exists $q_t \in \tilde{\mathbb{H}}_{t}$ s.t. $\| q_t - f^* \|_\infty \le \nu_1 t^{s/2}$
and $\| q_t \|^2_{\tilde{\mathbb{H}}_{t}} \le \nu_2 t^{- \varrho/2}$. 
Since $\nu_1 t^{s/2} < \varepsilon'$, we can insert this $q_t$ into the r.h.s. of the definition of $\phi_{f^*}^{t}$, and then we have
\begin{align}
\phi_{f^*}^{t}(\varepsilon')
& \le ||q_t||^2_{\tilde{\mathbb{H}}_{t}} 
       -\log \P [ ||f^{t}||_{\infty} \le \varepsilon' | \, t ] \nonumber \\
& \le  \nu_2 t^{- \varrho/2}
    + C_4 t^{- \varrho /2} (\log(\frac{1}{\sqrt{t}\varepsilon'}))^{D+1} \nonumber \\
& \le K_3 t^{-\varrho/2} (\log(\frac{1}{\sqrt{t}\varepsilon'}))^{D+1},
\quad K_3 : =  \nu_2 + C_4,    
\label{eq:bound-phi-f*-t-proof-2}
\end{align}
where the last inequality is by that $t < 1$ and $\log (1/\varepsilon') > 1$.

By Assumption \ref{assump:A3}(A3), 
there exist $c_1$, $c_2$, $c_3$, $a_1$, $a_2$, $K_1$, $K_2$, $C_1$, $C_2 > 0$, such that 
\begin{align}
      &\forall t \in [c_1 n^{\frac{-2}{2s+\varrho}} (\log n)^{\frac{ 2 (1+D)}{2s +\varrho}}, 
      			c_2 n^{\frac{-2}{2s+\varrho}}(\log n)^{\frac{ 2(1+D)}{2s +\varrho}} ], 
      \quad 
      p(t) \ge C_1t^{-a_1} \exp(-\frac{K_1}{t^{
    \varrho/2}}),   \label{eq:(A3)-row-1-proof-2} \\
      &\forall t \in 
      \rev{\left(0, c_3 n^{\frac{-2}{2s+\varrho}}
      (\log n)^{\frac{-4(1+D)}{(2+\varrho/s)\varrho}}\right]},
      \quad p(t) \le C_2t^{-a_2} \exp(-\frac{K_2}{t^{\varrho/2}}),
\end{align}
where $p(t)$ is the prior of $t$.
We define $C := 1/\nu_1$.
By (A3), $0 < c1 < c2$, then we can have a constant $c_4 >0$ s.t. 
\[
{c_1}/{c_2}< c_4 < 1.
\]
In the calculation below, we want to take an integral of $t$ on the interval 
\[
t \in [c_4(C\varepsilon')^{2/s}, (C\varepsilon')^{2/s}],
\]
on which we want to  use the lower bound of $p(t)$ in \eqref{eq:(A3)-row-1-proof-2}
and the upper bound of $\phi_{f^*}^{t}(\varepsilon')$ in \eqref{eq:bound-phi-f*-t-proof-2}. 
This requires $t \in [c_4(C\varepsilon')^{2/s}, (C\varepsilon')^{2/s}]$ to satisfy the range in \eqref{eq:range-t-eps'-proof-1} plus that in \eqref{eq:(A3)-row-1-proof-2}. 
Such requirement will be satisfied as long as 
\begin{equation}\label{eq:order-eps-prime}
[c_4(C\varepsilon')^{2/s}, (C\varepsilon')^{2/s}] 
\subset [c_1 n^{\frac{-2}{2s+\varrho}} (\log n)^{\frac{2( 1+D )}{2s +\varrho}}, 
    c_2 n^{\frac{-2}{2s+\varrho}} (\log n)^{\frac{2( 1+D)}{2s +\varrho}} ]
\end{equation}
and when $n$ is large enough s.t. 
$c_2 n^{\frac{-2}{2s+\varrho}} (\log n)^{\frac{2( 1+D)}{2s +\varrho}} < \min \{ \epsilon_0 ,1, r_0^2\} $.  
The condition \eqref{eq:order-eps-prime} poses a constraint on $\varepsilon'$, which we will choose $\varepsilon'$ below to satisfy. 
For now, for any $\varepsilon'$ that satisfies \eqref{eq:order-eps-prime}, we then have 
\begin{align} 
    \P [ ||f^{t} -f^*||_{\infty} &\le 2\varepsilon'] 
    \ge 
    \P \left[ ||f^{t} -f^*||_{\infty} \le 2\varepsilon', \, t \in [c_4(C\varepsilon')^{2/s}, (C\varepsilon')^{2/s}] \right]  \nonumber \\
    & \ge \int_{c_4(C\varepsilon')^{2/s}}^{(C\varepsilon')^{2/s}} e^{-\phi_{f^*}^{t}(\varepsilon')} p(t) d t  
    \quad \text{(by \eqref{eq:e-phi-f*-t-upper-bound-proof-1})}
    \nonumber \\
    &\ge  e^{- K_3c_4^{-\varrho/2}(C\varepsilon')^{-\varrho/s} 
        \left(\log( \frac{1}{   {c_4}^{1/2}(C\varepsilon')^{1/s}   \varepsilon'})\right)^{
    D+1}}   \nonumber  \\
   &~~~  C_1 e^{-K_1 c_4^{-\varrho/2}(C\varepsilon')^{-\varrho/s} } (C\varepsilon')^{-2a_1/s}  \nonumber \\
   &~~~  (1- c_4)(C\varepsilon')^{2/s},  \label{eq:estimation_around_true_func}
\end{align}
where the third inequality is by \eqref{eq:bound-phi-f*-t-proof-2}\eqref{eq:(A3)-row-1-proof-2}. 
Under \eqref{eq:order-eps-prime}, $(\varepsilon')^{2/s} \lesssim n^{-2/(2s + \varrho)} (\log n)^{2(1+D)/(2s +\varrho)} $, and thus $\varepsilon' = o(1)$.
Then, with large enough $n$ and consequently small enough $\varepsilon'$, we have the r.h.s. of \eqref{eq:estimation_around_true_func}  lower bounded by 
$e^{-K_4(\varepsilon')^{ -  \varrho/s}(\log(1/\varepsilon'))^{
    1+D}}$ where 
\begin{equation}\label{eq:def-K4-proof-thm3.1}
K_4 := K_3 c_4^{-\varrho/2 } C^{-\varrho/s}   \big(  (1+2/s)^{1+D}+ 1 \big)  > 0.
\end{equation}
To see this, we observe that the desired inequality holds if 
the following two inequalities are satisfied:
\begin{align*}
& K_3c_4^{-\varrho/2}(C\varepsilon')^{-\varrho/s} 
        \big(
        (1+ \frac{1}{s}) \log(\frac{1}{\varepsilon'})+\log \frac{1}{{c_4}^{1/2}C^{1/s}}
        \big)^{D+1}  \\
&~~~~~~~~~ 
 	\le  K_3 c_4^{-\varrho/2 } C^{-\varrho/s}  (1+2/s)^{1+D}   
     (\varepsilon')^{ -  \varrho/s}
     (\log \frac{1}{\varepsilon'})^{
    1+D}, \\
& - \log  C_1 
 + K_1 c_4^{-\varrho/2}(C\varepsilon')^{-\varrho/s}
 - \log (1- c_4)
 +  \frac{2(1-a_1)}{s} \log\frac{1} {C\varepsilon'}    \\
&~~~~~~~~~ 
 \le K_3 c_4^{-\varrho/2 } C^{-\varrho/s}
 (\varepsilon')^{ -  \varrho/s}
     (\log \frac{1}{\varepsilon'})^{
    1+D}.
\end{align*}
Thus, it suffices to show that the two inequalities hold  with small enough $\varepsilon'$.
For both of them, this can be done by collecting the dominating terms
and using that $\varepsilon' \to 0$.
As a result, 
\begin{equation}\label{eq:estimation_around_true_func-2}
  \P [ ||f^{t} -f^*||_{\infty} \le 2\varepsilon'] 
  \ge e^{-K_4(\varepsilon')^{ -  \varrho/s}(\log(1/\varepsilon'))^{
    1+D}}.
\end{equation}

We are ready to prove  \eqref{eq:pf-A1A1prime-goal1}.
Here, consider $\varepsilon' = \varepsilon_n'$
which is defined to satisfy that 
\begin{equation}\label{eq:def-varepsilonn-prime}
(C\varepsilon_n')^{2/s} = c_2 n^{\frac{-2}{2s+\varrho}} (\log n)^{\frac{2( 1+D)}{2s +\varrho}}.
\end{equation}
This $\varepsilon_n'$ satisfies \eqref{eq:order-eps-prime} because the right ends of the two intervals are the same and $c_4> c_1/c_2$.
Consequently, \eqref{eq:estimation_around_true_func} holds at $\varepsilon' = \varepsilon_n'$. 
We now specify
\begin{equation}\label{eq:choice-epsn-all-three-goals}
    \varepsilon_n
     = \max\{ \bar{C}_1' ( c_2^{s/2}/C), 
      {c_3^{-\varrho/4} K_2^{1/2}} \}
    n^{\frac{-s}{2s+\varrho}}
    (\log n)^{\frac{ 1+D}{2 +\varrho/s}},
\end{equation}
where $\bar{C}_1'$ is to be determined here,
and the constant factor ${c_3^{-\varrho/4} K_2^{1/2}}$ is to fulfill the proof in Part II. 
Recall our declared definition of $\varepsilon_n$ at the beginning of this proof, we see that $\bar{C}_1 = \max\{ (\bar{C}_1'/C) c_2^{s/2}, {c_3^{-\varrho/4} K_2^{1/2}} \}$, 
and thus the choice of the constant $\bar{C}_1'$ will equivalently determine $\bar{C}_1$. 
We will choose $\bar{C}_1' \ge 2$ s.t. when $n$ is large enough,
the r.h.s. of \eqref{eq:estimation_around_true_func-2} evaluated at $\varepsilon' = \varepsilon_n'$ can be lower bounded by $e^{-n \varepsilon_n^2} $.
This will imply  $\P [ ||f^{t} -f^*||_{\infty} \le 2\varepsilon'_n] \ge e^{-n \varepsilon_n^2} $.
Meanwhile, comparing \eqref{eq:choice-epsn-all-three-goals} with the definition of $\varepsilon_n'$ in \eqref{eq:def-varepsilonn-prime}, 
we also see that  $\varepsilon_n \ge \bar{C}_1' \varepsilon_n' $.
Then \eqref{eq:pf-A1A1prime-goal1} follows by that $2 \varepsilon_n' \le   \bar{C}_1' \varepsilon_n' \le \varepsilon_n$.

We claim that such $\bar{C}_1' $ can be chosen to be 
\[
\bar{C}_1' =  2 \vee K_4^{1/2} (c_2^{s/2}/C)^{-(2 + \varrho/s)/2}.
\]
To prove  \eqref{eq:pf-A1A1prime-goal1},
it remains to show that
$K_4(\varepsilon_n' )^{ -  \varrho/s}(\log \frac{1}{\varepsilon_n '} )^{1+D}
\le n \varepsilon_n^2$ with large $n$.
Inserting in the definitions of $\varepsilon_n'$ and $\varepsilon_n$,
and using that $\log \frac{1}{\varepsilon_n' }  \le \log n$ with large $n$ (to verify below),
it suffices to have
\[
K_4 
( \frac{c_2^{s/2}}{C})^{-\varrho/s}
 n^{\frac{\varrho}{2s+\varrho}} 
 (\log n)^{ \frac{-\varrho( 1+D)}{2s +\varrho}}
 (  \log n )^{1+D}
 \le 
(\bar{C}_1' \frac{ c_2^{s/2}}{C})^2
    n^{\frac{\varrho}{2s+\varrho}}
    (\log n)^{\frac{2( 1+D)}{2 +\varrho/s}}.
\]
This inequality is reduced to 
$
K_4 ( \frac{c_2^{s/2}}{C})^{-\varrho/s}
\le(\bar{C}_1' \frac{ c_2^{s/2}}{C})^2$
 and is guaranteed by our choice of $\bar C_1'$.
To see that $\log \frac{1}{\varepsilon_n' }  \le \log n$ with large $n$,
note that because 
$\varepsilon_n' = \frac{c_2^{s/2}}{C} n^{\frac{-s}{2s+\varrho}} (\log n)^{ \frac{s( 1+D)}{2s +\varrho}}$,
we have
\[
\log \frac{1}{\varepsilon_n' } 
= \frac{s}{2s + \varrho } \log n  + \log \frac{C}{c_2^{s/2}} - \frac{s( 1+D)}{2s +\varrho} \log \log n
= \left( \frac{s}{2s + \varrho } + o(1) \right) \log n, 
\]
and use that $ \frac{s}{2s + \varrho } < 1/2$.

\vspace{10pt}
\noindent
$\bullet$ \underline{Part II:  To prove  \eqref {eq:pf-A1A1prime-goal2}}.
%
Let $\mathbb{B}_1$ be the unit ball of $C(\mathcal{X})$.
Following the same construction as in the proof of  \cite[Theorem 3.1]{van2009adaptive} and  \cite[Theorem 2.1]{yang2016bayesian}, we introduce the set $B_{N,r,\delta,\varepsilon'}$ defined as

\begin{equation}\label{eq:def-B-N-r-delta-eps}
B_{N,r,\delta,\varepsilon'} 
: = \left( N 
(\frac{r}{\delta})^{D/2}
\tilde{\mathbb{H}}^1_{r^{-2}}+ \varepsilon' \mathbb{B}_1 \right) 
\cup \left( \bigcup_{t > \delta^{-2}}(N \tilde{\mathbb{H}}^1_{t})+\varepsilon' \mathbb{B}_1 \right),
\end{equation}
for positive numbers $r, \delta, N, \varepsilon'$ to be determined, where $r>\delta$.

By Lemma \ref{lemma:compare_rkhs}, one can verify that when $ t \in [ r^{-2},\delta^{-2} ]$, 
$ \tilde{\mathbb{H}}^1_{t} \subset ( r/\delta)^{D/2} \tilde{\mathbb{H}}^1_{r^{-2}} $.
As a result, 
\begin{equation*}
N \tilde{\mathbb{H}}^1_{t}+\varepsilon' \mathbb{B}_1
\subset B_{N,r,\delta,\varepsilon'},
\quad 
\forall t \in [ r^{-2},\delta^{-2} ].
\end{equation*} 
Meanwhile, when $ t > \delta^{-2}$, 
\[
N\tilde{\mathbb{H}}^1_{t}+\varepsilon' \mathbb{B}_1 
\subset \bigcup_{t > \delta^{-2}}(N \tilde{\mathbb{H}}^1_{t})+\varepsilon' \mathbb{B}_1  
\subset B_{N,r,\delta,\varepsilon'}.
\]
Putting together, we have 
\begin{equation}\label{eq:NHt1+epsB1-subset-B-proof1}
N \tilde{\mathbb{H}}^1_{t}+\varepsilon' \mathbb{B}_1
\subset B_{N,r,\delta,\varepsilon'},
\quad 
\forall t \ge r^{-2}. 
\end{equation} 
Next, we claim that if 
\begin{equation}\label{eq:cond-claim}
\begin{gathered}
    r > \rev{\max\{\delta,1\}},
    \quad 
    r^{-2} < c_3 n^{\frac{-2}{2s+\varrho}} (\log n )^{\frac{-4(1+D)}{(2+\varrho/s)\varrho}}, \\
    \quad e^{- \rev{\sup_{t\ge r^{-2}}\phi^t_0(\varepsilon')}} <1/4,
    \quad N \ge 4 \sqrt{\rev{\sup_{t\ge r^{-2}}\phi^t_0(\varepsilon')}},
\end{gathered}
\end{equation}
then, recalling the constants $K_2$, $C_2$ from (A3),
\rev{for all sufficiently large $n$} we have
\begin{equation}\label{eq:claim-Pft-1}
\P [ f^{t} \notin B_{N,r,\delta,\varepsilon'} ] 
\le \rev{\frac{4C_2}{K_2\varrho}r^{2a_2-2-\varrho}e^{-K_2r^\varrho}}
+ e^{-N^2/8}.
\end{equation}
We postpone the verification of this claim \eqref{eq:cond-claim}$\Rightarrow$\eqref{eq:claim-Pft-1}
till the end of the proof of the theorem.

Assuming \eqref{eq:cond-claim}$\Rightarrow$\eqref{eq:claim-Pft-1}
holds, we now provide sufficient conditions for \eqref{eq:cond-claim} to hold.
\rev{Suppose $r>\max\{1,1/r_0\}$ and $0<\varepsilon'<1/2$.
For every $t\ge r^{-2}$, set
\[
a_t:=\frac{1}{r\sqrt t}\le1,
\qquad \calX_t:=a_t\calX\subset[0,1]^D.
\]
For $0<u\le r_0$,
$
\calN(u,\calX_t,\|\cdot\|_\infty)
=\calN(u/a_t,\calX,\|\cdot\|_\infty)
\le\calN(u,\calX,\|\cdot\|_\infty)
\le C_{\calX}u^{-\varrho},
$
and thus every $\calX_t$ satisfies (A1) with the original parameters $\varrho,C_{\calX},r_0$. 
We are to apply Lemma \ref{lemma:small_ball} to $\calX_t$.
The ambient dimension and kernel are also unchanged,
so the explicit choice of $K'$, and hence of $C_4$ as introduced earlier in Part I
applies uniformly to all these domains $\calX_t$.
Let $g_t$ be the centered Gaussian process on $\calX_t$ with kernel
$h_{r^{-2}}$. Since
\[
h_{r^{-2}}(a_tx,a_ty)
=h\left(r^2a_t^2\|x-y\|^2\right)=h_t(x,y),
\qquad x,y\in\calX,
\]
the processes $x\mapsto g_t(a_tx)$ and $f^t$ have the same law.
Applying Lemma \ref{lemma:small_ball} on $\calX_t$ at bandwidth
$r^{-2}<\min\{r_0^2,1\}$, with this $C_4$, gives
$
\phi_0^t(\varepsilon')
=-\log\P\left(\|g_t\|_{\infty,\calX_t}\le\varepsilon'\right)
\le C_4r^{\varrho}(\log(r/\varepsilon'))^{D+1}.
$
Consequently,
\[
\sup_{t\ge r^{-2}}\phi_0^t(\varepsilon')
\le C_4r^{\varrho}(\log(r/\varepsilon'))^{D+1}.
\]
Separately, at any fixed $x\in\calX$, $f^t(x)$ has the $N(0,1)$
distribution for every $t$. Hence there exists $\varepsilon'_1>0$
such that, for $0<\varepsilon'<\varepsilon'_1$ and every $t>0$,
$
e^{-\phi_0^t(\varepsilon')}
\le\P(| \calN(0,1)|\le\varepsilon')<1/4.
$
This also implies the condition
$e^{-\sup_{t\ge r^{-2}}\phi_0^t(\varepsilon')}<1/4$
in \eqref{eq:cond-claim}.}
In summary, we have that
\begin{equation}\label{eq:cond-claim-sufficient}
\begin{split}
& \varepsilon' < \min\{ 1/2 , \varepsilon'_1\},
\quad 
N^2 \ge 16  C_4  r^\varrho (\log(r/\varepsilon'))^{1+D}, \\
&~~~
r > \delta, 
\quad
r > \max \left\{
1,  \frac{1}{r_0}, \frac{1}{\sqrt{c_3}}n^{\frac{1}{2s+ \varrho}}(\log n )^{\frac{2(1+D)}{(2+\varrho/s)\varrho}} \right\},
\end{split}
\end{equation}
will imply \eqref{eq:cond-claim}. 
Now we have that under the condition \eqref{eq:cond-claim-sufficient}, \eqref{eq:claim-Pft-1} holds.

Recall the definition of $\varepsilon_n = \bar{C}_1 n^{\frac{-s}{2s+\varrho}}
		(\log n)^{\frac{ 1+D}{2 +\varrho/s}}$ 
as in \eqref{eq:choice-epsn-all-three-goals}.
Define $r_n$ and $N_n$ by 
\begin{equation}
\label{eq:cond-claim-sufficient-final}
r_n^{\varrho} = \frac{8}{K_2}n \varepsilon^2_n, 
\ \ 
N_n^2 = \max\{32, \frac{128 C_4 }{K_2}\} n \varepsilon^2_n(\log(r_n/\varepsilon_n))^{1+D}.
\end{equation}
By \eqref{eq:choice-epsn-all-three-goals},  
$\bar{C}_1^2 \ge c_3^{-\varrho/2} K_2 > c_3^{-\varrho/2} K_2/8$,
which gives that 
$r_n > \frac{1}{\sqrt{c_3}}n^{\frac{1}{2s+ \varrho}}(\log  n )^{\frac{2(1+D)}{(2+\varrho/s)\varrho}} $.
Then, for any sequence of $\delta_n$ s.t. $\delta_n < r_n$ for large $n$ ($\delta_n$ to determined below), 
one can verify that for large enough $n$,
the quadruple $(N, r, \delta, \varepsilon' ) = (N_n, r_n, \delta_n, \varepsilon_n )$ satisfy \eqref{eq:cond-claim-sufficient}.

This gives that, assuming  $\delta_n < r_n$ for large $n$, 
then with large enough $n$, 
\eqref{eq:claim-Pft-1} holds at
 $(N, r, \delta, \varepsilon' ) = (N_n, r_n, \delta_n, \varepsilon_n )$, namely
\[
\P [ f^{t} \notin B_{N_n,r_n, \delta_n, \varepsilon_n} ] 
\le \rev{\frac{4C_2}{K_2\varrho}r_n^{2a_2-2-\varrho}e^{-K_2r_n^\varrho}}
+ e^{-N_n^2/8}.
\]
\rev{By \eqref{eq:cond-claim-sufficient-final}, $K_2r_n^\varrho=8n\varepsilon_n^2$,
$N_n^2/8\ge4n\varepsilon_n^2(\log(r_n/\varepsilon_n))^{1+D}$, and
$\log r_n=o(n\varepsilon_n^2)$.  Hence the
two terms on the right-hand side are together bounded by
$\exp(-4n\varepsilon_n^2)$ for all sufficiently large $n$.}
Thus, to prove \eqref {eq:pf-A1A1prime-goal2} with $B_n$ defined to be $B_{N_n,r_n, \delta_n, \varepsilon_n}$,
it suffices to choose $\delta_n$ s.t. $\delta_n < r_n$ for large $n$. 
We will show this is necessarily the case in our proof in Part III,
where we will choose $\delta_n$ to prove  \eqref{eq:pf-A1A1prime-goal3} which also involves $B_n = B_{N_n,r_n, \delta_n, \varepsilon_n}$.

\vspace{10pt}
\noindent
$\bullet$ \underline{Part III:  To prove  \eqref{eq:pf-A1A1prime-goal3}}.
We first derive two useful facts.
First,  let the constant $\tau_h$  be as in Lemma \ref{lemma:rkhs_lip},
and  $\tau_h$ is a fixed positive constant determined by the spectral measure $\mu$. 
 For any  $t > \delta^{-2}$,  by Lemma \ref{lemma:rkhs_lip}, 
every element of $N\tilde{\mathbb{H}}^1_{t}$ is uniformly at most $\delta  \sqrt{D} \tau_h N$ distant from a constant function for a constant in the interval $[-N, N]$. 
Therefore, we have

\vspace{5pt}
(Fact 1): For $\varepsilon' > \delta \sqrt{D}\tau_h N$ and $N > \varepsilon'$,
 \begin{equation}\label{eq:rkhs_cov_bound_1}
     \calN(3 \varepsilon', \bigcup_{t > \delta^{-2}}(N \tilde{\mathbb{H}}^1_{t})+\varepsilon' \mathbb{B}_1,||\cdot||_{\infty}) 
 \le \calN(\varepsilon',  [-N,N], |\cdot|) 
 \le \frac{2N}{\varepsilon'}.
 \end{equation}

Meanwhile, observe that we always have 
\begin{align*}
 \log \calN( 2 \varepsilon', N  {(\frac{r}{\delta})^{D/2}} \tilde{\mathbb{H}}^1_{r^{-2}}+\varepsilon' \mathbb{B}_1,||\cdot||_{\infty}) 
   & \le \log \calN(\varepsilon', N  {(\frac{r}{\delta})^{D/2}} \tilde{\mathbb{H}}^1_{r^{-2}}, ||\cdot||_{\infty}) \\
   & = \log \calN( \frac{\varepsilon'}{N}  { (\frac{\delta}{r})^{D/2} },  \tilde{\mathbb{H}}^1_{r^{-2}}, ||\cdot||_{\infty}). 
\end{align*}
To bound the r.h.s., we use Lemma \ref{lemma:unit_ball_covering_general},
and let the constant $K$ be as therein.
By Lemma \ref{lemma:unit_ball_covering_general},  
if $r^{-1} < r_0$ and $ \frac{\varepsilon'}{N} (\frac{\delta}{r})^{D/2} < 1/2$,
 then 
 \begin{align*}
   \log \calN( \frac{\varepsilon'}{N} {(\frac{\delta}{r})^{D/2}}, \tilde{\mathbb{H}}^1_{r^{-2}}, ||\cdot||_{\infty}) 
    \le K r^{\varrho} (\log(\frac{N {(r/\delta)^{D/2}}}{\varepsilon'}))^{D+1}.
   \end{align*}
This gives the following fact

\vspace{5pt}
(Fact 2):
As long as 
$  r > \delta$, $  r> \max\{ 1, \frac{1}{r_0} \}$ and $\varepsilon'/ N  < 1/2$,   
\begin{align}
    \log \calN( 2 \varepsilon', N  {(\frac{r}{\delta})^{D/2}} \tilde{\mathbb{H}}^1_{r^{-2}}+\varepsilon' \mathbb{B}_1,||\cdot||_{\infty}) 
   \le K r^\varrho (\log(\frac{N {(r/\delta)^{D/2}}}{\varepsilon'}))^{1+D}.
    \label{eq:controlling_rkhs_covering}
\end{align}

Having these two facts in hand, recall that $\varepsilon_n$, $N_n$, $r_n$ have been specified, we now set 
 \begin{equation}\label{eq:def-deltan-proof-3.1-goal3}
 \delta_n = \varepsilon_n/(2\sqrt{D}\tau_h N_n),
 \end{equation}
 and thus $\varepsilon_n > \delta_n \sqrt{D}\tau_h N_n$ for all $n$.
By definition, as $n$ increases, 
$\varepsilon_n = o(1)$, 
$N_n \to \infty$,
$r_n \to +\infty$, 
$\delta_n \sim \varepsilon_n/N_n = o(1)$.
Using our construction of $(N_n, r_n, \delta_n, \varepsilon_n)$, one can verify that for large enough $n$,
  \begin{equation}\label{eq:N-r-eqn}
\frac{1}{2}   N_n > \varepsilon_n >  \delta_n \sqrt{D}\tau_h N_n, 
     \quad r_n > \delta_n,
     \quad r_n > 
     \max\{ 1, \frac{1}{r_0} \}.
 \end{equation}
 By now, \eqref{eq:cond-claim-sufficient} fully holds with $(N, r, \delta, \varepsilon' ) = (N_n, r_n, \delta_n, \varepsilon_n )$.
\eqref{eq:N-r-eqn} ensures that $(\varepsilon', \delta, N) = ( \varepsilon_n, \delta_n, N_n)$ satisfies the requirement of (Fact 1), which implies that for large $n$, \eqref{eq:rkhs_cov_bound_1} holds with $( \varepsilon_n, \delta_n, N_n)$.
Meanwhile,  \eqref{eq:N-r-eqn}  and \eqref{eq:cond-claim-sufficient}
also ensure that $(N,r,\delta,\varepsilon')= (N_n, r_n, \delta_n, \varepsilon_n)$ satisfies the requirement of (Fact 2), and then we have \eqref{eq:controlling_rkhs_covering} hold with $(N_n, r_n, \delta_n, \varepsilon_n)$. 

Let $B_n = B_{N_n, r_n, \delta_n, \varepsilon_n}$,
by definition \eqref{eq:def-B-N-r-delta-eps}, 
\begin{align*}
 \calN( 3\varepsilon_n, B_n, ||\cdot||_{\infty}) 
 & \le  \calN( 3\varepsilon_n,  N {(\frac{r}{\delta})^{D/2}} \tilde{\mathbb{H}}^1_{r^{-2}}+ \varepsilon' \mathbb{B}_1 , ||\cdot||_{\infty})  \\
 &~~~
 +  \calN( 3\varepsilon_n, \bigcup_{t > \delta^{-2}}(N \tilde{\mathbb{H}}^1_{t})+\varepsilon' \mathbb{B}_1, ||\cdot||_{\infty}).
\end{align*}
By \eqref{eq:controlling_rkhs_covering}, the first term can be bounded by  
\[
\calN( 2\varepsilon_n,  N {(\frac{r}{\delta})^{D/2}} \tilde{\mathbb{H}}^1_{r^{-2}}+ \varepsilon' \mathbb{B}_1 , ||\cdot||_{\infty})
\le \exp \left\{ K r_n^\varrho (\log(\frac{N_n {(r_n/\delta_n)^{D/2}} }{\varepsilon_n}))^{1+D} \right\}.
\]
The second term can be bounded by $2N_n /\varepsilon_n$ due to \eqref{eq:rkhs_cov_bound_1}. 
Putting together, we have  
\begin{align} \label{eq:B_n_cover_bound_1}
 \calN( 3\varepsilon_n, B_n, ||\cdot||_{\infty}) 
 \le \exp  \left\{ K r_n^\varrho (\log(\frac{N_n {(r_n/\delta_n)^{D/2}} }{\varepsilon_n}))^{1+D} \right\}
 + \frac{2N_n}{\varepsilon_n}.
\end{align}
Note that $N_n > \varepsilon_n$, $r_n \to +\infty$, $r_n/\delta_n >1$ and $N_n/\varepsilon_n \to +\infty$, 
we have both terms in the r.h.s. of \eqref{eq:B_n_cover_bound_1} greater than 2 with large $n$. 
Then, by an elementary inequality that  for any $x \ge 2$, $y \ge 2$,
$\log(x+y) \le \log(x) + \log(y)$,
\eqref{eq:B_n_cover_bound_1} gives that, with large $n$,   
\begin{align}
     \log \calN( 3\varepsilon_n, B_n, ||\cdot||_{\infty}) 
 \le 
 K r_n^\varrho (\log(\frac{N_n {(r_n/\delta_n)^{D/2}} }{\varepsilon_n}))^{1+D}+ \log(\frac{2N_n}{\varepsilon_n}). 
 \label{eq:log-cover-proof-3}
 \end{align}

Recall our definition of $(N_n, r_n, \delta_n, \varepsilon_n)$ and also $\varepsilon_n$ and $\bar \varepsilon_n$, 
where we have chosen $\bar C_1$ and $\bar C_2$ is to be determined. 
Inserting all these into \eqref{eq:log-cover-proof-3}, 
we choose a large enough constant $\bar C_2$
to ensure that, at large enough $n$,
 the r.h.s. of \eqref{eq:log-cover-proof-3} is upper bounded by $ n \bar{\varepsilon}^2_n$ 
and $\bar{\varepsilon}_n > 3\varepsilon_n$.
Then we have 
 $$
 \log \calN(\bar{\varepsilon}_n,B_n,||\cdot||_{\infty}) 
 \le \log \calN(3 {\varepsilon}_n,B_n,||\cdot||_{\infty}) 
 \le n \bar{\varepsilon}^2_n,
 $$
and this proves \eqref{eq:pf-A1A1prime-goal3}.
Such $\bar C_2$ can be chosen as
\[
\bar C_2 = 4 \vee \sqrt{\frac{8K}{K_2}} \Big( \frac{D}{ (2s + \varrho) \wedge 2}+ 1 \Big)^{(1+D)/2}+1.
\]
First, $\bar C_2 \ge 4$ implies that $\bar{\varepsilon}_n > 3\varepsilon_n$ when $\log n > 1$.
It remains to show that 
$$
 K r_n^\varrho (\log(\frac{N_n {(r_n/\delta_n)^{D/2}} }{\varepsilon_n}))^{1+D}
 + \log(\frac{2N_n}{\varepsilon_n})
 \le n \bar{\varepsilon}^2_n
 = \bar{C}_2^2 n  \varepsilon_n^2 (\log n)^{1+D}.$$
By that
 $\bar C_2^2 \ge \frac{8K}{K_2} \Big( \frac{D}{ (2s + \varrho) \wedge 2}+ 1 \Big)^{1+D} +1$, it suffices to have 
\begin{align*}
K r_n^\varrho (\log \frac{N_n {(r_n/\delta_n)^{D/2}} }{\varepsilon_n} )^{1+D} 
& \le \frac{8K}{K_2} n  \varepsilon_n^2 \left( \frac{D}{ (2s + \varrho) \wedge 2}+ 1 \right)^{1+D}(\log n)^{1+D}, \\
 \log(\frac{2N_n}{\varepsilon_n})
& \le n  \varepsilon_n^2 (\log n)^{1+D}.
\end{align*}
Both can be verified at large $n$ by inserting the definitions of  $N_n$, $r_n$, $\delta_n$, $\varepsilon_n$, which gives that 
$\frac{r_n}{\delta_n} = C \frac{N_n}{\varepsilon_n} (n \varepsilon_n^2)^{1/\varrho}$ for a constant $C$, and that
$
\log \frac{N_n}{\varepsilon_n} = ( \frac{1}{2} + o(1)) \log n$,
and 
$
 \log ( n \varepsilon_n^2) = ( \frac{\varrho}{2s + \varrho} + o(1)) \log n.$

Finally, since our construction of $(N_n, r_n, \delta_n, \varepsilon_n)$ ensures that
$r_n > \delta_n$ for large $n$, as has been shown in \eqref{eq:N-r-eqn}, 
by the argument at the end of the proof of Part II,
we have also finished the proof of \eqref {eq:pf-A1A1prime-goal2}.

\vspace{5pt}
\noindent
\underline{Proof of \eqref{eq:claim-Pft-1} under \eqref{eq:cond-claim}}:
Note that 
\[
\P [ f^{t} \notin B_{N,r,\delta,\varepsilon'} ] 
\le 
\P [t < r^{-2} ]
+ \int_{r^{-2}}^{\infty} 
\P [ f^{t}\notin B_{N,r,\delta,\varepsilon'}|t] p(t) dt.
\]

For the first term on the r.h.s.,
since  $r^{-2} 
< c_3 n^{\frac{-2}{2s+\varrho}} (\log n )^{\frac{-4(1+D)}{(2+\varrho/s)\varrho}}$, i.e., the first condition in \eqref{eq:cond-claim},
together with our assumption on the prior in Assumption \ref{assump:A3}(A3), we have
\begin{align}\label{eq:proof-claim-1st-term}
\P (t < r^{-2}) 
&\le \int_0^{r^{-2}} C_2 t^{-a_2} \exp(-K_2 t^{-\varrho/2}) dt \nonumber\\
&\rev{=\frac{2C_2}{\varrho}K_2^{(2-2a_2)/\varrho}
\int_{K_2r^\varrho}^{\infty}u^{(2a_2-2)/\varrho-1}e^{-u}\,du}\nonumber\\
&\rev{\le \frac{4C_2}{K_2\varrho}
r^{2a_2-2-\varrho}e^{-K_2r^\varrho}}.
\end{align}
\rev{The last inequality uses
$\int_x^\infty u^{b-1}e^{-u}\,du\le2x^{b-1}e^{-x}$ for all sufficiently
large $x$, with $b=(2a_2-2)/\varrho$; the condition in
\eqref{eq:cond-claim} ensures that $K_2r^\varrho\to\infty$.}

To bound the second term, for any fixed $t > r^{-2}$, we have 
\begin{align}
\P [ f^{t} \notin B_{N,r,\delta,\varepsilon'}|t ] 
& \le \P (f^{t} \notin N \tilde{\mathbb{H}}^1_{t}+\varepsilon' \mathbb{B}_1|t) 
\quad \text{(by \eqref{eq:NHt1+epsB1-subset-B-proof1})}\nonumber \\
&\le 1- \Phi(\Phi^{-1}(e^{-\phi^{t}_0(\varepsilon')})+N) 
\nonumber \\ 
&\le 1- \Phi(\Phi^{-1}(e^{- \rev{\sup_{u\ge r^{-2}}\phi^u_0(\varepsilon')}})+N),
\label{eq:proof-rhs-second-term}
\end{align}
where the second inequality is by  Borell's inequality \cite{borell1975brunn}, see also \cite[Theorem 5.1]{van2008reproducing};
\rev{the third inequality follows from
$\phi^t_0(\varepsilon')\le\sup_{u\ge r^{-2}}\phi^u_0(\varepsilon')$
and monotonicity of $\Phi$ and $\Phi^{-1}$.}
To proceed, under the last two conditions in \eqref{eq:cond-claim}, 
by the estimate of the quantile of normal density in \cite[Lemma 4.10]{van2009adaptive}, 
we have $N \ge -2 \Phi^{-1}( \exp \{- \rev{\sup_{t\ge r^{-2}}\phi^t_0(\varepsilon')}\})$.
As a result, the right hand side of \eqref{eq:proof-rhs-second-term}
is bounded by $1 - \Phi(N/2) \le e^{- N^2/8}$. 
This gives that 
\[
\P [ f^{t} \notin B_{N,r,\delta,\varepsilon'}|t ]  \le 
e^{- N^2/8}, \quad \forall t > r^{-2}.
\]
Then, 
\begin{equation}\label{eq:proof-claim-2nd-term}
\int_{r^{-2}}^{\infty} 
\P [ f^{t}\notin B_{N,r,\delta,\varepsilon'}|t] p(t) dt
\le e^{- N^2/8}
\int_{r^{-2}}^{\infty}  p(t) dt
\le e^{- N^2/8}.
\end{equation}
Combining \eqref{eq:proof-claim-1st-term} and \eqref{eq:proof-claim-2nd-term}
proves \eqref{eq:claim-Pft-1}.
 \end{proof}

To prove Theorems \ref{thm:fix-design-estimator} and \ref{thm:contraction_2_norm},
we  let $\varepsilon_n $ and $\bar \varepsilon_n $ be defined as in \eqref{eq:def_ep_ep_bar-thm3.1} 
 with the constants $\bar{C}_1$, $\bar{C}_2$ chosen as in the proof of Theorem \ref{thm:contraction_n_norm}.
 We also need the following lemma adapted from \cite[Lemma 6.1]{yang2016bayesian}, and we include a proof for completeness.

\begin{lemma}\label{lemma:scale_prob_k}
Under the condition of Theorem \ref{thm:contraction_n_norm},

 (i) Fixed design. 
 Let $\Pr^{(n)}_{Y | X}$ and $ \E^{(n)}_{Y|X}$ be for the joint distribution of $\{Y_i \}_{i=1}^{n}$ conditioning on fixed $\{ X_i \}_{i=1}^n$. 
  There exist  $c_{5,Y} >0$
  and a sequence of measurable sets $A_{n,Y}$ under $\Pr^{(n)}_{Y | X}$ satisfying  
  that $\Pr^{(n)}_{Y|X } (A_{n,Y}^c) \to 0$
  and, when $n$ is sufficiently large, 
     $$
     \E^{(n)}_{Y|X} \left( 
     {\bf 1}_{A_{n,Y}} \Pi(||f^t - f^* ||_n \ge \bar{\varepsilon}_n|\{ X_i,Y_i \}_{i=1}^{n}) \right)
     \le  \exp\{-c_{5,Y} n\varepsilon_n^2 \}.
     $$

 (ii) Random design. 
 Let $\Pr^{(n)}$ and $ \E^{(n)}$ be for the joint distribution of $\{X_i, Y_i \}_{i=1}^{n}$.
  There exist  $c_5 >0$
  and a sequence of measurable sets $A_{n}$ under $\Pr^{(n)}$ satisfying  
  that $\Pr^{(n)} (A_{n}^c) \to 0$
  and, when $n$ is sufficiently large, 
     $$
     \E^{(n)} \left( 
     {\bf 1}_{A_{n}} \Pi(||f^t - f^*  ||_n \ge \bar{\varepsilon}_n|\{ X_i,Y_i \}_{i=1}^{n}) \right)
     \le  \exp\{-c_5n\varepsilon_n^2 \}.
     $$     
 \end{lemma}

  \begin{proof}[Proof for Lemma \ref{lemma:scale_prob_k}] 
  First, we prove (ii) in the random design case.
  We recall the definition of 
  $ \varepsilon_n,  r_n,  N_n, \delta_n$,
  and $B_n = B_{N_{n},r_{n},\delta_{n},\varepsilon_{n}}$
from the proof of  Theorem \ref{thm:contraction_n_norm}.
We have already shown in  \eqref{eq:pf-A1A1prime-goal2}\eqref{eq:pf-A1A1prime-goal3} that
\begin{align}\label{eq:3_3_prob_covering}
& \P ( f^{t} \notin B_n ) \le \exp(-4 n \varepsilon^2_{n}), \quad 
 \log \calN(\bar{\varepsilon}_{n},B_{n},||\cdot||_{\infty}) 
  \le n \bar{\varepsilon}^2_n .    
\end{align}
  As a consequence, 
  by inserting our definition of $B_n$ in the proof of Theorem 2.1
in \cite{ghosal2000convergence} and following their steps, we
obtain a sequence of measurable sets $A_n$ such that $\Pr^{(n)}(A_n^c) \to 0$,
and for constant $c_5 > 0$, 
$\E^{(n)}( {\bf 1}_{A_n} \Pi(|| f^t - f^*  || \ge  \bar{\varepsilon}_n|\{ X_i,Y_i \}_{i=1}^{n})) \le  e^{-c_5 n \varepsilon_n^2}$.
This proves the case (ii). 

The proof for (i) in the fixed design case is by adapting the above argument 
from \cite{ghosal2000convergence} to the fixed design case.
Specifically, though that theorem focused on i.i.d. observations, 
its proof can be
adapted to independent but not identically distributed observations.
The extension to  regression with fixed-design follows by applying the techniques in \cite{ghosal2007convergence}. 
  \end{proof}

\begin{proof}[Proof of Theorem \ref{thm:fix-design-estimator}] 
We denote by $X = \{X_i\}_{i=1}^n$ and $Y = \{Y_i\}_{i=1}^n$,
and write $\Pi ( \cdot | \{ X_i,Y_i \}_{i=1}^{n})$ as $ \Pi( \cdot | X, Y) $ for notation brevity. 
Under the assumption of the theorem, Lemma \ref{lemma:scale_prob_k}(i) applies.
Let the event $A_{n,Y}$ be as therein, for large $n$, we have that 
    \begin{equation}\label{eq:apply-lemma7.1(i)-2}
    {\Pr}^{(n)}_{Y|X} 
    \left[ {\bf 1}_{A_{n,Y}}  \Pi(||f^t-f^*||_n \ge \bar{\varepsilon}_n|  X, Y )   >  e^{-c_{5,Y} n\varepsilon_n^2/2}  \right]
     \le e^{-c_{5,Y} n\varepsilon_n^2/2} \to 0.
    \end{equation}    
As a result, there exists a sequence of events $C_{n,Y} \subset A_{n,Y}$ s.t. $\Pr^{(n)}_{Y|X}(C_{n,Y}^c) \to 0$ and with large $n$,
    \begin{equation}\label{eq:them_3_2_eq_1-i}
    \text{under the event $C_{n,Y}$}, 
    \quad
      \Pi( ||f^t-f^*||_n \ge \bar{\varepsilon}_n | X, Y) 
      \le  e^{-c_{5,Y}n\varepsilon_n^2/2}.
     \end{equation}
By definition of $\hat f$, we have
\begin{align}
\hat f - f^*
&=
\int (f^t  -f ^*) {\bf 1}_{  \{ ||f^t-f^*||_n < \bar{\varepsilon}_n \} } d\Pi(f^t | X, Y)
+ \int f^t  {\bf 1}_{  \{ ||f^t-f^*||_n \ge \bar{\varepsilon}_n \} } d\Pi(f^t | X, Y)  \nonumber \\
&~~~ 
-  \int f ^* {\bf 1}_{  \{ ||f^t-f^*||_n \ge \bar{\varepsilon}_n \} } d\Pi(f^t | X, Y)
=: \textcircled{1}+ \textcircled{2} - \textcircled{3}, \label{eq:proof-fix-design-estimator-123}
 \end{align}
 and below we bound $\|  \textcircled{1} \|_n$, $\|  \textcircled{2} \|_n$ and $\|  \textcircled{3} \|_n$ respectively.
 
Because $f^*$ is bounded, let $\|f^*\|_\infty \le M$ for some positive constant $M$.
 The bounds for  $\|  \textcircled{1} \|_n$ and  $\|  \textcircled{3} \|_n$ are straightforward:
 \begin{equation}\label{eq:proof-fix-design-estimator-1}
 \| \textcircled{1}  \|_n \le 
 \int \| f^t  -f ^*\|_n {\bf 1}_{  \{ ||f^t-f^*||_n < \bar{\varepsilon}_n \} } d\Pi(f^t | X, Y)
 \le \bar{\varepsilon}_n.
 \end{equation}
For $\textcircled{3}$, because $ \| f^* \|_n \le \| f^* \|_\infty \le M$, we have
$$
 \| \textcircled{3}  \|_n 
 \le   \int \| f ^* \|_n {\bf 1}_{  \{ ||f^t-f^*||_n \ge \bar{\varepsilon}_n \} } d \Pi(f^t | X, Y ) 
  \le M \Pi ( ||f^t-f^*||_n \ge \bar{\varepsilon}_n  | X, Y ),$$
and thus, by \eqref{eq:them_3_2_eq_1-i}, we have that with large $n$,
  \begin{align}
  \text{under the event $C_{n,Y}$,} \quad
   \| \textcircled{3}  \|_n  \le M e^{-c_{5,Y}n\varepsilon_n^2/2}.
  \label{eq:proof-fix-design-estimator-3}
  \end{align}
  To bound  $\|  \textcircled{2} \|_n$, first observe that 
  \begin{align}
\|  \textcircled{2} \|_n
& \le   \int \| f^t \|_n {\bf 1}_{  \{ ||f^t-f^*||_n \ge \bar{\varepsilon}_n \} } d\Pi(f^t | X, Y)  \nonumber \\
& \le \left( \int \| f^t \|_n^2  d\Pi(f^t | X, Y) \right)^{1/2}
   \Pi(  ||f^t-f^*||_n \ge \bar{\varepsilon}_n   | X, Y)^{1/2}, \label{eq:proof-fix-design-estimator-2-a}
  \end{align}
  where the second inequality is by Cauchy-Schwarz.
  The second factor in \eqref{eq:proof-fix-design-estimator-2-a} can be bounded as $e^{-c_{5,Y}n\varepsilon_n^2/4}$ by restricting to the event $C_{n,Y}$;
to control the first factor we utilize more property of GP. Specifically, denote by $p(t | X, Y)$ the marginal posterior of $t$, we have
  \begin{equation}\label{eq:proof-bound-circle2-2nd-moment-rhof-1}
\int \| f^t \|_n^2  d\Pi(f^t | X, Y) 
= \int \int  \frac{1}{n}\| \rho_X (f^t) \|_2^2  dP(  \rho_X (f^t) | X, Y, t)  p(t | X, Y) dt,
 \end{equation}
 where $\rho_X (f) : = (f(X_1), \cdots, f(X_n)) \in \R^n$ for function $f $ on $\calX$.
 For each $t > 0$, the vector $\rho_X (f)$ has the conditional posterior distribution as (see, e.g. \cite{williams2006gaussian})
 \[
 \rho_X (f) | X, Y, t  \sim \calN( \hat{\mu}_t, \hat{\Sigma}_t ),
 \quad
 \hat{\mu}_t := K_t( K_t + \sigma^2 I)^{-1}Y, 
 \quad 
\hat{\Sigma}_t := K_t - K_t( K_t + \sigma^2 I)^{-1} K_t,
 \]
 where $K_t := [ h(X_i, X_j) ]_{i,j=1}^n$ is the $n$-by-$n$ PSD kernel matrix built from $X$.
 As a result, 
 \[
 \int \| \rho_X (f^t) \|_2^2  dP(  \rho_X (f^t) | X, Y, t)
 = \| \hat{\mu}_t \|_2^2 + {\rm Tr}( \hat{\Sigma}_t  ).
 \]
 Using the spectral representation of the matrix $K_t = \sum_{k=1}^n \lambda_k u_k u_k^T$, 
 where $u_k$ are eigenvectors, 
 $[ u_1 | \cdots | u_k]$ forms an orthogonal matrix,
 and $\lambda_k$ are the associated eigenvalues of $K_t$, 
 one can verify  that the operator norm 
 \[ 
 \| K_t( K_t + \sigma^2 I)^{-1} \|_{op} 
 = \max_{k=1,\cdots, n} \frac{\lambda_k}{ \lambda_k + \sigma^2} \le 1, 
 \] 
 and that
 \[
{\rm Tr}  ( \hat{\Sigma}_t  )
 = \sigma^2 \sum_{k=1}^n \frac{\lambda_k}{ \lambda_k + \sigma^2} \le n \sigma^2.
 \]
Then, we have $\| \hat \mu_t  \|_2 \le \| Y\|_2$, and that 
 \[
 \int  \| \rho_X (f^t) \|_2^2  dP(  \rho_X (f^t) | X, Y, t)
 \le  \| Y \|_2^2 + n \sigma^2,
 \]
  which holds for any $t$. Inserting back to \eqref{eq:proof-bound-circle2-2nd-moment-rhof-1}, we have
  \begin{equation}\label{eq:proof-bound-circle2-2nd-moment-rhof-2}
  \int \| f^t \|_n^2  d\Pi(f^t | X, Y) 
\le  \frac{1}{n} \| Y \|_2^2 +  \sigma^2
\le \| Y\|_\infty^2 + \sigma^2.
  \end{equation}
  We can bound $\| Y\|_\infty   \lesssim  \sqrt{\log n}$ with high probability using a union bound (which only utilizes the marginal distribution of each $Y_i | X$):
  given fixed $X$, 
for each $i$,  $\Pr [ |w_i|  > \sigma \alpha] \le 2 e^{-\alpha^2/2}$, thus 
$ \Pr [ \max_{i =1,\cdots,n }|w_i| > \sigma \sqrt{4 \log n} ] \le 2/n$.
This means that, under a sequence of events $E_{n,Y}$ s.t. $\Pr^{(n)}_{Y|X} ( E_{n,Y}^c) \le 2/n \to 0$, we have
\[
|Y_i |\le  |f^*(X_i)| + | w_i| \le M + \sigma \sqrt{4 \log n}, \quad \forall i=1,\cdots,n.
\]
Putting back to \eqref{eq:proof-bound-circle2-2nd-moment-rhof-2} gives that 
\[
  \text{under the event $E_{n,Y}$,} \quad
  \int \| f^t \|_n^2  d\Pi(f^t | X, Y)  \le (  M + \sigma \sqrt{4 \log n} )^2 + \sigma^2.
\]

We are ready to continue \eqref{eq:proof-fix-design-estimator-2-a} as follows:  with large $n$,  
  \[
  \text{ under $E_{n,Y} \cap  C_{n,Y}$,}  \quad
  \|  \textcircled{2} \|_n \le
  3  (M + \sigma\sqrt{ \log n} )
 e^{-c_{5,Y}n\varepsilon_n^2/4}.
\]
Combined with \eqref{eq:proof-fix-design-estimator-1}\eqref{eq:proof-fix-design-estimator-3}, 
this allows us to apply triangle inequality to \eqref{eq:proof-fix-design-estimator-123} and have that, under $E_{n,Y} \cap  C_{n,Y}$,
\begin{align*}
\| \hat f - f^*\|_n
& \le \|  \textcircled{1} \|_n + \|  \textcircled{2} \|_n + \|  \textcircled{3} \|_n \\
& \le  \bar{\varepsilon}_n
+ M e^{-c_{5,Y}n\varepsilon_n^2/2}
+   3  (M + \sigma\sqrt{ \log n} )
e^{-c_{5,Y}n\varepsilon_n^2/4}.
\end{align*}
 When $n$ is large enough such that 
 both $M e^{-c_{5,Y}n\varepsilon_n^2/2}$ and 
 $  3  (M + \sigma\sqrt{ \log n} )
 e^{-c_{5,Y}n\varepsilon_n^2/4}$ are less than $ \bar{\varepsilon}_n$, 
we have 
     $||\hat{f} -f^*||_n \le 3 \bar \varepsilon_n $
     under the event $E_{n,Y} \cap  C_{n,Y}$.
Since $\Pr^{(n)}_{Y|X}( E_{n,Y} \cap  C_{n,Y}) \to 1$, this proves the theorem. 
 \end{proof}

\subsection{Random design results}

To prove the random design results, we need the following lemma for comparing $||\cdot||_n$ and $||\cdot||_2$ based on empirical process theory \cite{geer2000empirical}.
Let $H_{B}(\varepsilon, \mathcal{F}, ||\cdot||)$ denote the $\varepsilon$-bracketing entropy of a function space $\mathcal{F}$ with respect to a norm $||\cdot||$.

\begin{lemma}[Lemma 5.16 in \cite{geer2000empirical}]
\label{lemma:ep_2_n_compare}
    Suppose $X_i$ are i.i.d drawn from a distribution,
    and denote by $\Pr^{(n)}_{X}$ the joint law of $\{X_i\}_{i=1}^n$.
If for some 
\rev{$\bar{M} \ge 1$},
the function class $\mathcal{F}$ satisfies that 
$\sup_{f \in \mathcal{F}} ||f||_{\infty}\le \bar M$, 
$\omega > 0 $ satisfies that 
$n {\bar M}^{-2}\omega^2 \ge H_B( {\bar M}^{-1}\omega, \mathcal{F}, ||\cdot||_2)$, 
and  $\eta \in (0,1)$,
then there exists $C_5 > 0$ such that
$$
{\Pr}_{X}^{(n)}
\left[ \sup_{f \in \mathcal{F},||f||_2 \ge 32 \omega/\eta} \left| \frac{||f||_n}{||f||_2} -1 \right| \ge \eta \right] 
\le 8 \exp\{ -C_5 n \bar M^{-2} \omega^2 \eta^2 \}.
$$
 \end{lemma}

We are ready to prove 
Theorem \ref{thm:contraction_2_norm}
using Lemma \ref{lemma:scale_prob_k}(ii)
and Lemma \ref{lemma:ep_2_n_compare}.

\begin{proof}[Proof of Theorem \ref{thm:contraction_2_norm}] 
We first prove the posterior contraction rate in $||\cdot||_2$ under the random design.

For $ B_{n} = {{B_{N_n, r_n,\delta_n, \varepsilon_n}}}$ as constructed in the proof of Theorem \ref{thm:contraction_n_norm}, 
we have that \eqref{eq:3_3_prob_covering} holds. 
Under the assumption of the theorem, Lemma \ref{lemma:scale_prob_k}(ii) applies. Let $A_n$ be as therein.
\rev{By the first inequality in \eqref{eq:3_3_prob_covering}, the prior mass bound
\eqref{eq:pf-A1A1prime-goal1}, and the argument in the proof of
\cite[Lemma 1]{ghosal2007convergence}, using the high-probability lower bound for the
posterior denominator in that argument and Markov's inequality, we may replace $A_n$
by a subset (still denoted by $A_n$) such that $\Pr^{(n)}(A_n)\to1$ and the conclusion
of Lemma \ref{lemma:scale_prob_k}(ii) remains valid. Then, for some $c_6>0$,}
\begin{equation}\label{eq:3_3_condition_ball_prob}
\text{under the event $A_n$}, \quad
\Pi(f^t \notin B_n|\{X_i,Y_i \}_{i}^n) \le e^{-c_6 n \varepsilon_n^2}.
\end{equation}
Meanwhile, 
following the same argument as in the proof of Theorem \ref{thm:fix-design-estimator} to derive \eqref{eq:them_3_2_eq_1-i}
but now applied under the random design by Lemma \ref{lemma:scale_prob_k}(ii),
we have that there exists a sequence of events $C_n \subset A_n$ s.t. 
$\Pr^{(n)}(C_{n}) \to 1$ and 
   \begin{equation}\label{eq:them_3_2_eq_2}
  \text{under $C_n$}, \quad 
     \Pi(||f^t-f^*||_n \ge \bar{\varepsilon}_n|\{ X_i,Y_i \}_{i=1}^{n}) \le  e^{-c_{5} n\varepsilon_n^2/2}.
     \end{equation}
Combining \eqref{eq:3_3_condition_ball_prob} and \eqref{eq:them_3_2_eq_2}, 
with $c_7 := \min \{ c_5/2, c_6 \}$,
we have that, under $A_n \bigcap C_n = C_n$, 
 \begin{equation}\label{eq:3_3_prob_tail_2}
     \Pi(||f^t-f^*||_n \le \bar \varepsilon_n, f^t \in B_n|\{X_i,Y_i\}_{i=1}^n) \ge 1 - 2e^{-c_7n \varepsilon_n^2}.
 \end{equation}
Let $B_n^{(M)} = \{f_{M}: f \in B_n \}$.
By definition, 
$f^t \in B_n$ implies that $f^t_M \in B_n^{(M)}$.
Then, \eqref{eq:3_3_prob_tail_2} gives that
\begin{align}\label{eq:tail_under_An}
\text{under $C_n$,} \quad
    \Pi( ||f^t_M-f^*||_n \le \bar{\varepsilon}_n, f^t_M \in B^{(M)}_n|\{X_i,Y_i\}_{i=1}^n) 
    \ge 1 - 2e^{-c_7n \varepsilon^2_n}. 
\end{align}

We are to apply Lemma \ref{lemma:ep_2_n_compare} with
$\mathcal{F} = B^{(M)}_n - f^*$, 
\rev{$\bar M = \max\{2M,1\}$},
$\rev{\omega = 2\bar M \bar{\varepsilon}_n}$, 
and $\eta= 1/2$.
To verify that the needed conditions by Lemma \ref{lemma:ep_2_n_compare} are satisfied:
the boundedness of members in $\mathcal{F} $ in $\| \cdot \|_\infty$ by $\bar M$ is by construction,
and it remains to verify that
\rev{$H_B( 2\bar \varepsilon_n, \mathcal{F}, ||\cdot||_2) 
\le 4n \bar \varepsilon_n^2 $.}
\rev{Indeed, an $\bar\varepsilon_n$-cover in the sup norm gives brackets of
$L^2$-width at most $2\bar\varepsilon_n$,}
and the covering entropy of $B^{(M)}_n$ is upper bounded by that $B_n$,
then 
$$
H_B ( \rev{2\bar \varepsilon_n}, \mathcal{F}, \|\cdot \|_2)
\le 
\log \calN  ( \bar \varepsilon_n, B^{(M)}_n, \| \cdot \|_\infty)
\le 
\log \calN  ( \bar \varepsilon_n, B_n, \| \cdot \|_\infty)
\le n \bar \varepsilon_n^2
\rev{\le4n\bar\varepsilon_n^2},
$$
where its {{second but}} last inequality is by the 2nd inequality in \eqref{eq:3_3_prob_covering}.

By Lemma \ref{lemma:ep_2_n_compare}, there exist a sequence of events $E_n$ with 
$\Pr^{(n)}(E_n) \to 1$ as $ n \to \infty$
($E_n$ is originally event with respect to  $\Pr_X^{(n)}$, and can be viewed as event with respect to $\Pr^{(n)}$),
such that
\begin{equation}\label{eq:ep-bound-proof-3}
\text{under $E_n$,} \quad
    \rev{\frac{1}{2}\le
    \frac{||f_M^t-f^*||_n}{||f_M^t-f^*||_2}
    \le\frac{3}{2},\quad
    \forall f_M^t\in B_n^{(M)}\ \text{such that}\quad
    ||f_M^t-f^*||_2\ge}
    \rev{128\bar M}\rev{\bar\varepsilon_n.}
\end{equation}
We restrict to when $\{X_i,Y_i\}_{i=1}^n$ are 
under the event $C_n \bigcap E_n$,
and then both \eqref{eq:tail_under_An} and \eqref{eq:ep-bound-proof-3} hold. 
We consider the set 
$\{ ||f^t_M-f^*||_n \le \bar{\varepsilon}_n, f^t_M \in B^{(M)}_n \}$
on the l.h.s. of \eqref{eq:tail_under_An}. 
Restricted to this set, either $ ||f_M^t - f^*||_2 < 
\rev{128\bar M} \bar{\varepsilon}_n$,
or, together with \eqref{eq:ep-bound-proof-3}, 
$||f_M^t - f^*||_2 \le 2 ||f_M^t - f^*||_n \le 2 \bar{\varepsilon}_n$.
Thus, 
$ \{ ||f^t_M-f^*||_n \le \bar{\varepsilon}_n, f^t_M \in B^{(M)}_n \}
\subset 
\{ ||f^t_M-f^*||_2 \le \max\{ 
\rev{128\bar M},2 \} \bar{\varepsilon}_n, f^t_M \in B^{(M)}_n  \}$.
As a result, defining 
\[ C_7 : =  
\rev{128\bar M} \vee 2, 
\]
we have that,
under $C_n \bigcap E_n$ whose probability goes to 1,
  \begin{align}
 \Pi( ||f^t_M-f^*||_2  & \le C_7 \bar{\varepsilon}_n| \{X_i,Y_i\}_{i=1}^n) 
  \ge 
  \Pi ( ||f^t_M-f^*||_2 \le C_7 \bar{\varepsilon}_n, f^t_M \in B^{(M)}_n 
  |\{X_i,Y_i\}_{i=1}^n  ) \nonumber \\
 &~~~~~~
  \ge   \Pi (  ||f^t_M-f^*||_n \le \bar{\varepsilon}_n, f^t_M \in B^{(M)}_n  |\{X_i,Y_i\}_{i=1}^n )
  \ge 1 - 2e^{-c_7n \varepsilon^2_n},
 \label{eq:pos_contr_2_norm_1}
  \end{align}
  where the last inequality is by \eqref{eq:tail_under_An}.
This shows that the posterior contraction rate in $|| \cdot||_2$ is at least $C_7 \bar{\varepsilon}_n$, and proves the claimed  posterior contraction rate
with $c= \rev{256}$ by that
\[
C_7 \le \rev{256M+128\le256(M+1)}.
\]

\vspace{5pt}
Next, we prove the convergence rate of the posterior mean estimator. 
Observe that 
\[
\int ||f_M^t - f^*||^2_2 d\Pi(f|\{X_i,Y_i\}_{i=1}^n)  
\le  C_7^2 \bar{\varepsilon}^2_n + 4M^2  \Pi(||f^t_M-f^*||_2 > C_7 \bar{\varepsilon}_n|\{X_i,Y_i\}_{i=1}^n). 
\]
Under the event $C_n \bigcap E_n$, 
by  \eqref{eq:pos_contr_2_norm_1}, the r.h.s. is upper bounded by
$ C_7^2 \bar{\varepsilon}^2_n + 4M^2 \cdot  2e^{-c_7n \varepsilon^2_n}$.
When $n$ is large enough such that $ 8M^2  e^{-c_7n \varepsilon^2_n} <  \bar{\varepsilon}^2_n$, we have  that, 
\[
\int ||f_M^t - f^*||^2_2 d\Pi(f|\{X_i,Y_i\}_{i=1}^n)  
\le (C^2_7  +1)\bar{\varepsilon}^2_n.
\]
Note that 
$
\int ||f^t_M - f^*||^2_{2}d\Pi(f^t|\{ X_i,Y_i \}_{i=1}^{n}) =   \int ||f^t_M - \hat{f}_M||^2_{2}d\Pi(f^t|\{ X_i,Y_i \}_{i=1}^{n})  + ||\hat{f}_M -f^*||_2^2,
$
and then
$$
||\hat{f}_M -f^*||_2^2 \le \int ||f_M^t - f^*||^2_2 d\Pi(f|\{X_i,Y_i\}_{i=1}^n) \le (C_7^2  +1) \bar{\varepsilon}^2_n, 
$$
which holds under $C_n \bigcap E_n$, where $\Pr^{(n)}(C_n \bigcap E_n) \to 1$.
Because
\[
\sqrt{ C_7^2+1 }\le C_7 +1 \le 
\rev{256M+129\le256(M+1)},
\] 
this proves the convergence rate of the (truncated) posterior mean estimator
with $c= \rev{256}$.
 \end{proof}

\subsection{Adaptive rate with mis-specified prior}
\label{sec:sub_rate}

In this section, we show that when the intrinsic dimension is not known, under some conditions, we can still obtain posterior contraction but potentially with a sub-optimal rate. We first modify Assumption \ref{assump:A3}(A3) on the prior to (A3'). 

\begin{assumption}\label{assump:A3-mis}
(A3') Given positive constants $\varrho$ and $s$, 
there exist positive constants
\rev{$c_2 > c_1 >0$,}
$c_3$, $a_1$, $a_2$, $K_1$, $K_2$, $C_1$, $C_2$, and $\varrho_{+} \ge \varrho_{-} > 0$, 
such that 
$\varrho_{+} \ge \varrho $ and
 \begin{align}
      &p(t) \ge C_1t^{-a_1} \exp\Big(-\frac{K_1}{t^{
    \varrho_{+}/2}}\Big)\ \mbox{for}\ t \in [c_1 n^{\frac{-2}{2s+\varrho_{+}}} (\log n )^{\frac{2(1+D)}{2s +\varrho_{+}}}, 
    							c_2 n^{\frac{-2}{2s+\varrho_{+}}}(\log n)^{\frac{2(1+D)}{2s +\varrho_{+}}} ], \nonumber \\
      &p(t) \le C_2t^{-a_2} \exp\Big(-\frac{K_2}{t^{\varrho_{-}/2}}\Big)\ \mbox{for}\ t \in (0, c_3 n^{\frac{-2\varrho_{+}}{(2s+\varrho_{+})\varrho_{-}}}		 		
      						(\log n )^{\frac{-4(1+D)}{(2+\varrho_{+}/s)\varrho_{-}}} ]. \nonumber
\end{align}
\end{assumption}

The condition $\varrho_{+} \ge \varrho_{-}$ ensures that, 
when $n$ is large enough,
the two intervals 
for the lower and upper bounds of $p(t)$ to hold
will not overlap,
and thus the required lower and upper bounds can be satisfied at the same time.

\begin{theorem}
\label{thm:contraction-mis}
     Suppose Assumption \ref{assump:A1-A2-prime}(A1)(A2) and 
     Assumption \ref{assump:A3-mis}
     (A3') are satisfied with the same positive factors $\varrho$ and $s$.
     If 
     \[
     \varrho_{-} > \frac{\varrho_{+}}{2s+\varrho_{+}} \varrho,
     \]
     then the posterior contraction rate 
     with respect to $||\cdot||_n$ 
     is at least a multiple of $n^{-r(\varrho, \varrho_\pm,s)} (\log n )^{k}$ where 
       \begin{equation} 
     r(\varrho, \varrho_\pm,s) = 
     \frac{1}{2} 
     \left( 1 - \frac{\varrho_+}{(
     \varrho_-  \wedge \varrho)} \frac{\varrho}{ (2s  + \varrho_+)}
     \right),
     \end{equation}
     with $k= (1+D)/(2 + \varrho_{+}/s)$ if  $\varrho < \varrho_{-}$,
     and 
     $k = \frac{\varrho(1+D)}{2\varrho_{-}+\varrho_{+}\varrho_{-}/s}+\frac{1+D}{2}$ if $\varrho\ge\varrho_{-}$. 
     
     If additionally $\| f^*\|_\infty \le M$ for some constant $M$, 
     then there exists $C > 0$ s.t.
\[
    \Pi( \| f_M - f^*\|_2 > 
    C n^{-r(\varrho, \varrho_\pm,s)} (\log n)^{k}
    | \{X_i,Y_i\}_{i=1}^{n} ) 
    \to 0 \ \ \textit{in probability as} \ n \to \infty,  
\]
with $k$ as above for the two cases respectively.
\end{theorem}

\begin{figure} 
\centering
\includegraphics[width= 1\textwidth]{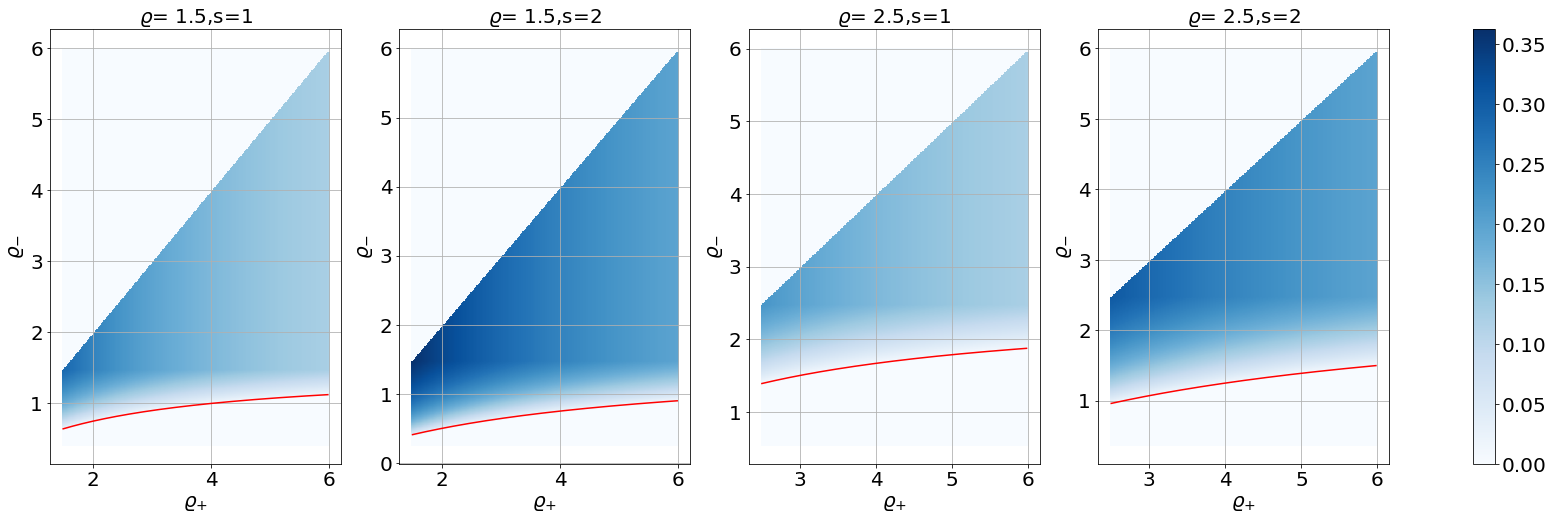}
\caption{
The exponential factor $r(\varrho, \varrho_\pm,s)$ 
plotted as a color field 
over the varying values of $\varrho_\pm$, for fixed $\varrho$ and $s$.
The red lines denote the lower bound
${{\varrho_{-} = \frac{\varrho_{+}}{2s+\varrho_{+}}\varrho}}$,
Given $\varrho$ and $s$,
$r(\varrho, \varrho_\pm,s)$ changes with $\varrho_{\pm}$. $r(\varrho, \varrho_\pm,s)$ gets larger when $\varrho_{+}$ and $\varrho_{-}$ get closer and finally converges to $\frac{1}{2}(1-\frac{\varrho}{2s+\varrho}) = \frac{s}{ 2s + \varrho }$, the minus logarithm of rate with known dimension, when $\varrho_\pm = \varrho$. }
\label{fig1}
\end{figure}

This theorem suggests that the posterior 
can adapt to the low dimensional structure of $\mathcal{X}$ even when we put a prior on the bandwidth not dependent on the intrinsic dimension.
The rate $n^{-r(\varrho, \varrho_\pm,s)}$ is slower than the rate with known dimension $\varrho$,
which is $n^{-{s}/{(2s+\varrho)}} = n^{-\frac{1}{2}( 1-\frac{\varrho}{2s+\varrho})}$. The rate $n^{-r(\varrho, \varrho_\pm,s)}$ gets better when $\varrho_{+}$ and $\varrho_{-}$ get closer, and  when $\varrho_{-} = \varrho_{+} = \varrho$,
we recover the rate
$n^{-{s}/{(2s+\varrho)}}$
as has been proved in Section \ref{subsec:adapt-rates-general}.
Larger $s$ and smaller $\varrho_{+}$ impose less constraints on the feasible $\varrho_{-}$ to prevent the rate  $n^{-r(\varrho, \varrho_\pm,s)}$ from degenerating.
Illustration of the feasible region and the change of $r(\varrho, \varrho_\pm,s)$  with $\varrho_{\pm}$ are shown in Figure \ref{fig1}.

 \begin{proof}[Proof of Theorem \ref{thm:contraction-mis}]
 We first prove the fixed design case, i.e., the first half of the theorem.
 Following the framework in Section \ref{subsec:proof-sec3-fixed-design}, it suffices to prove that, for some 
$\varepsilon_n, \bar{\varepsilon}_n$ to be determined later,
the three inequalities \eqref{eq:pf-A1A1prime-goal1}\eqref{eq:pf-A1A1prime-goal2}\eqref{eq:pf-A1A1prime-goal3}  are satisfied for some Borel measurable subsets $B_n$ of $C(\mathcal{X})$
and $n$ sufficiently large.

Recall that  in the proof of Theorem \ref{thm:contraction_n_norm},
\eqref{eq:pf-A1A1prime-goal1}\eqref{eq:pf-A1A1prime-goal2}\eqref{eq:pf-A1A1prime-goal3} are proved by the equations  \eqref{eq:estimation_around_true_func-2}\eqref{eq:claim-Pft-1}\eqref{eq:log-cover-proof-3} respectively.
Here, we will prove the following counterparts of 
\eqref{eq:estimation_around_true_func-2}\eqref{eq:claim-Pft-1}\eqref{eq:log-cover-proof-3}, under the condition (A3') instead of (A3).

We recall some constants:  
$r_0$ is from (A1),
$\epsilon_0$ is from (A2), 
$a_1$, $a_2$, $c_1$, $c_2$, $\varrho_{+}$, $\varrho_{-}$, $K_1$, $K_2$ are from (A3'),
$K$ is form Lemma \ref{lemma:unit_ball_covering_general}, 
$\tau_h$ is from Lemma \ref{lemma:rkhs_lip}; 
$K_3$ as in \eqref{eq:bound-phi-f*-t-proof-2},
$B_{N,r,\delta,\varepsilon'}$ as in \eqref{eq:def-B-N-r-delta-eps}
and $C = 1/\nu_1$ are defined in the same way
as in the proof of Theorem \ref{thm:contraction_n_norm},
where $\nu_1$ is from (A2).
\rev{We also retain $C_4$ and $\varepsilon'_1$ from the uniform
small-ball argument in that proof.}

\vspace{5pt}
\noindent
$\bullet$ Counterpart of \eqref{eq:estimation_around_true_func-2}: 
We claim that there exist $c_4 > 0$ satisfying $c_1/c_2<c_4 < 1$,
and constant $K_4 > 0 $,
such that as long as  $\varepsilon' $  satisfies the condition 
\begin{equation*}
    [c_4(C\varepsilon')^{2/s},  (C\varepsilon')^{2/s}] \subset [c_1 n^{\frac{-2}{2s+\varrho_{+}}} (\log n )^{\frac{2(1+D)}{2s +\varrho_{+}}}, 
    											c_2 n^{\frac{-2}{2s+\varrho_{+}}}(\log  n )^{\frac{2(1+D)}{2s +\varrho_{+}}} ],
\end{equation*}
then for large enough $n$ (and subsequently small enough $\varepsilon'$, because $\varepsilon'=o(1)$), we have
\begin{equation}
    \P (||f^{t} -f^*||_{\infty} \le 2\varepsilon')  
    \ge  e^{-K_4 (\varepsilon')^{-  \varrho_{+}/s}(\log(1/\varepsilon'))^{
    1+D}}.
    \label{eq:A1_counter_1}
\end{equation}

\vspace{5pt}
\noindent
$\bullet$ Counterpart of  \eqref{eq:claim-Pft-1}:
We claim that, 
if $(N,r, \delta, \varepsilon') $ satisfy that
\begin{align*}
& N^2 \ge 16 C_4 r^{\varrho} (\log(r/\varepsilon'))^{1+D},
\quad  
r > \max(1, \delta,\frac{1}{r_0}, 
	\rev{\frac{1}{\sqrt{c_3}}n^{\frac{\varrho_+}{(2s+ \varrho_{+})\varrho_{-}}}(\log  n )^{\frac{2(1+D)}{(2+\varrho_{+}/s)\varrho_{-}}}}),\\
&~~~~~~
 \varepsilon' < \rev{\min\{1/2,\varepsilon'_1\}},
\end{align*}
then \rev{for all sufficiently large $n$},
 \begin{equation}
     \P (f^{t} \notin B_{N,r,\delta,\varepsilon'}) \le
     \rev{\frac{4C_2}{K_2\varrho_-}r^{2a_2-2-\varrho_-}e^{-K_2r^{\varrho_-}}}
     + e^{-N^2/8},
    \label{eq:A1_counter_2}
 \end{equation}

\vspace{5pt}
\noindent
$\bullet$ Counterpart of  \eqref{eq:log-cover-proof-3}:
We claim that, if $(N,r,\delta, \varepsilon')$ satisfy that
\[
N^2 \ge 16 C_4 r^{\varrho} (\log(r/\varepsilon'))^{1+D},
\quad
r > \max(1, \delta,\frac{1}{r_0}),
\quad
\varepsilon' < 1/2,
\quad \rev{N > 2\varepsilon' },
\quad \rev{\delta\sqrt D\tau_h N<\varepsilon'},
\]
then
\begin{equation}
    \log \calN(3\varepsilon',B_{N,r,\delta,\varepsilon'},||\cdot||_{\infty}) \le K r^{\varrho}(\log
    (\frac{N (r/\delta)^{D/2}}{\varepsilon'} )
    )^{1+D}+ \log(\frac{2N}{\varepsilon'}).
    \label{eq:A1_counter_3}
\end{equation}

We prove the above three claims respectively.

\vspace{5pt}
\noindent
- Proof of \eqref{eq:A1_counter_1} as counterpart of \eqref{eq:estimation_around_true_func-2}:
Similar as in the proof of \eqref{eq:estimation_around_true_func}\eqref{eq:estimation_around_true_func-2}, 
under the condition stated for  \eqref{eq:A1_counter_1}, 
\begin{align*}
    \P (||f^{t} -f^*||_{\infty} &\le 2\varepsilon') 
    \ge 
    \P (||f^{t} -f^*||_{\infty} \le 2\varepsilon', t \in [c_4(C\varepsilon')^{2/s},  (C\varepsilon')^{2/s}] ) \\
    & \ge \int_{c_4(C\varepsilon')^{2/s}}^{(C\varepsilon')^{2/s}} 
    e^{-\phi_{f^*}^{t}(\varepsilon')} p(t) d t \\
    &\ge  e^{- K_3c_4^{-\varrho/2}(C\varepsilon')^{-\varrho/s} 
        \left(\log( \frac{1}{  {c_4}^{1/2}(C\varepsilon')^{1/s}   \varepsilon'})\right)^{
    D+1}}   \nonumber  \\ 
   &~~~  C_1 e^{-K_1 c_4^{-\varrho_{+}/2}(C\varepsilon')^{-\varrho_{+}/s} } (C\varepsilon')^{-2a_1/s}  \nonumber \\
   &~~~  (1- c_4)(C\varepsilon')^{2/s} \\
    &\ge  e^{-K_4 (\varepsilon')^{  - \varrho_{+}/s}(\log(1/\varepsilon'))^{
    1+D}},
\end{align*}
where  $   K_4 := K_3 c_4^{-\varrho/2 } C^{-\varrho/s}  (1+2/s)^{1+D} + K_1 c_4^{- 
\varrho_+
/2 } C^{-  \varrho_+ /s}   + 1 > 0$,
and the above inequality holds for large enough $n$ and subsequently small enough $\varepsilon'$.
The only difference between the above proof to 
that of \eqref{eq:estimation_around_true_func}\eqref{eq:estimation_around_true_func-2} lies in that, in the third inequality,
we used the lower bound of $p(t)$ in (A3') instead of that in (A3).

\vspace{5pt}
\noindent
- Proof of \eqref{eq:A1_counter_2} as counterpart of \eqref{eq:claim-Pft-1}:
The proof uses the same strategy.
Under the conditions stated before \eqref{eq:A1_counter_2}, one can verify that
\begin{equation}\label{eq:App_cond-claim}
\begin{gathered}
    r > \rev{\max\{\delta,1\}},
    \quad 
    r^{-2} < c_3n^{\frac{-2\varrho_+}{(2s+ \varrho_{+})\varrho_{-}}}(\log(n))^{-\frac{4(1+D)}{(2+\rev{\varrho_{+}}/s)\varrho_{-}}},
    \\ \quad e^{-\rev{\sup_{u\ge r^{-2}}\phi^u_0(\varepsilon')}} <1/4,
    \quad N \ge 4 \sqrt{\rev{\sup_{u\ge r^{-2}}\phi^u_0(\varepsilon')}},
\end{gathered}
\end{equation}
in the same way as how we derive \eqref{eq:cond-claim} from \eqref{eq:cond-claim-sufficient}.
\rev{Here we use the uniform small-ball bound with $C_4$ from the
proof of Theorem \ref{thm:contraction_n_norm}; that bound depends on
(A1) and the kernel, and is unchanged under (A3').}
We are to prove \eqref{eq:A1_counter_2} under \eqref{eq:App_cond-claim},
following the same way to prove \eqref{eq:claim-Pft-1} under \eqref{eq:cond-claim}. 
Specifically, all the previous proofs remain unchanged 
except for that \eqref{eq:proof-claim-1st-term} now becomes 
\begin{align*}
\P (t < r^{-2}) 
&\le \int_0^{r^{-2}} C_2 t^{-a_2} \exp(-K_2 t^{-\varrho_{-}/2}) dt \\
&\rev{=\frac{2C_2}{\varrho_-}K_2^{(2-2a_2)/\varrho_-}
\int_{K_2r^{\varrho_-}}^\infty u^{(2a_2-2)/\varrho_--1}e^{-u}\,du}\\
&\rev{\le \frac{4C_2}{K_2\varrho_-}
r^{2a_2-2-\varrho_-}e^{-K_2r^{\varrho_-}}},
\end{align*}
\rev{where the last inequality is the same upper incomplete-gamma bound as
in \eqref{eq:proof-claim-1st-term}.}
{{This}} follows by (A3') and
$r^{-2} < c_3n^{\frac{-2\varrho_+}{(2s+ \varrho_{+})\varrho_{-}}}
(\log(n))^{-\frac{4(1+D)}{(2+\rev{\varrho_{+}}/s)\varrho_{-}}}$.

\vspace{5pt}
\noindent
- Proof of  \eqref{eq:A1_counter_3} as counterpart of \eqref{eq:log-cover-proof-3}:
\rev{Under the stated conditions, Facts 1 and 2 in the proof of
Theorem \ref{thm:contraction_n_norm} apply. Repeating the derivation of
\eqref{eq:log-cover-proof-3} proves \eqref{eq:A1_counter_3};} 
the argument does not use any assumption on the prior $p(t)$.
\vspace{5pt}

Now we have proved the three claims, we use them to prove the \eqref{eq:pf-A1A1prime-goal1}\eqref{eq:pf-A1A1prime-goal2}\eqref{eq:pf-A1A1prime-goal3}.

To prove \eqref{eq:pf-A1A1prime-goal1}, 
we take $\varepsilon' = \varepsilon'_n$ where $(C\varepsilon'_n)^{2/s} = c_2 n^{\frac{-2}{2s+\varrho_{+}}}(\log n )^{\frac{2(1+D)}{2s +\varrho_{+}}}$,
and this $\varepsilon' = \varepsilon'_n$ satisfies the conditions of \eqref{eq:A1_counter_1}. 
Next, we let
\[
\varepsilon_n
     = \max\{ (\bar{C}_1'/C) c_2^{s/2}, 
      {c_3^{-\varrho_{-}/4} K_2^{1/2}} \}
    n^{\frac{-s}{2s+\varrho_{+}}}
     (\log n)^{\frac{ 1+D}{2 +\varrho_{+}/s}}
\]
 with $k_1 = (1+D)/(2 + \varrho_+/s)$,
 and we choose $\bar{C}_1'>0$ to be a large enough constant
to make the r.h.s. of \eqref{eq:A1_counter_1} lower bounded by 
$e^{- n \varepsilon_n^2}$.
The choices of  $ \varepsilon'_n$ and $ \varepsilon_n$ here are similar to part I of the proof for Theorem \ref{thm:contraction_n_norm}, where in the exponent of $n$ and $\log(n)$ the $\varrho$ is replaced with $\varrho_+$. The rest of the argument is same with the proof of part I of Theorem \ref{thm:contraction_n_norm} and we have \eqref{eq:pf-A1A1prime-goal1} hold.

To prove \eqref{eq:pf-A1A1prime-goal2} and \eqref{eq:pf-A1A1prime-goal3}, 
we let  $(N, r, \delta, \varepsilon') = (N_n,r_n,\delta_n, \varepsilon_n) $ in \eqref{eq:A1_counter_2}\eqref{eq:A1_counter_3}, 
and we are to choose the $N_n,r_n,\delta_n$ to satisfy the conditions stated before the claims \eqref{eq:A1_counter_2}\eqref{eq:A1_counter_3}.
In addition, we are to choose suitable $\bar \varepsilon_n$ such that $n\bar \varepsilon_n^2$ is larger than the r.h.s. of \eqref{eq:A1_counter_3}. 
As long as such $N_n$, $r_n$, $\delta_n$, $\bar \varepsilon_n$ can be specified, using the same proof in Part II and III of Theorem \ref{thm:contraction_n_norm}, we can prove \eqref{eq:pf-A1A1prime-goal2} by  \eqref{eq:A1_counter_2} and \eqref{eq:pf-A1A1prime-goal3} by \eqref{eq:A1_counter_3} respectively. 

To choose such $N_n$, $r_n$, $\delta_n$, $\bar \varepsilon_n$, we consider two cases separately:
\begin{enumerate}
    \item When $\varrho < \varrho_{-}$, we choose $N_n$, $r_n$, $\delta_n$ such that
$$
r_n^{\varrho_{-}} = \frac{8}{K_2}n \varepsilon^2_n, \ \ N_n^2 = \max(32, \frac{128 C_4}{K_2})n \varepsilon^2_n(\log(r_n/\varepsilon_n))^{1+D}, \ \ \delta_n = \varepsilon_n/(2\sqrt{D}\tau_h N_n),
$$
and take 
\rev{$\bar{\varepsilon}_n  = 4{\varepsilon}_n$}.
The specification of $N_n$, $r_n$, $\delta_n$ are the same as in the proof of Theorem \ref{thm:contraction_n_norm} except for that the $r_n^{\varrho}$ is changed to $r_n^{\varrho_-} \ge r_n^{\varrho} $.
One can verify that $(N, r, \delta, \varepsilon') = (N_n,r_n,\delta_n, \varepsilon_n) $ satisfy the conditions stated before \eqref{eq:A1_counter_2}\eqref{eq:A1_counter_3} for large enough $n$. 
\rev{Since $\bar\varepsilon_n=4\varepsilon_n>3\varepsilon_n$, monotonicity of the covering number gives
\[
\log\calN(\bar\varepsilon_n,B_{N_n,r_n,\delta_n,\varepsilon_n},\|\cdot\|_\infty)
\le
\log\calN(3\varepsilon_n,B_{N_n,r_n,\delta_n,\varepsilon_n},\|\cdot\|_\infty).
\]}
Meanwhile, $\bar\varepsilon_n$ makes $n\bar \varepsilon_n^2$ larger than the r.h.s. of \eqref{eq:A1_counter_3} with large enough $n$.

This proves \eqref{eq:pf-A1A1prime-goal2}\eqref{eq:pf-A1A1prime-goal3}, and  the overall rate ${\varepsilon}_n \sim n^{-s/(2s+\varrho_{+})}(\log(n))^{k_1}$, where the exponent $s/(2s+\varrho_{+}) = r(\varrho, \varrho_\pm,s)$ in this case.

\item When $\varrho \ge \varrho_{-}$, we take 
\begin{align*}
& r_n^{\varrho_{-}} = \frac{8}{K_2}n \varepsilon^2_n, \ \ 
N_n^2 = \max(32,  (\frac{128 C_4}{K_2})^{\varrho / \varrho_-})(n \varepsilon^2_n)^{\varrho / \varrho_-}(\log(r_n/\varepsilon_n))^{1+D},
 \\ 
 & \delta_n = \varepsilon_n/(2\sqrt{D}\tau_h N_n).
\end{align*}
One can verify that $(N, r, \delta, \varepsilon') = (N_n,r_n,\delta_n, \varepsilon_n) $ satisfy the conditions stated before \eqref{eq:A1_counter_2}\eqref{eq:A1_counter_3} for large enough $n$. 
We also take
\[
\bar{\varepsilon}_n = C_2' n^{-\frac{1}{2} + \frac{ \varrho_{+}\varrho}{2(2s+\varrho_+ )\varrho_{-} }}(\log  n )^{k_2},
\] 
with $k_2 := \frac{\varrho(1+D)}{(2+\varrho_+ /s )\varrho_{-}}+\frac{1+D}{2}$
and positive constant $C_2'$ to be determined.
The theorem assumes that $\varrho_{-} > \frac{\varrho_+ \varrho}{2s+\varrho_+}$, and this ensures that $\bar{\varepsilon}_n  = o(1)$.
One can choose $C_2'$ to a large enough constant s.t. 
$n\bar \varepsilon_n^2$ is larger than the r.h.s. of \eqref{eq:A1_counter_3} with large enough $n$. 

This proves \eqref{eq:pf-A1A1prime-goal2}\eqref{eq:pf-A1A1prime-goal3}, and the overall rate is $\bar{\varepsilon}_n \vee \varepsilon_n 
\sim n^{-\frac{\varrho_{-}(2s+\varrho_+) - \varrho_+\varrho}{2(2s+\varrho_+ )\varrho_{-} }}(\log n )^{k_2}$, where the exponent $\frac{1}{2} - \frac{\varrho_+ }{{{\varrho_-}}} \frac{\varrho}{ 2(2s  + \varrho_+)} = r(\varrho, \varrho_\pm,s)$.
\end{enumerate}

\vspace{5pt}
 The random design case, i.e., the second half of the theorem,
 follows the same strategy of the proof of Theorem \ref{thm:contraction_2_norm}.
To be specific, respectively under the two cases above,
one can plug in the new definitions of ${\varepsilon}_n$ and $\bar{\varepsilon}_n$ in Lemma \ref{lemma:scale_prob_k}
and prove the same statement of the lemma under the condition of the current theorem. 
With this new version of  Lemma \ref{lemma:scale_prob_k},
the rest of the proof is the same as in the proof of Theorem \ref{thm:contraction_2_norm}, with the new definitions of ${\varepsilon}_n$ and $\bar{\varepsilon}_n$. We then proved the random design case with the claimed rate. 
 \end{proof}

\section{Proof of Lemma  \ref{lemma:22_holder}}\label{sec:proof-lemma-4.1}

Below we give the proof to Lemma  \ref{lemma:22_holder} in several steps. The other proofs in Section \ref{sec:rate_manifold} are given in Appendix \ref{app:proof-sec4},
and technical lemmas in Appendix \ref{ap:A}.
Under Assumption \ref{assump:A4}, recall that  $\calM $ is embedded through $\iota:\calM \rightarrow [0,1]^D \subset \R^D$.

\subsection{Proof of equation \eqref{G epsilon expansion with remainder} and Lemma \ref{lemma:22_holder}(i)}

Consider an arbitrary point $x \in \mathcal{M}$,
and let $\delta(\epsilon) := \sqrt{(d+k+1)\epsilon \log(\frac{1}{\epsilon})}$.
Denote by $B_{r}(x)$ be the geodesic ball on $\calM$ of radius $r$  centered at $x$,
and we denote by  $B^{\R^D}_{r}(\iota(x))$ the Euclidean ball in $\R^D$.
Recall that $\xi > 0$ is the injectivity radius of $\calM$, 
and $\tau > 0$ is the reach of $\iota (\calM)$.
By Lemma \ref{lemma:manifold-reach}, $\forall x, y \in \calM$ s.t. $\|\iota(x)-\iota(y)\|_{\mathbb{R}^D} < \tau/2$, we have
\begin{equation}\label{eq:metric-comparison-sqrt2}
d_{\mathcal{M}}(x,y) 
\ge \|\iota(x)-\iota(y)\|_{\mathbb{R}^D} 
\ge  \frac{1}{2} d_{\mathcal{M}}(x,y).
\end{equation}
We let $\epsilon_{1,\calM} >0 $ be a constant
depending on $d$, $k$, $\xi$, and $\tau$ such that 
$\epsilon < \epsilon_{1,\calM}$ would guarantee that $ 2 \delta(\epsilon ) < \min\{ \tau/2 , \xi, 1 \}$.
Suppose $\epsilon <\epsilon_{1,\calM} $, 
then $\delta(\epsilon) < \tau/4 $, and 
we can verify that $B^{\mathbb{R}^D}_{{\delta(\epsilon)}}(\iota(x)) \cap \iota(\mathcal{M}) \subset \iota(B_{2 {\delta(\epsilon)}}(x))$:
For any $y \in \calM$ with $ \|\iota(x)-\iota(y)\|_{\mathbb{R}^D} < \delta(\epsilon) < \tau/2$, we have \eqref{eq:metric-comparison-sqrt2} holds and then $d_\calM(x,y) \le {2}  \|\iota(x)-\iota(y)\|_{\mathbb{R}^D}  < 2 \delta(\epsilon)$.
 Hence, if $y \not \in B_{2 {\delta(\epsilon)}}(x)$, $\|\iota(x)-\iota(y)\|_{\mathbb{R}^D} \geq {\delta(\epsilon)}$. 
Observe that 
\begin{align} 
G_\epsilon (f) (x)
& =  \frac{1}{(2\pi \epsilon)^{d/2}} \int_{B_{2 {\delta(\epsilon)}}(x)} h\Big(\frac{\|\iota(x)-\iota(y)\|^2_{\mathbb{R}^D}}{\epsilon}\Big)f(y)dV(y)  \nonumber \\
& ~~~
+  \frac{1}{(2\pi \epsilon)^{d/2}} \int_{\mathcal{M} \setminus B_{2 {\delta(\epsilon)}}(x)} h\Big(\frac{\|\iota(x)-\iota(y)\|^2_{\mathbb{R}^D}}{\epsilon}\Big)f(y)dV(y) \nonumber \\
& =: G_\epsilon^{(1)} (f) (x) +  {\rm R}^{(2)} (x). 
\label{Lemma: integral in and out}
\end{align}
We first show that $| {\rm R}^{(2)} (x) |$ is uniformly small and can be put to remainder term $R_{f ,\epsilon} $.
Since $\|\iota(x)-\iota(y)\|_{\mathbb{R}^D} \geq {\delta(\epsilon)}$ when $y \in \mathcal{M} \setminus B_{2 {\delta(\epsilon)}}(x)$, 
we have $h\Big(\frac{\|\iota(x)-\iota(y)\|^2_{\mathbb{R}^D}}{\epsilon}\Big) \leq \epsilon^{\frac{d+k+1}{2}}$ by that $h(r)=e^{-r/2}$ and the definition of $\delta(\epsilon)$.
Thus, 
\begin{align}
| {\rm R}^{(2)} (x) |
& \le \frac{1}{(2\pi \epsilon)^{d/2}} \int_{\mathcal{M} \setminus B_{2 {\delta(\epsilon)}}(x)} 
    	h\Big(\frac{\|\iota(x)-\iota(y)\|^2_{\mathbb{R}^D}}{\epsilon}\Big) |f(y)|dV(y)  \nonumber\\
& \leq \frac{Vol(\mathcal{M})}{(2\pi)^{d/2}}  \|f\|_{\infty} \epsilon^{\frac{k+1}{2}}
	\le C_{R,2} \|f\|_{k,\beta} \epsilon^{(k+1)/2},
	\quad C_{R,2}:= \frac{Vol(\mathcal{M})}{(2\pi)^{d/2}}, \label{1st part of the integral}
\end{align}
and this upper bound is uniform for all $x \in \calM$.

For the term
{{$G_\epsilon^{(1)} (f) (x)= \frac{1}{(2\pi\epsilon)^{d/2}}\int_{B_{2 {\delta(\epsilon)}}(x)} h\Big(\frac{\|\iota(x)-\iota(y)\|^2_{\mathbb{R}^D}}{\epsilon}\Big)f(y)dV(y)$,}}
since $2 {\delta(\epsilon)} < \xi$, 
we can parametrize $B_{2 {\delta(\epsilon)}}(x)$ through normal coordinates at $x$.  
Specifically, we utilize the polar coordinates $(t, \theta)$ on $\R^d \cong T_x \calM$,
where $\theta \in S^{d-1} \subset T_x \calM$, and $0 \le t < \xi$,
and write $y = \exp_x ( t \theta)$.
We then have that
\begin{align}
G_\epsilon^{(1)} (f) (x) 
& =  \frac{1}{(2\pi \epsilon)^{d/2}} \int_{B_{2 {\delta(\epsilon)}}(x)} h\Big(\frac{\|\iota(x)-\iota(y)\|^2_{\mathbb{R}^D}}{\epsilon}\Big)f(y)dV(y) \nonumber \\
& = \frac{1}{(2\pi \epsilon)^{d/2}} \int_{S^{d-1}} \int_{0}^{2{\delta(\epsilon)}} h\Big(\frac{{{\|\iota \circ \exp_x( t \theta )-\iota(x)\|^2_{\mathbb{R}^D}}}}{\epsilon}\Big)f(\exp_x(t \theta)) V(x, \theta, t) t^{d-1}
	dt d\theta, \label{2nd part of the integral}
\end{align}
where $V(x, \theta, t) t^{d-1}dt d\theta$ is the volume form. We expand each term in the integrand in terms of $t$ in the next a few steps.
Next we focus on the computation of $G_\epsilon^{(1)} (f) (x)$.

\vspace{10pt}
\noindent
$\bullet$ \underline{\textbf{Expansion of $f$, kernel, and volume form in $B_{2{\delta(\epsilon)}}(x)$}}
\vspace{5pt}

Recall that  $t = d_\calM(x,y)$ satisfies that  $0 \leq t \leq 2\delta(\epsilon) < \min\{ \tau/2, \xi, 1\}$,
then  Lemma \ref{expansion-of-volume-of-form-and-Euclidean-distance}(i) and (ii) both apply.
Meanwhile, in this case $ \|\iota(x)-\iota(y)\|_{\mathbb{R}^D}  \le t < \tau/2 $
 and then \eqref{eq:metric-comparison-sqrt2} holds.
We start with the expansion of $h\Big(\frac{{{\|\iota \circ \exp_x( t \theta )-\iota(x)\|^2_{\mathbb{R}^D}}}}{\epsilon}\Big)$.
By Lemma \ref{expansion-of-volume-of-form-and-Euclidean-distance}(ii), 
we can expand ${{\|\iota \circ \exp_x( t \theta )-\iota(x)\|^2_{\mathbb{R}^D}}}$ as  
\begin{align}\label{expansion iota expn t proof}
{{\|\iota \circ \exp_x( t \theta )-\iota(x)\|^2_{\mathbb{R}^D}}}
 =t^2+  \sum_{j =4}^{2k} q_j (x, \theta)t^j + r_k 
 :=  t^2 + \tilde{r},
\end{align}
where the remainder $r_k$ and $\tilde r$ all depend on $t$ (when $k=0$ or 1, $r_k = \tilde r $), and 
\begin{equation}\label{eq:proof-bound-cr-cq}
|\tilde r| \le c_{\tilde r} t^4 \text{ always}, 
\quad 
| r_k | \le c_{r}(k) t^{2k+1} \text{ when $k \ge 2$};
\quad \sup_{x,\theta} | q_j(x,\theta) | \le c_q(k), \, \forall j=4,\cdots, 2k.
\end{equation}
For $j \ge 4$, each $q_j$ is determined by the second fundamental form  $ \Second$  of $\iota(\mathcal{M})$ 
and its covariant derivatives up to ($j-4$)-th order.
The constants $c_{\tilde r}$, $c_{r}(k)$, $c_{q}(k)$ depend on $ \Second$ and its covariant derivatives:
 when $k = 0$ or 1, we only have $c_{\tilde r}$ and it involves the 
  $\| \cdot \|_\infty$ norm of $ \Second$ and its 1st covariant derivative,
 when $k \ge 2$, 
 $c_q(k)$ and $c_r(k) $ involve 
 the $\| \cdot \|_\infty$ norm of $ \Second$'s covariant derivatives up to
  ($2k-4$)-th and 
   ($2k-2$)-th order respectively.
In addition, by Lemma \ref{expansion-of-volume-of-form-and-Euclidean-distance}(ii) b),
$q_j(x,\theta) = \bar q_j(x)( \theta ,\cdots, \theta)$ for a tensor field $\bar q_j$ of order $j$.
Since $ \bar q_j(x)( -\theta ,\cdots, - \theta) = (-1)^j \bar q_j(x)( \theta ,\cdots, \theta)$ by linearity of tensor, 
we have $q_j(x, -\theta)=(-1)^j q_j(x, \theta)$.

Next, by that $h(r) = e^{-r/2}$, 
we can expand $h\Big( \frac{t^2+\tilde{r}}{\epsilon} \Big)$ as
\begin{equation}\label{eq:proof-expansion-h-in-l}
h(  \frac{t^2+\tilde{r}}{\epsilon}) 
= h( \frac{t^2}{\epsilon}) 
	+ \sum_{ \ell =1}^{ \lfloor k/2 \rfloor} \frac{1}{ \ell! } h^{(\ell)}( \frac{t^2}{\epsilon} ) \frac{\tilde r^\ell }{\epsilon^\ell} 
	+ h_r
=:  h( \frac{t^2}{\epsilon})  + {\rm (II)} + h_r,
\end{equation}
where $h_r : = \frac{1}{( \lfloor k/2 \rfloor  +1)! } h^{(\lfloor k/2 \rfloor + 1)}( \frac{t^2 + \tilde r'}{\epsilon} )
(\frac{\tilde r}{\epsilon})^{\lfloor k/2 \rfloor + 1}$,
 and $\tilde r'$ is between 0 and $\tilde r$.
Because  $t = d_\calM(x,y)$, 
$ t^2 + \tilde r = \| \iota (y) - \iota(x)\|_{\R^D}^2 $,
by \eqref{eq:metric-comparison-sqrt2}, 
$t^2 \ge t^2 + \tilde r \ge t^2/4$, 
and thus $t^2 + \tilde r' \ge t^2/4$.
Meanwhile, by that $h^{( \ell )}(r) = 
{{\frac{h(r)}{(-2)^\ell}
= \frac{e^{-r/2}}{(-2)^\ell}
}}$, we have $|h^{(\ell)}(r)| \le h(r)$ for any $\ell$ and $r$, 
and then by changing variable to $u := t/\sqrt{\epsilon}$, we have
$
|h_r | \le 
h( \frac{u^2}{4})
(\frac{ | \tilde r |}{\epsilon})^{ \lfloor k/2 \rfloor + 1}
\le c_{\tilde r}^{\lfloor k/2 \rfloor + 1 } h( \frac{u^2}{4}) u^{4 (\lfloor k/2 \rfloor + 1)} \epsilon^{\lfloor k/2 \rfloor + 1}.
$
We will assume $\epsilon < 1/e < 1$,
then  $\epsilon^{\lfloor k/2 \rfloor + 1} \le \epsilon^{(k+1)/2}$, and we have
\[
|h_r | 
= O( \epsilon^{(k+1)/2}) h( {u^2}/{4}) u^{4 (\lfloor k/2 \rfloor + 1)}.
\]
Here, we use the big-O notation for convenience
where the constant dependence can be tracked, and we will summarize constant dependence later. 

The second term in expansion \eqref{eq:proof-expansion-h-in-l} that sums over $\ell$,
denoted as ${(\rm II)}$, 
involves the power of $\tilde r$
and thus is more complicated. 
For $k \ge 2$ (otherwise ${(\rm II)} = 0$) and using the variable $u$ instead of $t$, we have
\begin{align*}
{(\rm II)} 
& = \sum_{ \ell =1}^{ \lfloor k/2 \rfloor} \frac{h^{(\ell)}( u^2) }{ \ell! \epsilon^\ell }  
	\Big( \sum_{j=4}^{2k} q_j (x, \theta) u^j \epsilon^{j/2}+r_k \Big)^\ell \\
& = \sum_{\ell =1}^{ \lfloor k/2 \rfloor} \frac{h^{(\ell)}( u^2) }{\ell ! \epsilon^\ell }  
	\bigg[
	\Big( \sum_{j=4}^{2k} q_j (x, \theta) u^j  \epsilon^{j/2} \Big)^\ell
	+ \sum_{m=1}^\ell  {\ell \choose m}  r_k^m  \Big( \sum_{j=4}^{2k} q_j(x, \theta) u^j \epsilon^{j/2} \Big)^{ \ell-m} 
	\bigg] \\
& =: 	{(\rm II)}_1 + {(\rm II)}_2,
\end{align*}
where, for the second term, by that $|h^{( \ell)}(r)| \le h(r)$ and \eqref{eq:proof-bound-cr-cq},
(we omit dependence on $k$ in $c_r$ and $c_q$)
\begin{align*}
| {(\rm II)}_2| 
& \le  \sum_{\ell =1}^{ \lfloor k/2 \rfloor} 
	\frac{h( u^2) }{ \ell ! \epsilon^\ell }  
		\sum_{m=1}^\ell {\ell \choose m}
	  \left( c_r  u^{2k+1} \epsilon^{k+1/2} \right)^m \Big( c_q \sum_{j=4}^{2k}  u^j \epsilon^{j/2} \Big)^{\ell -m} \\
& \le   \sum_{\ell =1}^{ \lfloor k/2 \rfloor} 
	\frac{h( u^2) }{ \ell! }  
		\sum_{m=1}^\ell {\ell \choose m}
		( c_r u^{ 2k+1 } )^m
	    \Big(  c_q \sum_{j=4}^{2k}  u^j \Big)^{\ell-m} 	 \epsilon^{(k+1/2)m + 2(\ell -m)-\ell}  \\
&\le \epsilon^{(k+1)/2}	  
 	\bigg[   \sum_{ \ell =1}^{ \lfloor k/2 \rfloor} \frac{h( u^2) }{ \ell! }  
		\sum_{m=1}^\ell {\ell \choose m}
		( c_r u^{2k+1} )^m
	    \Big(  c_q \sum_{j=4}^{2k}  u^j \Big)^{\ell -m}  \bigg],
\end{align*}
because one can verify that $ (k+1/2)m + 2(\ell-m)-\ell \ge (k+1)/2 $ using $k\ge 2$
and recall that $\epsilon <1$.

For the first term ${(\rm II)}_1$, we will separate the leading terms and a remainder of $O(\epsilon^{(k+1)/2})$.
\begin{align*}
{(\rm II)}_1 
& =  \sum_{\ell =1}^{ \lfloor k/2 \rfloor} \frac{h^{(\ell)}( u^2) }{ \ell! \epsilon^\ell }  
	\Big( \sum_{j=4}^{2k} q_j(x, \theta) u^j \epsilon^{j/2} \Big)^\ell  
= 	 \sum_{\ell =1}^{ \lfloor k/2 \rfloor} \frac{h^{( \ell)}( u^2) }{\ell !  }  
 	\Big( \sum_{i =4}^{k^2} A^{(k)}_{i, \ell}(x,\theta) u^i \epsilon^{i/2 - \ell} \Big),
\end{align*}
where for $i\ge 4$ we define 
\[
A^{(k)}_{i,\ell} (x, \theta)
: = 
  	 \sum_{\substack{j_1+ \cdots + j_\ell =i\\
 		4\le j_1, \cdots ,j_\ell {\le 2k}}} q_{j_1}(x, \theta) \cdots q_{j_\ell}(x, \theta), 
\]
and set $A^{(k)}_{i,\ell} (x, \theta) = 0$ if the valid combination of $\{ j_1,\cdots, j_\ell \}$ is empty. 
The summation inside $( \cdots )$ over $i$ is from 4 to $k^2$ because the highest power of $u$ is $ 2 k \ell \le k^2 $.
We now separate the summation into two categories where $i \le k+2 \ell $ and $i \ge k+2 \ell +1$ respectively,
and for the latter, in each term the factor $\epsilon^{i/2- \ell} \le \epsilon^{(k+1)/2}$.
As a result, 
using the equivalent expression $( \sum_{j=4}^{2k} q_j(x, \theta) u^j \epsilon^{j/2} )^\ell $ of the summation,
the absolute value of the sum in the second category can be upper bounded by
$O(\epsilon^{(k+1)/2}) ( c_q \sum_{j=4}^{2k}   u^j )^\ell $.
 We then have
\[
{(\rm II)}_1 
=  \sum_{\ell =1}^{ \lfloor k/2 \rfloor} \frac{h^{( \ell)}( u^2) }{ \ell!  }  
  	\sum_{i =4}^{k+2 \ell} A^{(k)}_{i,\ell }(x,\theta) u^i \epsilon^{i/2 - \ell }
	+  O(\epsilon^{(k+1)/2})  \sum_{\ell =1}^{ \lfloor k/2 \rfloor} \frac{h( u^2) }{ \ell!  }  \Big( c_q \sum_{j=4}^{2k}   u^j \Big)^\ell.
\]
Now, in the leading terms in ${(\rm II)}_1$, we have $i \le k+2 \ell$.
When $i \le k+2 \ell$, in a valid combination of $\{j_1,\cdots, j_\ell\}$ in the definition of $ A^{(k)}_{i,\ell}$ we must have $j_m \le i \le k+2 \lfloor k/2\rfloor \le 2k$, for all $m=1,\cdots, \ell$. Thus, in these leading terms, we can drop the requirement that $j_m \le 2k$ in the definition and let
\begin{equation}\label{eq:def-Ail-no-k}
A_{i,\ell} (x, \theta)
: = 
  	 \sum_{\substack{j_1+ \cdots + j_\ell =i\\
 		4\le j_1, \cdots ,j_\ell }} q_{j_1}(x, \theta) \cdots q_{j_\ell}(x, \theta),
\end{equation}
which is {\it independent} from $k$, and  $A_{i,\ell} (x, \theta) = 0$  if no valid combination of $\{ j_1,\cdots, j_\ell \}$ exists.
One can verify that $A_{i,\ell} (x, \theta) $ is only non-zero when $  \ell \le \lfloor i/4 \rfloor $.
As a result, in the leading term the $\epsilon$'s power is always positive, i.e., $i/2-\ell \ge i/2 - \lfloor i/4 \rfloor \ge 1$
by that $i \ge 4$.
Putting together ${(\rm II)}_1$ and ${(\rm II)}_2$, we have
\[
{(\rm II)} =
{(\rm II)}_0 
	+  O(\epsilon^{(k+1)/2})  \sum_{ \ell=1}^{ \lfloor k/2 \rfloor} \frac{h( u^2) }{\ell!  }  
			\Big(c_r u^{ 2k+1} + c_q \sum_{j=4}^{2k}   u^j \Big)^\ell,
\]
where the leading term  
$
{(\rm II)}_{\rm 0} := \sum_{\ell =1}^{ \lfloor k/2 \rfloor} \frac{h^{(\ell)}( u^2) }{ \ell!  }  
  	\sum_{i =4}^{k+2\ell} A_{i,\ell}(x,\theta) u^i \epsilon^{i/2 - \ell}$.

Back to \eqref{eq:proof-expansion-h-in-l}, we have that, when $k \ge 2$, 
\begin{align*}
& h\Big(\frac{{{\|\iota \circ \exp_x( t \theta )-\iota(x)\|^2_{\mathbb{R}^D}}}}{\epsilon}\Big)
 = h(u^2) + {(\rm II)}_0 \\
&~~~~~~~~~~~~~~~~~~~~~~~~~~~~~~
	+ O( \epsilon^{(k+1)/2}) 
	 \bigg[ 
	h(u^2) \sum_{\ell =1}^{ \lfloor k/2 \rfloor} \frac{1 }{\ell!  }  
			\Big(   \sum_{j=4}^{2k+1}   u^j \Big)^\ell
	+ h(\frac{u^2}{4}) u^{4 (\lfloor k/2 \rfloor + 1)} 
	\bigg],
\end{align*}
and when $k =0,1$, the expression is $h(u^2) +  O( \epsilon^{(k+1)/2})  h(u^2/4 ) u^4$.
Combining both cases, we have
\begin{align}\label{eq:expansion-h-RA-bound}
& h\Big(\frac{{{\|\iota \circ \exp_x( t \theta )-\iota(x)\|^2_{\mathbb{R}^D}}}}{\epsilon}\Big)
 =   \sum_{ \ell=0}^{ \lfloor k/2 \rfloor} \frac{h^{(\ell)}( u^2) }{ \ell!  }  
  	\sum_{i = 0}^{k+2 \ell} A_{i, \ell}(x,\theta) u^i \epsilon^{i/2 - \ell}
+ R_A(u),  \\
&~~~ 
\max_{ 0 \le \ell \le \lfloor \frac{k}{2} \rfloor, \, 0 \le i \le k +2 \ell}\sup_{x, \theta} |A_{i, \ell}(x,\theta) | \le C_A,  \nonumber \\
&~~~ 
|R_A(u)|  \le C_A \epsilon^{(k+1)/2} h( \frac{u^2}{4}) \sum_{j=4}^{J_k} u^j,
\quad J_k := \max\{ (2k+1)\lfloor \frac{k}{2}\rfloor, \, 4(\lfloor \frac{k}{2} \rfloor + 1) \}, \nonumber
\end{align}
where in deriving the upper bound of $|R_A|$ we used that
and using that $h(u^2) \le h(u^2/4)$;
The constant $C_A$ depends on $k$,
the $\| \cdot \|_\infty$ norm of $ \Second$'s covariant derivatives up to  $\max\{2k-2,1\}$-th order.
In the leading terms of the expansion \eqref{eq:expansion-h-RA-bound},
we take the sum starting from $\ell =0$ and $i=0$ by allowing  
$0 \le i\le 3$ in the definition of $A_{i,\ell}$ in \eqref{eq:def-Ail-no-k} 
and extending the definition of $A_{i,\ell}$ for $\ell=0$ as
\begin{equation}\label{eq:def-Ail-no-k-small-i}
A_{i,\ell}(x, \theta)  :=
\begin{cases}
	1, &  \ell =0, \, i=0,   \\
	0, &  \ell  =0, \, i \ge 1.  \\
\end{cases}
\end{equation}
Note that when $i<4$ and $\ell \ge 1$, $A_{i, \ell} =0$ because $\ell > \lfloor i/4 \rfloor = 0$.
 By the definition of $A_{i,\ell}$ and that $ q_j (x, -\theta)=(-1)^j q_j (x, \theta)$, we have
 $A_{i, \ell}(x, - \theta) = (-1)^i A_{i, \ell} (x, \theta)$,
 and this holds for all $i \ge 0$ and $\ell \ge 0$.

\vspace{5pt}
Next, we expand $f(\exp_x(t \theta))$: 
Because $t \le 2\delta(\epsilon) < \xi$, we apply Lemma \ref{lemma:taylor-f-intrinsic-Holder} to have that
\[
f(\exp_x(t \theta)) = \sum_{i=0}^k \frac{1 }{i!} \nabla^i_\theta f(x) t^i  + r_B(t),
\quad |r_B(t)| \le \frac{1}{k!}L_{k,\beta} (f,x) t^{k+\beta} \le \|f \|_{k, \beta} t^{k+\beta}.
\]
Change the variable to $u=t/\sqrt{\epsilon}$, the upper bound of $r_B$ becomes 
$\|f \|_{k, \beta} \epsilon^{(k+\beta)/2} u^{k+\beta}$,
and since $u \ge 0$, $0 < \beta \le 1$,
$u^{k+\beta} \le \max \{u^k, u^{k+1} \}\le u^k + u^{k+1}$.
Thus we have
\begin{align}\label{Lemma:expansion 2}
& f(\exp_x(t \theta))
 = \sum_{i=0}^k B_i(x, \theta)  u^i \epsilon^{i/2} +R_B(u),
\quad B_i(x, \theta) := \frac{1 }{i!} \nabla^i_\theta f(x),   \\
\|f \|_{\infty}, \, 
& \max_{0 \le i\le k} \sup_{x, \theta} |B_i(x,\theta)| \le  C_B, 
\quad | R_B(u) | \le  C_B \epsilon^{(k+\beta)/2}  ( u^{k} + u^{k+1}),
\quad C_B:= \|f\|_{k, \beta}.
\nonumber
\end{align}
Note that $B_i(x, -\theta)=(-1)^i B_i(x, \theta)$ because $\nabla^i f(x)$ is an order-$i$ tensor.

\vspace{5pt}

Finally, by Lemma \ref{expansion-of-volume-of-form-and-Euclidean-distance}(i) b), when $k\ge 2$,
\begin{align*}
V(x, \theta, t)
=   
 1 + \sum_{i=2}^k V_i (x, \theta) t^i  + r_V (t),
\quad
| r_V(t) | \le c_V(k)  t^{k+1},
\end{align*}
where for each $i \geq 2$, 
$V_i(x, \theta)$  is determined by the curvature tensor  of $\calM$ and its covariant derivatives at $x$ up to ($i-2$)-th order;
The constant $c_V(k)$ depends on $d$ and the uniform bounds of up to ($k+1$)-th intrinsic derivatives of the Riemann metric tensor $g$.
When $k = 0, 1$, 
by Lemma \ref{expansion-of-volume-of-form-and-Euclidean-distance}(i) a),
we have $ V(x, \theta, t) = 1 + O(t^2)$, 
and the constant in big-O is bounded by the uniform bounds of up to the 2nd intrinsic derivative of $g$.
Combining both cases and changing variable to $u=t/\sqrt{\epsilon}$, we have
\begin{align}\label{Lemma:expansion 3}
&V(x, \theta, t)
=  
  \sum_{i= 0}^k V_i (x, \theta) u^i \epsilon^{i/2} + R_V (u), \\
&
\sup_{x,\theta}|V(x,\theta,t)|, \, \max_{ 0 \le i \le k }\sup_{x,\theta} | V_i(x,\theta) | \le C_{V},  
\quad
| R_V(u) | \le C_{V} \epsilon^{(k+1)/2}   u^{\max \{2, k+1 \}}, \nonumber
\end{align}
where for the case $k \le 1$ we used that $\epsilon \le \epsilon^{(k+1)/2}$,
and we define $V_0 (x,\theta) =1$ and $V_1 (x,\theta)= 0$.
The constant $C_V$ depends on $d$, 
the $\| \cdot \|_\infty$ norm of the curvature tensor and its covariant derivatives  up to $\max\{ k-2, 0 \}$-th order,
and the uniform bounds of up to $\max\{ k+1, 2\}$-th intrinsic derivatives of $g$.
Meanwhile, by Lemma \ref{expansion-of-volume-of-form-and-Euclidean-distance}(i) b),
$V_i(x,\theta) = \bar V_i (x) ( \theta ,\cdots, \theta) $ for a tensor field $\bar V_i$ of order $i$.
Then again by  linearity of tensor, we have $V_i(x, -\theta)=(-1)^iV_i(x, \theta)$ for $2 \le i \le k$.
The same relation also holds for $i=0,1$.

\vspace{10pt}
\noindent
$\bullet$ \underline{\textbf{Derivation of equation \eqref{G epsilon expansion with remainder} and the remainder}}
\vspace{5pt}

Substituting \eqref{eq:expansion-h-RA-bound}\eqref{Lemma:expansion 2}\eqref{Lemma:expansion 3} into \eqref{2nd part of the integral}
and changing variable to $u=t/\sqrt{\epsilon}$, 
we get
\begin{align*}
& G_\epsilon^{(1)} (f) (x)
= \frac{1}{(2\pi )^{d/2}}
	\int_{S^{d-1}} \int_{0}^{\frac{ 2 \delta(\epsilon)}{\sqrt{\epsilon}} }  
	\Big(  \sum_{ \ell = 0 }^{ \lfloor k/2 \rfloor} \frac{h^{(\ell)}( u^2) }{\ell !  }  
  		\sum_{i_1 = 0}^{k+2 \ell} A_{i_1, \ell}(x,\theta) u^{i_1} \epsilon^{i_1 /2 - \ell}
		+ R_A(u) \Big)   \nonumber \\
&~~~~~~~~~~~~~~~~~~~
	\Big( \sum_{i_2=0}^k B_{i_2}(x, \theta)  u^{i_2} \epsilon^{i_2/2} +R_B(u)  \Big) 
	\Big( \sum_{i_3=0}^k V_{i_3} (x, \theta) u^{i_3} \epsilon^{i_3/2} + R_V (u)  \Big) 
	 u^{d-1} du d\theta, \nonumber
\end{align*}
and  the three big brackets multiplied in the integrand represent 
$ h\Big(\frac{{{\|\iota \circ \exp_x( t \theta )-\iota(x)\|^2_{\mathbb{R}^D}}}}{\epsilon}\Big)$,
 $f(\exp_x(t \theta))$
 and $ V(x, \theta, t)$ respectively. 
We are to collect terms up to a remainder of order $O(\epsilon^{(k+\beta)/2})$.

To proceed, we define the leading terms in the three brackets as 
\begin{align*}
L_A(u) 
&:= \sum_{\ell = 0 }^{ \lfloor k/2 \rfloor} \frac{h^{(\ell)}( u^2) }{ \ell!  }  
  		\sum_{i = 0}^{k+2\ell} A_{i,\ell}(x,\theta) u^{i} \epsilon^{i /2 - \ell}, \\
L_B(u) 
& : = 	 \sum_{i=0}^k B_{i}(x, \theta)  u^{i} \epsilon^{i/2}, \quad
L_V(u) := 	\sum_{i=0}^k V_{i} (x, \theta) u^{i} \epsilon^{i/2}.
\end{align*}
Recall the upper bounds of $|A_{i,\ell}|$, $|B_i|$ and $|V_i|$ derived in \eqref{eq:expansion-h-RA-bound}\eqref{Lemma:expansion 2}\eqref{Lemma:expansion 3}, 
and note that the $\epsilon$-factor always has non-negative power and thus is bounded by 1,
then we have, with a constant $c_A(k)$ depending on $k$,
\[
|L_A(u)| \le C_A c_A(k) h(u^2)  \sum_{j=0}^{2k} u^j, \quad 
|L_B(u)| \le C_B \sum_{j=0}^{k} u^j, \quad
|L_V(u)| \le C_V \sum_{j=0}^{k} u^j.
\]
Meanwhile, $R_A$, $R_B$ and $R_V$ at all  $O(\epsilon^{(k+\beta)/2})$  as shown in  \eqref{eq:expansion-h-RA-bound}\eqref{Lemma:expansion 2}\eqref{Lemma:expansion 3}.
As a result, when we multiply the three brackets $(L_A + R_A)(L_B + R_B)(L_V + R_V)$, 
we can bound all the other terms except from $L_AL_BL_V$ to be  $O(\epsilon^{(k+\beta)/2})$.
Specifically, with a positive constant $c_2(k)$ depending on $k$, we have
{$
 {|L_AL_BR_V|}\begingroup\color{cyan},\endgroup
 {|L_AR_BL_V|}\begingroup\color{cyan},\endgroup
 {|L_AR_BR_V|}\begingroup\color{cyan},\endgroup
\begingroup\color{cyan}|R_AL_BL_V|,
|R_AL_BR_V|,\endgroup
 {|R_AR_BL_V|}\begingroup\color{cyan},\endgroup
 {|R_AR_BR_V|}$}
are all upper bounded by 
\[
 C_A C_B C_V  c_2(k)  \epsilon^{(k+\beta)/2}  h( \frac{u^2}{4}) 
	\sum_{j=0}^{J_k + 2k+3} u^j,
\]
where we used that $J_k \ge 2k$
and $h(u^2) \le h(u^2/4)$. 

The term $ L_AL_B L_V$ consists of terms having half-integer powers of $\epsilon$,
i.e. $\epsilon^{i/2}$ where $i=0, \cdots, 3k$.
We separate these terms into two parts
where $i\le k$ terms are kept and and the rest go to the remainder:
\[
L_AL_B L_V 
= \sum_{i=0}^k 
	\epsilon^{{i}/{2}}
	\sum_{\substack{i_1+i_2+i_3-2\ell=i \\ 
		 {{i_1,i_2,i_3\ge0}}\\ 
		0 \leq \ell \leq \lfloor i_1/4 \rfloor}}  
	\frac{h^{(\ell)}(u^2) }{ \ell!} A_{i_1,\ell}(x,\theta)B_{i_2}(x, \theta)V_{i_3}(x,\theta)
	u^{i+2\ell} 
 + R_1(u),
\]
where, for a positive constant $c_1(k)$ depending on $k$, 
\[
|R_1(u)| \le C_A C_B C_V  c_1(k)  \epsilon^{(k+1)/2}  h(u^2) \sum_{j=0}^{4k} u^j.
\]
In the expression, we rewrite the summation limit of $i_1, i_2, i_3, \ell$, but the expression is equivalent as before
because the summed terms are non-zero only when $ 0 \le i_1-2 \ell \le k$, $0 \le i_2, i_3 \le k$ and $\ell \le \lfloor k/2\rfloor$.
(Since $\ell \le \lfloor i_1/4 \rfloor \le i_1/2$, 
we have $i_1-2 \ell \ge 0$, 
thus $0 \le i_1-2\ell, \, i_2, \, i_3 \le i \le k$. 
Then  $i_1 \le k + 2  \lfloor i_1/4 \rfloor$ gives that $i_1 \le 2k$, and then $\lfloor i_1/4 \rfloor  \le \lfloor k/2 \rfloor $.)

Putting things together, we have 
\begin{align}
 G_\epsilon^{(1)} (f) (x)
& =
	\sum_{i=0}^k 
	\sum_{\substack{i_1+i_2+i_3-2\ell=i \\ i_1, i_2, i_3  \ge 0 \\ 0 \leq \ell \leq \lfloor i_1/4 \rfloor}}
 	 \frac{ \epsilon^{{i}/{2}} }{(2\pi)^{d/2}} \int_{S^{d-1}} \int_{0}^{\frac{2{\delta(\epsilon)}}{\sqrt{\epsilon}}}  
	\frac{h^{(\ell)}(u^2)}{ \ell!} 
	A_{i_1,\ell}(x,\theta)B_{i_2}(x, \theta)
	\nonumber\\
&~~~~~~~~~ 
	V_{i_3}(x,\theta)  u^{i+2\ell+d-1} 
	du d\theta 
	+ \frac{1}{(2\pi)^{d/2}} \int_{S^{d-1}} \int_{0}^{\frac{2{\delta(\epsilon)}}{\sqrt{\epsilon}}}  R(u)  u^{d-1} du d\theta,  
\label{eq:lemma:expansion4-a}
\end{align}
where, again by $J_k \ge 2k$ and $\epsilon <1$, 
we have that for a positive constant $c(k)$ depending on $k$,
$ |R(u)| \le C_A C_B C_V  c(k)  \epsilon^{(k+\beta)/2}  h(u^2/4 ) 
	\sum_{j=0}^{J_k + 2k+3} u^j$.
Recall that $C_B = \|f\|_{k,\beta}$ 
and the constants $C_A$, $C_V$ only depend on the manifold geometry,
we can write the bound as
\begin{equation}\label{eq:proof-bound-expansion-Ru}
    |R(u)| \le C_R \| f\|_{k,\beta} \epsilon^{(k+\beta)/2}  h( \frac{u^2}{4}) 
	\sum_{j=0}^{J_k + 2k+3} u^j,
\end{equation}
where $C_R$  depends  on $k$ and $\calM$,
inheriting the dependence on manifold geometry from the constants $C_A$ and $C_{V}$
as declared beneath \eqref{eq:expansion-h-RA-bound} and \eqref{Lemma:expansion 3} respectively.

We define the last term in \eqref{eq:lemma:expansion4-a} as 
\begin{equation*}
{\rm R}^{(3)} (x)
: = \frac{1}{(2\pi)^{d/2}} \int_{S^{d-1}} \int_{0}^{\frac{2{\delta(\epsilon)}}{\sqrt{\epsilon}}}  R(u) u^{d-1} du d\theta,
\end{equation*}
and will show that it belongs to the remainder. Next, we explore the $i$th term in the summation in  \eqref{eq:lemma:expansion4-a}. Note that 
$$
A_{i_1, \ell }(x, -\theta) B_{i_2}(x, -\theta) V_{i_3}(x,-\theta) 
= (-1)^{i_1+i_2+i_3}
    A_{i_1, \ell}(x,\theta)B_{i_2}(x, \theta)V_{i_3}(x,\theta).
$$ 
\rev{Hence, by the symmetry of $S^{d-1}$, the sphere integral vanishes when $i_1+i_2+i_3$ is odd.  Thus, the $i$th term in the summation can be nonzero only when $i=i_1+i_2+i_3-2\ell$ is even.}
As a result,  only terms with integer powers of $\epsilon$ remain in the summation,
and \eqref{eq:lemma:expansion4-a} can be written as
\begin{align}
G_\epsilon^{(1)} (f) (x)
&= \sum_{j=0}^{ \lfloor k/2 \rfloor } \epsilon^{ j } 
\int_0^{\frac{2{\delta(\epsilon)}}{\sqrt{\epsilon}}} 
	\sum_{\substack{i_1+i_2+i_3-2\ell= 2j \\
             i_1, i_2, i_3 \ge 0 \\
             0 \leq \ell \leq \lfloor i_1/4 \rfloor }}  
             \frac{h^{( \ell )}(u^2) }{(2\pi)^{d/2}}
             \Big( \int_{S^{d-1}} \frac{1}{ \ell! i_2!} 
             	A_{i_1,\ell}(x, \theta)
		\nonumber \\
&~~~~~~~~~~~~~~~~~~~~~~~~~~~~~~~~~~~
	V_{i_3}(x,\theta) 
		\nabla^{i_2}_\theta f(x) 
		d\theta 
		\Big) 
 	u^{2j +2\ell+d-1} du 
 	+ {\rm R}^{(3)} (x), 
 \label{Lemma:expansion 4}
\end{align}
where we also insert in the definition of $B_i$ as in \eqref{Lemma:expansion 2} to make explicit the dependence on the function $f$
(since $A_{i, \ell}$ and $V_{i}$ are determined by $\calM$ and do not involve $f$).

Inside the  summation over $j$ in \eqref{Lemma:expansion 4}, 
the $du$-integral limit ${2{\delta(\epsilon)}}/{\sqrt{\epsilon}}$ depends on $\epsilon$.
Note that  ${2{\delta(\epsilon)}}/{\sqrt{\epsilon}} \sim \sqrt{\log (1/\epsilon)}$ is large when $\epsilon$ is small, 
and we will show that the contribution from $\int_{\frac{2{\delta(\epsilon)}}{\sqrt{\epsilon}}}^{\infty} \cdots du$ can be all put to the remainder term. 
As a result, we can use the contribution from $\int_{0}^{\infty} \cdots du$ to construct the functions $f_j$ in the desired expansion \eqref{G epsilon expansion with remainder}, which does not involve $\epsilon$.
Specifically, we define
\begin{align}
\label{DEfinition of f_j}
 f_{j}(x) 
& : = \int_{0}^{ \infty } 
  	\sum_{\substack{i_1+i_2+i_3-2\ell= 2j \\
             i_1, i_2, i_3 \ge 0 \\
             0 \leq \ell \leq \lfloor i_1/4 \rfloor }}  
              \frac{h^{( \ell )}(u^2)}{(2\pi)^{d/2}}  \\
&~~~~~~~~~~~~~~~~~~
  	\bigg( \int_{S^{d-1}}
  	\frac{1}{ \ell! i_2! } A_{i_1,\ell}(x, \theta) 		
	V_{i_3}(x,\theta) 
 	\nabla^{i_2}_\theta f(x)d\theta 
	 \bigg) u^{2j+2\ell+d-1} du.  
	 \nonumber
\end{align}
To calculate the expression, we use $h^{(\ell)}(r) = {{\frac{h(r)}{(-2)^\ell}=\frac{ e^{-r/2}}{(-2)^\ell}
}}$ and introduce 
 the $i$th moment 
$
 \mathfrak{M}_i : = \int_{0}^{\infty}\frac{h(u^2)}{(2\pi)^{d/2}} u^i du$.
 We also define
\begin{equation}\label{eq:def-Silii-at-x}
S_{i_1, \ell, i_2, i_3}(x)
: = \int_{S^{d-1}}A_{i_1,\ell}(x, \theta)V_{i_3}(x,\theta) \nabla^{i_2}_\theta f(x)d\theta,
\end{equation}
and then we have
\begin{equation}\label{eq:proof-fj-equiv-expression}
f_j(x)
=  \sum_{\substack{i_1+i_2+i_3-2\ell= 2j \\
             i_1, i_2, i_3 \ge 0 \\
             0 \leq \ell \leq \lfloor i_1/4 \rfloor }}   
	 \mathfrak{M}_{2j+2\ell+d-1} 
	\frac{1}{(-2)^{\ell} \ell!i_2!}  S_{i_1, \ell, i_2, i_3}(x). 
\end{equation}
In particular, we can show that $f_0=f$: 
When $j=0$, 
the summation only contains one term where $i_1=i_2=i_3=\ell=0$. 
By definition, $A_{0,0} =1$, $V_0 = 1$, and then 
 $S_{0, 0, 0, 0}(x)=|S^{d-1}|f(x)$. 
Then we have $f_0(x)= \mathfrak{M}_{d-1} |S^{d-1}|f(x) = f(x)$
by that  $\mathfrak{M}_{d-1}= 1/|S^{d-1}| $.

Meanwhile, we define 
\[
{\rm R}^{(4)}_j (x) := 
	\sum_{\substack{i_1+i_2+i_3-2\ell= 2j \\
             i_1, i_2, i_3 \ge 0 \\
             0 \leq \ell \leq \lfloor i_1/4 \rfloor }}   
	\bigg( \int_{\frac{2{\delta(\epsilon)}}{\sqrt{\epsilon}}}^{\infty}  \frac{h(u^2)}{(2\pi)^{d/2}} u^{2j+2\ell+d-1} du \bigg)
	 \frac{1}{(-2)^{\ell} \ell!i_2!}   S_{i_1, \ell, i_2, i_3}(x), 
\]
and then  \eqref{Lemma:expansion 4}  can be written as
\begin{equation}\label{interior expansion f epsilon j}
G_\epsilon^{(1)} (f) (x) 
= \sum_{j=0}^{ \lfloor k/2 \rfloor } \epsilon^{ j } ( f_j(x) - {\rm R}^{(4)}_j (x) ) + {\rm R}^{(3)} (x). 
\end{equation}
Putting together \eqref{Lemma: integral in and out} and \eqref{interior expansion f epsilon j},
we obtain the expansion in the form as \eqref{G epsilon expansion with remainder} where
\begin{align}\label{Lemma: remainder expression}
R_{f,\epsilon}(x)
= {\rm R}^{(2)} (x)
+ {\rm R}^{(3)} (x)
- \sum_{j=0}^{ \lfloor k/2 \rfloor }  \epsilon^j  {\rm R}^{(4)}_j (x).
\end{align}
To prove the lemma, we are to bound the remainder term $\| R_{f,\epsilon} \|_\infty$ and verify the stated properties of $f_j$.

\vspace{10pt}
\noindent
$\bullet$ \underline{\textbf{Bound the remainder $R_{f,\epsilon}$}}
\vspace{5pt}

We first bound $|{\rm R}^{(3)} (x)|$. By definition, we have
\begin{align*}
| {\rm R}^{(3)} (x) |
& \le  
 \frac{|S^{d-1}|}{(2\pi)^{d/2}} 
	 \int_{0}^{\frac{2{\delta(\epsilon)}}{\sqrt{\epsilon}}}  
 	|R(u)| u^{d-1} du  
 \leq  
 \frac{|S^{d-1}|}{(2\pi)^{d/2}} 
	 \int_{0}^{\infty}  
 	|R(u)| u^{d-1} du.
\end{align*}
By the upper bound of $|R(u)|$
as in \eqref{eq:proof-bound-expansion-Ru}, we have
\begin{align*}
| {\rm R}^{(3)} (x) |
& \le  C_R |S^{d-1}| \| f\|_{k,\beta} 
	\epsilon^{(k+\beta)/2}
	\sum_{j=0}^{J_k + 2k+3}
	\int_{0}^{\infty}  \frac{ h(u^2 /4 ) }{(2\pi)^{d/2}}   u^{j+d-1} du \\
& = C_R |S^{d-1}| \| f\|_{k,\beta} 
	\left(\sum_{j=0}^{J_k + 2k+3} \mathfrak{m}_{j+d-1} \right)
	\epsilon^{(k+\beta)/2},
	\quad \mathfrak{m}_i := \int_{0}^{\infty}  \frac{ h(u^2 /4 ) }{(2\pi)^{d/2}}   u^{i} du.
\end{align*}
Recall that $J_k = \max\{ (2k+1)\lfloor \frac{k}{2}\rfloor, \, 4(\lfloor \frac{k}{2} \rfloor + 1) \}$. Hence, $\sum_{j=0}^{J_k + 2k+3}  \mathfrak{m}_{j+d-1}$ can be bounded by a constant depending on $k$ and $d$.
Then  we have
\begin{align}\label{Lemma:remainder bound 1}
|{\rm R}^{(3)} (x)| \leq C_{R,3}  \|f\|_{k, \beta} \epsilon^{{(k+\beta)}/{2}},
\end{align}
where $C_{R,3}$ depends on $d$, $k$, and $\calM$, 
inheriting the $\calM$-dependence from the constants $C_R$ as declared beneath \eqref{eq:proof-bound-expansion-Ru}.

Next, we bound $|\sum_{j=0}^{ \lfloor k/2 \rfloor }  \epsilon^j  {\rm R}^{(4)}_j (x)|$. 
We have derived above before \eqref{eq:lemma:expansion4-a}  that a combination of $\{ i_1,\ell, i_2, i_3 \}$
 that contributes non-zero-ly to the summation in the definition of  $f_j$ (and thus to the definition of ${\rm R}^{(4)}_j (x)$) must satisfy
 $ 0 \le i_1-2 \ell \le k$, $0 \le i_2, i_3 \le k$ and $\ell \le \lfloor k/2\rfloor$,
 and thus we have 
\begin{equation}\label{eq:valid-combi-i1-l-i2-i3}
0 \le i_1 - 2\ell, \, i_2, \, i_3 \le 2j \le k, 
\quad 0 \leq i_1 \leq 2k,
\quad  0 \leq 2\ell \leq k.
\end{equation}
Recall that $\frac{2\delta(\epsilon)}{ \sqrt{\epsilon}}=2\sqrt{(d+k+1)\log(\frac{1}{\epsilon})}$,
and we consider $j, i_1, i_2, i_3$ and $\ell$ that satisfy \eqref{eq:valid-combi-i1-l-i2-i3}.
Recall that $ 0 \le 2j+2\ell \le 2k$,
by Lemma \ref{upper incomplete gamma bound}, when $\epsilon< {1}/{e}$,  
$
\int_{\frac{2{\delta(\epsilon)}}{\sqrt{\epsilon}}}^{\infty}  \frac{h(u^2)}{(2\pi)^{d/2}} u^{2j+2\ell+d-1} du 
\leq \frac{c(k,d)}{(2\pi)^{d/2}}\epsilon^{d+k+1}.
$
By the definition of $S_{i_1, \ell, i_2, i_3}(x)$ in \eqref{eq:def-Silii-at-x} and \eqref{eq:expansion-h-RA-bound}\eqref{Lemma:expansion 2}\eqref{Lemma:expansion 3},
\begin{align*}
|S_{i_1, \ell, i_2, i_3}(x)|  \leq  |S^{d-1}| C_A C_V  \|f\|_{k, \beta}.
\end{align*}
Recall that the range of valid indices $ \{ i_1, \ell, i_2, i_3\}$ as in \eqref{eq:valid-combi-i1-l-i2-i3}.

\rev{For each fixed $j$, there are at most $k^3+1$ terms.  Indeed,
letting $a=i_1-2\ell$, the number of terms is at most
$\sum_{a=0}^{2j}(\lfloor a/2\rfloor+1)(2j-a+1)\le(2j)^3\le k^3$
when $j\ge1$, while for $j=0$ there is one term.}
Therefore, we have
\begin{align}
|{\rm R}^{(4)}_j (x)| \leq & \sum_{\substack{i_1+i_2+i_3-2\ell= 2j \\
             i_1, i_2, i_3 \ge 0 \\
             0 \leq \ell \leq \lfloor i_1/4 \rfloor }}  ( \int_{\frac{2{\delta(\epsilon)}}{\sqrt{\epsilon}}}^{\infty}  \frac{h(u^2)}{(2\pi)^{d/2}} u^{2j+2\ell+d-1} du) |S_{i_1, \ell, i_2, i_3}(x)| \nonumber \\
\leq& 
\rev{(k^3+1)}
\frac{c(k,d)}{(2\pi)^{d/2}}|S^{d-1}| C_A C_V  \|f\|_{k, \beta} \epsilon^{d+k+1}.  \nonumber
\end{align}
Since $\epsilon < {1}/{e} < 1$, we have
\begin{align}\label{Lemma:remainder bound 2}
& \sum_{j=0}^{ \lfloor k/2 \rfloor }  \epsilon^j  |{\rm R}^{(4)}_j (x)| 
\leq  \rev{(\lfloor k/2\rfloor+1)(k^3+1)} \frac{c(k,d)}{(2\pi)^{d/2}}|S^{d-1}| C_A C_V  \|f\|_{k, \beta} \epsilon^{d+k+1}
:=C_{R,4}\|f\|_{k, \beta} \epsilon^{d+k+1}.
\end{align}
where $C_{R,4}$ depends on $d$, $k$, 
and $\calM$, inheriting the $\calM$-dependence from the constants $C_A$ and $C_{V}$.

At last, apply triangle inequality to \eqref{Lemma: remainder expression}, we have
$$ 
|R_{f,\epsilon}(x)| \leq
			 |{\rm R}^{(2)} (x)|+|{\rm R}^{(3)} (x)|+
			\sum_{j=0}^{ \lfloor k/2 \rfloor }  \epsilon^j  | {\rm R}^{(4)}_j (x)|.
$$
By substituting the bounds \eqref{1st part of the integral}, \eqref{Lemma:remainder bound 1}, and \eqref{Lemma:remainder bound 2}, the upper bound for $|R_{f ,\epsilon}(x)|$ in the statement (i) of the lemma follows,
where $ \tilde{C}_1(\mathcal{M},d, k) =  C_{R,2}+ C_{R,3}+ C_{R,4} $ satisfies the declared dependence on $k$ and manifold geometric quantities.

Finally, we collect the requirement on the smallness of $\epsilon$: the needed conditions so far are
$
\epsilon < \epsilon_1 :=\min\{ \epsilon_{1,\calM},  {1}/{e}\}.
$
Hence, $\epsilon_1$ is a constant depending on $d$, $k$, $\xi$ and $\tau$.
This finishes the proof of \eqref{G epsilon expansion with remainder} and statement (i) of the lemma.

\subsection{Proof of $f_j \in C^{k-2j, \beta}(\calM)$ and 
Lemma \ref{lemma:22_holder}(ii)}

Recall the definition of $f_j(x)$ in \eqref{DEfinition of f_j}, 
and we have shown that $f_0 = f$.
Thus, the statement
$\|f_j\|_{k-2j,\beta} \leq \tilde{C}_2 \|f\|_{k, \beta}$
 trivially holds when $j=0$ with constant $\tilde C_2 = 1$.

To prove the cases for $j \ge 1$, 
we will need to analyze the differential property of $f_j$ on the manifold. 
Using the symbol $S_{i_1, \ell, i_2, i_3}(x)$ defined in \eqref{eq:def-Silii-at-x}, we have the equivalent expression of $f_j$ as in \eqref{eq:proof-fj-equiv-expression}, so we focus on the differential property of $S_{i_1, \ell, i_2, i_3}$. 
Recall that a combination of indices $\{ i_1, \ell, i_2, i_3\}$ that contribute to the summation in  \eqref{eq:proof-fj-equiv-expression},
which we call a {\it valid} combination, 
must satisfy \eqref{eq:valid-combi-i1-l-i2-i3}.
Since $ \lfloor k/2 \rfloor \ge j \ge 1$, we consider $k \ge 2$.

Strictly speaking, the expression \eqref{eq:def-Silii-at-x} stands for the value of $S_{i_1, \ell, i_2, i_3}$ at a point $x$ only, where $\theta$ is unit vector in  $T_x \calM$.  
When $x$ moves on $\calM$, the tangent plane $T_x \calM$ also changes, and thus the formal definition of $S_{i_1, \ell, i_2, i_3}(x)$ as a function of $x$ should be
\begin{equation}\label{eq:expression-Silii-1}
S_{i_1, \ell, i_2, i_3}(x)=\int_{S^{d-1}_x}A_{i_1,\ell}(x, \theta)V_{i_3}(x,\theta) \nabla^{i_2}_\theta f(x)d\theta,
\end{equation}
where $S^{d-1}_x$ is the unit ($d-1$)-sphere in $T_x \calM$.
To analyze the covariant derivative of  $S_{i_1, \ell, i_2, i_3}(x)$, we will introduce a parallel frame which provides a differentiable mapping $T( y, \theta)$ that maps from every $y$ (in a neighborhood of $x$)
 and $\theta \in {\rm\bf S}^{d-1}$,  the unit ($d-1$)-sphere in  $\R^d$, 
 to a unit vector in $T_y \calM$.
Using this mapping, we will show that 
\begin{equation}\label{eq:expression-Silii-2}
S_{i_1, \ell, i_2, i_3}(y)=\int_{ {\rm \bf S}^{d-1} } A_{i_1,\ell}(y, T(y, \theta) )V_{i_3}(y, T(y, \theta)) \nabla^{i_2}_{T(y,\theta)} f(y) d \theta, 
\quad \forall y\in B_\xi (x),
\end{equation}
where the domain of $d\theta$ is an ``absolute'' ($d-1$)-sphere in $\R^d$ and is independent from $x$.
Then, the covariant derivative can be taken inside the integral of $d\theta$ and considered for each fixed $\theta$.
The construction of the parallel frame will allow convenient evaluation of the covariant derivative when the mapping $T$ is involved.

\vspace{10pt}
\noindent
$\bullet$ \underline{\textbf{Parallel frame and the covariant derivatives of $S_{i_1, \ell, i_2, i_3}(x)$}}
\vspace{5pt}

We introduce the {\it parallel frame} 
$\{\mathcal{E}_i\}_{i=1}^d$ 
defined on $B_\xi(x) \subset \calM$:
For any $y \in B_{\xi}(x)$, 
recall that $P_{x,y}: T_x \calM \to T_y \calM$ denote the parallel transport from $x$ to $y$;
Let $ \{ E_i\}_{i=1}^d$ be an orthonormal basis of $T_x\mathcal{M}$,
and we define $\mathcal{E}_i(y) = P_{x,y} E_i$, $i=1,\cdots, d$.
At $x$, this gives that $\calE_i(x) = E_i$.
As a result,  $\{\mathcal{E}_i(y)\}_{i=1}^d$ form an orthonormal basis of $T_y\mathcal{M}$.
Meanwhile, because $\calE_i$ is parallel along each radial geodesic,  for any $v \in S_x^{d-1} \subset T_x \calM$, we have
\begin{equation}\label{eq:vanishing-nabla-along-geodesic-calEi}
\nabla_{\dot \gamma} \calE_i = 0  \quad \text{along the geodesic $\gamma(t) = \exp_x( tv)$, $| t | < \xi$.}
\end{equation}
With the parallel frame, we define the mapping $T(y,\theta)$ as 
\[
(y,  \theta = (u_1,\cdots, u_d))
\mapsto T(y, \theta) = \sum_{i=1}^d u_i \mathcal{E}_i(y) \in T_y \calM,
\quad  \forall y \in B_\xi(x), \, \theta \in {\rm\bf S}^{d-1}.
\]
 The mapping $T(y, \cdot): {\rm \bf S}^{d-1} \to S^{d-1}_y$ preserves the measure on ${\rm\bf S}^{d-1}$,
because at any $y$, $\{\mathcal{E}_i(y) \}_{i=1}^d$ form an orthonormal basis. 
Evaluating \eqref{eq:expression-Silii-1} at any $y \in B_\xi(x)$,
we then have \eqref{eq:expression-Silii-2} hold by change of variable of $\theta$.

The usage of \eqref{eq:expression-Silii-2} lies in that we now have an expression of $S_{i_1, \ell, i_2, i_3}$ on a neighborhood $B_\xi (x)$ of $x$. The idea to compute and analyze the covariance derivative of $S_{i_1, \ell, i_2, i_3}$ is by leveraging a ``tensor-field view'' of the integrand on the r.h.s. of  \eqref{eq:expression-Silii-2}.
\rev{
Here and below, if $T_1$ and $T_2$ are tensor fields of orders $p$
and $r$, respectively, we use $T_1T_2$ for the order-$(p+r)$ tensor
field defined by
\[
(T_1T_2)(u_1,\ldots, u_{p}, v_1, \dots, v_r)
:=
T_1(u_1,\ldots,u_{p})
T_2(v_1,\ldots,v_{r}).
\]
We use the usual convention when one of the tensor fields has order zero.
In particular, when all slots are filled by the same vector field $U$,
then $
(T_1T_2)(U,\ldots,U)
=
T_1(U,\ldots,U)T_2(U,\ldots,U).$}

Specifically, for each fixed $\theta \in {\rm\bf S}^{d-1}$, we define $U_\theta(y) := T(y,\theta)$ and then $U_\theta$ is a vector field on $B_\xi(x)$.
We will show that $\forall y \in B_\xi(x)$,
\begin{equation}\label{eq:tensor-field-extension-AVF-goal}
A_{i_1,\ell}(y, T(y, \theta) ) V_{i_3}(y, T(y, \theta)) \nabla^{i_2}_{T(y,\theta)} f(y)
= ( \bar A_{i_1,\ell} \bar V_{i_3} \bar F_{i_2}) ( \underbrace{ U_\theta, \cdots, U_\theta}_{\text{$i_1+i_2+i_3$ many}})|_y, 
\end{equation}
where $\bar A_{i_1,\ell}$,  $\bar V_{i_3}$, and  $\bar F_{i_2}$ are tensor fields on $\calM$ of order $i_1$, $i_3$, and $i_2$ respectively, 
satisfying 
\begin{align}
 A_{i_1,\ell}( y, U_\theta(y) ) & = \bar A_{i_1, \ell} (U_\theta, \cdots, U_\theta)|_y,   \label{eq:tensor-field-extension-A-goal}\\
 V_{i_3}(y, U_\theta(y) ) & = \bar V_{i_3}  (U_\theta, \cdots, U_\theta)|_y,  \label{eq:tensor-field-extension-V-goal} \\
 \nabla^{i_2}_{U_\theta(y)} f(y) & = \bar F_{i_2} (U_\theta, \cdots, U_\theta)|_y. \label{eq:tensor-field-extension-F-goal}
\end{align}
The construction of $\bar F_{i_2}$ is direct by the covariant derivative of $f$: 
we let $\bar F_{i_2} = \nabla^{i_2} f$ which is an order-$i_2$ tensor field.  
For any vector field $U$ on $\calM$, 
we have $\bar F_{i_2}(U, \cdots, U)|_y = \nabla^{i_2}_{U(y)} f(y)$ for any $y \in \calM$ by definition of the covariant derivative. Thus we have \eqref{eq:tensor-field-extension-F-goal} hold for all $0 \le i_2 \le k$, and we also have that 
\begin{equation}\label{eq:nabla-barF-i2-expression}
\nabla^m \bar F_{i_2} = \nabla^{m+ i_2 } f, \quad \forall m \le k - i_2.
\end{equation}

The construction of $\bar A_{i_1, \ell} $ and $\bar V_{i_3} $ are results of Lemma \ref{expansion-of-volume-of-form-and-Euclidean-distance}.
We first consider $\bar V_{i_3} $.
By Lemma \ref{expansion-of-volume-of-form-and-Euclidean-distance}(i) b), when $i_3 \ge 2$, there exists an order-$i_3$ tensor field $\bar V_{i_3} $ on $\calM $ s.t., for any vector field $U$ on $\calM$,
$\bar V_{i_3} (U, \cdots, U)|_y 
= \bar V_{i_3}(y) (U(y), \cdots, U(y))
= V_{i_3 } ( y, U(y))$, $\forall y \in \calM$.
The tensor field $\bar V_{i_3}$ consists of sums of products of the curvature tensor and its covariant derivatives up to ($i_3-2$)-th order,
including a contraction of the tensors.
When $i_3 = 0$ or 1, we set $\bar V_0 = 1$ and $\bar V_1 =0$ which are constant tensor fields.
We then have \eqref{eq:tensor-field-extension-V-goal} hold  for all $0 \le i_3 \le k$, 
and $\nabla^m \bar V_{i_3}$ is a tensor field 
determined by
the curvature tensor and its covariant derivatives up to $\max\{ m+i_3-2, 0\}$-th order.

To construct $\bar A_{i_1, \ell} $, recall the definition of $A_{i_1, \ell}$ in \eqref{eq:def-Ail-no-k} for $i_1 \ge 4$ and $\ell \ge 1$.
For $4 \le i_1 \le 2k$ and $1 \le l \le \lfloor i_1/4 \rfloor$, we define 
$
\bar{A}_{i_1,\ell} : = 
  	 \sum_{\substack{j_1+ \cdots + j_\ell =i_1\\
 		4\le j_1, \cdots ,j_\ell }} \bar{q}_{j_1} \cdots \bar{q}_{j_\ell},
$
where, applying Lemma \ref{expansion-of-volume-of-form-and-Euclidean-distance}(ii) b), 
each $\bar q_{j_i}$ is an order-$j_i$ tensor field on $\calM$ s.t. 
for any vector field $U$ on $\calM$,
$ \bar q_{j_i}(U, \cdots, U)|_y 
= \bar q_{j_i}(y) (U(y), \cdots, U(y))
= q_{j_i } ( y, U(y))$, $\forall y \in \calM$.
In addition, $\bar q_{j_i}$ can be expressed through dot products and sums of the second fundamental form $\Second$ 
and its covariant derivatives up to ($j_i-4$)-th order with coefficients depending on $j_i$.
\rev{Based on the construction of $\bar q_{j_i}$ in Lemma
\ref{expansion-of-volume-of-form-and-Euclidean-distance}(ii) b), we also
know that $\nabla^r\bar q_{j_i}$ involves covariant derivatives of
$\Second$ only through order $r+j_i-4$.}
For $i_1 \le 3$ or $\ell = 0$, we let $\bar A_{i_1, \ell} $ be the constant tensor fields in line with \eqref{eq:def-Ail-no-k-small-i}.
This construction ensures that
$\bar{A}_{i_1,\ell}$ is a tensor field of order $i_1$
satisfying  \eqref{eq:tensor-field-extension-A-goal} for all valid $i_1$ and $\ell$, 
and $\nabla^m \bar{A}_{i_1,\ell}$ is a tensor field consisting of dot products and sums of $\Second$  and and its covariant derivatives up to $\max\{ m+i_1-4, 0\}$-th order. 

By now our construction fulfills \eqref{eq:tensor-field-extension-A-goal}\eqref{eq:tensor-field-extension-V-goal}\eqref{eq:tensor-field-extension-F-goal}, then by definition of $U_\theta$ we have \eqref{eq:tensor-field-extension-AVF-goal} hold. 
Since this holds for any $\theta \in  {\rm \bf S}^{d-1} $,
we can go back to \eqref{eq:expression-Silii-2} and rewrite it as
$$
S_{i_1, \ell, i_2, i_3}(y)= \int_{ {\rm \bf S}^{d-1} } 
	(\bar A_{i_1,\ell} \bar V_{i_3} \bar F_{i_2}) (U_\theta, \cdots, U_\theta)|_y  
	d \theta,
	\quad 
\forall y\in B_\xi (x).$$
In view of the tensor field, this gives that
\[
S_{i_1, \ell, i_2, i_3}= \int_{ {\rm \bf S}^{d-1} } 
	(\bar A_{i_1,\ell} \bar V_{i_3} \bar F_{i_2}) (U_\theta, \cdots, U_\theta)
	d \theta,
	\quad \text{on $B_\xi(x)$.}
\]

We want to compute $\nabla^m_v S_{i_1, \ell, i_2, i_3}(x)$ for an arbitrary $v \in S_x^{d-1} \subset T_x \calM$.
For any such $v$, we consider covariant derivative along the radial geodesic $\gamma(t)$ with $\dot \gamma(0) = v$, then we have
\begin{align*}
\nabla_{\dot \gamma}  S_{i_1, \ell, i_2, i_3} 
& = \int_{ {\rm \bf S}^{d-1} } 
	\nabla_{\dot \gamma} [ (\bar A_{i_1,\ell} \bar V_{i_3} \bar F_{i_2}) (U_\theta, \cdots, U_\theta)  ]
	d \theta \\
& = \int_{ {\rm \bf S}^{d-1} } 
	\nabla(\bar A_{i_1,\ell} \bar V_{i_3} \bar F_{i_2}) (\dot \gamma, U_\theta, \cdots, U_\theta)  
	d \theta
	 \quad \text{along $\gamma(t)$, $|t| < \xi$,}
\end{align*}
where in the second equality we used that $\nabla_{\dot \gamma}  U_\theta = 0$ because of \eqref{eq:vanishing-nabla-along-geodesic-calEi} and $U_\theta \in \text{span}\{ \calE_i, i=1,\cdots, d\}$.
Because $\nabla(\bar A_{i_1,\ell} \bar V_{i_3} \bar F_{i_2})$ is also a tensor field on $\calM$, 
and $\nabla_{\dot \gamma} \dot \gamma =0$ along $\gamma(t)$ as well,
we have 
$
\nabla^2_{\dot \gamma}  S_{i_1, \ell, i_2, i_3}  
= \int_{ {\rm \bf S}^{d-1} } 
	\nabla^2(\bar A_{i_1,\ell} \bar V_{i_3} \bar F_{i_2}) ( \dot \gamma, \dot \gamma, U_\theta, \cdots, U_\theta)  
	d \theta$ along $\gamma(t)$, $|t| < \xi$.
Repeating this argument recursively, we have
\begin{equation}\label{eq:nabla-dot-gamma-S-along-gamma}
\nabla^m_{\dot \gamma}  S_{i_1, \ell, i_2, i_3} 
= \int_{ {\rm \bf S}^{d-1} } 
	\nabla^m(\bar A_{i_1,\ell} \bar V_{i_3} \bar F_{i_2}) 
	(\underbrace{\dot \gamma,  \cdots, \dot \gamma}_{\text{$m$ many}},
	U_\theta, \cdots, U_\theta)  
	d \theta
	 \quad \text{along $\gamma(t)$, $|t| < \xi$,}
\end{equation}
and this goes up to a high order of $m$ as long as the covariant derivatives exist. 
Assuming that tensor fields $\bar A_{i_1,\ell}$ and $\bar V_{i_3}$ based on manifold geometric quantities have sufficient regularity, 
the only constraint is that $ m +i_2 \le k $ for $\nabla^{i_2 + m} f$ to exist.
Since $i_2 \le 2j$ by \eqref{eq:valid-combi-i1-l-i2-i3},
we can always take $m$ up to $k-2j$.

Evaluating \eqref{eq:nabla-dot-gamma-S-along-gamma} at $t=0$, recall that $\gamma(0) = x$, $\dot \gamma (0) = v$,  we have
\begin{align*}
\nabla^m_{v}  S_{i_1, \ell, i_2, i_3}(x) 
& = \int_{ {\rm \bf S}^{d-1} } 
	\nabla^m(\bar A_{i_1,\ell} \bar V_{i_3} \bar F_{i_2})(x) 
	(\underbrace{ v,  \cdots, v}_{\text{$m$ many}},
	U_\theta(x), \cdots, U_\theta(x))  
	d \theta \\
& = \int_{ { S}_x^{d-1} } 
	\nabla^m(\bar A_{i_1,\ell} \bar V_{i_3} \bar F_{i_2})(x) 
	(\underbrace{ v,  \cdots, v}_{\text{$m$ many}},
	\theta, \cdots, \theta)  
	d \theta,
\end{align*}
where  the second equality is by that $U_\theta(x) = T(x,\theta)$ and we use change of variable again to put the integral of $d\theta$ on $S_x^{d-1}$.
We now introduce a condense notation:  for order-$r$ tensor $T^{(r)}$, we define
\begin{equation}\label{eq:condense-notation-tensor-Ti}
\nabla^m_v T^{(r)} (x)( \theta)
: = \nabla^m T^{(r)} (x) (  \underbrace{v, \cdots, v}_{\text{$m$ many}}, \underbrace{\theta, \cdots, \theta}_{\text{$r$ many}}).
\end{equation}
Then we obtain the expression
\begin{equation}\label{eq:compute-dir-S-expression-any-x}
\nabla^m_v S_{i_1, \ell, i_2, i_3}(x)
= \int_{S_x^{d-1}} \nabla^m_v (\bar A_{i_1,\ell} \bar V_{i_3} \bar F_{i_2}) (x)( \theta) d\theta, 
\quad  \forall v \in T_x \calM,  \, \forall 0 \le m \le k-2j.
\end{equation}
The argument to compute $\nabla^m_v S_{i_1, \ell, i_2, i_3}(x)$ so far chooses an $x$ to begin with (to construct the parallel frame),
but the argument holds for arbitrary $x$, 
and thus we have  \eqref{eq:compute-dir-S-expression-any-x} hold for any $x \in \calM$. 

By now,
assuming sufficient regularity of the manifold $\calM$, we have shown that $S_{i_1, \ell, i_2, i_3} \in C^{k-2j}(\calM)$ for each valid combination of the indices $\{i_1, \ell, i_2, i_3\}$.
This implies that $f_j \in C^{k-2j}(\calM)$. Next, we will upperbound the $\| \cdot \|_\infty$ norm of $ \nabla^m S_{i_1, \ell, i_2, i_3} $ for $m\le k-2j$
and also $L_{k-2j, \beta} (S_{i_1, \ell, i_2, i_3})$ by a multiple of $\| f\|_{k, \beta}$, 
which then bounds the $\| \cdot \|_{k-2j, \beta}$ of $S_{i_1, \ell, i_2, i_3}$ and subsequently that of $f_j$.

\vspace{10pt}
\noindent
$\bullet$
\underline{\textbf{Bound $\| \nabla^m S_{i_1, \ell, i_2, i_3}\|_{\infty}$ by $\|f\|_{k, \beta}$, for $0 \leq m \leq k-2j$ }}
\vspace{5pt}

Recall that we have $j \ge 1$ and $k \ge 2$.
We compute the integrand in \eqref{eq:compute-dir-S-expression-any-x} by the Product Rule:
For any $x \in \calM$ and any $\theta, \, v \in S_x^{d-1} \subset T_x \calM$,
\begin{equation}\label{eq:prod-rule-AVF-proof-1}
\nabla^m_v ( \bar A_{i_1,\ell} \bar V_{i_3} \bar F_{i_2}) (x)( \theta)
= \sum_{i=0}^m \binom{m}{i}  \nabla_v^{m-i} ( \bar A_{i_1,\ell} \bar V_{i_3} ) (x)(\theta) 
								\nabla_v^{i} (\nabla^{i_2} f)(x)(\theta),
\end{equation}
where we inserted \eqref{eq:nabla-barF-i2-expression} to reveal the covariant derivatives of $f$ in the expression.
We first derive a bound of the $ \nabla_v^{m-i} ( \bar A_{i_1,\ell} \bar V_{i_3} ) (x)(\theta) $ term.

We will consider up to ($k-2j+1$)-th covariant derivative of $\bar A_{i_1,\ell} \bar V_{i_3} $,
and will ensure sufficient manifold regularity later.
\rev{For $0\le p\le k-2j+1$ and independent derivative vectors
$w_1,\ldots,w_p$, the product rule gives}
\begin{equation}\label{eq:prod-rule-AV-proof-1}
\rev{
\begin{aligned}
&\nabla^p(\bar A_{i_1,\ell}\bar V_{i_3})(x) 
(w_1,\ldots,w_p,a_1,\ldots,a_{i_1},b_1,\ldots,b_{i_3})\\
&=\sum_{I\subseteq\{1,\ldots,p\}}
\nabla^{|I|}\bar A_{i_1,\ell}(x)(w_I,a_1,\ldots,a_{i_1})
\nabla^{p-|I|}\bar V_{i_3}(x)(w_{I^c},b_1,\ldots,b_{i_3}),
\end{aligned}
}
\end{equation}
\rev{where $w_I$ and $w_{I^c}$ are listed in the order inherited from
$w_1,\ldots,w_p$.}
Recall that $\nabla^i \bar V_{i_3}$ is a tensor field 
determined by
the curvature tensor and its covariant derivatives up to $\max\{ i+i_3-2, 0\}$-th order,
	then we have
	
	\rev{
	\[
	\bar C_V:=
	\sup_{\substack{0\le q+i_3\le k+1\\
	                  0\le q,\ 0\le i_3\le k}}
	\sup_{x\in\calM}
	\sup_{\substack{w_1,\ldots,w_q\in S_x^{d-1}\\
	                  u_1,\ldots,u_{i_3}\in S_x^{d-1}}}
	\left|
	\nabla^q\bar V_{i_3}(x)
	(w_1,\ldots,w_q,u_1,\ldots,u_{i_3})
	\right|.
	\]
	}
	\rev{When $q=0$ or $i_3=0$, the corresponding list of vectors is omitted.}

{{The constant $\bar C_V$ depends on $d$, $k$, and the sup norms of the curvature tensor of $\calM$ and its covariant derivatives up to order $k-1$.}}
	
	\rev{Because $0\le q\le p\le k-2j+1$ and $0\le i_3\le2j$ by
	\eqref{eq:valid-combi-i1-l-i2-i3}, we have $q+i_3\le k+1$, so this
	definition covers every derivative of $\bar V_{i_3}$ used below.}

Similarly, since $\nabla^i \bar{A}_{i_1,\ell}$ is a tensor field consisting of dot products and sums of $\Second$  and and its covariant derivatives up to $\max\{ i+i_1-4, 0\}$-th order we have 
	
	\rev{
	\[
	\bar C_A:=
	\sup_{\substack{0\le i_1\le2k,\ 0\le\ell\le\lfloor i_1/4\rfloor\\
	                  q\ge0,\ q+i_1\le2k+1}}
	\sup_{x\in\calM}
	\sup_{\substack{w_1,\ldots,w_q\in S_x^{d-1}\\
	                  u_1,\ldots,u_{i_1}\in S_x^{d-1}}}
	\left|
	\nabla^q\bar A_{i_1,\ell}(x)
	(w_1,\ldots,w_q,u_1,\ldots,u_{i_1})
	\right|.
	\]
	}
	\rev{When $q=0$ or $i_1=0$, the corresponding list of vectors is omitted.}

{{The constant $\bar C_A$ depends on $k$ and the sup norms of $\Second$ and its covariant derivatives up to order $2k-3$.}}
	
	\rev{Similarly, $0\le p-q\le k-2j+1$ and
	$i_1\le2j+2\ell$ imply
	$(p-q)+i_1\le k+1+2\ell\le2k+1$, so the definition of $\bar C_A$
	covers every derivative of $\bar A_{i_1,\ell}$ used below.  Both
	$\bar C_A$ and $\bar C_V$ are finite because the relevant geometric tensor
	fields and their required covariant derivatives are continuous on the compact
	manifold $\calM$.}

	\rev{Since the sum in \eqref{eq:prod-rule-AV-proof-1} has $2^p$ subsets,
	the strengthened constants give}
	
	\begin{equation}\label{eq:prod-rule-AV-proof-2}
	\rev{
	\begin{aligned}
	&\sup_{x\in\calM}
	\sup_{\substack{w_1,\ldots,w_p\in S_x^{d-1}\\
	                  a_1,\ldots,a_{i_1}\in S_x^{d-1}\\
	                  b_1,\ldots,b_{i_3}\in S_x^{d-1}}}
	\big|\nabla^p(\bar A_{i_1,\ell}\bar V_{i_3})(x)
	(w_1,\ldots,w_p,a_1,\ldots,a_{i_1},b_1,\ldots,b_{i_3})\big| 
	\le2^p\bar C_A\bar C_V,
	\quad 0\le p\le k-2j+1.
	\end{aligned}
	}
	\end{equation}

	\rev{\medskip\noindent\textbf{Auxiliary symmetrization and commutator estimate.}}
	\rev{For a  tensor field $T$ of order $m\ge1$, define its full
	symmetrization by}
	\[
	\rev{
	T^{\mathrm{sym}}(v_1,\ldots,v_m)
	:=\frac{1}{m!}\sum_{\sigma\in\mathfrak S_m}
	T(v_{\sigma(1)},\ldots,v_{\sigma(m)}).
	}
	\]
	\rev{Permuting the arguments of $T^{\mathrm{sym}}$ only reindexes the sum,
	so $T^{\mathrm{sym}}$ is symmetric; on the diagonal every summand equals
	$T(v,\ldots,v)$, so $T^{\mathrm{sym}}$ has the same diagonal as $T$.
	Therefore, Banach's Theorem, applied only to the symmetric tensor
	$T^{\mathrm{sym}}$, gives}
	\[
	\rev{
	\sup_{v_1,\ldots,v_m\in S_x^{d-1}}
	|T^{\mathrm{sym}}(v_1,\ldots,v_m)|
	=
	\sup_{v \in S_x^{d-1}}
	|T^{\mathrm{sym}}(v,\ldots, v)|
	=\sup_{v\in S_x^{d-1}}|T(v,\ldots,v)|.
	}
	\]
	\rev{For a scalar function $f$, define the symmetry defect}
	\[
	\rev{
	\mathcal D_m(f):=\nabla^m f-(\nabla^m f)^{\mathrm{sym}}.
	}
	\]
	\rev{Set $\mathcal D_0(f)=0$.  Also $\mathcal D_1(f)=0$, and
	$\mathcal D_2(f)=0$ because the Hessian of a scalar function is symmetric for
	the Levi--Civita connection.}
	\rev{We next record the curvature structure of the defect when $m\ge3$.
	The final two slots of $\nabla^m f$ are the two Hessian slots.  Thus}
	\[
	\rev{
	\nabla^m f(v_1,\ldots,v_{m-2},v_{m-1},v_m)
	-
	\nabla^m f(v_1,\ldots,v_{m-2},v_m,v_{m-1})=0,
	}
	\]
	\rev{so exchanging these final two slots has zero cost.  To exchange the
	first two slots, set $Q=\nabla^{m-2}f$.  At a fixed point, extend the inserted
	vectors locally so that their covariant derivatives vanish at that point.
	The Ricci commutator identity gives}
	\[
	\rev{
	\begin{aligned}
	&\nabla^m f(v_1,v_2,v_3,\ldots,v_m)
	-\nabla^m f(v_2,v_1,v_3,\ldots,v_m)
	=
	\big(\nabla_{v_1}\nabla_{v_2}-\nabla_{v_2}\nabla_{v_1}
	      -\nabla_{[v_1,v_2]}\big)Q(v_3,\ldots,v_m).
	\end{aligned}
	}
	\]
	\rev{The right-hand side is a finite sum of contractions of
	$\texttt{R}(v_1,v_2)$ with $\nabla^{m-2}f$.  Inductively, for
	$1\le j\le m-2$, exchanging slots $j$ and $j+1$ produces a finite linear
	combination of contractions and slot permutations formed from the factors}
	\[
	\rev{
	\nabla^a\texttt{R}
	\quad\text{and}\quad
	\nabla^{m-2-a}f,
	\qquad 0\le a\le j-1.
	}
	\]
	\rev{Indeed, the case $j=1$ is the displayed Ricci commutator identity for
	the exchange of slots $1$ and $2$.  For $j\ge2$,
	write $\nabla^m f=\nabla(\nabla^{m-1}f)$ and differentiate the corresponding
	adjacent-slot defect of $\nabla^{m-1}f$.  The product rule either differentiates
	the curvature factor or the derivative of $f$, giving terms formed respectively
	from $\nabla^{a+1}\texttt{R}$ and $\nabla^{m-3-a}f$, or from
	$\nabla^a\texttt{R}$ and $\nabla^{m-2-a}f$.}
	\rev{This proves the induction.  Every permutation is a product of adjacent
	transpositions.  Telescoping along such products, averaging over all
	permutations, and using the zero-cost final exchange show that
	$\mathcal D_m(f)$ is a finite linear combination of contractions and slot
	permutations formed from the factors}
	\[
	\rev{
	\nabla^a\texttt{R}
	\quad\text{and}\quad
	\nabla^{m-2-a}f,
	\qquad 0\le a\le m-3.
	}
	\]
	\rev{For $q=0$ the claim is immediate. For $q\ge1$, write
\[
\nabla^qf=(\nabla^qf)^{\mathrm{sym}}+\mathcal D_q(f).
\]
By Banach's theorem, the symmetric part is bounded by the diagonal
quantity appearing in $\|f\|_{k,\beta}$. The curvature representation
of $\mathcal D_q(f)$ involves only derivatives of $f$ of order at most
$q-2$. Induction on $q$ therefore yields the required independent-slot
bound for every $0\le q\le k$.}

\rev{Similarly, differentiating the curvature representation of
$\mathcal D_k(f)$ produces only contractions and slot permutations
formed either from $\nabla^{a+1}\texttt{R}$ and
$\nabla^{k-2-a}f$, or from $\nabla^a\texttt{R}$ and
$\nabla^{k-1-a}f$, where $0\le a\le k-3$. Thus this derivative is also
controlled by $\|f\|_{k,\beta}$. Since only finitely many derivative
orders are involved, there exists a finite constant
$C_{\nabla}=C_{\nabla}(\calM,d,k)\ge1$, depending only on $d,k$ and,
when $k\ge3$, the sup-norm bounds of $\texttt{R}$ and its covariant
derivatives through order $k-2$, such that the following two estimates
hold simultaneously:}
\begin{equation*}\tag{A.89a}
\label{eq:bound-independent-slot-nabla-q-f}
\rev{
\begin{aligned}
\sup_{\substack{x\in\calM\\
v_1,\ldots,v_q\in S_x^{d-1}}}
|\nabla^qf(x)(v_1,\ldots,v_q)|
&\le C_{\nabla}\|f\|_{k,\beta},
&&0\le q\le k,\\
\sup_{\substack{x\in\calM\\
w,v_1,\ldots,v_k\in S_x^{d-1}}}
|\nabla\mathcal D_k(f)(x)(w,v_1,\ldots,v_k)|
&\le C_{\nabla}\|f\|_{k,\beta}.
\end{aligned}
}
\end{equation*}
\rev{In the first line the vector list is omitted when $q=0$. For
$k\le2$, the second line is immediate because $\mathcal D_k(f)=0$.
In particular, that estimate does not require $\nabla^{k+1}f$: it is
obtained by differentiating the curvature representation of
$\mathcal D_k(f)$ rather than $\nabla^kf$ directly.\medskip}

	We are ready to go back to \eqref{eq:prod-rule-AVF-proof-1}.
	Recall that at any $x \in \calM$,
	
	\begin{equation}\label{eq:bound-nabla-i-i2-f-proof1}
	\rev{
	\sup_{v,\,\theta\in S_x^{d-1}}
	|\nabla_v^i(\nabla^{i_2}f)(x)(\theta)|
	\le C_{\nabla}\|f\|_{k,\beta},
	}
	\end{equation}
	
	\rev{Indeed, $i\le m\le k-2j$ and $i_2\le2j$ imply
	$i+i_2\le k$.  Thus, applying the first estimate in
	\eqref{eq:bound-independent-slot-nabla-q-f}
	with $q=i+i_2$, using $i$ copies of $v$ and $i_2$ copies of
	$\theta$, yields \eqref{eq:bound-nabla-i-i2-f-proof1}.}
Together with \eqref{eq:prod-rule-AV-proof-2}, we have
\begin{align*}
| \nabla^m_v ( \bar A_{i_1,\ell} \bar V_{i_3} \bar F_{i_2}) (x)( \theta) |
& \le 
\sum_{i=0}^m \binom{m}{i}  | \nabla_v^{m-i} ( \bar A_{i_1,\ell} \bar V_{i_3} ) (x)(\theta)| 
									| \nabla_v^{i} (\nabla^{i_2} f)(x)(\theta)| \\
	&{{\le
	\textcolor{cyan}{C_{\nabla}}\Big( \sum_{i=0}^m \binom{m}{i}  2^{m-i} \Big) \bar C_A \bar C_V
	\|f \|_{k, \beta}
	=\textcolor{cyan}{C_{\nabla}}3^m\bar C_A\bar C_V\|f\|_{k,\beta},}}
	\end{align*}
and this holds for any $x\in \calM$ and any $v, \, \theta \in S_x^{d-1}$.
Then \eqref{eq:compute-dir-S-expression-any-x} gives that
	
	{{\[
	\begin{aligned}
	| \nabla^m_v S_{i_1, \ell, i_2, i_3}(x) |
	&\le \int_{S_x^{d-1}} | \nabla^m_v (\bar A_{i_1,\ell} \bar V_{i_3} \bar F_{i_2}) (x)( \theta) | d\theta\\
	&\le \textcolor{cyan}{C_{\nabla}}|{\rm \bf S}^{d-1}|3^m\bar C_A\bar C_V\|f\|_{k,\beta}.
	\end{aligned}
	\]
	}}
for any $x\in \calM$ and any $v \in S_x^{d-1}$, which means that, for all $0 \le m \le k-2j$,
	
	\begin{equation}\label{eq:bound-nabla-m-S-1}
	{{\begin{aligned}
	\| \nabla^m S_{i_1, \ell, i_2, i_3} \|_\infty
	&= \sup_{x \in \calM} \sup_{v \in S_x^{d-1}}
	| \nabla^m_v S_{i_1, \ell, i_2, i_3}(x) |\\
	&\le \textcolor{cyan}{C_{\nabla}}|{\rm \bf S}^{d-1}| 3^m
	\bar C_A \bar C_V  \|f \|_{k, \beta}.
	\end{aligned}
	}}
	\end{equation}

\vspace{10pt}
\noindent
$\bullet$
\underline{\textbf{Bound $L_{k-2j, \beta}(S_{i_1, \ell, i_2, i_3})$ by $\|f\|_{k, \beta}$}}
\vspace{5pt}

We consider fixed $i_1, \ell, i_2, i_3$ and omit the subscript in the notations of $S_{i_1, \ell, i_2, i_3}$, 
$\bar A_{i_1, \ell }$, $\bar V_{i_3}$, and $\bar F_{i_2}$
for brevity. Let $m = k-2j$,  and recall that $L_{m, \beta}(S) = \sup_{x \in \calM} L_{m,\beta} ( S, x) $, where for any $x \in \calM$,
$
L_{m, \beta} (S, x) 
= \sup_{y \in B_\xi (x)} \sup_{v \in S_x^{d-1}} 
	{| \nabla_v^{m} S(x) - \nabla_{P_{x,y} v}^{m} S(y)  | }/{ d_\calM(x,y)^\beta}.
$
By \eqref{eq:compute-dir-S-expression-any-x}, we have
\begin{align*}
&  \nabla_v^{m} S(x) - \nabla_{P_{x,y} v}^{m} S(y)
=   \int_{S_x^{d-1}} \nabla^{m}_v (\bar A \bar V \bar F) (x)( \theta) d\theta
	- \int_{S_y^{d-1}} \nabla^{m}_{P_{x,y}v} (\bar A  \bar V \bar F) {{(y)}}( \theta) d\theta 	\\
& ~~~
= \int_{S_x^{d-1}} \big(  
			\nabla^{m}_v (\bar A \bar V \bar F ) (x)( \theta) 
 			-  \nabla^{m}_{P_{x,y} v} (\bar A \bar V \bar F) (y)( P_{x,y} \theta) 
			\big) d\theta 
			\\
& ~~~
= \sum_{i=0}^m \binom{m}{i}  \int_{S_x^{d-1}}  
		\big(  \nabla_v^{m-i} ( \bar A \bar V ) (x)(\theta) 
		\nabla_v^{i} (\nabla^{i_2} f)(x)(\theta)	 \\
&~~~~~~~~~~~~~~~~~~~~~~~~~~~~~~~~~~~		
		- \nabla_{P_{x,y} v }^{m-i} ( \bar A \bar V) (y) (P_{x,y} \theta) 
		\nabla_{P_{x,y} v}^{i} (\nabla^{i_2} f)(y)(P_{x,y} \theta) 
		 \big) d\theta,
\end{align*}
where the second equality is by change of variable, and the last equality is by Product Rule.
By triangle inequality,
\begin{align} 
&   | \nabla_v^{m} S(x) - \nabla_{P_{x,y} v}^{m} S(y)|   \nonumber  \\
 \le  & \sum_{i=0}^m 
 	\binom{m}{i}  \int_{S_x^{d-1}} 
 	\Big(  
		|\nabla_v^{m-i} ( \bar A \bar V ) (x)(\theta)| 
		\big| \nabla_v^{i} (\nabla^{i_2} f)(x)(\theta)
		- \nabla_{P_{x,y} v}^{i} (\nabla^{i_2} f)(y)( P_{x,y} \theta) \big|
			\nonumber \\
&~~~~~~	
		+  | \nabla_{P_{x,y} v}^{i} (\nabla^{i_2} f)(y)( P_{x,y} \theta) |
		\big|   \nabla_v^{m-i} ( \bar A \bar V ) (x)(\theta) - \nabla_{P_{x,y} v }^{m-i} ( \bar A \bar V ) (y)(P_{x,y} \theta) \big|
		\Big) d\theta, 		\label{eq:bound-Sx-Sy-proof-1} 
\end{align}
and recall that $i + i_2 \le k$.
For each $0 \le i \le m$ fixed,
\rev{the first estimate in \eqref{eq:bound-independent-slot-nabla-q-f},
with $q=i+i_2$, gives
$| \nabla_{P_{x,y} v}^{i} (\nabla^{i_2} f)(y)( P_{x,y} \theta) |
\le C_{\nabla}\|f\|_{k,\beta}$.}
Meanwhile, 
though the indices of $\bar A \bar V$ is omitted in \eqref{eq:bound-Sx-Sy-proof-1}, we have \eqref{eq:prod-rule-AV-proof-2} applicable to bound $|\nabla_v^{m-i} ( \bar A \bar V ) (x)(\theta)| $ because $m-i \le m =k-2j$.
This gives 
\begin{align}
& | \nabla_v^{m} S(x) - \nabla_{P_{x,y} v}^{m} S(y)| \nonumber  \\
 & \le   \sum_{i=0}^m \binom{m}{i}  \int_{S_x^{d-1}}  
		\Big(  
		 2^{m-i}  \bar C_A \bar C_V 
		\big| \nabla_v^{i} (\nabla^{i_2} f)(x)(\theta)
		- \nabla_{P_{x,y} v}^{i} (\nabla^{i_2} f)(y)(P_{x,y} \theta) \big|
			\nonumber \\
&~~~~~~ 			
		+  \rev{C_{\nabla}\| f\|_{k, \beta}}
		\big|   \nabla_v^{m-i} ( \bar A \bar V ) (x)(\theta) - \nabla_{P_{x,y} v }^{m-i} ( \bar A \bar V ) (y)(P_{x,y} \theta) \big|
		\Big) 
		d\theta. 		\label{eq:bound-Sx-Sy-proof-2}
\end{align}
We claim that 
\begin{equation}\label{eq:bound-Lip-derivative-AV-goal}
\big|   \nabla_v^{m-i} ( \bar A \bar V ) (x)(\theta) - \nabla_{P_{x,y} v }^{m-i} ( \bar A \bar V ) (y)(P_{x,y} \theta) \big|
\le 2^{m-i+1} \bar C_A \bar C_V d_{\calM}(x,y),
\end{equation}

\begin{equation}\label{eq:bound-Lip-derivative-f-goal}
\rev{
\begin{aligned}
&\big| \nabla_v^{i} (\nabla^{i_2} f)(x)(\theta)
	- \nabla_{P_{x,y} v}^{i} (\nabla^{i_2} f)(y)(P_{x,y} \theta) \big|\\
&\qquad\le
\left(1+C_{\nabla}\max\{\texttt{diam}(\calM),1\}\right)
\|f\|_{k,\beta}d_{\calM}(x,y)^\beta,
\qquad i+i_2\le k.
\end{aligned}
}
\end{equation}

{{We will use the following distance comparison:}}
\begin{equation*}\tag{A.95a}
\label{eq:distance-beta-comparison}
{{d_{\calM}(x,y)
\le\max\{\texttt{diam}(\calM),1\}d_{\calM}(x,y)^\beta.
}}
\end{equation*}
Indeed, since $ 0 < \beta \le 1$,
 if $\texttt{diam}(\calM) \le 1$ then $d_{\calM}(x,y) \le 1$ and $d_{\calM}(x,y) \le d_{\calM}(x,y)^\beta$;
if $\texttt{diam}(\calM) > 1$ then $d_{\calM}(x,y)^{1-\beta} \le\texttt{diam}(\calM)^{1-\beta} \le  \texttt{diam}(\calM)$.

\rev{Since $C_{\nabla}\ge1$ and $\max\{\texttt{diam}(\calM),1\}\ge1$,
inserting \eqref{eq:bound-Lip-derivative-AV-goal} and
\eqref{eq:bound-Lip-derivative-f-goal} into
\eqref{eq:bound-Sx-Sy-proof-2} gives}
\rev{
\begin{align*}
&|\nabla_v^mS(x)-\nabla_{P_{x,y}v}^mS(y)|\\
&\quad\le
\sum_{i=0}^m\binom mi\int_{S_x^{d-1}}
4C_{\nabla}\max\{\texttt{diam}(\calM),1\}
2^{m-i}\bar C_A\bar C_V\|f\|_{k,\beta}
d_{\calM}(x,y)^\beta\,d\theta\\
&\quad=
4|{\rm\bf S}^{d-1}|C_{\nabla}3^m
\max\{\texttt{diam}(\calM),1\}\bar C_A\bar C_V
\|f\|_{k,\beta}d_{\calM}(x,y)^\beta.
\end{align*}}

\rev{Taking the supremum over $x$ and recalling that $m=k-2j$ gives}
\begin{equation}\label{eq:bound-L-S-proof-3}
\rev{
L_{k-2j,\beta}(S_{i_1,\ell,i_2,i_3})
\le4|{\rm\bf S}^{d-1}|C_{\nabla}3^{k-2j}
\max\{\texttt{diam}(\calM),1\}\bar C_A\bar C_V
\|f\|_{k,\beta}.
}
\end{equation}
This also shows that $S_{i_1, \ell, i_2, i_3}$, and subsequently $f_j$, is in $C^{k-2j, \beta}(\calM)$.

It remains to establish  \eqref{eq:bound-Lip-derivative-AV-goal}\eqref{eq:bound-Lip-derivative-f-goal} to finish the proof of \eqref{eq:bound-L-S-proof-3}.
To do so, we utilize Lemma \ref{lemma:bound-Lip-tensor-field-derivative} proved in Appendix. 

\vspace{5pt}
\noindent
\underline{Proof of \eqref{eq:bound-Lip-derivative-AV-goal}}: 

\rev{Apply Lemma \ref{lemma:bound-Lip-tensor-field-derivative} to the
order-$(i_1+i_3)$ tensor field $\bar A\bar V$ with $p=m-i$.
The independent-slot quantity $M_{m-i+1}$ in that lemma is bounded by
$2^{m-i+1}\bar C_A\bar C_V$ by \eqref{eq:prod-rule-AV-proof-2}, since
$m-i+1\le m+1=k-2j+1$.  Lemma
\ref{lemma:bound-Lip-tensor-field-derivative} therefore gives
\eqref{eq:bound-Lip-derivative-AV-goal}.}

\vspace{5pt}
\noindent
\underline{Proof of \eqref{eq:bound-Lip-derivative-f-goal}}:

\rev{First suppose that $i+i_2\le k-1$.  Apply Lemma
\ref{lemma:bound-Lip-tensor-field-derivative} to the order-$i_2$ tensor
field $\nabla^{i_2}f$ with $p=i$.  By the first estimate in
\eqref{eq:bound-independent-slot-nabla-q-f}, with
$q=i+i_2+1\le k$, the independent-slot quantity $M_{i+1}$ in that lemma
is bounded by $C_{\nabla}\|f\|_{k,\beta}$.  Hence}
\[
\rev{
\begin{aligned}
&\big|\nabla_v^i(\nabla^{i_2}f)(x)(\theta)
-\nabla_{P_{x,y}v}^i(\nabla^{i_2}f)(y)(P_{x,y}\theta)\big| 
\le C_{\nabla}\|f\|_{k,\beta}d_{\calM}(x,y).
\end{aligned}
}
\]
\rev{Combined with the distance comparison \eqref{eq:distance-beta-comparison},
this proves \eqref{eq:bound-Lip-derivative-f-goal} in this case.}

\rev{Now suppose that $i+i_2=k$.  Let $\{E_s\}_{s=1}^d$ be an
orthonormal basis of $T_x\calM$, and use
$\{P_{x,y}E_s\}_{s=1}^d$ as the orthonormal basis of $T_y\calM$.
Let $A_x$ and $A_y$ be the coordinate representations of
$\nabla^kf(x)$ and $\nabla^kf(y)$ in these two parallel frames; these
tensors need not be symmetric.  If $\mathbf v$ and
$\boldsymbol\theta$ are the coordinate vectors of $v$ and $\theta$ at
$x$, then they are also the coordinate vectors of $P_{x,y}v$ and
$P_{x,y}\theta$ at $y$.  Therefore the difference on the left-hand side
of \eqref{eq:bound-Lip-derivative-f-goal} equals}
\[
\rev{
(A_x-A_y)(\underbrace{\mathbf v,\ldots,\mathbf v}_{\text{$i$ many}},
\underbrace{\boldsymbol\theta,\ldots,\boldsymbol\theta}_{\text{$i_2$ many}}).
}
\]
\rev{By linearity of symmetrization,}
\[
\rev{
A_x-A_y=(A_x^{\mathrm{sym}}-A_y^{\mathrm{sym}})
+\big[(A_x-A_x^{\mathrm{sym}})-(A_y-A_y^{\mathrm{sym}})\big].
}
\]
\rev{The first term is symmetric, and its diagonal agrees with that of
$A_x-A_y$.  Banach's theorem and the definition of
$L_{k,\beta}(f,x)$ therefore give}
\[
\rev{
\begin{aligned}
&\sup_{\mathbf u_1,\ldots,\mathbf u_k\in{\rm\bf S}^{d-1}}
|(A_x^{\mathrm{sym}}-A_y^{\mathrm{sym}})
(\mathbf u_1,\ldots,\mathbf u_k)|\\
&\qquad =
\sup_{\mathbf w\in{\rm\bf S}^{d-1}}
|(A_x-A_y)(\mathbf w,\ldots,\mathbf w)| 
\le L_{k,\beta}(f,x)d_{\calM}(x,y)^\beta.
\end{aligned}
}
\]
\rev{The tensors $A_x-A_x^{\mathrm{sym}}$ and
$A_y-A_y^{\mathrm{sym}}$ represent $\mathcal D_k(f)$ at $x$ and $y$,
respectively.  When $k\le2$ they vanish.  When $k\ge3$, the construction
of the second estimate in \eqref{eq:bound-independent-slot-nabla-q-f}
shows that $\mathcal D_k(f)$ is $C^1$; applying
Lemma \ref{lemma:bound-Lip-tensor-field-derivative} with $p=0$ and the
second estimate in \eqref{eq:bound-independent-slot-nabla-q-f} gives}
\[
\rev{
\begin{aligned}
&\sup_{\mathbf u_1,\ldots,\mathbf u_k\in{\rm\bf S}^{d-1}}
\big|\big[(A_x-A_x^{\mathrm{sym}})-(A_y-A_y^{\mathrm{sym}})\big]
(\mathbf u_1,\ldots,\mathbf u_k)\big|
\le C_{\nabla}\|f\|_{k,\beta}d_{\calM}(x,y).
\end{aligned}
}
\]
\rev{Combining the symmetric-part and symmetry-defect bounds, using
$L_{k,\beta}(f,x)\le\|f\|_{k,\beta}$ and
\eqref{eq:distance-beta-comparison}, proves
\eqref{eq:bound-Lip-derivative-f-goal} when $i+i_2=k$.}

\vspace{10pt}
\noindent
$\bullet$
\underline{\textbf{Combine the previous steps}}
\vspace{5pt}

\rev{Let $m=k-2j$.  By \eqref{eq:bound-nabla-m-S-1} and
\eqref{eq:bound-L-S-proof-3}, for every
$1\le j\le\lfloor k/2\rfloor$,}
\rev{
\begin{align*}
\|S_{i_1,\ell,i_2,i_3}\|_{m,\beta}
&\le |{\rm\bf S}^{d-1}|C_{\nabla}\bar C_A\bar C_V
\|f\|_{k,\beta}
\left(\sum_{q=0}^m3^q
+4\max\{\texttt{diam}(\calM),1\}3^m\right)\\
&\le |{\rm\bf S}^{d-1}|C_{\nabla}\bar C_A\bar C_V
\max\{\texttt{diam}(\calM),1\}3^{m+2}\|f\|_{k,\beta}\\
&\le |{\rm\bf S}^{d-1}|C_{\nabla}\bar C_A\bar C_V
\max\{\texttt{diam}(\calM),1\}3^k\|f\|_{k,\beta},
\end{align*}
}
\rev{where the last inequality uses $m=k-2j\le k-2$.
The expression \eqref{eq:proof-fj-equiv-expression} and the triangle
inequality then give}
\rev{
\begin{align*}
\|f_j\|_{k-2j,\beta}
&\le
\Big (
\sum_{\substack{i_1+i_2+i_3-2\ell=2j\\
i_1,i_2,i_3\ge0\\0\le\ell\le\lfloor i_1/4\rfloor}}
\frac{\mathfrak M_{2j+2\ell+d-1}}{2^\ell\ell!i_2!}
\Big )
|{\rm\bf S}^{d-1}|C_{\nabla}\bar C_A\bar C_V 
\max\{\texttt{diam}(\calM),1\}3^k\|f\|_{k,\beta}.
\end{align*}
}
{{By \eqref{eq:valid-combi-i1-l-i2-i3}, the sum has at most $k^3$
terms and
$d-1\le2j+2\ell+d-1\le2k+d-1$; hence the finitely many moments in the
sum are bounded by a quantity depending only on $d$ and $k$.
}\textcolor{cyan}{The constant $C_{\nabla}$ depends only on $d,k$ and curvature bounds
through order $k-2$, $\bar C_V$ depends only on $d,k$ and curvature
bounds through order $k-1$, and $\bar C_A$ depends only on $k$ and the
second fundamental form and its derivatives through order $2k-3$.
Hence there exists a constant
$\tilde C_2(\calM,d,k)\ge1$, with the dependence declared in Lemma
\ref{lemma:22_holder}, such that}}
\[
{{\|f_j\|_{k-2j,\beta}
\le\tilde C_2(\calM,d,k)\|f\|_{k,\beta},
\qquad0\le j\le\lfloor k/2\rfloor.
}}
\]
{{For $j=0$ this follows from $f_0=f$ and
$\tilde C_2(\calM,d,k)\ge1$.  This proves statement (ii).}}

\vspace{5pt}
Finally, we collect the needed regularity of the manifold.
In the proof of statement (i),
when $k\ge 2$,
we need up to ($k-2$)-th (continuous) covariant derivatives of curvature tensor   
and up to ($k+1$)-th intrinsic derivatives of the Riemann metric tensor $g$,               
as well as up to ($2k-2$)-th (continuous) covariant derivatives of $\Second $;		
when $k=0,1$ we need 2nd intrinsic derivatives of $g$						
and continuous 1st covariant derivative of $\Second $.						
In the proof of statement (ii), we only need regularity of $\calM$ when $k\ge 2$: 
up to ($k-1$)-th continuous covariant derivatives of curvature tensor  			
and again up to ($2k-3$)-th continuous covariant derivatives of $\Second $.               
Overall, it suffices to have $ \max\{  2k, 3 \} $ regularity of $\calM$,
that is, $\calM$ is $C^{ \max\{  2k, 3 \}}$.

\begin{remark}[More general $h$]
\label{rk:diff-decay-h-manifold-integral}
The proof only uses the differentiability and decay property of the function $h(r)$, and thus can be extended beyond when $h(r) = e^{-r/2}$ 
-- specifically, Assumption \ref{assump:general-h}(i) suffices,
\rev{with the integral operator normalized by 
$\int_{\mathbb R^d}h(\|u\|^2)\,du > 0$ instead of $(2\pi)^{d/2}$.}
To extend the proof, first truncate the $2\delta(\epsilon)$ geodesic ball by letting $\delta(\epsilon) = \sqrt{\frac{1}{2 a}(d+k+1)\epsilon \log(\frac{1}{\epsilon})}$, which will ensure that when $y \in \mathcal{M} \setminus B_{2 {\delta(\epsilon)}}(x)$, $h\Big( {\|\iota(x)-\iota(y)\|^2_{\mathbb{R}^D}}/{\epsilon}\Big) \leq a_0 \epsilon^{ {(d+k+1)}/{2}}$ by that $|h(r) | \le a_0 e^{-ar}$.
We define $h_a(r):= e^{-ar}$, then $|h^{(l)}(r)| \le a_l h_a(r) $, and $h_a$ is a monotonically decay function. 
In all the bounding of the remainder terms, we replace $h(u^2)$ to be $ h_a(u^2)$, 
and $h(u^2/4)$ to be $h_a(u^2/4)$, multiplied by a constant depending on $k$ which is $\max_{ 0 \le \ell \le \max\{ k, 1\}} a_\ell$, 
and we use $h_a( u^2) \le h_a(u^2/4)$ by monotonicity of $h_a$. 
All the additional $k$-dependent constant (including the dependence on $\{ a_\ell, \ell \le  k\}$)
multiplied to upper bounds
can be absorbed into the constants in front, 
and we have additional dependence on $a$.
Meanwhile, we keep ``$h^{(\ell)}(u^2)$'' in the definition of $f_j$ 
(and subsequently in ${\rm R}^{(4)}_j (x)$) 
and in the definition of the moments 
$\mathfrak{M}^{(\ell)}_i : = \int_{0}^{\infty}\frac{h^{(\ell)}(u^2)}{(2\pi)^{d/2}} u^i du$, removing the $(-2)^{\ell }$ factor in the denominator.
In bounding $| {\rm R}^{(4)}_j (x) |$,  
we apply Lemma \ref{upper incomplete gamma bound} with a change of variable $u \mapsto \sqrt{2a} u$, 
resulting in a factor of $\max\{1, (2a)^{- k - d/2} \}$ multiplied to the constant $c(k,d)$.
Thus the constant $C_{R,4}$ also depends on $a$. 
%
All this will give the same bounds of the remainder terms as in Lemma \ref{lemma:22_holder}(i) 
where the constants $C_{R,2}$,  $C_{R,3}$,  $C_{R,4}$ are modified to absorb the additional factors depending on $a$,
$d$, $k$. 
In the proof of Lemma \ref{lemma:22_holder}(ii), the analysis of $S_{i_1, \ell, i_2, i_3}$ is not affected; 
In the last step when we bound  $\| f_j \|_{k-2j,\beta}$ from $\| S_{i_1, \ell, i_2, i_3}\|_{k-2j, \beta} $,
 we use $| \mathfrak{M}^{(\ell)}_i | \le \int_{0}^{\infty}\frac{ a_\ell h_a(u^2)}{(2\pi)^{d/2}} u^i du$ in the upper bound, 
 the summation of which can be bounded by a constant depending on $k$, $d$, and $a$.
 This again will give the same bounds as before, where the constants are modified in its dependence on $a$, $d$, and $k$.
\end{remark}

\rev{
In the proof above, we use the full operator norm
$\sup_{v_1,\ldots,v_k\in S_x^{d-1}}
|T(x)(v_1,\ldots,v_k)|$
for a covariant tensor field $T$ of order $k$.
Since $\nabla^k f$ need not be symmetric, one can define
a H\"older norm of $f$ using this full operator norm,
taking suprema over independent unit vectors in each tensor slot
(in both the derivative bounds and the H\"older seminorm).
The resulting H\"older norm is equivalent to ours, with
equivalence constants depending only on the manifold geometry
and the regularity parameters $k,\beta$, as follows from
the symmetrization and Ricci commutator estimates above.
}

\section{More proofs in Section \ref{sec:rate_manifold} and extension}\label{app:proof-sec4}

\subsection{Proofs of Proposition \ref{prop:appr_s_l_inf} and Corollary \ref{cor:contraction_n_norm_manifold}}

\begin{proof}[Proof of Proposition \ref{prop:appr_s_l_inf} ]
Let the constants $\epsilon_1(\calM,d,k)$, $\tilde{C}_1(\mathcal{M}, d, k)$ and $\tilde{C}_2(\mathcal{M},d, k)$ be as in Lemma \ref{lemma:22_holder},
and for notation brevity, below we omit the dependence on $(\calM, d)$ in the constant notation, and write as $\epsilon_1(k)$, $\tilde{C}_1(k)$, $\tilde{C}_2(k)$, and so on. The dependence on manifold geometry is inherited. 
Define 
\[
\tilde{C}_3( k)=\max \{\tilde{C}_1(0), \tilde{C}_1( 1), \cdots, \tilde{C}_1( k)\},
\]
\[
\tilde{C}_4( k)=\max \{\tilde{C}_2( 0), \tilde{C}_2( 1), \cdots, \tilde{C}_2(k), 1\},
\]
\[
\epsilon_2(k)=\min \{\epsilon_1( 0), \epsilon_1(1), \cdots, \epsilon_1(k) \}.
\]
Note that $\epsilon_1(\ell) \le 1/e $ for all $\ell$, then $\epsilon_2(k) \le 1/e < 1/2$.

We construct $F_i$ inductively and show that for 
all $i=0, \cdots, \lfloor k/2 \rfloor$,
\begin{align}\label{Proposition: induction condition}
F_i \in C^{k-2i, \beta}(\calM), \quad 
\|F_i\|_{k- 2i ,\beta} \leq (k+1)^i \tilde{C}_4( k)^i \|f\|_{k, \beta}. 
\end{align}
The claim \eqref{Proposition: induction condition} will lead to the proof of \eqref{eq:cond-A2_holder_1} and \eqref{eq:cond-A2_holder_2}.
\vspace{5pt}

\noindent 
\underline{Proof of claim \eqref{Proposition: induction condition}:}
First, we let $F_0(x)=f(x)$. Then, \eqref{Proposition: induction condition} holds when $i=0$. 

Next, suppose for some integer $ 0 \le \ell < \lfloor k/2 \rfloor$, $F_i$  has been constructed and \eqref{Proposition: induction condition} holds for all $0 \le i \le \ell$, 
we want to construct $F_{\ell+1}$ and show that \eqref{Proposition: induction condition} also holds for $i=\ell+1$.

For each $i={{0,\cdots,\ell}}$, we apply Lemma \ref{lemma:22_holder} to $F_i  \in C^{k-2i, \beta}(\calM)$ (the ``$k$'' in the lemma is $k-2i$) to obtain the expansion of $G_\epsilon (F_i)$, and denote the resulting sequence of functions as $F_{i,j}$ with $F_{i,0} = F_i$.
We have that, when $\epsilon < \epsilon_2(k) \le \epsilon_1(k-2i)$,
\begin{equation}\label{eq:apply-expansion-lemma-to-Fi}
G_\epsilon(F_i)(x)=\sum_{j=0}^{\lfloor k/2-i \rfloor} \epsilon^j F_{i,j}(x)  + R_{F_i ,\epsilon}(x),
\end{equation}
where, by Lemma \ref{lemma:22_holder}(i),
\begin{equation}\label{R F i epsilon proof main proposition}
\|R_{F_i, \epsilon}\|_{\infty} \le \tilde{C}_1( k -2i) \|F_i\|_{k-2i, \beta} \epsilon^{(k+\beta)/2-i} \leq  \tilde{C}_3( k) \|F_i\|_{k-2i, \beta} \epsilon^{(k+\beta)/2-i},
\end{equation}
and the second inequality is by our definition of $\tilde{C}_3$;
By Lemma \ref{lemma:22_holder}(ii),  
$\forall 0 \leq j \leq \lfloor k/2-i \rfloor$,
\[
\|F_{i,j}\|_{k-2i-2j ,\beta} 
\leq \tilde{C}_2( k-2i) \| F_i\|_{k-2i,\beta}
\leq \tilde{C}_4( k) \| F_i\|_{k-2i,\beta},
\]
and the second inequality is by our definition of $\tilde{C}_4$. 
Inserting the induction hypothesis \eqref{Proposition: induction condition}
into the r.h.s. of the above display,
we have that, for  $0 \le i \le \ell$, 
\begin{align}\label{Proposition: Fij bound 1}
 \|F_{i,j}\|_{k-2i-2j ,\beta} 
 & \leq    (k+1)^i \tilde{C}_4( k)^{i+1} \|f\|_{k, \beta},
 \quad \forall 0 \leq j \leq \lfloor k/2-i \rfloor.
\end{align}
We now construct $F_{\ell+1}$ as 
\begin{align}\label{Proposition: Fl formula}
F_{\ell+1} = -\sum_{i=0}^\ell F_{i, \ell+1-i},
\end{align}
and verify that  $F_{\ell+1}$ also satisfies \eqref{Proposition: induction condition}.

First, because
$F_{i, \ell+1-i} \in
\rev{C^{k-2i-2(\ell+1-i),\beta}(\calM)=C^{k-2\ell-2,\beta}(\calM)}$
for all $0 \le i \le \ell$,
we have $F_{\ell +1 } \in C^{k - 2\ell-2,\beta} (\calM)$. 
In addition,  taking $j= \ell+1-i \le \lfloor k/2-i \rfloor$ in \eqref{Proposition: Fij bound 1}, we have
\[
\|F_{i, \ell+1-i}\|_{k-2\ell -2 ,\beta} \leq (k+1)^i \tilde{C}_4( k)^{i+1} \|f\|_{k, \beta}.
\]
Putting to \eqref{Proposition: Fl formula} and by triangle inequality, we have 
\begin{align}\|F_{\ell+1}\|_{k - 2\ell-2 ,\beta}
&\leq 
 \sum_{i=0}^{\ell}\|F_{i, \ell+1-i}\|_{k - 2\ell-2 ,\beta} \nonumber \\
& \le 
 \sum_{i=0}^\ell (k+1)^i   \tilde{C}_4( k)^{i+1} \|f\|_{k, \beta} \nonumber\\
& \le 
  ({{\ell}}+1)(k+1)^\ell \tilde{C}_4( k)^{\ell+1} \|f\|_{k, \beta}
	  \quad \text{(by that $i\le {{\ell}}$ and  $\tilde C_4(k) \ge 1$)}\nonumber\\
& \leq 
  (k+1)^{\ell+1} \tilde{C}_4( k)^{\ell+1} \|f\|_{k, \beta}
  \quad \text{(by that $\ell < \lfloor k/2 \rfloor \le k$)}. \label{main: F l+1 k-2l-2 beta}
\end{align}
This finishes the verification of \eqref{Proposition: induction condition}.
\vspace{5pt}

We have constructed all $F_i$, 
and now we let $F = \sum_{i=0}^{\lfloor k/2 \rfloor} \epsilon^i  F_i  $. We  are ready to prove \eqref{eq:cond-A2_holder_1} and \eqref{eq:cond-A2_holder_2}.  
Recall that $F_{i,0} = F_i$, $F_{0,0}=F_{0}=f$,
and $\lfloor k/2-i \rfloor=\lfloor k/2\rfloor-i$, we have
\begin{align*}
G_\epsilon (F)
& =  \sum_{i=0}^{\lfloor k/2 \rfloor} \epsilon^i G_\epsilon(F_i) \\
& =
\sum_{i=0}^{\lfloor k/2 \rfloor}  \epsilon^i F_{i} 
+\sum_{i=0}^{\lfloor k/2 \rfloor}  \sum_{j=1}^{\lfloor k/2\rfloor -i }  \epsilon^{i+j} F_{i,j} 
+\sum_{i=0}^{\lfloor k/2 \rfloor} \epsilon^i  R_{F_i ,\epsilon}
\quad \text{(by \eqref{eq:apply-expansion-lemma-to-Fi})}\\
& = F_{0,0}+\sum_{l=1}^{\lfloor k/2 \rfloor}  \epsilon^l F_{l}
+\sum_{l=1}^{\lfloor k/2 \rfloor}  \sum_{i=0}^{l-1}  \epsilon^{l} F_{i,l-i} 
+\sum_{i=0}^{\lfloor k/2 \rfloor} \epsilon^i  R_{F_i ,\epsilon}\nonumber \\
& = F_{0,0}+ \sum_{l=1}^{\lfloor k/2 \rfloor} \epsilon^{l} \Big(F_{l} + \sum_{i=0}^{l-1} F_{i,l-i} \Big) 
 +\sum_{i=0}^{\lfloor k/2 \rfloor}R_{F_i ,\epsilon}\epsilon^i \nonumber \\
& = f+\sum_{i=0}^{\lfloor k/2 \rfloor}R_{F_i ,\epsilon}\epsilon^i, \nonumber 
\end{align*}
where the last step used that $F_{l}+ \sum_{i=0}^{l-1} F_{i,l-i}=0$ following our construction \eqref{Proposition: Fl formula}.
Therefore, for any $x \in \calM$, 
\begin{align}\label{G epsilon F - f}
|G_\epsilon (F)(x) -f(x)| 
& \leq \sum_{i=0}^{\lfloor k/2 \rfloor}| R_{F_i ,\epsilon}(x)| \epsilon^i \nonumber \\
& \leq  \tilde{C}_3( k) \epsilon^{(k+\beta)/2} \sum_{i=0}^{\lfloor k/2 \rfloor} \|F_i\|_{k-2i, \beta} 
\quad \text{(by \eqref{R F i epsilon proof main proposition})}
\nonumber\\
& \leq  
  \tilde{C}_3( k)  \epsilon^{(k+\beta)/2}  \sum_{i=0}^{\lfloor k/2 \rfloor} (k+1) ^i\tilde{C}_4( k)^i \|f\|_{k, \beta}
  \quad \text{(by \eqref{Proposition: induction condition})}
  \nonumber \\
& \leq 
  \tilde{C}_3( k) \epsilon^{(k+\beta)/2}
  (k+1) ^{k+1}\tilde{C}_4( k)^k \|f\|_{k, \beta},
\end{align}
where we use $i\le \lfloor k/2 \rfloor \le k$ and $\tilde{C}_4 \ge 1$ in the last step. This proves \eqref{eq:cond-A2_holder_1} with the constant 
\[
\gamma_1:= (k+1)^{k+1} \tilde{C}_3( k)\tilde{C}_4( k)^k,
\]
and this constant $\gamma_1(\calM, d, k)$ satisfies the declared manifold dependence as in the proposition.

Finally, denote ${\mathbb{H}}_{\epsilon}(\calM)$ by $\tilde{\mathbb{H}}_{\epsilon}$,
and we are to bound $\|G_\epsilon (F) \|^2_{\tilde{\mathbb{H}}_{\epsilon}}$. 
Because $F \in C^{0, \beta}(\calM) \subset C(\calM)$, 
we have that $F \in L^2(\calM, dV)$;
the kernel $h_\epsilon(x,y)$ satisfies the needed condition in Lemma \ref{lemma:computation-Hnorm} on $(\calM, dV)$
by continuity of $h$ and compactness of $\calM$, 
and then the lemma applies to give that 
\begin{align}
        \|G_\epsilon(F)\|^2_{\tilde{\mathbb{H}}_{\epsilon}} &= \frac{1}{(2\pi \epsilon)^d}\int_{\mathcal{M}}\int_{\mathcal{M}} h\Big(\frac{\|\iota(x)-\iota(y)\|^2_{\mathbb{R}^D}}{\epsilon}\Big) F(x)F(y)dV(x)dV(y) \nonumber \\
        &\le \|F\|^2_{\infty} \frac{1}{(2\pi \epsilon)^{d/2}} \int_{\mathcal{M}} dV(x) \frac{1}{(2\pi \epsilon)^{d/2}} \int_{\mathcal{M}}h\Big(\frac{\|\iota(x)-\iota(y)\|^2_{\mathbb{R}^D}}{\epsilon}\Big)dV(y).
        \label{eq:proof-rkhs-var-bound-1}
\end{align}
Since $\epsilon < \epsilon_2(k) \le \epsilon_1(0)$,
we can apply  Lemma \ref{lemma:22_holder} with $f = 1$,  $k=0$, $\beta = 1$ to give 
\begin{align}\label{eq:proof-bound-rkhs-var-integral-vol}
\left| 
\frac{1}{(2\pi \epsilon)^{d/2}} \int_{\mathcal{M}} h\Big(\frac{\|\iota(x)-\iota(y)\|^2_{\mathbb{R}^D}}{\epsilon}\Big)dV(y) - 1 \right| 
\le \tilde{C}_1(0) \epsilon^{1/2} \le \tilde{C}_1(0),
\end{align}
where in the second inequality we used that $\epsilon < \epsilon_2(k) \le 1$;
To bound $\|F\|_{\infty}$, since $\epsilon<\epsilon_2(k) \le {1}/{2}$, we have 
\begin{align}
\|F\|_{\infty}
& \leq   \sum_{i=0}^{\lfloor k/2 \rfloor} \|F_i\|_{\infty} \epsilon^i 
  \leq  \sum_{i=0}^{\lfloor k/2 \rfloor} \|F_i\|_{k- 2i ,\beta}  \epsilon^i \nonumber \\ 
&  \leq \sum_{i=0}^{\lfloor k/2 \rfloor}  
   (k+1)^i \tilde{C}_4(k)^i \|f\|_{k, \beta}  \epsilon^i 
   \quad \text{(by \eqref{Proposition: induction condition})}\nonumber \\
& \leq   (k+1)^k \tilde{C}_4(k)^k \|f\|_{k, \beta} \sum_{i=0}^{\lfloor k/2 \rfloor} \epsilon^i  
   \quad \text{($i\le \lfloor k/2 \rfloor \le k$ and $\tilde{C}_4 \ge 1$)}
\nonumber \\
& \leq 2 (k+1)^k \tilde{C}_4( k)^k \|f\|_{k, \beta}.
\quad \text{(by $\epsilon < 1/2$)}
\label{F infty f k b}
\end{align}
Inserting \eqref{eq:proof-bound-rkhs-var-integral-vol}\eqref{F infty f k b} back to \eqref{eq:proof-rkhs-var-bound-1},  we have 
\begin{align*}
\|G_\epsilon(F)\|^2_{\tilde{\mathbb{H}}_{\epsilon}}  \leq 
\frac{(1+\tilde{C}_1(0)) Vol(\mathcal{M})}{(2\pi)^{d/2}} \Big(2 (k+1)^k\tilde{C}_4( k)^k \|f\|_{k, \beta}\Big)^2  \epsilon^{-d/2},
\end{align*}
and this proves \eqref{eq:cond-A2_holder_2} with the constant 
\[
\gamma_2 := \frac{4(1+\tilde{C}_1(0))Vol(\mathcal{M})}{(2\pi)^{d/2}} (k+1)^{2k}\tilde{C}_4( k)^{2k},
\]
and this constant $\gamma_2(\calM, d, k)$ satisfies the declared manifold dependence as in the proposition.

Finally, the small $\epsilon$ threshold needed is $\epsilon < \epsilon_2(k)$ whose dependence is described in the proposition.
The following fact is not used in the proof but can help explain how $\epsilon_2(k)$ is like: 
In the definition of $\epsilon_2(k)$, 
as shown in the proof of Lemma \ref{lemma:22_holder}, 
the dependence on $k$ in $\epsilon_1(k)$ is only via the requirement that $\delta(\epsilon)= \sqrt{(d+k+1)\epsilon \log(\frac{1}{\epsilon})} < \frac{1}{2} \min\{ \tau/2 ,\xi, 1\}$. 
Thus if we always set $\epsilon_1(k)$ to be the largest possible value, it would be a descending sequence as $k$ increases.
In this case, $\epsilon_2(k) = \min\{ \epsilon_1(k), 1/2\}$.
\end{proof}

\begin{proof}[Proof of Corollary \ref{cor:contraction_n_norm_manifold}]
We want to apply Theorems \ref{thm:contraction_n_norm}, \ref{thm:fix-design-estimator} and \ref{thm:contraction_2_norm} to prove the corollary.
Since Assumption \ref{assump:A3}(A3) is already satisfied with $ s = k+\beta$ and $\varrho = d$, it suffices to verify that Assumption \ref{assump:A1-A2-prime} is satisfied with the same $s$ and $\varrho$. 

First, (A1) holds with $\varrho = d$ because $\calX = \calM$ is a $d$-dimensional manifold, see Example \ref{eg:manifold}.
Meanwhile, (A2) is satisfied as a result of Proposition \ref{prop:appr_s_l_inf}.
Specifically,  let constants $\epsilon_2(\calM)$, $\gamma_1(\calM, d,k)$, $\gamma_2(\calM, d,k)$ be as in Proposition \ref{prop:appr_s_l_inf}.
We set
\[
\nu_1 = \gamma_1(\calM, d,k) \|f^* \|_{k, \beta},
\quad
\nu_2 = \gamma_2(\calM, d,k) \|f^* \|_{k, \beta}^2.
\]
When $\epsilon <   \epsilon_2(\calM)$ set to be $\epsilon_0$,
applying Proposition \ref{prop:appr_s_l_inf} with $f=f^*$,
we have the two bounds \eqref{eq:cond-A2_holder_1} and \eqref{eq:cond-A2_holder_2} hold with the function $G_\epsilon( F)$.
This allows to use  $G_\epsilon( F)$ as the needed $F^\epsilon \in {\mathbb{H}}_{\epsilon}(\calM)$ in (A2) to approximate $f^*$, 
and the two bounds imply  \eqref{eq:cond-A2-prime}.
Thus, we have shown that (A2) holds with $ s = k+\beta$  and $\varrho = d$.

Consequently, the fixed-design result follows from Theorems  \ref{thm:contraction_n_norm} and \ref{thm:fix-design-estimator},
and the random-design result follows from Theorem \ref{thm:contraction_2_norm}. 
\end{proof}

 \begin{remark}[Dependence on $\sigma$]\label{rk:sigma-dependence}
\rev{One can consider a standardized formulation:}
 rescale $Y_i$ by dividing by $\sigma$
\rev{(and place the GP prior \eqref{eq:prior} on $f/\sigma$)},
 thereby considering the $\sigma = 1$ case and replacing $f^*$ with  $f^*/\sigma$ in the analysis. 
Substituting into our definitions of the constants $\bar C_1$ and $\bar C_2$ in the proof of Theorem \ref{thm:contraction_n_norm},
we note that  only $\nu_1 = \gamma_1 \|f^* \|_{k, \beta} $ and $\nu_2 = \gamma_2 \|f^* \|_{k, \beta}^2$ 
depend on $f^*$ (which is $f^*/\sigma$).
In particular, $\bar C_2$ does not involve $f^*$.
This will lead to $\|  f - f^*\|_n \le C n^{-\frac{s}{2s+\varrho}}(\log n)^{k_1+k_2} $ in the posterior contract rate
where the dependence of $C$ on $f^*$ and $\sigma$ can be explicitly tracked.
Specifically, one can show that
$
  C \le    \|f^* \|  \max \{ 
  	C_1',  C_2' \frac{\sigma}{\|f^* \|}, 
	C_2' ( 1 +  \frac{\sigma^2}{\|f^* \|^2}  
	)^{1/2} 
	\}
	\le C' \|f^* \|  
	( 1 +  \frac{\sigma^2}{\|f^* \|^2}  
	)^{1/2},
$
where $\|f^* \| = \|f^* \|_{k, \beta}$ 
and the constants $C_1'$, $C_2'$, $C_3'$ and $C'$ do not involve $f^*$ or $\sigma$.
 \end{remark}

\subsection{Proofs in Section \ref{sec:knn_prior}}

\begin{lemma}[Concentration of $\hat{v}_n(t)$ uniform over $t$]
    \label{lemma:concentrate_v_hat}
    Under Assumption \ref{assump:A6},
    $\hat{v}_n(t) $ defined as in \eqref{eq:def_vn},
    Then, 
    there exists  $n_1(\calM, p_X)$
    s.t. when $n > n_1$,
    with probability $\ge 1 - n^{-10}$, \[
    \frac{1}{4} (2\pi)^{d/2} p_{\min} t^{d/2}
    \le 
    \hat{v}_n(t) 
    \le \frac{7}{4} (2\pi)^{d/2} p_{\max} t^{d/2},
    \quad 
    \forall t \in  [n^{{-2}/{d}}(\log n)^{3/d},  t_0].
    \]
    The constant 
    $t_0 := \min\{ 1, {\epsilon_1}/2, {1}/{(2c_{\mathcal{M}})}\}
    $ only depends on $\calM$,
    where $\epsilon_1= \epsilon_1(\calM, d, 1)$
    and
    $c_{\mathcal{M}} = \tilde C_1 (\mathcal{M}, d, 1)$ 
    are as defined in Lemma \ref{lemma:22_holder}.
\end{lemma}

\begin{proof}[Proof of Lemma \ref{lemma:concentrate_v_hat}]
Before we prove the lemma,  we first introduce some notations and estimates. 
By definition \eqref{eq:def_vn},
\begin{equation}
\begin{aligned}
\hat{v}_n(t) = \frac{1}{n} \sum_{i=1}^n  \hat{V}_i(t),
\quad
\hat{V}_i(t) := \frac{1}{n-1}\sum_{j \ne i} h_{t}(X_i, X_j).
\end{aligned} 
\end{equation}
For $ i = 1, \cdots, n$,
\[
|\hat{V}_i(t)| 
= \big| \frac{1}{n-1} \sum_{j \ne i} h_t(X_i,X_j) \big| 
\le 1, 
\quad \forall t \in \R^+,
\] 
by the fact that
$h_t(X_i,X_j) = \exp(-\frac{||X_i-X_j||^2}{2t}) 
\le 1$.
Moreover, we know that 
\begin{equation}
\label{eq:def_L}
    0\le 
\frac{h_t(X_i,X_j)}{t^{d/2}} 
\le {t^{-d/2}}
=: L(t), 
\quad \forall t \in \R^+.
\end{equation}

Meanwhile, 
for each $i$, condition on $X_i$, 
we can bound the conditional variance of $h_t(X_i, X_j)$ over the randomness of $X_j$, $j\neq i$, as  
\begin{align}
{\rm Var} \left( \left. \frac{h_t(X_i,X_j)}{t^{d/2}} \right| X_i 
    \right) 
&\le \E \left( \left. (\frac{h_t(X_i,X_j)}{t^{d/2}})^2 \right|X_i \right) \nonumber \\
&=  t^{-d/2} \int_{\mathcal{M}} t^{-d/2} 
 e^{ - \| X_i -y\|^2/t} 
p_X(y)dV(y) \nonumber  \\
      &\le p_{\max} t^{-d/2}\int_{\mathcal{M}} 
      t^{-d/2}  e^{ - \| X_i -y\|^2/t}  dV(y).
      \label{eq:var_exp_1}
\end{align}
Under the assumption of the current lemma,
$t/2\le t_0/2 < \epsilon_1$,
and then we can apply  Lemma \ref{lemma:22_holder} with $f=1$, $k=1$, $\beta =1$
and $\epsilon = t/2$
to obtain that 
\[
t^{-d/2}  
\int_{\mathcal{M}} 
    e^{-\| X_i -y\|^2/t}
    dV(y)
    = \pi^{d/2}( 1 + r_{t}(X_i) ),
    \quad |r_t(X_i)| \le c_\calM t/2,
\]
where $c_\calM = \tilde C_1(\calM, d, 1)$.
Putting back to \eqref{eq:var_exp_1}, we have
\begin{align} \label{eq:def_nu}
 {\rm Var} \left( \left. \frac{h_t(X_i,X_j)}{t^{d/2}} \right|X_i 
    \right)    
    &\le p_{\max} t^{-d/2}\pi^{d/2}(1 + c_\calM t/2) \nonumber \\
    &\le \frac{3}{2}p_{\max} \pi^{d/2} t^{-d/2} 
    := \nu(t),
\end{align}
and the second inequality due to 
$c_\calM t  \le  1$,
which is guaranteed by that $t \le t_0 < 1/c_\calM$.

Similarly, we can compute and bound 
\begin{align}\label{eq:expectation_bound_0}
\E \left( \left. \frac{h_t(X_i,X_j)}{t^{d/2}} \right| X_i \right) 
= \int_{\calM}t^{-d/2} e^{-||X_i-y||^2/(2t)}p_X(y)dV(y)
\end{align}
by applying Lemma \ref{lemma:22_holder} again 
(with $f=1$ and $\epsilon = t \le t_0 < \epsilon_1$)
and we then have 
\begin{equation}\label{eq:expectation_bound}
    \frac{1}{2}p_{\min}(2\pi)^{d/2} \le \int_{\calM}t^{-d/2} h_{t}(X_i,y)p_X(y)dV(y) \le \frac{3}{2}p_{\max}(2\pi)^{d/2},
\end{equation}
where we used that $c_\calM t \le 1/2$,
which holds by that $t\le t_0$.

Next, we prove the concentration of $\hat{V}_i(t)$ at its expectation for a fixed $t$.
Specifically, the claim is that
$\forall \alpha > 0$, 
if $n>\max\{ n_2(\mathcal{M}),n_3(\alpha, p_X), 2\}$
(where $n_2$, $n_3$ defined below are independent of $t$),
then, for any fixed $i$
and any fixed $t \in [n^{{-2}/{d}}(\log n)^{3/d}, t_0]$,  
with probability $1-2n^{-{\alpha}/{4}}$,
\begin{equation}\label{eq:berstein_1}
    -\sqrt{\nu(t) \frac{2 \alpha \log(n)}{n}}
    \le \frac{\hat{V}_i(t)}{t^{d/2}} 
    - \E \Big(  \frac{\hat{V}_i(t) }{t^{d/2}} \Big| X_i \Big) 
    \le \sqrt{\nu(t) \frac{2 \alpha \log(n)}{n}}.
\end{equation}
Specifically, $n_2(\calM)$ is to ensure that when $n> n_2$,
\begin{equation} \label{eq:n_2_def}
    n^{{-2}/{d}}(\log n)^{3/d}< t_0,
\end{equation}
and then the interval of $t$ is nonempty.
The requirement $n > n_{3}$ is needed when we apply the Berstein inequality (Lemma \ref{lemma:berstein}) to prove the claim \eqref{eq:berstein_1}:
condition on $X_i$, let 
\[
\xi_j := \frac{h_t(X_i,X_j)}{t^{d/2}} - \E \Big( \frac{h_t(X_i,X_j)}{t^{d/2}} \Big| X_i \Big),
\]
which are  $n-1$ many i.i.d. mean-zero random variables.
By definition,
\[
 \frac{\hat{V}_i(t)}{t^{d/2}} 
 - \E \Big( \frac{\hat{V}_i(t)}{t^{d/2}} \Big| X_i \Big)
 = \frac{1}{n-1} \sum_{j\ne i} \xi_j.
\]
By \eqref{eq:def_L}, we have $|\xi_j| \le L (t)$.
By \eqref{eq:def_nu}, we have $\E \xi_j^2 \le \nu(t)$.
We apply Lemma \ref{lemma:berstein} with $\tau (t) = \sqrt{\nu(t) \frac{\alpha \log n}{n-1}}$.
Here, to simplify notation, we omit the dependence on $t$ 
in the notation of $L$, $\nu$, $\tau$ 
in the rest of proof of \eqref{eq:berstein_1}. 
Inserting the definitions of $L$ and $\nu$, 
one can verify that $\tau L < 3 \nu$ holds if 
\begin{equation}\label{eq:condition_n_3}
     \frac{\alpha \log n}{n-1} < \frac{27}{2}p_{\max}\pi^{d/2}t^{d/2}.
\end{equation}
This will require a largeness of $n$,
where, to ensure that the threshold is uniform for all $t$, we employ the lower bound that $t \ge n^{-2/d} (\log n)^{3/d}$.
Then \eqref{eq:condition_n_3} can be ensured if 
\begin{equation}\label{eq:n_3_condition}
    \alpha \frac{n}{n-1}\frac{1 }{ (\log n)^{1/2}} \le \frac{27}{2}p_{\max} \pi^{d/2}.
\end{equation}
There exists $n_3$ depending on constants $\alpha$ (to be determined below) and ${{p_{\rm max}}}$,
 and independent of $t$,
such that \eqref{eq:n_3_condition} holds  when $n > n_3$. 
The choice  of $\tau$ ensures that 
$\exp\{-\frac{1}{4}\frac{N\tau^2}{\nu}\} = n^{-\alpha /4}$
where $N = n-1$.
The Bernstein gives that  the deviation is bounded by 
$ \sqrt{\nu \frac{\alpha \log n}{ n-1}}$,
which is further upper bounded by 
$ \sqrt{\nu \frac{2 \alpha \log n}{ n}}$ as long as $n>2$.
Thus, when $n > n_3$ (and $n> \max\{ n_2, 2\}$), 
the claim \eqref{eq:berstein_1} holds.

By \eqref{eq:expectation_bound} and \eqref{eq:expectation_bound_0}, we have 
\[
    \frac{1}{2}p_{\min}(2\pi)^{d/2} 
    \le \E \Big( \frac{h_t(X_i,X_j)}{t^{d/2}} \Big| X_i \Big)
    \le \frac{3}{2}p_{\max}(2\pi)^{d/2}.
\]
Together with  \eqref{eq:berstein_1}, we have 
\begin{align}\label{eq:berstein_single}
    \frac{1}{2}p_{\min}(2\pi)^{d/2} -\sqrt{\nu (t) \frac{2\alpha \log(n)}{n}}\le 
    \frac{\hat{V}_i(t)}{t^{d/2}} 
    \le  \frac{3}{2}p_{\max}(2\pi)^{d/2} + \sqrt{\nu (t) \frac{2\alpha \log(n)}{n}}.
\end{align}

To prove the lemma, we will need to bound the concentration uniformly over $t$. 
We do this by leveraging the Lipschitz continuity of $\hat{v}_n(t)$ as a function of $t$.
Specifically, we first bound the derivative of $\hat{V}_i(t)$ for each $i$ as 
\begin{align*}
    \Big| \frac{d\hat{V}_i(t)}{dt} \Big| 
    &= \Big| \frac{1}{n-1}  \sum_{j \ne i} \frac{d}{dt}\exp\{ -\frac{||X_i-X_j||^2}{2t}\}  \Big| \\
    &= \Big| \frac{1}{n-1}  \sum_{j \ne i} \exp \{ -\frac{||X_i-X_j||^2}{2t} \}\frac{||X_i-X_j||^2}{2t^2} \Big| \\
    &\le \frac{1}{n-1}  \sum_{j \ne i} \frac{1}{e}  \frac{1}{t} 
    \le \frac{1}{et},
\end{align*}
where 
in the first inequality  we use the fact that
$x e^{-x} \le {1}/{e}$ for all $x \ge 0$. 
As a result, 
\begin{align}
    \Big| \frac{d}{dt}(\frac{ \hat{V}_i(t)}{t^{d/2}}) \Big| 
    &= \Big| \frac{d}{dt}(\hat{V}_i(t))\frac{1}{t^{d/2}} + \hat{V}_i(t)\frac{d}{dt}(\frac{1}{t^{d/2}}) \Big| 
    \le \frac{1}{et^{d/2 + 1}} 
        + \hat{V}_i(t) \frac{d/2}{t^{d/2 + 1}}  \nonumber \\
    &\le \frac{1}{et^{d/2 + 1}} + \frac{d/2}{t^{d/2 + 1}} 
    \le \frac{d}{t^{d/2 + 1}}. \label{eq:lip-bound-2}
\end{align}

We derive a covering of the interval 
$$
I_{(n)}: = [n^{-{2}/{d}}(\log n)^{3/d},  t_0]
$$
and then apply a union-bound argument:
We divide the interval  $I_{(n)}$ 
into $M$ even length adjacent close sub-intervals $\{ I_j \}_{j=1}^{M}$,
then $I_{(n)} \subset \cup_{j=1}^M I_j$
and the midpoint of each $I_j$ is inside $I_{(n)}$. 
Let $M = n^4$,  because $I_{(n)}$ is contained in $(0,1]$,
the length of each $I_j$ is at most $n^{-4}$.
Let the midpoint of each $I_j $ be denoted as $t_j$.
For each $\hat V_i(t)$,
 we apply the lower and upper bounds in \eqref{eq:berstein_single} at each $t_j$, which holds under a good event $E_{i,j}$ that happens w.p. $\ge 1 - 2n^{-\alpha /4}$.
Then, under the intersection of all the $M n$ events $\{ E_{i,j}, \, i=1,\cdots, n, \, j=1,\cdots, M \}$,
we have that for all $i$ and $j$,
\begin{align}\label{eq:berstein_single_2}
    \frac{1}{2}p_{\min} (2\pi)^{d/2}   -\sqrt{\nu(t_j) \frac{2\alpha \log(n)}{n}}\le \frac{\hat{V}_i(t_j)}{t_j^{d/2}} \le \frac{3}{2}p_{\max} (2\pi)^{d/2} + \sqrt{\nu(t_j) \frac{2\alpha \log(n)}{n}}. 
\end{align}
The intersection of all  $M n = n^5$ good events happens w.p. 
$\ge 1-2n^{-{\alpha}/{4}}n^{5}$.
We set $\alpha = 64$, then this probability is at least  $1-n^{-10}$ when $n > 2$.

For any $t \in [n^{ - {2}/{d}}(\log n)^{3/d},  t_0]$, we can
find $j$ such that $t \in I_j$.
Because $t_j$ is the midpoint of $I_j$, 
$
|t-t_j| \le |I_j|/2 \le n^{-4}/2.
$
Then, for each $i$, by the Lipschitz bound \eqref{eq:lip-bound-2}, 
\begin{align*}
\Big|
\frac{\hat{V}_i(t)}{t^{d/2}} -  \frac{\hat{V}_i(t_j)}{t_j^{d/2}} \Big|
& \le |t-t_j| \frac{d}{ ( t' )^{d/2 + 1}}, 
\quad \text{for some $t'$ between $t$ and $t_j$}, \\
&\le \frac{n^{-4}}{2} \frac{d}{ (n^{-{2}/{d}}(\log n)^{3/d} )^{d/2 + 1}}  \\
& \le dn^{-1}  \quad \text{(by that $\log n > 1$ when $n >2$)}
\end{align*}
where in the 2nd inequality,
we used that $t' \ge n^{-{2}/{d}}(\log n)^{3/d}$,
the left end of $I_{(n)}$,
since both $t$ and $t_j$ are inside $I_{(n)}$.

Combined with \eqref{eq:berstein_single_2} where $\alpha = 64$
and, by that $t_j \ge n^{-{2}/{d}} (\log n)^{3/d}$,
\[
\nu(t_j) = \frac{3}{2} { p_{\rm max }} \pi^{d/2} t_j^{-d/2} 
\le { \frac{3}{2}  p_{\rm max } \pi^{d/2} n (\log n)^{-3/2}  },
\] 
we have that, for any $i$ and any $t \in I_{(n)}$,
$
    \frac{1}{2}p_{\min}(2\pi)^{d/2}   -\sqrt{\frac{3 p_{\rm max } \cdot 64 \pi^{d/2}}{ (\log n)^{1/2}}} - dn^{-1} 
    \le \frac{\hat{V}_i(t)}{t^{d/2}} 
    \le \frac{3}{2}p_{\max}(2\pi)^{d/2} + \sqrt{\frac{3 p_{\rm max } \cdot 64 \pi^{d/2}}{ (\log n)^{1/2}}} +  dn^{-1}.
$
There is $\rev{n_4(d,p_{\min},p_{\max})}$ s.t. when $n > n_4$,
\begin{equation}\label{eq:def_n_4}
   \sqrt{\frac{3 p_{\rm max } \cdot 64 \pi^{d/2}}{ (\log n )^{1/2}}} + dn^{-1} \le \frac{1}{4}p_{\min}(2\pi)^{d/2},
\end{equation}
and then we have
\begin{align}
    \frac{1}{4} (2\pi)^{d/2} p_{\min}
    \le 
    \frac{\hat{V}_i(t)}{t^{d/2}} 
    \le \frac{7}{4}(2\pi)^{d/2} p_{\max},
    \quad 
    \forall t \in  I_{(n)}, \,  i=1,\dots, n.
    \label{eq:bound-Ri/t-bound-2}
 \end{align}
This holds under the intersection of all $E_{ij}$ which happens w.p.  $\ge 1 - n^{-10}$,
and when
\begin{equation}
    n > n_1 = \max \{ n_2(\mathcal{M}),n_3, 2,
    \rev{n_4(d,p_{\min},p_{\max})} \},
\end{equation}
where $n_2$,
$n_3$,
and $n_4$
are introduced to ensure 
\eqref{eq:n_2_def},
  \eqref{eq:n_3_condition},
and  \eqref{eq:def_n_4} respectively.

The lemma directly follows by \eqref{eq:bound-Ri/t-bound-2} and that 
$\hat{v}_n(t) = \frac{1}{n} \sum_{i=1}^n \hat{V}_i(t)$. 
\end{proof}

\begin{lemma}[Concentration of $k$NN distance]
    \label{lemma:concentrateT}
    Under Assumption \ref{assump:A6}, suppose $p_X \in C^2(\mathcal{M})$, 
    let $\hat R_k$ be as defined in \eqref{eq:def-Tn-knn} 
    and  $k = \lceil \gamma_2 ( \log n)^2 \rceil$, where $\gamma_2 > 0$ is a fixed constant.
    Then, there exists $n_5(\calM, p_X)$ s.t. when $n > n_5$,
    with probability larger than $1 - n^{-10}$,
     for all  $i=1,..., n$, 
    \begin{equation}\label{eq:lemma-bound-hatRk-conc}
    0.9 
    \big( \frac{\gamma_2}{p_{\max} \nu_d} \big)^{1/d} 
    \big(\frac{ (\log n)^{2}}{n} \big)^{1/d} 
    \le \hat R_k(X_i) \le 1.2
    \big( \frac{\gamma_2}{p_{\min} \nu_d} \big)^{1/d} 
    \big(\frac{(\log n)^{2}}{n} \big)^{1/d}.
    \end{equation}
\end{lemma}

\begin{proof}
    Recall that $\hat R_k(X_i)$ is the distance from $X_i$ to its $k$NN in $\{ X_1, \dots, X_n\}$
    where $ k = \lceil \gamma_2 (\log n)^{2} \rceil$.
    This choice of $k$ satisfies the requirement of Lemma \ref{lemma:concentrate_rho_hat}.
    By the lemma and the definition of $\hat{\rho}$ and $\bar{\rho}$, we have that 
    when $n > n_5'$ for some $n_5'(\calM, p_X)$,
    with probability $\ge 1-n^{-10}$, 
$$
\bar{R}_k (X_i)(1-\delta_n) 
\le \hat R_k(X_i) \le \bar{R}_k(X_i) (1+\delta_n),
\quad \forall i=1,\cdots, n,
$$
where, with $k = \lceil \gamma_2 (\log n)^{2} \rceil$, 
\[
\bar{R}_k (x) = p_X^{-1/d}(x) \left( \frac{1}{\nu_d}\frac{ {k} }{n} \right)^{1/d},
\quad 
\delta_n= C_{1,X}    (\frac{k}{n})^{2/d} + \frac{3\sqrt{13}}{d} \sqrt{\frac{ \log n }{k}},
\]
assuming $\delta_n<1$.
Since $\delta_n = o(1)$ as $n$ increases, 
there exists $n_{5,1} > n_5'$ s.t. then $n >n_{5,1} $, $\delta_n < 0.1$. 
Then we have
\[
0.9 \bar{R}_k (X_i)   
\le \hat R_k(X_i) 
\le 1.1\bar{R}_k(X_i).
\]
Meanwhile,  by the expression of $\bar{R}_k$,  for any $X_i$,
\[
{{p_{\rm max}^{-1/d}}} \nu_d^{-1/d} 
 (\frac{\lceil \gamma_2 (\log n)^{2} \rceil}{n})^{1/d}
\le \bar{R}_k (X_i)  
\le {{p_{\rm min}^{-1/d}}} \nu_d^{-1/d} 
 (\frac{\lceil \gamma_2 (\log n)^{2} \rceil}{n})^{1/d}.
\]
Putting together, we have $\hat R_k(X_i)$ satisfies the lower bound in \eqref{eq:lemma-bound-hatRk-conc} and the upper bound 
\[
\hat R_k(X_i) 
\le 1.1  {{p_{\rm min}^{-1/d}}} \nu_d^{-1/d} 
 (\frac{\lceil \gamma_2 (\log n)^{2} \rceil}{n})^{1/d}
 \le 1.2  {{p_{\rm min}^{-1/d}}} \nu_d^{-1/d} 
 (\frac{ \gamma_2 (\log n)^{2} }{n})^{1/d},
\]
whenever 
$
1.1  ({\lceil \gamma_2 (\log n)^{2} \rceil}/{n})^{1/d}
 \le 1.2  
 ({ \gamma_2 (\log n)^{2} }/{n})^{1/d}
$
which holds if $n > n_{5,2}$ for some $n_{5,2}$. 
Thus, when $n > \max \{ n_{5,1}, n_{5,2}\}= : n_5$,  
and under the good event of Lemma \ref{lemma:concentrate_rho_hat} which happens with probability $\ge 1-n^{-10}$, we have \eqref{eq:lemma-bound-hatRk-conc} hold for all $i=1,\cdots, n$. 
\end{proof}

\begin{lemma}[Theorem 2.3 in \cite{cheng2022convergence}]
    \label{lemma:concentrate_rho_hat}
    Assume Assumption \ref{assump:A6}, and $p_X \in C^2(\mathcal{M})$. 
    Let $\bar{\rho}(x) = p_X(x)^{-1/d}$ and 
    $
    \hat{\rho}(x) = \hat{R}_k (x)
    \big( \frac{1}{ \nu_d }\frac{k}{n} \big)^{-1/d},
    $
    where $\nu_d$ is the volume of the unit $d$-ball. 
    If as $n \to \infty$,
    $k=o(n)$ and $k=\Omega(\log n)$, then when $n$ is sufficiently large, with probability higher than $1-n^{-10}$,
    $$
    \sup_{x \in \calM } 
    \frac{|\hat{\rho}(x)-\bar{\rho}(x)|}{\bar{\rho}(x)} \le C_{1,X}
    \big(\frac{k}{n} \big)^{2/d} + \frac{3\sqrt{13}}{d} \sqrt{\frac{\log n}{k}},
    $$
    where the constant $C_{1,X}$
    and the large-$N$ threshold depend on $p_X$ and $\calM$. 
\end{lemma}

Strictly speaking, Theorem 2.3 in \cite{cheng2022convergence} assumed $p_X \in C^\infty (\calM)$. However, only $C^2$ regularity of $p_X$ is used in the proof therein.

   \begin{figure}[t] 
    \centering
    \includegraphics[width= 1\textwidth]{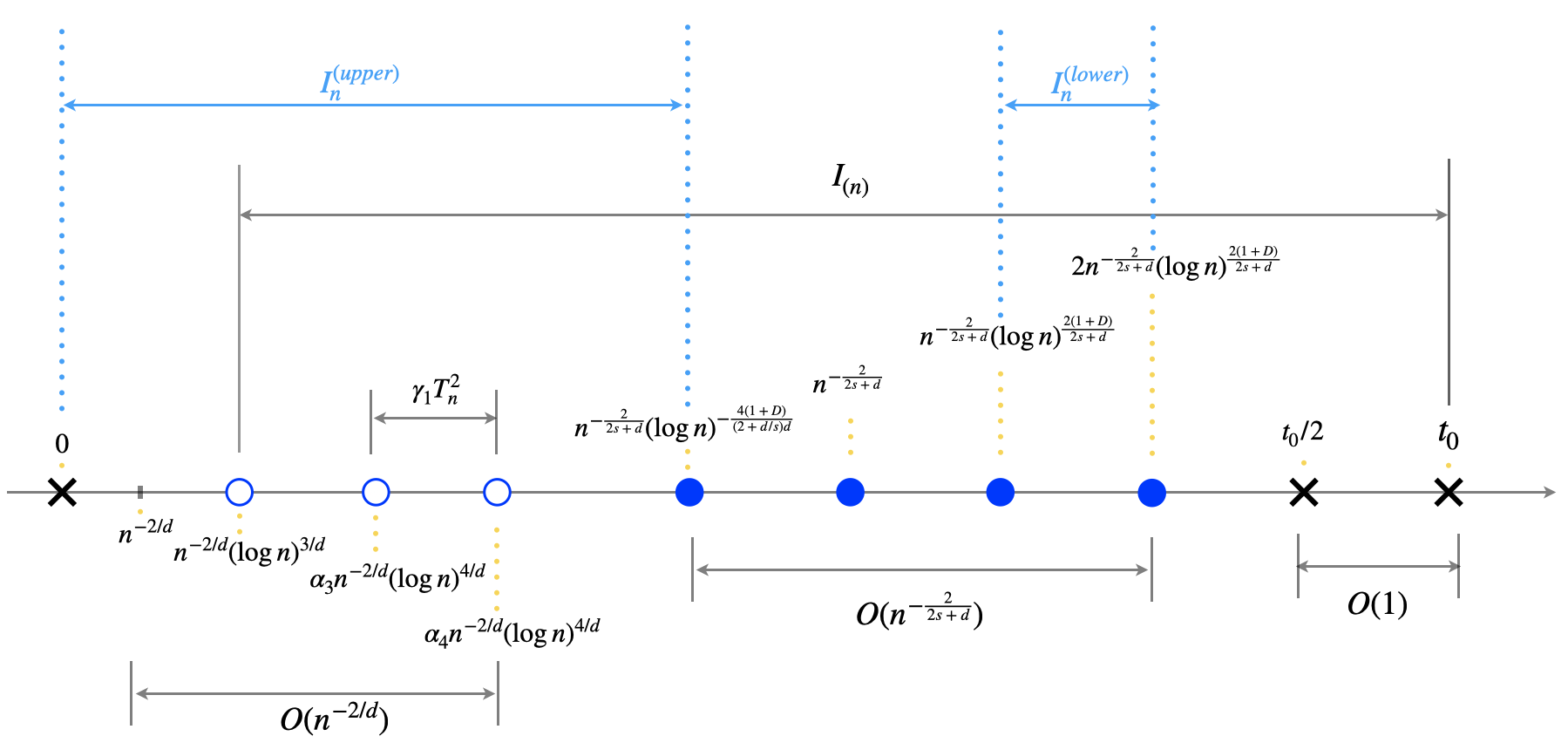}
    \caption{
    Illustration of \eqref{eq:n_func_relation}.
    The intervals $I_n^{(upper)}$
    and $I_n^{(lower)}$ are as in 
    \eqref{eq:prior-upper}
    and 
    \eqref{eq:prior-lower} respectively, 
    in Assumption \ref{assump:A3}.}
    \label{fig0}
    \end{figure}

\begin{proof}[Proof of Proposition \ref{prop:v_hat}]
Let the prior $p(t)$ be as in  \eqref{eq:EBprior}, we are to verify that it satisfies Assumption \ref{assump:A3}
with  $\varrho = d$ and the given $s > 0$. 
The positive constants $c_1$, $c_2$, $c_3$, $a_1$, $a_2$, $K_1$, $K_2$, $C_1$, $C_2$ are to be specified below. 

Under the assumption of the proposition, 
Lemmas \ref{lemma:concentrate_v_hat} and \ref{lemma:concentrateT} apply.
Because the two-sided bound \eqref{eq:lemma-bound-hatRk-conc} in Lemma \ref{lemma:concentrateT} holds for all $i$ (under the good event therein), 
for any subset $S \subset [n]$, the averaged $k$NN distance $T_n$ defined in \eqref{eq:def-Tn-knn} also satisfies the same bound, namely, 
\[
 0.9  ( \frac{\gamma_2}{p_{\max} \nu_d} )^{1/d} 
    (\frac{(\log n)^{2}}{n})^{1/d} 
 \le T_n \le 1.2
    ( \frac{\gamma_2}{p_{\min} \nu_d} )^{1/d} 
    (\frac{(\log n)^{2}}{n})^{1/d}.
\]
We now consider the intersection of the two good events in the two lemmas respectively, 
which happens with probability $\ge 1-2 n^{-10}$
as long as $n > \max\{ n_1, n_5 \}$,
where the two thresholds $n_1$ and $n_5$ defined in the two lemmas  depend on $(\mathcal{M}, p_X)$. 
Then, the following two claims hold simultaneously:
\begin{itemize}
    \item (Claim 1)
    $\forall t \in [n^{-{2}/{d}} (\log n )^{3/d},  t_0]$, 
    $ \alpha_1 t^{d/2}\le \hat{v}_n(t) \le \alpha_2 t^{d/2}$, 
    where
    $\alpha_1 :=\frac{1}{4}(2\pi)^{d/2}p_{\min}$, 
    $\alpha_2 :=\frac{7}{4}(2\pi)^{d/2}p_{\max}$,
    and 
    $ 0 < t_0 \le 1$ is a constant defined in Lemma \ref{lemma:concentrate_v_hat}.  
    
    \item (Claim 2)
    $ \alpha_3 n^{-{2}/{d}} (\log n )^{4/d} \le \gamma_1 T^2_n \le  \alpha_4 n^{-2/d} (\log n )^{4/d}$, 
    where \\
    $\alpha_3 : = 0.9^2  \gamma_1 \gamma_2^{2/d} ( {p_{\max} \nu_d} )^{-2/d}$  and  $\alpha_4 : = 1.2^2 \gamma_1  \gamma_2^{2/d}
    ( p_{\min} \nu_d )^{-2/d} $.
\end{itemize}
It remains to verify that the two claims jointly will imply the needed conditions in 
Assumption \ref{assump:A3}, namely \eqref{eq:prior-lower} and 
\eqref{eq:prior-upper}, 
with proper constants.

To proceed, we assume large enough $n$ such that the scaling of $n$ dominates the ordering of the following quantities:
there exits $n_7(s, \mathcal{M},p_X)$, such that whenever $n > n_7(s, \mathcal{M},p_X)$, we have 
\begin{equation}\label{eq:n_func_relation}
    \begin{split}
    0 
    & < n^{-2/d} < n^{-2/d} (\log n)^{ 3/d } 
     < \alpha_3 n^{-2/d} ( \log n)^{ 4/d } < \alpha_4 n^{-2/d} (\log n)^{ 4/d } 
     \\
    & < n^{-\frac{2}{2s+d}} (\log n)^{-\frac{4 {(1+D)}}{(2+d/s)d}} 
     <  n^{-\frac{2}{2s+d}} 
     < n^{-\frac{2}{2s+d}} (\log n)^{\frac{2 {(1+D)}}{2s+d}} 
     \\
    & < 2 n^{-\frac{2}{2s+d}} (\log n )^{\frac{2 {(1+D)}}{2s+d}}
     < t_0/2 < t_0 \le 1.
 \end{split}
\end{equation}
This ordering is illustrated in Figure \ref{fig0}.

We now specify the needed constants in Assumption \ref{assump:A3}.
Suppose the two parameters $a_0, b_0 > 0$ in  \eqref{eq:EBprior} have been chosen and fixed. 
Let the needed positive constants be as follows,
\begin{align*}
c_1 & =1,  \quad c_2 =2 , \quad c_3 = 1, \quad  a_1  = a_2 = a_0, \\
K_1 & =   b_0/\alpha_1 , \quad K_2=   b_0 / \alpha_2 ,  \\
C_1 &=
\left( \int_{0}^{t_0} t^{-a_0}\exp(-\frac{b_0}{{{\alpha_2}} t^{d/2}}) dt + \frac{1-t_0}{t^{a_0}_0} \right)^{-1}, \\
C_2 & = 2 t_0^{a_0-1} \exp(\frac{b_0}{{{\alpha_1}} (t_0/2)^{d/2}}), 
\end{align*}
where $C_1, C_2 > 0$ because  $0 < t_0 \le 1$.

The desired lower and upper bounds \eqref{eq:prior-lower} and 
\eqref{eq:prior-upper} of $p(t)$ 
call to bounds the normalizing constants in the expression of $p(t)$. 
Specifically, by the definition of $p(t)$ in \eqref{eq:EBprior},
\begin{equation}\label{eq:def_z_n}
p(t) 
=  
\frac{1}{\hat Z_n}
 t^{-a_0}\exp\Big( - \frac{b_0}{\hat{v}_n(t)}\Big)
 {\bf 1}_{\{ \gamma_1T_n^2 < t  \le 1  \} },
 \quad 
 \hat Z_n : = 
 \int_{\gamma_1T_n^2}^1 t^{-a_0}\exp\Big( - \frac{b_0}{\hat{v}_n(t)}\Big) dt. 
\end{equation}
We now make another claim that 
\begin{equation}\label{eq:1/Zn-lower-upper-to-prove}
 C_1 \le {\hat Z_n}^{-1} \le C_2,
\end{equation}
which we will verify later
based on (Claim 1)(Claim 2) and \eqref{eq:n_func_relation}.
Assuming \eqref{eq:1/Zn-lower-upper-to-prove} holds, we finish the rest of the proof as follows.

\vspace{5pt}
\noindent
\underline{To prove  \eqref{eq:prior-lower}}:
\eqref{eq:n_func_relation} implies that 
   $$ 
   n^{- {2}/{d}} (\log n )^{3/d} <  n^{-{2}/{(2s+d)}} (\log n)^{\frac{2 {(1+D)}}{2s +d}} <  2n^{-{2}/{(2s+d)}} (\log n)^{\frac{2 {(1+D)}}{2s + d}} < t_0, 
   $$
   and thus 
   \begin{equation}
   I_n^{\rm (lower)}
   : = [n^{-{2}/{(2s+d)}} (\log n)^{\frac{2 {(1+D)}}{2s +d}}, 
       2n^{-{2}/{(2s+d)}} (\log n)^{\frac{2 {(1+D)}}{2s + d}}] 
        \subset [n^{- {2}/{d}} (\log n )^{3/d}, t_0].
   \end{equation}
   As a result, 
   the lower bound of $\hat v_n(t)$ in (Claim 1) 
   and 
   that $\hat Z_n^{-1} \ge C_1$, i.e., the lower bound in \eqref{eq:1/Zn-lower-upper-to-prove},
   together guarantee that
   \[
p(t) \ge C_1 t^{-a_1} \exp\Big(-\frac{K_1}{t^{d/2}}\Big),
\quad \forall t \in    I_n^{\rm (lower)}, 
   \]
and observe that $I_n^{\rm (lower)}$ is the interval of $t$ in  \eqref{eq:prior-lower}.
This  implies \eqref{eq:prior-lower} 
with the constants $c_1$, $c_2$, $K_1$, $a_1$ and $C_1$ as above.

\vspace{5pt}
\noindent
\underline{To prove \eqref{eq:prior-upper}}:
By (Claim 2), $\gamma_1T_n^2 \in [\alpha_3 n^{-2/d} (\log n)^{4/d}, \alpha_4 n^{-2/d} (\log n)^{4/d}]$, 
and this interval lies inside the interval 
$( n^{-{2}/{d}} (\log n)^{3/d},  
   n^{-{2}/{(2s+d)}} (\log n )^{\frac{-4{(1+D)}}{(2+ d/s) d}} )$
by \eqref{eq:n_func_relation}.
As a result, we have 
\begin{equation}\label{eq:T_n_interval_cont}
    0 < n^{-{2}/{d}} (\log n )^{3/d} 
     < \gamma_1T_n^2  
     < n^{-{2}/{(2s+d)}} (\log n )^{\frac{-4{(1+D)}}{(2+d/s)d}} < t_0,
\end{equation}
and thus 
  \begin{equation}
   I_n^{\rm (upper)}
    := [0, n^{-{2}/{(2s+d)}} (\log n )^{\frac{-4{(1+D)}}{(2+d/s)d}}] 
       \subset 
        [0, \gamma_1T_n^2] \bigcup [n^{-{2}/{d}} (\log n )^{3/d}, t_0]. 
  \end{equation}
We now derive the upper bound of $p(t)$ on $I_n^{\rm (upper)}$. 
First, $\forall t \in  [0, \gamma_1T_n^2] $,
$p(t)=0  \le C_2t^{-a_0} \exp\Big(-\frac{K_2}{t^{d/2}}\Big)$. 
When $t$ is in the interval 
$ [n^{- {2}/{d}} (\log n)^{3/d}, t_0]$,
 (Claim 1) holds. 
 The upper bound of $\hat v_n(t)$ in (Claim 1)
 together with that $\hat Z_n^{-1} \le C_2$, i.e., the upper bound in \eqref{eq:1/Zn-lower-upper-to-prove}, 
imply that 
  \begin{equation} \label{eq:p_t_bound-2}
p(t) \le C_2t^{-a_0} \exp\Big(-\frac{K_2}{t^{d/2}}\Big),
          \quad 
          \forall t \in  [n^{- {2}/{d}} (\log n)^{3/d}, t_0].
\end{equation}
Putting together, we have that 
    \begin{align}
        p(t)  \le C_2t^{-a_2} \exp\Big(-\frac{K_2}{t^{d/2}}\Big),
              \quad \forall t \in  I_n^{\rm (upper)},
        \end{align}
and recall that $I_n^{\rm (upper)}$ is the interval of $t$ in \eqref{eq:prior-upper}.
Thus, this implies \eqref{eq:prior-upper} with
the constants $c_3$, $K_2$, $a_2$ and $C_2$ as above.

The largeness of $n$ needs $n > n_{0} = \max\{ n_1,n_5, n_7\}$,
and the three thresholds are required for 
(Claim 1)(Claim 2) and \eqref{eq:n_func_relation} to hold.
$n_0$ depends on $(\mathcal{M}, p_X)$.
In addition, \eqref{eq:prior-lower} and \eqref{eq:prior-upper} hold under the same good event as (Claim 1)(Claim 2),
which happens  with probability  $\ge 1-2n^{-10}$.

\vspace{5pt}
It remains to verify \eqref{eq:1/Zn-lower-upper-to-prove} to finish the proof of the proposition. 
We do this under (Claim 1)(Claim 2) and \eqref{eq:n_func_relation}, 
which we have established with large enough $n$ and under the needed good events.

\vspace{5pt}
$\bullet$ Proof of $\hat Z_n^{-1} \le C_2$:
By (Claim 2) and \eqref{eq:n_func_relation}, we also have
$$
n^{-{2}/{d}} (\log n )^{3/d} < \gamma_1T_n^2  < t_0/2 < t_0 
\le 1,
$$
and then
\begin{equation}\label{eq:contain_1}
    [t_0/2, t_0] \subset [\gamma_1T_n^2,1] \ \bigcap \  [n^{-{2}/{d}} (\log n )^{3/d},  t_0].
\end{equation}
Applying the lower bound of $\hat v_n(t)$ in (Claim 1) on 
$[ n^{-{2}/{d}} (\log n )^{3/d},  t_0]$, 
we have
\begin{equation}\label{eq:exp_rel_1}
    \exp(-\frac{b_0}{\hat{v}_n(t)}) \ge \exp(-\frac{b_0}{\alpha_1 t^{d/2}}),
    \quad \forall t \in [t_0/2, t_0].
\end{equation}
Recall the definition of $\hat Z_n$ in \eqref{eq:def_z_n}, we have
\begin{align*}
    \hat Z_n &\ge
\int_{t_0/2}^{t_0} t^{-a_0}\exp\Big( - \frac{b_0}{\hat{v}_n(t)}\Big) dt \quad \text{(by  \eqref{eq:contain_1})}
\\ & \ge \int_{t_0/2}^{t_0} t^{-a_0}\exp\Big( - \frac{b_0}{\alpha_1 t^{d/2}}\Big) dt \quad  \text{(by  \eqref{eq:exp_rel_1})} \\
&\ge \frac{t_0}{2}\frac{1}{t^{a_0}_0} \exp(- \frac{b_0}{{{\alpha_1}} (t_0/2)^{d/2}}) = 
C_2^{-1}, 
\end{align*}
where in the last inequality we use the fact that $ t_0/2 \le t \le t_0$.
This proves that 
${\hat Z_n}^{-1} \le C_2$.

\vspace{5pt}
$\bullet$ Proof of $\hat Z_n^{-1} \ge C_1$:
Under (Claim 2) and \eqref{eq:n_func_relation}, we have \eqref{eq:T_n_interval_cont},
which implies that 
\begin{equation}\label{eq:contain_2}
 [\gamma_1T_n^2,t_0] \subset [n^{-{2}/{d}} (\log n )^{3/d},  t_0].
\end{equation}
By definition, 
\[
  \hat Z_n 
  = \left( \int_{\gamma_1T_n^2}^{t_0} +
   \int_{t_0}^{1} \right)
   t^{-a_0}\exp\Big( - \frac{b_0}{\hat{v}_n(t)}\Big) dt.  
\]
On $[\gamma_1T_n^2, t_0]$, by \eqref{eq:contain_2}, 
the upper bound $\hat v_n(t) \le \alpha_2 t^{d/2}$ holds by (Claim 1), and then we have 
\begin{equation*}
    \exp(-\frac{b_0}{\hat{v}_n(t)}) \le \exp(-\frac{b_0}{\alpha_2 t^{d/2}}), 
    \quad \forall t \in [\gamma_1T_n^2, t_0].
\end{equation*}
On $[t_0, 1]$, we have that
$t^{-a_0}\exp\Big( \frac{-b_0}{\hat{v}_n(t)}\Big) \le t_0^{-a_0}$.
Putting together, we have
\begin{align*}
    \hat Z_n 
    & \le \int_{\gamma_1T_n^2}^{t_0} t^{-a_0}\exp(-\frac{b_0}{\alpha_2t^{d/2}}) dt 
    + \int_{t_0}^{1} t_0^{-a_0}  dt \\
    &\le \int_{0}^{t_0} t^{-a_0}\exp(-\frac{b_0}{\alpha_2t^{d/2}}) dt + \frac{1-t_0}{t^{a_0}_0} = C_1^{-1},
    \quad \text{(by $\gamma_1T_n^2 \ge 0$).} 
\end{align*}
This proves that 
$ {\hat Z_n}^{-1} \ge C_1$.
\end{proof}

\subsection{Extension to finite union of disjoint manifolds}\label{app:extension-theory-manifold-mixed-d}

\begin{assumption}[Finite union of manifolds]\label{A4:disconnected}
The data domain $\calX =\cup_{i=1}^m \calM_i$, 
where each $\calM_i$ is a $d_i$-dimensional smooth connected closed Riemannian manifolds isometrically embedded in $[0,1]^D \subset \R^D$, and the $m$ manifolds are mutually disjoint. Let $ \bar d = \max_{ 1 \le i \le m} d_i$.
\end{assumption}

We denote by $\iota: \cup_{i=1}^m \calM_i \rightarrow \mathbb{R}^D$ be  the isometric embedding of $\calX = \cup_{i=1}^m \calM_i$ in $[0,1]^D$.
Because the $m$ manifolds are disjoint, we have 
\begin{equation}\label{eq:def-m0-finite-union}
\mathfrak{m}_0 =\min_{i \not= j}\min_{x\in \calM_i, y \in \calM_j}\| \iota(x)-\iota(y)\|_{\mathbb{R}^D}
> 0.
\end{equation}
If we allow $\calX$ to change when the sample size $n$ increases,
our analysis directly extends when $\mathfrak{m}_0$ is $O(1)$, meaning uniformly bounded away from zero, see the proof of Lemma \ref{lemma:22_holder_finite-union} below.
This assumption on the separation between $\calM_i$ can be further relaxed by considering small $\epsilon$ depending on $n$, and we postpone such extension for exposition simplicity.

Our definition of H\"older class on manifold (Definition \ref{def:Holder-manifold}) also naturally extends to the case of finite union of manifolds. 
Specifically, under Assumption \ref{A4:disconnected},
for $k = 0,1,\cdots$,   $ 0 <\beta \le 1$,
we say  $ f \in C^{k, \beta}(\calX)$
if $f|_{\mathcal{M}_i} \in C^{k, \beta}(\mathcal{M}_i)$ for each $i$,
and we define $\|f\|_{k,\beta}=\max_{1 \le i \le m} \|f|_{\mathcal{M}_i}\|_{k,\beta}$.

The integral operator $G_\epsilon$ previously defined in \eqref{eq:def-Geps-integral-operator} is now taking the integral over the union of the $m$ manifolds. 
Specifically, suppose that $f \in L^1(\calX)$, i.e. $f|_{\calM_i} \in L^1(\calM_i)$ for all $i$, and let $dV_i$ be the volume form of $\calM_i$. 
For any $x \in \calX$, we define
\begin{align}\label{eq:G-eps-f-mixed-d}
G_\epsilon (f)(x)
 =\sum_{i=1}^m \int_{\mathcal{M}_i}  \frac{1}{(2\pi \epsilon)^{d_i/2}} h\Big(\frac{\|\iota(x)-\iota(y)\|^2_{\mathbb{R}^D}}{\epsilon}\Big)f(y)dV_i(y).
\end{align}

We first prove the following lemma as a counterpart to Lemma \ref{lemma:22_holder}. 

\begin{lemma}\label{lemma:22_holder_finite-union}
Under Assumption \ref{A4:disconnected},  given nonnegative integer $k$ and $0<\beta \leq 1$, 
there exists a constant $\epsilon_1(\calX, \bar d, k)$
such that when $\epsilon < \epsilon_1$, 
for any $f \in  C^{k,\beta}(\calX)$,  there exist $f_j \in C^{k-2j, \beta}(\calX)$, $j=1,\cdots, \lfloor k/2 \rfloor $,
and $R_{f ,\epsilon} \in C(\calX)$ s.t.
\begin{align}\label{G epsilon expansion with remainder disconnected}
G_\epsilon (f)(x)=f(x)+\sum_{j=1}^{\lfloor k/2 \rfloor} f_j(x) \epsilon^j + R_{f,\epsilon}(x),
\end{align}

(i) The remainder $R_{f ,\epsilon}(x) $ satisfies that 
$\|R_{f ,\epsilon}\|_{\infty} \le \tilde{C}_1(\calX, \bar d , k) \|f\|_{k, \beta} \epsilon^{(k+\beta)/2}$, 

(ii) For all $0 \leq j \leq \lfloor k/2 \rfloor$, $\|f_j\|_{k-2j,\beta} \leq \tilde{C}_2(\calX, \bar d , k) \|f\|_{k, \beta}$ 
(when $j=0$, $f_0 = f$), 

\noindent 
where 
$\tilde{C}_1(\calX, \bar d , k)= \max_{1 \le i \le m} \tilde{C}_{1,i} + \sum_{i=1}^m Vol(\mathcal{M}_i)$ 
and 
$\tilde{C}_2(\calX,  \bar d , k)= \max_{ 1 \le i \le m} \tilde{C}_{2,i}$,
where  for each $i$,
$\tilde{C}_{1,i} = \tilde{C}_{1,i}(\calM_i, d_i, k)$
and $\tilde{C}_{2,i} = \tilde{C}_{2,i} (\calM_i, d_i, k)$ 
are the constants in Lemma \ref{lemma:22_holder}(i) and (ii) respectively for manifold $\calM_i$,
which inherit the dependence on the geometric quantities of $\calM_i$.
 The threshold $\epsilon_1(\calX, \bar d, k)$ depends on $\bar d$, $k$,
$\mathfrak{m}_0$,  the reach of $\iota(\calM_i)$, and  the injectivity radius of $\calM_i$ for all $i$. 
\end{lemma}

\begin{proof}[Proof of Lemma \ref{lemma:22_holder_finite-union}]

Let $\epsilon_3 > 0 $ be a constant depending on $\bar d$, $k$ and $\mathfrak{m}_0$ s.t.
\begin{equation}\label{eq:epsilon1-separation-m0}
\text{
$\epsilon < \epsilon_3 $ would guarantee that 
$ \sqrt{(\bar d +k+1)\epsilon \log(\frac{1}{\epsilon})} <  \mathfrak{m}_0$}.
\end{equation}
For each $\calM_\ell$, $\ell = 1, \cdots, m$,
we let denote the $ \epsilon_1( \calM_\ell, d_\ell, k) \le 1/e$ in Lemma \ref{lemma:22_holder} as $\epsilon_{1, \ell}$.
We define 
\[
\epsilon_1(\calX, \bar d, k) :=  \min\{ \epsilon_{1, \ell}, \, \ell =1, \cdots, m \} \wedge \epsilon_3,
\]
and we have $\epsilon_1(\calX, \bar d, k) \le 1/e <1/2$. Below we assume $\epsilon < \epsilon_1(\calX, \bar d, k)$.

Recall the definition of  $G_\epsilon(f)(x)$ in \eqref{eq:G-eps-f-mixed-d}, and suppose $ x \in \calM_\ell$ for some $\ell$.
Consider the two cases that $i  = \ell $ and $i \neq \ell$ in the summation, we have
\begin{align}
G_\epsilon (f)(x)
= & \frac{1}{(2\pi \epsilon)^{d_\ell/2}} \int_{\calM_\ell} h\Big(\frac{\|\iota(x)-\iota(y)\|^2_{\mathbb{R}^D}}{\epsilon}\Big)f(y)dV_\ell(y)  \nonumber \\
&+\sum_{i=1, i\not=\ell}^m \int_{\mathcal{M}_i} \frac{1}{(2\pi \epsilon)^{d_i/2}} h\Big(\frac{\|\iota(x)-\iota(y)\|^2_{\mathbb{R}^D}}{\epsilon}\Big)f(y)dV_i(y)
=: G_\epsilon^{(\ell)}(x) + {\rm R}^{c}_{\ell}(x).
\nonumber
\end{align}
We observe that ${\rm R}^{c}_{\ell}(x)$ contributes to the remainder:
because $x \in \calM_\ell$ and $y \in \calM_i$, and $ i \neq \ell$, 
we have $\|\iota(x)-\iota(y)\|_{\mathbb{R}^D} \geq \mathfrak{m}_0 \geq \sqrt{(\bar d +k+1)\epsilon \log(\frac{1}{\epsilon})}$,
where the second inequality is by that $\epsilon < \epsilon_3$.
Consequently, $h\Big(\frac{\|\iota(x)-\iota(y)\|^2_{\mathbb{R}^D}}{\epsilon} \Big) \leq \epsilon^{{( \bar d +k+1)}/{2}}$. 
Thus, for any $x \in \calM_\ell$,
\begin{align*}
|{\rm R}^{c}_{\ell}(x)| 
& \leq
 \sum_{i=1, i\not=\ell}^m \frac{1}{(2\pi)^{d_i/2}}Vol(\mathcal{M}_i)\|f\|_{\infty}\epsilon^{ {(\bar d -d_i+k+1)}/{2}}  
 \leq 
\big(\sum_{i=1, i\not=\ell}^m Vol(\mathcal{M}_i) \big )\|f\|_{\infty}\epsilon^{{(k+\beta)}/{2}}. \nonumber
\end{align*}

The term $G_\epsilon^{(\ell)}(x)$ consists of the kernel integration of $f$ on manifold $\calM_\ell$ only, which has been analyzed in Lemma \ref{lemma:22_holder},
and here we have assumed that $\epsilon < \epsilon_{1, \ell}$ so the lemma applies. 
We then have 
\begin{align}\label{Appendix:interior expansion f epsilon j}
G_\epsilon^{(\ell)}(x)=f(x)+\sum_{j=1}^{\lfloor k/2 \rfloor} f_j(x) \epsilon^j + R_{\ell}(x),
\end{align}
where $f_j$ is defined in \eqref{DEfinition of f_j} with $d=d_\ell$ and all the geometric quantities are with respect to $\calM_\ell$,
and in particular $f_0 = f$;
By Lemma \ref{lemma:22_holder}(i), $R_{\ell}(x)$ satisfies that
\[
| R_{\ell}(x)| 
\le \tilde{C}_{1, \ell} \|f|_{\calM_\ell} \|_{k, \beta} \epsilon^{(k+\beta)/2},
\quad \forall x\in \calM_\ell.
\]
Putting together, this gives \eqref{G epsilon expansion with remainder disconnected}
and for $x \in \calM_\ell$,
$ R_{f,\epsilon}(x) = {\rm R}^{c}_{\ell}(x) + R_{\ell}(x)$.
Thus, putting together the bounds of $| {\rm R}^{c}_{\ell}(x)|$ and $|R_{\ell}(x)|$, we have that for any $x \in \calM_\ell$,
\begin{align*}
| R_{f,\epsilon}(x)|
& \le |{\rm R}^{c}_{\ell}(x)| + |R_{\ell}(x)|
\le \big( \sum_{i=1, i\not=\ell}^m Vol(\mathcal{M}_i) 
	+ \tilde{C}_{1, \ell} \big)
	\| f \|_{k, \beta} \epsilon^{(k+\beta)/2} \\
& \le \tilde{C}_1(\calX, \bar d , k)\| f \|_{k, \beta} \epsilon^{(k+\beta)/2},
\end{align*}
where in the second inequality we used that $\|f\|_{k,\beta} =\max_{1 \le i \le m} \|f|_{\mathcal{M}_i}\|_{k,\beta}$
and $ \| f\|_\infty \le \|f\|_{k,\beta}$,
and the last inequality is by our definition of $\tilde{C}_1(\calX, \bar d , k)$.
The above bound holds for $x\in \calM_\ell$ for all $\ell$, and this proves Lemma \ref{lemma:22_holder_finite-union}(i).

To prove Lemma \ref{lemma:22_holder_finite-union}(ii),
recall that our $f_j$  has been defined on $\calM_\ell$ for each $\ell$,
and Lemma \ref{lemma:22_holder}(ii) gives that 
$$ 
\| f_j|_{\calM_\ell} \|_{k-2j,\beta} 
\le  \tilde{C}_{2,\ell}  \|f|_{\calM_\ell}\|_{k, \beta}, 
\quad \forall 0 \le j \le \lfloor k/2 \rfloor.
$$
Again, by that $\|f\|_{k,\beta}=\max_{1 \le i \le m} \|f|_{\mathcal{M}_i}\|_{k,\beta}$
and our definition of $\tilde{C}_2(\calX,  \bar d , k)$,
we have
$
\| f_j|_{\calM_\ell} \|_{k-2j,\beta}  
\le \tilde{C}_2(\calX,  \bar d , k) \|f \|_{k, \beta}$.
This bounds holds for  $f_j|_{\calM_\ell}$ for all $\ell$, 
and then Lemma \ref{lemma:22_holder_finite-union}(ii) follows. 
\end{proof}

\begin{proposition}
\label{prop:appr_s_l_inf disconnected}
Under Assumption \ref{A4:disconnected}, 
there exists $\epsilon_2(\mathcal{X}, \bar d, k)$ 
such that when $\epsilon < \epsilon_2$,
for any $f \in  C^{k ,\beta}(\calX)$, 
we can find 
$F = \sum_{j=0}^{\lfloor k/2 \rfloor} \epsilon^j F_j $ 
with $F_j \in C^{k - 2j ,\beta}(\calX)$ and
\begin{equation}   
\|G_\epsilon (F)- f\|_{\infty} 
	\le \gamma_1(\calX, \bar d , k) \|f\|_{k, \beta}  \epsilon^{(k+\beta)/2}, \label{G epsilon (F)- f_{infty} disjoint 1}
\end{equation}
\begin{equation}\label{G epsilon (F) H disjoint 2}
\|G_\epsilon (F) \|^2_{{\mathbb{H}}_{\epsilon}(\calX) } 
	\le \gamma_2(\calX, \bar d , k) \|f\|^2_{k, \beta} \epsilon^{-\bar d /2}, 
\end{equation}
 where the constants  $\gamma_1$ and $\gamma_2$ inherit the dependence of $\calX$ from the constants $\tilde{C}_1$ and $\tilde{C}_2$ as in Lemma \ref{lemma:22_holder_finite-union}. 
 The threshold $\epsilon_2$ inherit the dependence of $\calX$ from the threshold $\epsilon_1$ as in Lemma \ref{lemma:22_holder_finite-union}.
 \end{proposition}

\begin{proof}[Proof of Proposition \ref{prop:appr_s_l_inf disconnected}]
The proof follows the same strategy as that of Proposition \ref{prop:appr_s_l_inf}.
Let $\tilde C_1(k)$, $\tilde C_2(k)$, $\epsilon_1(k)$ be as in Lemma \ref{lemma:22_holder_finite-union}, 
where we omit the dependence on $(\calX, \bar d)$ in the constant notation.
We define $\tilde{C}_3(k)$, $\tilde{C}_4(k)$ and $\epsilon_2(k)$ same as before.
We consider  $\epsilon < \epsilon_2(k)$ below.

Again, let $F_0 = f$, and we recursively define $F_i$ by \eqref{Proposition: Fl formula},
where $F_{i,j}$ is provided by  the expansion \eqref{eq:apply-expansion-lemma-to-Fi} 
 with the remainder $\|R_{F_i, \epsilon}\|_{\infty} $ bounded 
 	as in \eqref{R F i epsilon proof main proposition} by Lemma \ref{lemma:22_holder_finite-union}(i).
 In addition, we have \eqref{Proposition: Fij bound 1} hold for  $0 \le i \le \lfloor k/2 \rfloor$ by  Lemma \ref{lemma:22_holder_finite-union}(ii).
 Then, by the same argument as before, for all $ 0 \le i \le \lfloor k/2 \rfloor$,
\begin{align}\label{Appendix:Proposition: induction condition}
F_i \in C^{k-2i, \beta}(\calX), \quad 
\|F_i\|_{k- 2i ,\beta} \leq (k+1)^i \tilde{C}_4( k)^i \|f\|_{k, \beta}. 
\end{align}
Let $F = \sum_{j=0}^{\lfloor k/2 \rfloor} \epsilon^j F_j $,
the same argument as before proves \eqref{G epsilon (F)- f_{infty} disjoint 1} 
with  $\gamma_1(\calX, \bar d, k) = (k+1) ^{k+1}  \tilde{C}_3(k) \tilde{C}_4^k(k)$.

To prove \eqref{G epsilon (F) H disjoint 2}, 
note that, by  $dV(x) :=\sum_{i=1}^m dV_i(x) {\bf 1}_{\calM_i}(x)$,
\begin{align*}
G_\epsilon (F)(x)
 =&\sum_{i=1}^m \int_{\mathcal{M}_i}  \frac{1}{(2\pi \epsilon)^{d_i/2}} h
	 \Big( \frac{\|\iota(x)-\iota(y)\|^2_{\mathbb{R}^D}}{\epsilon} \Big)
 	F(y)dV_i(y) \\
 =& \int_{\calX} h\Big(\frac{\|\iota(x)-\iota(y)\|^2_{\mathbb{R}^D}}{\epsilon}\Big) 
 	\Big( \sum_{i=1}^m \frac{ {\bf 1}_{\calM_i}(y)}{(2\pi \epsilon)^{d_i/2}} F(y) \Big) 
	dV(y).
\end{align*}
Again, $F \in L^2(\calX, dV)$ because $F \in C^{0, \beta}(\calX) \subset C(\calX)$,
and the kernel $h_\epsilon(x,y)$ satisfies the needed condition by Lemma \ref{lemma:computation-Hnorm} 
by the continuity of $h$ and compactness of $\calX$.
Hence, by Lemma \ref{lemma:computation-Hnorm}, 
\begin{align}\label{eq:proof-prop-manifold-union-2}
&  \|G_\epsilon(F)\|^2_{{\mathbb{H}}_{\epsilon}(\calX)} 
 = \int_{\calX}\int_{\calX} h\Big(\frac{\|\iota(x)-\iota(y)\|^2_{\mathbb{R}^D}}{\epsilon}\Big) 
 	\Big( \sum_{i=1}^m \frac{ {\bf 1}_{\calM_i}(x)}{(2\pi \epsilon)^{d_i/2}} F(x) \Big) \\
&~~~~~~~~~~~~~~~~~~~~~~~~~~~~~~~~~~~~~~~~~	
	\Big( \sum_{j =1}^m \frac{ {\bf 1}_{\calM_j}(y)}{(2\pi \epsilon)^{d_j /2}} F(y) \Big) 
	 dV(x) dV(y) \nonumber \\
&~~~  \le \|F\|^2_{\infty} \int_{\calX}\int_{\calX} h\Big(\frac{\|\iota(x)-\iota(y)\|^2_{\mathbb{R}^D}}{\epsilon}\Big) 
   	 \sum_{i=1}^m \frac{ {\bf 1}_{\calM_i}(x)}{(2\pi \epsilon)^{d_i/2}}   
	 \sum_{j =1}^m \frac{ {\bf 1}_{\calM_j}(y)}{(2\pi \epsilon)^{d_j /2}}  
	dV(x) dV(y)  \nonumber \\
&~~~ =  \|F\|^2_{\infty} 
 	\int_{\calX} G_\epsilon( {\bf 1} ) (x)
	\sum_{i=1}^m \frac{ {\bf 1}_{\calM_i}(x)}{(2\pi \epsilon)^{d_i/2}}   
	dV(x)   \nonumber \\
&~~~ \leq \|F\|^2_{\infty}\|G_\epsilon({\bf 1}  )\|_\infty \Big(\sum_{i=1}^m Vol(\mathcal{M}_i)\Big) \epsilon^{-\bar d /2}, \nonumber
\end{align}
where ${\bf 1}$ is the one-constant function on $\calX$,
and we used $2 \pi >1$ and  $\epsilon<1/2$ in the last step. 
Since $\epsilon <\epsilon_2( k)\le \epsilon_1(0)$, we have
$
\|G_{\epsilon}( {\bf 1}) - 1\|_\infty 
\le \tilde{C}_1( 0) \epsilon^{1/2} \le \tilde{C}_1( 0)$
by applying  Lemma \ref{lemma:22_holder_finite-union} with $f = 1$,  $k=0$, $\beta = 1$.
Since  $\epsilon< {1}/{2}$,  
applying the same method as in \eqref{F infty f k b},  
we have $\|F\|_{\infty} \leq 2 (k+1)^k \tilde{C}_4( k)^k \|f\|_{k, \beta}$. 
Putting both bounds back to \eqref{eq:proof-prop-manifold-union-2}, 
we prove \eqref{G epsilon (F) H disjoint 2} with 
$\gamma_2 := 4(1+\tilde{C}_1(0)) \big(\sum_{i=1}^m Vol(\mathcal{M}_i) \big) (k+1)^{2k} \tilde{C}_4( k)^{2k} $.
\end{proof}

Proposition \ref{prop:appr_s_l_inf disconnected} serves as the counterpart of Proposition \ref{prop:appr_s_l_inf}.
Using the proposition, we are ready to prove the convergence rates. 

\begin{corollary}\label{cor:manifold-mixed-d-rate}
  Under Assumption  \ref{A4:disconnected},
     given $f^* \in C^{k,\beta}(\calX)$ for some $k = 0, 1, \cdots$ and $0< \beta \le 1$,
     suppose the prior on the kernel bandwidth $\epsilon$ satisfies Assumption \ref{assump:A3}
     with  $s = k+\beta$ and $\varrho = \bar d$.
     Then, under both fixed and random designs,
     the same posterior contraction rate 
     and posterior mean convergence rate 
     as in Corollary \ref{cor:contraction_n_norm_manifold} hold by replacing $d$ to be $\bar d $.
\end{corollary}

\begin{proof}[Proof of Corollary \ref{cor:manifold-mixed-d-rate}]
By the same argument as the proof of Corollary \ref{cor:contraction_n_norm_manifold}, it suffices to check that Assumption \ref{assump:A1-A2-prime} is satisfied with $s = k + \beta$ and $\varrho = \bar d$,
and then the convergence rates follow by Theorems \ref{thm:contraction_n_norm}, \ref{thm:fix-design-estimator}, and \ref{thm:contraction_2_norm},

To verify Assumption \ref{assump:A1-A2-prime}(A1), 
we say that a set $S$ satisfies (A1)  with dimension $d$
if there exists $C_S$ and $r_0$ such that  \eqref{eq:cond-A1-prime} holds with $\varrho=d$.
Here, ``dimension'' of $S$ is not unique, since for any $d' > d$,
$S$ also satisfies (A1) with dimension $d'$. 
Applying Lemma \ref{Lemma: dimension of the union} $m$ times, we have that $\calX$ satisfies  (A1)  with dimension $\bar d$.
Assumption \ref{assump:A1-A2-prime}(A2) can be verified similarly as in the proof of Corollary \ref{cor:contraction_n_norm_manifold} based on Proposition \ref{prop:appr_s_l_inf disconnected}.
\end{proof}

\begin{lemma}\label{Lemma: dimension of the union}
Suppose $\calX_1$, $\calX_2 \subset [0,1]^D$
 satisfy Assumption \ref{assump:A1-A2-prime}(A1) 
with dimensions $\varrho_1$ and $\varrho_2$ respectively,
then $\calX=\calX_1 \cup \calX_2$ satisfies (A1) with dimension $\varrho=\max \{ \varrho_1, \varrho_2 \}$.
\end{lemma}
\begin{proof}
By definition, there exist  $0<r_1<1$, $0<r_2<1$, and $C_{\calX_1}>0$ and $C_{\calX_2}>0$ such that 
$\calN( r, \calX_1, \| \cdot \|_\infty) \leq C_{\calX_1} r^{-\varrho_1}$ for all $0<r\leq r_1$ and $\calN( r, \calX_2, \| \cdot \|_\infty)  \leq C_{\calX_2} r^{-\varrho_2}$ for all $0<r\leq r_2$.  
Since the union of any covers of $ \calX_1$ and $\calX_2$ is a cover of $\calX= \calX_1 \cup \calX_2$, 
we have that $ \calN( r, \calX, \| \cdot \|_\infty) \leq  \calN( r, \calX_1, \| \cdot \|_\infty) + \calN( r, \calX_2, \| \cdot \|_\infty) $ always holds.
 Hence, for all $0< r \leq r_0 := \min\{ r_1, r_2 \} \in (0, 1]$,  we have
\begin{align*}
\calN( r, \calX, \| \cdot \|_\infty) 
& \leq  \calN( r, \calX_1, \| \cdot \|_\infty) + \calN( r, \calX_2, \| \cdot \|_\infty) \leq C_{\calX_1} r^{-\varrho_1} +C_{\calX_2} r^{-\varrho_2} \\
& \leq ( C_{\calX_1}+ C_{\calX_2}) r^{-\max \{ \varrho_1, \varrho_2 \}},
\end{align*}
and this proves the claim with $\varrho=\max \{ \varrho_1, \varrho_2 \}$.
\end{proof}

At last, to ensure the validity of the EB prior \eqref{eq:EBprior} and extend Proposition \ref{prop:v_hat}, 
we modify the definition of $\hat v_n$ as 
\begin{equation}\label{eq:def_vn-mixed-d}
\hat{v}_n(t) = \left(  \frac{1}{n} \sum_{i=1}^n  \hat{V}_i(t)^{-1} \right)^{-1},
\quad
\hat{V}_i(t) := \frac{1}{n-1}\sum_{j \ne i} h_{t}(X_i, X_j);
\end{equation}
The $k$NN-based $T_n$  is as in \eqref{eq:def-Tn-knn} with $S = [n]$. 
In practical computation of $\hat V_i(t)^{-1}$,
 the numerical singularity due to using very small $t$ compared to the distances $\|X_i - X_j\|^2$
can be avoided 
by restricting  $t > \gamma_1 T_n^2$, i.e., a multiple of the (squared averaged) $k$NN distance, as proposed in \eqref{eq:EBprior}.

To extend the theory, we assume that the data distribution is a mixture on the $m$ manifolds, that is, 
$p_X = \sum_{\ell = 1}^m \alpha_\ell p_{X, \ell}$,
where $\alpha_\ell > 0$, $\sum_{\ell=1}^m \alpha_\ell = 1$,
and each $p_{X, \ell}$ is $C^2$ and uniformly bounded from below and above on $\calM_\ell$.

For $\hat v_n(t)$, we  can extend the proof of Lemma \ref{lemma:concentrate_v_hat} to show that 
$ \hat v_n(t) \sim t^{\bar d /2}$ when $t > C n^{-2/ \bar d}$ up to a log factor.
To see this, for each $i$, we condition on $X_i$ and consider $h_t(X_i, X_j)$ over the randomness of $X_j$, $j \neq i$.
Suppose $X_i \in \calM_\ell$, 
there exists $\bar t_0 \le 1$ that depends on $\calX$, $\bar d$ and also on $\mathfrak{m}_0$ s.t. when $t \le \bar t_0$,
the mean $\E [ h_t(X_i, X_j) | X_i ]$
and variance ${\rm Var} (  h_t(X_i, X_j) | X_i )$
are both dominated by the contribution from the integration on $\calM_\ell$.
Then, similar to \eqref{eq:bound-Ri/t-bound-2}, we can show that $\hat V_i(t) \sim t^{d_\ell/2}$ uniformly for $ t\in  \bar I_{(n)} := [n^{-{2}/{\bar d}}(\log n)^{3/ \bar d},  \bar t_0]$ 
and all $i$,
at large $n$ and with high probability. 
We introduce another good event under which $\#\{i, X_i \in \calM_\ell \}/n $ concentrates around $\alpha_\ell$, for all $\ell$.
There is at least one $\bar \ell$ where $d_{\bar \ell} = \bar d$,
and then $\frac{1}{n} \sum_{i=1}^n  \hat{V}_i(t)^{-1} \ge \frac{1}{n} \sum_{i, \, X_i \in \calM_{\bar \ell}}  \hat{V}_i(t)^{-1} \sim t^{- \bar d/2} $;
Meanwhile $t^{- d_\ell /2} \le t^{- \bar d /2} $ for all $ t \in \bar I_{(n)}$ because $t \le 1$.
Then, taking the union bound over all the good events, we can show that 
$ c_1  t^{- \bar d/2}
\le \frac{1}{n} \sum_{i=1}^n  \hat{V}_i(t)^{-1} 
 \le c_2  t^{- \bar d/2}$ uniformly for $ t\in  \bar I_{(n)}$
 for positive constants $c_1$ and $c_2$.
 This proves that 
 $ {{c_2^{-1}  t^{ \bar d/2}}}
\le \hat v_n(t)
 \le {{c_1^{-1}  t^{ \bar d/2}}}$ uniformly for $ t\in  \bar I_{(n)}$ at large $n$ with high probability. 

 For $T_n$, we can extend Lemma \ref{lemma:concentrateT} to show that when  $n$ is large, 
with high probability and for all $i$, $\hat R_k(X_i) \sim (k/n)^{1/ d_\ell}$ when $X_i \in \calM_\ell$.
Because $S = [n]$, 
we have  $T_n = \frac{1}{n}\sum_{i=1}^n  \hat{R}_k (X_i)$ and can be shown to satisfy that 
$ c_3 n^{-1 / \bar d}  (\log n)^{2/\bar d}  \le T_n \le  c_4 n^{-1 / \bar d}  (\log n)^{2/\bar d} $ for positive constants $c_3$ and $c_4$.

The rest of the proof of Proposition \ref{prop:v_hat} applies with $d$ replaced to be $\bar d$.
This shows that the prior $p(t)$ with modified $\hat v_n(t)$
satisfies Assumption \ref{assump:A3} with $\varrho = \bar d$,  $s = k +\beta$
at large $n$ with high probability.

\section{Technical lemmas and proofs}
\label{ap:A}

\subsection{RKHS lemmas}\label{appsubsec:proofs-sec-4}

\subsubsection{Lemmas about RKHS on a subset of $\R^D$}

We consider a general class of $h$ to be specified in Assumption \ref{assump:general-h} below.
Because $h_\epsilon( x,x')$ is the covariance function of the Gaussian process $F^\epsilon_x$,
when $\epsilon = 1$, 
there exists a finite measure $d\mu$ on $\R^D$, namely the {\it spectral measure}  of $F^1$, 
such that 
\begin{equation}\label{eq:def-spectral-measure-h1}
h_1(x,x') 
= h( \|  x- x'\|^2)
= \int_{\mathbb{R}^D} e^{-i  \lambda^T (x-x')} d\mu(\lambda).
\end{equation}
In particular, we have
\[
h(0) = h_1( x,x) = \int_{\R^D} d\mu(\lambda) > 0 \text{ and is finite.}
\]
For other values of $\epsilon > 0$, the spectral measure is $d \mu_\epsilon$, and $\mu = \mu_1$. We have
\[
h_\epsilon( x, x') 
= \int_{\mathbb{R}^D} e^{-i  \lambda^T (x-x')} d\mu_\epsilon(\lambda)
= \int_{\mathbb{R}^D} e^{-i  \lambda^T (x-x')/\sqrt{\epsilon} } d\mu(\lambda).
\]
In our setting, the spectral measure will have a density, denoted as $f_\epsilon$, namely, $d \mu_\epsilon(\lambda) = f_\epsilon(\lambda) d\lambda$. By change of variable, for any $\epsilon > 0$,
\begin{equation}\label{eq:expression-ft-from-f1}
f_\epsilon( \lambda) = \epsilon^{D/2} f_1(\sqrt{\epsilon} \lambda). 
\end{equation}
In addition, because $h_1( x, x') = h( \| x -x'\|^2)$ is radial symmetric, then so is $f_1(\lambda)$ as the Fourier transform of the function $h( \|x\|^2)$, 
that is,
$f_1 (\lambda) = f_{1, {\rm r}} ( \| \lambda \|)$ for some positive function $f_{1, {\rm r}}$ on $[0,\infty)$.

\begin{assumption}[General kernel function]\label{assump:general-h}
The kernel $h_\epsilon( x,x') =  h ({\|x-x'\|^2}/{\epsilon} ) $ 
where $h: [ 0, \infty) \to \R$ is associated with the spectral measure $d\mu$
and satisfies that 
\begin{itemize}
\item[(i)] 
Differentiability and decay of $h$.
$h \in C[0,\infty) \cap C^\infty( 0, \infty)$ and there exist $a, a_l >0 $ s.t. 
\[
|h^{(l)}(r)| \le a_l e^{- a r}, \quad \forall r \ge 0, \, \forall l = 0, 1, \cdots.
\]
$h(0) > 0 $ and without loss of generality we assume $h(0) =1$. 
\item[(ii)] 
Spectral measure has subexponential decay. 
There exist  $\delta_h >0 $ and $c_h > 1$ s.t. 
\[
\int_{\R^D} e^{\delta_h \| \lambda \|} d\mu(\lambda ) \le c_h^{2D}, \quad \forall D = 1, 2, \cdots,
\]
namely the constants $\delta_h $ and $c_h $ are uniform for all $D$.
\item[(iii)]  The spectral measure $d \mu$ has density $f_1$,
and $f_1$ has monotonic radial decay, i.e., for any $\lambda \in \R^D$ and any $a \ge 1$,
$f_1(  a \lambda) \le f_1(\lambda)$.
\end{itemize}
\end{assumption}
In Assumption \ref{assump:general-h}, 
(i) is to ensure the extension of Lemma \ref{lemma:22_holder} 
used in the RKHS approximation analysis,
see Remark \ref{rk:diff-decay-h-manifold-integral}.
The proof only uses up to ($\lfloor k/2 \rfloor+1$)-th derivative of $h(r)$,
and thus the $C^\infty$ of $h$ can be relaxed to be $C^{\lfloor k/2 \rfloor+1}$, that is, to ``match'' that of  $f^*$.
Assumption \ref{assump:general-h}
(ii)(iii) are by following the setup in \cite{van2009adaptive} to enable a series technical estimates of the RKHS. 
The assumption covers the squared exponential kernel  in \eqref{eq:def-kernel},
corresponding to $h(r) = e^{-r/2}$, as a special case:
(i) is satisfied by letting $a_l =1$, $a = 1/2$;
With $h(r) = e^{-r/2}$, $h_1(x,x')$ is a Gaussian kernel in $\R^D$, 
and its spectral measure also has a density that is Gaussian on $\R^D$.
Then (iii) holds,
 and (ii) holds with $\delta_h = 1/2$ and
$c_h = \sqrt{2}$.

The next three lemmas largely follow the arguments in \cite{yang2016bayesian, van2009adaptive},
and we prove under our Assumption \ref{assump:general-h}.
With such general $h_\epsilon$, 
let $\mathbb{H}_{\epsilon}(\mathcal{X})$  and $ \mathbb{H}_{\epsilon}([0,1]^D)$
be the RKHS associated with the kernel $h_\epsilon$ on $\calX$ and $[0,1]^D$ respectively.
The first lemma characterizes the relationship between the two.
    
\begin{lemma}
    \label{lemma:two_rkhs}
    Suppose $h_\epsilon$ satisfies Assumption \ref{assump:general-h}.
    Given $\calX \subset [0,1]^D$,
    
  (i)  For any $f \in \mathbb{H}_{\epsilon}(\mathcal{X})$, there exists a unique $\bar{f} \in \mathbb{H}_{\epsilon}([0,1]^D)$ such that 
    $\bar{f}|_\mathcal{X} = f$ and $||\bar{f}||_{\mathbb{H}_{\epsilon}([0,1]^D)} = ||f||_{\mathbb{H}_{\epsilon}(\mathcal{X})}$. 
    Moreover, for any $g \in  \mathbb{H}_{\epsilon}([0,1]^D)$ with $g|_\mathcal{X}= f$, it holds that $||g||_{\mathbb{H}_{\epsilon}([0,1]^D)} \ge ||f||_{\mathbb{H}_{\epsilon}(\mathcal{X})}$,
    and ``$=$'' is achieved only when $g = \bar f$ in $ \mathbb{H}_{\epsilon}([0,1]^D)$.
    
   (ii) For any $g \in   \mathbb{H}_{\epsilon}([0,1]^D)$, $g|_\calX \in  \mathbb{H}_{\epsilon}(\mathcal{X})$.
\end{lemma}
\begin{proof}[Proof of Lemma \ref{lemma:two_rkhs}]
The lemma follows \cite[Lemma 5.1]{yang2016bayesian} which assumed square-exponential kernel and that $\cal$ is a submanifold, but the argument extends here. 

To prove (i), we construct an isometry $\bf{\Phi}$ between 
$\mathbb{H}_{\epsilon}(\mathcal{X})$
and a complete subspace of $\mathbb{H}_{\epsilon}([0,1]^D)$
such that ${\bf \Phi}$ maps every member in 
\[
\tilde{\mathcal{H}} := 
\left\{ 
\sum_{i=1}^m a_i h_\epsilon( x_i, \cdot), \quad a_1, \cdots, a_m \in \R, \quad x_1, \cdots, x_m \in \calX, \quad  m \in \mathbb{N} \right\},
\]
viewed as a function on $\calX$ to the same function on the domain $[0,1]^D$.
Note that $\tilde{\mathcal{H}}$ is dense in $\mathbb{H}_{\epsilon}(\mathcal{X})$ so such ${\bf \Phi}$ can be constructed by extension.

For any $f \in \mathbb{H}_{\epsilon}(\mathcal{X})$,  we first prove the existence of $\bar f$.
Let $\bar f = {\bf \Phi}(f)$, 
then  $||\bar{f}||_{\mathbb{H}_{\epsilon}([0,1]^D)} = ||f||_{\mathbb{H}_{\epsilon}(\mathcal{X})}$ by that ${\bf \Phi}$ is an isometry.
Meanwhile, one can construct a sequence $f_n$ in $\tilde{\mathcal{H}}$ that converges to $f$ in $\mathbb{H}_{\epsilon}(\mathcal{X})$,
and also pointwisely on $\calX$: by reproducing property of kernel and Cauchy-Schwarz,
\[
| f_n(x) - f(x) | 
= | \langle f_n - f , h_\epsilon( x, \cdot) \rangle_{\mathbb{H}_{\epsilon}(\calX) }  |
\le \| f_n - f\|_{\mathbb{H}_{\epsilon}(\calX) } \| h_\epsilon( x, \cdot) \|_{\mathbb{H}_{\epsilon}(\calX) }, 
\]
where 
$ \| h_\epsilon( x, \cdot) \|_{\mathbb{H}_{\epsilon}(\calX) }
= h_\epsilon(x,x)^{1/2}= h(0)^{1/2} =1$,
and thus $f_n(x) - f(x) \to 0$ for any $x \in \calX$.
By that ${\bf \Phi}$ is an isometry and again $ \| h_\epsilon( x, \cdot) \|_{\mathbb{H}_{\epsilon}([0,1]^D) } = h_\epsilon(x,x)^{1/2}  =1$,
one can similarly show that  ${\bf \Phi} ( f_n) $ converges to $ {\bf \Phi}(f) $ pointwisely on $[0,1]^D$.
Combined with that ${\bf \Phi}(f_n) |_\calX = f_n$
by the definition of ${\bf \Phi}$ on $\tilde{\mathcal{H}}$, 
we have that that  ${\bf \Phi}(f) |_\mathcal{X} = f$. 

Now consider any  $g \in \mathbb{H}_{\epsilon}( [0,1]^D)  $ satisfying $g|_\mathcal{X}= f$.
By that $ (g-\bar f)|_\calX =0 $,
we have that $g-\bar f$ is in the  orthogonal complement of $ {\bf \Phi}( \mathbb{H}_{\epsilon}(\mathcal{X}) )$ in $\mathbb{H}_{\epsilon}([0,1]^D)$,
then 
\[
||g||_{\mathbb{H}_{\epsilon}([0,1]^D)}^2
= \| g- \bar f\|_{\mathbb{H}_{\epsilon}([0,1]^D)}^2 + 
 ||\bar{f}||_{\mathbb{H}_{\epsilon}([0,1]^D)}^2
 \]
by the Pythagorean theorem. 
This means that $||g||_{\mathbb{H}_{\epsilon}([0,1]^D)} \ge  ||\bar{f}||_{\mathbb{H}_{\epsilon}([0,1]^D)} $ and ``$=$'' is only achieved when 
$g = \bar f$ in $ \mathbb{H}_{\epsilon}([0,1]^D)$.  
This also proves that $\bar f = {\bf \Phi}(f)$ is the unique extension of $f$ in $\mathbb{H}_{\epsilon}([0,1]^D)$ such that preserves the RKHS norm of $f$ in $\mathbb{H}_{\epsilon}(\mathcal{X})$.

\vspace{5pt}

To prove (ii), again by the orthogonal decomposition $\mathbb{H}_{\epsilon}([0,1]^D) = {\bf \Phi}( \mathbb{H}_{\epsilon}(\mathcal{X}) ) \oplus {\bf \Phi}( \mathbb{H}_{\epsilon}(\mathcal{X}) )^\perp$, there exists $h \in  \mathbb{H}_{\epsilon}(\mathcal{X})$ such that 
 $g - {\bf \Phi}(h)  \in {\bf \Phi}( \mathbb{H}_{\epsilon}(\mathcal{X}) )^\perp$. This means that 
 \[
 \langle g - {\bf \Phi}(h) , h_\epsilon( x, \cdot) \rangle_{\mathbb{H}_{\epsilon}([0,1]^D)} = 0, \quad \forall x \in \calX,
 \]
 and because $h_\epsilon$ is the reproducing kernel, the l.f.s. equals $g(x) -  {\bf \Phi}(h)(x)$.
 By (i), we have that ${\bf \Phi}(h)(x) = h(x)$ for any $x\in \calX$. This means that 
 $g|_\calX = h $ which is in $ \mathbb{H}_{\epsilon}(\mathcal{X})$.
\end{proof}

We denote the unit ball in $\mathbb{H}_{\epsilon}(\mathcal{X})$ as   
  $\mathbb{H}^1_{\epsilon}(\mathcal{X})$.
The next lemma derives boundedness and a Lipschitz bound of functions in the unit RKHS ball.

\begin{lemma}
    \label{lemma:rkhs_lip}
    Under Assumption \ref{assump:general-h}, for any $\epsilon > 0 $ and 
   any $q \in \mathbb{H}^{1}_{\epsilon}(\mathcal{X})$, 
   it satisfies that $| q(x)| \le 1$, $\forall x\in \calX$;
Meanwhile,
    $|q(x) - q(x')| \le \epsilon^{- {1}/{2}} \tau_h ||x-x'||_{\R^D}$ for any $x, x' \in \calX$, 
    where  
    $$
    \tau_h^2 : = \int_{\R^D} \|\lambda \|^2 d\mu(\lambda).
    $$ 
\end{lemma}
\begin{proof}[Proof of Lemma \ref{lemma:rkhs_lip}]
The proof is the same as that of lemma 8.2 in \cite{yang2016bayesian}, 
which appiles when $\calX$ is a subset of $[0,1]^D$ and $h(0) =1$,
the latter implying that    $ \| h_\epsilon( x, \cdot) \|_{\mathbb{H}_{\epsilon}(\calX) }
= h_\epsilon(x,x)^{1/2}= h(0)^{1/2} =1$. We include the specifics for completeness.

For any  $q \in \mathbb{H}^{1}_{\epsilon}(\mathcal{X})$ and $ x \in \calX$,
by Cauchy-Schwarz,
\[
|q(x)| = |  \langle q , h_\epsilon( x, \cdot) \rangle_{\mathbb{H}_{\epsilon}(\calX) } | 
\le \|q \|_{\mathbb{H}_{\epsilon}(\calX) } \| h_\epsilon( x, \cdot) \|_{\mathbb{H}_{\epsilon}(\calX) }
= \|q \|_{\mathbb{H}_{\epsilon}(\calX) }  \le 1.
\]
For any $x, x' \in \calX$, similarly,
$| q(x) - q(x') | = | \langle q , h_\epsilon( x, \cdot) - h_\epsilon( x', \cdot)  \rangle_{\mathbb{H}_{\epsilon}(\calX) }|
\le  \|q \|_{\mathbb{H}_{\epsilon}(\calX) }  \| h_\epsilon( x, \cdot) - h_\epsilon( x', \cdot) \|_{\mathbb{H}_{\epsilon}(\calX) }
\le  \| h_\epsilon( x, \cdot) - h_\epsilon( x', \cdot) \|_{\mathbb{H}_{\epsilon}(\calX) }$, and, 
by that $\int_{\R^D} d\mu_\epsilon(\lambda) = h_\epsilon(x,x) = h(0)= 1$, 
\[
\| h_\epsilon( x, \cdot) - h_\epsilon( x', \cdot) \|_{\mathbb{H}_{\epsilon}(\calX) }^2
= 2 (1- h_\epsilon(x,x'))
= 2 \int_{\R^D} ( 1- e^{i \lambda^T(x-x')} )d \mu_\epsilon(\lambda).
\]
Because $d\mu_\epsilon$ is radial symmetric, 
we have $d\mu_\epsilon(-\lambda) =d\mu_\epsilon(\lambda) $,
and then $\int_{\R^D} \lambda d\mu_\epsilon(\lambda) = 0$. Then
\begin{align*}
\int_{\R^D} ( 1- e^{i \lambda^T(x-x')} )d \mu_\epsilon(\lambda)
& = \int_{\R^D} ( 1 + i\lambda^T(x-x')- e^{i \lambda^T(x-x')}  ) d \mu_\epsilon(\lambda) \\
& \le \int_{\R^D}  \frac{1}{2}  |\lambda^T(x-x') |^2 d \mu_\epsilon(\lambda)
 \le  \frac{1}{2} \| x-x'\|^2 \int_{\R^D}  \| \lambda\|^2 d \mu_\epsilon(\lambda),
\end{align*}
where the 1st inequality is by that $ |1 + i \xi - e^{i \xi}| \le \xi^2/2$ for any $\xi \in \R$.
The claim follows by that $ \int_{\R^D}  \| \lambda\|^2 d \mu_\epsilon(\lambda) = \frac{1}{\epsilon} \int_{\R^D}  \| \lambda\|^2 d \mu(\lambda) = \tau_h^2/\epsilon$.
\end{proof}

The third lemma characterizes the nested property between RKHS 
$\mathbb{H}_{\epsilon}(\mathcal{X})$ when $\epsilon$ decreases,
and intuitively, the smaller the $\epsilon$ the richer the space.

\begin{lemma}
    \label{lemma:compare_rkhs}
    Suppose $h_\epsilon$ satisfies Assumption \ref{assump:general-h},
    then $\epsilon_1 \ge \epsilon_2 > 0$ implies that
    $$
    \mathbb{H}^1_{\epsilon_1}(\mathcal{X}) 
    \subset 
    (\epsilon_1/\epsilon_2)^{ {D}/{4}}
    \mathbb{H}^1_{\epsilon_2}(\mathcal{X}).
    $$
\end{lemma}

The proof is based on the following lemma which characterizes RKHS on $[0,1]^D$ by Fourier representation. 
\begin{lemma}
\label{lemma:fourier-character-RKHS-RD}
Suppose $h_\epsilon$  satisfies 
Assumption \ref{assump:general-h},
for any $\epsilon > 0$,  
 $ \mathbb{H}_\epsilon ([0,1]^D)$  consists of real parts of the functions 
\[
h_\psi(x) = \int_{\R^D} e^{i \lambda^T x } \psi(\lambda) d\mu_\epsilon(\lambda), \quad x\in [0,1]^D,
\]
where $\psi$ runs through the complex-valued space $L^2( \mu_\epsilon)$.
Moreover, for any $f \in \mathbb{H}_\epsilon ([0,1]^D)$, there exists $\psi \in L^2( \mu_\epsilon) $ 
s.t. $f = h_\psi$, which is real-valued,
and 
 $\| f \|_{ \mathbb{H}_\epsilon ([0,1]^D)} = \| \psi \|_{L^2( \mu_\epsilon)}$.
\end{lemma}

\begin{proof}[Proof of Lemma \ref{lemma:fourier-character-RKHS-RD}]
The lemma follows \cite[Lemma 4.1]{van2009adaptive} applied to domain $[0,1]^D \subset \R^D$.
To show the statement of $f =h_\psi$, recall that the proof of \cite[Lemma 4.1]{van2009adaptive} is by first letting $\psi$ run though the complex-number linear span $\mathcal{L}$ of the sets of functions $\{e_x(\lambda) = e^{-i \lambda^T x}, \, x \in [0,1]^D\}$
and then take the $L^2(\mu_\epsilon)$ closure $\bar \calL$ of $\mathcal{L}$.
Under Assumption \ref{assump:general-h}(ii), 
the spectral measure $\mu_\epsilon$ satisfies the subexponential decay condition Eqn. (3.3) in \cite{van2009adaptive},
and then their Lemma 4.1 proved that $\bar \calL = L^2(\mu_\epsilon)$.

There is a mapping $P^{\R}$ from $\calL$ to its subset $\calL^{\R}$, which consists of real-number linear span of $\{ e_x\}$,
by taking the real-part of the coefficients in the linear combination.
One can verify that $\forall \varphi \in \calL$,  $h_{P^{\R} \varphi} = {\rm Re} h_\varphi$.
We then extend $P^{\R}$ to $\bar \calL = L^2(\mu_\epsilon)$ by taking the closure, 
then $\forall \varphi \in L^2(\mu_\epsilon)$,  ${\rm Re} h_\varphi = h_{ P^{\R} \varphi} $.
Thus, using the first part of the lemma we have $f = {\rm Re} h_\varphi$ for a $\varphi \in L^2(\mu_\epsilon)$,
and letting $\psi = P^{\R} \varphi$ gives that $f = h_\psi$ and is real-valued.
Finally, to show that $\| f \|_{ \mathbb{H}_\epsilon ([0,1]^D)} = \| \psi \|_{L^2( \mu_\epsilon)}$,
first verify that  $\| h_\psi \|_{ \mathbb{H}_\epsilon ([0,1]^D)} = \| \psi \|_{L^2( \mu_\epsilon)}$ for any $\psi \in \calL^{\R}$,
and this means that the mapping $H: \psi \mapsto h_\psi$ is an isometry from $ \calL^{\R}$ to $H( \calL^{\R})$.
By taking the closure $\bar \calL^\R$ of $\calL^\R$ in $L^2(\mu_\epsilon)$,
$H$ is also an isometry on $\bar \calL^\R$. 
One can verify that $P^{\R} (\bar \calL) = \bar \calL^\R$,
and our $\psi = P^{\R} \varphi \in P^{\R} (\bar \calL)$,
thus $\| h_\psi \|_{ \mathbb{H}_\epsilon ([0,1]^D)} = \| \psi \|_{L^2( \mu_\epsilon)}$ for our $\psi$.
\end{proof}

\begin{proof}[Proof of lemma \ref{lemma:compare_rkhs}]
We first prove the claim for RKHS in $[0,1]^D$: 
\begin{equation}\label{eq:lemma-compare-rkhs-euclidean-space}
    \mathbb{H}^1_{\epsilon_1}( [0,1]^D ) 
    \subset 
    (\epsilon_1/\epsilon_2)^{ {D}/{4}}
    \mathbb{H}^1_{\epsilon_2}( [0,1]^D ),
\end{equation}
which was addressed in \cite[Lemma 4.7]{van2009adaptive}.
However, we believe that the bound there misses a power of $D$ in the factor multiplied in front of the RKHS unit ball. 
We include a proof here for completeness.

For any $f \in \mathbb{H}^1_{\epsilon_1}( [0,1]^D ) $,
by Lemma \ref{lemma:fourier-character-RKHS-RD}, there exists $\psi \in L^2(\mu_{\epsilon_1})$ s.t.   $f = h_\psi$ and 
$ \| f\|_{\mathbb{H}_{\epsilon_1}( [0,1]^D )} =  \| \psi\|_{ L^2(\mu_{\epsilon_1}) } \le 1$.
Note that $f(x) = \int e^{i \lambda^T x} \psi(\lambda) f_{\epsilon_1}(\lambda) d\lambda
= \int e^{i \lambda^T x} \varphi(\lambda )f_{\epsilon_2}(\lambda) d\lambda$, where $ \varphi = \psi f_{\epsilon_1}/f_{\epsilon_2} $.
One can verify that $\varphi  \in L^2( \mu_{\epsilon_2})$: 
Note that
\[ 
\frac{f_{\epsilon_1} (\lambda)}{f_{\epsilon_2} (\lambda) } 
= (\frac{\epsilon_1}{\epsilon_2})^{D/2}
\frac{  f_1(\sqrt{\epsilon_1} \lambda)}{ f_1(\sqrt{\epsilon_2} \lambda)}
\le (\frac{\epsilon_1}{\epsilon_2})^{D/2}, \quad \forall \lambda \in \R^D,
\]
because $ f_1(\sqrt{\epsilon_1} \lambda) \le  f_1(\sqrt{\epsilon_2} \lambda)$ by radial monotonicity Assumption \ref{assump:general-h}(iii) and that $\epsilon_1 \ge \epsilon_2$.
Thus
\[
\| \varphi\|_{  L^2( \mu_{\epsilon_2})}^2
= \int   |\psi(\lambda)|^2  \frac{ f_{\epsilon_1} (\lambda)}{ f_{\epsilon_2} (\lambda)} f_{\epsilon_1} (\lambda) d\lambda
\le (\frac{\epsilon_1}{\epsilon_2})^{D/2} \int   |\psi(\lambda)|^2  f_{\epsilon_1} (\lambda) d\lambda 
\le (\frac{\epsilon_1}{\epsilon_2})^{D/2}.
\]
As a result, $f \in \mathbb{H}_{\epsilon_2}( [0,1]^D ) $ 
and $ \| f\|_{ \mathbb{H}_{\epsilon_2}( [0,1]^D )  } = \| \varphi\|_{  L^2( \mu_{\epsilon_2})} \le  ({\epsilon_1}/{\epsilon_2})^{D/4}$.
This finishes the proof of \eqref{eq:lemma-compare-rkhs-euclidean-space}.

The lemma then follows from  \eqref{eq:lemma-compare-rkhs-euclidean-space} combined with Lemma \ref{lemma:two_rkhs},
and the argument is the same as in  \cite[Lemma 8.1]{yang2016bayesian}.
Specifically, for any $f \in  \mathbb{H}^1_{\epsilon_1}(\mathcal{X}) $, by Lemma \ref{lemma:two_rkhs}(i),
there exists $\bar f \in \mathbb{H}_{\epsilon_1}( [0,1]^D ) $ s.t. 
$\bar f|_\calX = f$ and
$\| \bar f \|_{ \mathbb{H}_{\epsilon_1}( [0,1]^D ) } = \|f \|_{\mathbb{H}_{\epsilon_1}( \calX )} \le 1$.
By \eqref{eq:lemma-compare-rkhs-euclidean-space},
$\bar f \in \mathbb{H}_{\epsilon_2}( [0,1]^D ) $ and
$\| \bar f\|_{ \mathbb{H}_{\epsilon_2}( [0,1]^D ) } 
\le (\epsilon_1/\epsilon_2)^{ {D}/{4}} \| \bar f\|_{\mathbb{H}_{\epsilon_1}([0,1]^D ) }
\le (\epsilon_1/\epsilon_2)^{ {D}/{4}} $.
Then, by Lemma \ref{lemma:two_rkhs}(ii), $\bar f|_\calX = f \in \mathbb{H}_{\epsilon_2}( \calX ) $,
and by Lemma \ref{lemma:two_rkhs}(i), $  \| \bar f\|_{ \mathbb{H}_{\epsilon_2}( [0,1]^D)}  \ge \| f\|_{ \mathbb{H}_{\epsilon_2}( \calX ) }$.
This means that $ \| f\|_{ \mathbb{H}_{\epsilon_2}( \calX ) } \le  (\epsilon_1/\epsilon_2)^{ {D}/{4}} $, namely
$f \in (\epsilon_1/\epsilon_2)^{ {D}/{4}} \mathbb{H}^1_{\epsilon_2}( \calX )$.
\end{proof}

We also introduce a lemma to compute the RKHS norm of functions expressed as kernel integral operator applied to another function. 
\begin{lemma}\label{lemma:computation-Hnorm}
Suppose $h_\epsilon$  satisfies Assumption \ref{assump:general-h},
given $\calX \subset [0,1]^D$,
for any $\epsilon >0$, let $\mathbb{H}_\epsilon (\calX)$ 
be the RKHS associated with $h_\epsilon$.
Let $d\nu $ be a measure on $\calX$,
suppose $h_\epsilon (\cdot, y)$ is in  $L^2(\calX, d\nu)$ for any $y \in [0,1]^D$,
and 
$\int_\calX \int_\calX h_\epsilon(x,y)^2 d\nu(x) d\nu(y) < \infty$.
Then,
 for any $g \in L^2(\calX, d\nu)$, 
the function 
$f(x) = \int_{\calX} h_\epsilon( x, y) g(y)  d\nu(y)$  is in $\mathbb{H}_\epsilon (\calX)$, and 
$\| f \|_{\mathbb{H}_\epsilon (\calX)}^2 = \int_\calX \int_\calX h_\epsilon(x,y) g(x) g(y) d\nu(x ) d\nu(y)$.
\end{lemma}

It is possible to prove the conclusion in even more general settings.
In this work, we apply Lemma \ref{lemma:computation-Hnorm} to when $(\calX, dx) = (\calM, dV)$ or when $\calX$ is a finite union of disjoint manifolds (Assumption \ref{A4:disconnected}),
when $\calX$ is always a subset of $[0,1]^D$.
In our usage, the needed integrability conditions by Lemma \ref{lemma:computation-Hnorm} are always satisfied 
because $h_\epsilon$ is continuous under  Assumption \ref{assump:general-h} and $\iota( \calM)$ is continuous and compact domain.

\begin{proof}[Proof of Lemma \ref{lemma:computation-Hnorm}]
Consider $\epsilon > 0 $ fixed, and denote the kernel $h_\epsilon$ as $k$.
We first verify that $f(x)$ is well defined on $[0,1]^D$, and thus also on $\calX$.
For any $x\in [0,1]^D$, 
$|f(x)| \le \|g \|_{ L^2(\calX, d\nu) } ( \int_{\calX} k(x,y)^2 d\nu(y) )^{1/2} $
by Cauchy-Schwarz,
and then $f(x)$ is finite due to that both $g$ and $k(x, \cdot)$
are in $L^2(\calX, d\nu)$ ($k$ is symmetric).
This also gives that 
$\int_\calX f(x)^2 d\nu(x) 
\le \|g \|_{ L^2(\calX, d\nu) }^2  \int_\calX    \int_{\calX} k(x,y)^2 d\nu(y)   d\nu(x)  <\infty$,
and thus  $f \in L^2(\calX, d\nu)$.

Denote $\mathbb{H}_\epsilon (\calX)$ by $\tilde{ \mathbb{H}}$.
Suppose $f(x)$ is in $\tilde{ \mathbb{H}}$, then 
\begin{equation}\label{eq:rkhs-norm-integral-f-formal-1}
\| f \|_{\tilde{ \mathbb{H}}}^2 
= \langle \int_\calX k(\cdot, y) g(y) d\nu(y), \int_\calX k(\cdot, y') g(y') d\nu(y')  \rangle_{\tilde{ \mathbb{H}}}
= \int_\calX \int_\calX k(y,y') g(y)  g(y') d\nu(y) d\nu(y')
\end{equation}
by the producing property of $k$, i.e.,  $ \langle   k(\cdot, y) ,  k(\cdot, y')   \rangle_{\tilde{ \mathbb{H}}} = k(y,y')$.
We now show that the r.h.s. is finite: by definition of $f$ and Cauchy-Schwarz, 
\begin{align}\label{eq:rkhs-norm-f-finite-1}
\int_\calX \int_\calX k(x,y) g(x)  g(y) d\nu(x) d\nu(y) 
 = \int_\calX g(x) f(x) d\nu(x)
 \le \| g\|_{L^2(\calX, d\nu)} \|f\|_{L^2(\calX, d\nu)} < \infty,
\end{align}
because both $g$ and $f$ are in $L^2(\calX, d\nu)$.
It remains to show that $f$ is in $\tilde{ \mathbb{H}}$ to finish the proof. 

To do this, we will first show that $f$ (as a function on $[0,1]^D$) is in the RKHS ${ \mathbb{H}} = \mathbb{H}_\epsilon ([0,1]^D)$,
then $f|_\calX$ is in $\tilde{ \mathbb{H}} $ by Lemma \ref{lemma:two_rkhs}(ii).

Because $h_\epsilon(x,y)$ is real-valued, using the spectral measure representation, we have
\[ 
h_\epsilon(x,y) 
= \int_{\R^D} e^{-i \lambda^T(x-y)}  d\mu_\epsilon(\lambda)
= \int_{\R^D} e^{i \lambda^T(x-y)}  d\mu_\epsilon(\lambda).
\]
Inserting into the definition of $f$ and using the notation in Lemma \ref{lemma:fourier-character-RKHS-RD},
we have
\[
f(x) = \int_\calX  \int_{\R^D} e^{ i \lambda^T(x-y)} g(y) d\nu(y)  d\mu_\epsilon(\lambda)
= h_\psi(x), 
\quad \psi(\lambda) = \int_{\calX} e^{-i \lambda^T y} g(y) d\nu(y),
\]
and this $h_\psi(x) = f(x)$ is real-valued. 
Thus, by Lemma \ref{lemma:fourier-character-RKHS-RD}, to show that $f \in { \mathbb{H}}  $ it suffices to show that $\psi \in L^2(\R^D, \mu_\epsilon) $,
and this is the case because 
\begin{align*}
\int_{\R^D} |\psi(\lambda)|^2 d\mu_\epsilon(\lambda)
& =  \int_{\R^D} \int_{\calX}  \int_{\calX} e^{-i \lambda^T (y - y')} g(y)    g(y') d\nu(y) d\nu(y') d\mu_\epsilon(\lambda)\\
& = \int_{\calX}  \int_{\calX} h_\epsilon(y, y') g(y)    g(y') d\nu(y) d\nu(y'), 
\end{align*}
which is finite as shown in \eqref{eq:rkhs-norm-f-finite-1}.
As a result, $f \in { \mathbb{H}}  $ and  $f|_\calX \in \tilde{ \mathbb{H}}$.
\end{proof}

\subsubsection{RKHS covering lemma and small ball probability on a general subset}

These important estimates are used in the proofs in Section \ref{sec:pos_rate}.

\begin{lemma}[RKHS covering bound]
    \label{lemma:unit_ball_covering_general}
    Suppose  $\mathcal{X} \subset [0,1]^D$ satisfies Assumption \ref{assump:A1-A2-prime}(A1) with positive constants $r_0 $ and $C_\calX$ as therein,
    and  $\mathbb{H}^1_{t}(\mathcal{X})$ is the unit ball in the RKHS on $\calX$ associated with kernel $h_t$ satisfying Assumption \ref{assump:general-h}.
       Then, there exist $K_1>1$ and $c > 4$ s.t., 
    for any $0 < t < r_0^2$, we have
   \begin{align*}
   \log \calN(\varepsilon', \mathbb{H}^1_{t}(\mathcal{X}), ||\cdot||_{\infty}) 
    \le K_1 t^{-\varrho/2} (\log { \frac{c^D}{\varepsilon'} } )^{D+1}, \quad \forall 0 < \varepsilon ' < 1,
   \end{align*}
  and $\calN(\varepsilon', \mathbb{H}^1_{t}(\mathcal{X}), ||\cdot||_{\infty})  =1$ when $\varepsilon ' \ge 1$.
  The constant $c$ depends on $\mu$, 
  and $K_1$ depends on  $\mu$,  $D$ and $\calX$. 
  In particular, there is $K>1$ which depends on $\mu$, $D$ and $\calX$,
   s.t.
   \begin{align*}
   \log \calN(\varepsilon', \mathbb{H}^1_{t}(\mathcal{X}), ||\cdot||_{\infty}) 
    \le K t^{-\varrho/2} (\log { \frac{ 1}{\varepsilon'} } )^{D+1}, \quad \forall  0 < \varepsilon' < 1/2.
   \end{align*}
\end{lemma}

The proof adopts techniques from Lemma 4.5 in \cite{van2009adaptive},
and when we construct a net to cover the domain $\calX$ we invoke Assumption \ref{assump:A1-A2-prime}(A1) and bring in the factor $\varrho$ in the scaling.

\begin{proof}[Proof of Lemma \ref{lemma:unit_ball_covering_general}]

We apply the following result from \cite[Corollary A.8]{hamm2021adaptive}:
For all $t, \, \varepsilon' > 0$, we have 
$$
\log \calN(\varepsilon',\mathbb{H}^1_{t}(\mathcal{X}), ||\cdot||_{\infty}) \leq \calN(\sqrt{t},\mathcal{X}, ||\cdot||_{\infty}) \log\calN(\varepsilon',\mathbb{H}^1_{1}([-1,1]^D), ||\cdot||_{\infty}), 
$$
where in the covering number of $\calX$, $\| \cdot\|_\infty$ is in $\R^D$, and in the covering numbers of RKHS balls, $\| \cdot\|_\infty$ stands for the functional infinity norm on the corresponding domains.
By Assumption \ref{assump:A1-A2-prime}(A1), when $0<t<r_0^2$, $\calN(\sqrt{t},\mathcal{X}, ||\cdot||_{\infty}) \leq C_{\calX} t^{-\varrho/2}$. 
Hence, we need to bound  $\log\calN(\varepsilon',\mathbb{H}^1_{1}([-1,1]^D), ||\cdot||_{\infty})$.

Recall that $\mu=\mu_1$ is the spectral measure of kernel $h_1$. 
By Lemma \ref{lemma:two_rkhs} and Lemma \ref{lemma:fourier-character-RKHS-RD}, 
 any element of $\mathbb{H}^1_{1}([-1,1]^D)$ can be expressed as
$$
h_{\psi}(x)=\int_{\mathbb{R}^D}e^{-i \lambda^\top x} \psi(\lambda) d \mu(\lambda), \quad x \in [-1,1]^D,
$$
which is real-valued,
where $\psi \in L^2(\mu)$ and $\| \psi \|_{L^2(\mu)} = \| h_\psi \|_{\mathbb{H}_{1}([-1,1]^D) } \le 1$.
We extend $h_{\psi}(x)$ with $x \in \mathbb{R}^D$ to a function $h_{\psi}(z)$ with $z \in \mathbb{C}^D$.
Because $\int_{\R^D} | \psi(\lambda)|^2 d\mu(\lambda) \le 1$, 
by Cauchy-Schwartz,
$|h_{\psi}(z)|^2 \leq \int_{\mathbb{R}^D} e^{2\|\lambda\|_{\R^D} \|Im(z)\|_{\R^D}} d\mu(\lambda)$. 
Let $\delta_h$ and $c_h$ be the constants in Assumption \ref{assump:general-h}(ii),
and define \rev{$R := (\delta_h\wedge1)/2$}.
Let $\Omega=\{z \in \mathbb{C}^D,   \|Im(z)\|_{\mathbb{R}^D}<R\}$. 
The $\C$-valued function $h_{\psi}(z)$ is analytic on $\Omega$, and 
$$
|h_{\psi}(z)| \leq (\int_{\mathbb{R}^D} e^{\delta_h \|\lambda\|_{\R^D}} d\mu(\lambda) )^{1/2} \leq c_h^D,
\quad  \forall z \in  \Omega.
$$

We will construct a set of piecewise polynomials to approximate $h_\psi$ in $\|\cdot \|_\infty$,
which then provides a net of $\mathbb{H}^1_{1}([-1,1]^D)$.
To proceed, we use the multi-index notations: Let $n=(n_1, n_2, \cdots, n_D)$, $n!=n_1! n_2! \cdots n_D!$, and $|n|=n_1+n_2+\cdots+ n_D$. 
\rev{For any $x\in\mathbb R^D$ and $0<\rho<R/\sqrt D$, consider the complex polydisc
\[
\mathbb D_\rho(x):=\{z\in\mathbb C^D:|z_j-x_j|<\rho,\ j=1,\ldots,D\}.
\]
For $z\in\mathbb D_\rho(x)$, $\|\operatorname{Im}z\|_{\mathbb R^D}\le
\|z-x\|_{\mathbb C^D}<\sqrt D\,\rho<R$, so
$\mathbb D_\rho(x)\subset\Omega$.  The multivariate Cauchy formula on this
polydisc, followed by letting $\rho\uparrow R/\sqrt D$, gives}
\begin{align}\label{Appendix: bound cauchy formula 0}
\left| \frac{D^n h_{\psi}(x)}{n!} \right|
\rev{\leq c_h^D\left(\frac{\sqrt D}{R}\right)^{|n|}},
\quad \forall n=(n_1, n_2, \cdots, n_D).
\end{align}
\rev{Since $R\le1/2$, we can construct a net
$\{p_1,\ldots,p_m\}\subset[-1,1]^D$ such that the $\ell^\infty$-cells
\[
Q_i:=\left\{x\in\mathbb R^D:\|x-p_i\|_\infty
\le\frac{R}{2\sqrt D}\right\},\qquad i=1,\ldots,m,
\]
cover $[-1,1]^D$ and $m\le(4\sqrt D/R)^D$.}

We construct a set $\mathcal{S}$ of piecewise polynomials of degree at most $q$ associated with \rev{$\{Q_i\}$} on $[-1,1]^D$,
where $q$ is to be determined.
We denote the cardinal number $\mathcal{S}$ by  $|\mathcal{S}|$.
We partite
\rev{$[-c_h^D(\frac{\sqrt D}{R})^{|n|},
c_h^D(\frac{\sqrt D}{R})^{|n|}]$}
into intervals of length between
\rev{$\frac{\varepsilon'}{2}(\frac{\sqrt D}{R})^{|n|}$ and
$\varepsilon'(\frac{\sqrt D}{R})^{|n|}$},
and  let $a_{i,n}$ be any end point of these intervals. 
\rev{To remove overlaps, set
\[
\mathcal A_i:=
\left(\rev{Q_i}\cap[-1,1]^D\right)
\setminus\bigcup_{j<i}\rev{Q_j},
\qquad i=1,\ldots,m,
\]
where the union is empty for $i=1$.  The sets $\mathcal A_i$ are disjoint,
cover $[-1,1]^D$, and satisfy
$\mathcal A_i\subset\rev{Q_i}$.  We construct the
piecewise polynomials $P$ in $\mathcal S$ on this partition as follows:}
$$
P=\sum_{i=1}^m P_i
\rev{\chi_{\mathcal A_i}},  \quad\quad P_i(x)= \sum_{|n| \leq q}a_{i,n} [x-p_i]^n,
\quad\quad x \in \rev{\mathcal A_i},
$$
where for $x =(x_1, \cdots, x_D) \in \R^D$,  $[x]^n$ stands for $x_1^{n_1} \cdots x_D^{n_D}$.
Therefore, 
\begin{equation}\label{eq:logS-bound-1-proof}
\log | \mathcal{S} |
\leq m\rev{(q+1)^D}
\log\left(
\frac{\rev{2c_h^D(\sqrt D/R)^{|n|}}}
{\rev{(\varepsilon'/2)(\sqrt D/R)^{|n|}}}
\right)
=  m\rev{(q+1)^D} \log(\frac{4c_h^D}{\varepsilon'}).
\end{equation}

\rev{For every $\ell\geq0$,
\[
\#\{n\in\mathbb Z_{\geq0}^D:|n|=\ell\}
=\binom{\ell+D-1}{D-1}.
\]
The negative-binomial identity gives
\[
\sum_{\ell=0}^{\infty}\binom{\ell+D-1}{D-1}(3/4)^\ell
=4^D.
\]}
Then, for any $x \in \rev{Q_i}$, by \eqref{Appendix: bound cauchy formula 0} we have
\begin{align}\label{Appendix: polynomial approximation 0.1}
&|h_{\psi}(x)- \sum_{|n|\leq q}\frac{D^n h_{\psi}(p_i)}{n!}[x-p_i]^n| \nonumber \\
\leq & |\sum_{|n|>q}\frac{D^n h_{\psi}(p_i)}{n!}[x-p_i]^n| 
\leq \sum_{|n|>q}
\rev{c_h^D\left(\frac{\sqrt D}{R}\right)^{|n|}
\left(\frac{R}{2\sqrt D}\right)^{|n|}} \nonumber \\
&\rev{=c_h^D\sum_{\ell=q+1}^\infty
\binom{\ell+D-1}{D-1}2^{-\ell}}\nonumber\\
&\rev{=c_h^D\sum_{\ell=q+1}^\infty
\binom{\ell+D-1}{D-1}(3/4)^\ell(2/3)^\ell} \nonumber\\
&\rev{\le c_h^D(2/3)^q\sum_{\ell=q+1}^\infty
\binom{\ell+D-1}{D-1}(3/4)^\ell}
 \rev{\leq  (4c_h)^D  (\frac{2}{3})^q}.
\end{align}
Moreover, there exists $P \in \mathcal{S}$ such that
\rev{for every $x\in\mathcal A_i$,}
\begin{align}\label{Appendix: polynomial approximation 0.2}
|\sum_{|n|\leq q}\frac{D^n h_{\psi}(p_i)}{n!}[x-p_i]^n-P(x)|
&\leq \sum_{|n|\leq q}
\rev{\varepsilon'\left(\frac{\sqrt D}{R}\right)^{|n|}
\left(\frac{R}{2\sqrt D}\right)^{|n|}} \\
&\rev{\leq\varepsilon'\sum_{\ell=0}^{\infty}
\binom{\ell+D-1}{D-1}2^{-\ell}=2^D\varepsilon'}.
\end{align}
We require $c_h^D (\frac{2}{3})^q \leq \varepsilon'$ which is satisfied by choosing 
$q=\lceil 3 \log(\frac{c_h^D }{\varepsilon'}) \rceil$.  
 By \eqref{Appendix: polynomial approximation 0.1}\eqref{Appendix: polynomial approximation 0.2} and triangle inequality,   
 \rev{$\|h_{\psi}-P\|_{\infty} \leq (4^D+2^D)\varepsilon'
 \leq 2\cdot4^D\varepsilon'$.}
 This means that 
$$ 
\log \calN( \rev{2\cdot4^D} \varepsilon', \mathbb{H}^1_{1}([-1,1]^D), ||\cdot||_{\infty})
\leq \log  | \mathcal{S} |.
$$
We revisit \eqref{eq:logS-bound-1-proof} to continue.
\rev{Here $\#\{n\in\mathbb Z_{\ge0}^D:|n|\le q\}\le(q+1)^D$, and, for the
chosen $q$, $q+1\le3\log(4c_h^D/\varepsilon')$.}
By substituting the bounds of $m$, $q$ and
\rev{$R=(\delta_h\wedge1)/2$} respectively, we have
\begin{align*}
\log  | \mathcal{S} |
\leq& m\rev{(q+1)^D} \log(\frac{4 c_h^D}{\varepsilon'}) \\
&\rev{\leq  (8\sqrt{D}/\rev{(\delta_h\wedge1)})^D  3^D \log(\frac{4c_h^D }{\varepsilon'})^{D+1}} \\
\leq &  (24\sqrt{D}/\rev{(\delta_h\wedge1)})^D \log(\frac{4 c_h^D }{\varepsilon'})^{D+1}, 
\end{align*}
and the argument so far holds for
\rev{$0<\varepsilon'<1$}.
Define $\varepsilon : = \rev{2\cdot4^D} \varepsilon'$, then we have
\begin{equation}\label{Appendix:lemma:unit_ball_covering_manifold:final}
\log \calN(\varepsilon,\mathbb{H}^1_{1}([-1,1]^D), ||\cdot||_{\infty}) 
\leq  (24\sqrt{D}/\rev{(\delta_h\wedge1)})^D
\log(\frac{\rev{8(4c_h)^D} }{\varepsilon})^{D+1}.
\end{equation}
We will utilize this upper bound when $\varepsilon < 1$, because otherwise the covering number can be bounded trivially by 1, see below.

In conclusion, suppose $0<t<r_0^2$,  when $ 0 < \varepsilon < 1$, 
we have
\begin{align}\label{lemma:unit_ball_covering_manifold full result}
\log \calN(\varepsilon, \mathbb{H}^1_{t}( \calX ), ||\cdot||_{\infty}) 
\leq K_1 t^{-\varrho/2} \log(\frac{ c^D}{\varepsilon})^{D+1}, 
\end{align}
where $K_1 : = (C_{\calX} \vee 1) (\frac{24\sqrt{D}}{  \delta_h \wedge 1})^D>1$ 
and $c : =\rev{32c_h}>4$.
Meanwhile, 
by Lemma \ref{lemma:rkhs_lip}, for any $f^t \in \mathbb{H}^1_{t}(\mathcal{X})$ we always have $\|f^t\|_{\infty} \leq 1$. 
Hence, $\calN(\varepsilon, \mathbb{H}^1_{t}(\mathcal{X}), ||\cdot||_{\infty})=1$ for $\varepsilon \geq 1$.

Finally, to derive the claimed bound when $\varepsilon < 1/2$, we use the elementary relationship that $a+b \leq 3ab$ if $a \geq \log2$ and $b \geq \log2$. 
Thus, 
if $\varepsilon < \frac{1}{2}$, then $\log(\frac{c^D}{\varepsilon})=\log(c^D)+\log(\frac{1}{\varepsilon}) 
\leq 3 \log(c^D)\log(\frac{1}{\varepsilon})$. 
Substituting into \eqref{lemma:unit_ball_covering_manifold full result}, this gives 
$$
\log \calN(\varepsilon,\mathbb{H}^1_{t}( \calX), ||\cdot||_{\infty}) \leq  K t^{-\varrho/2}( \log \frac{1}{\varepsilon})^{D+1},
\quad \varepsilon < {1}/{2},
$$
where $K= (C_{\calX} \vee 1) (\frac{24\sqrt{D}}{  \delta_h \wedge 1})^D
(3D\log c )^{D+1} >1$. 
\end{proof}

\begin{lemma}[Small ball probability of Gaussian measure]
    \label{lemma:small_ball}
    Let $\calX$ and $h_t$ be as in Lemma \ref{lemma:unit_ball_covering_general},
    and
    let $f^t$ be the Gaussian process on $\calX$ associated with kernel $h_t$. 
Then, there exists $ C > 1$ and s.t. 
for any 
$0<t < \min\{ r_0^2, 1 \}$ and $0<\varepsilon' <  1/2$, we have 
    $$ 
    \phi_0^{t}(\varepsilon') 
    = -\log
    \P [ ||f^{t}||_{\infty} \le \varepsilon' | \, t] 
    \le C  t^{- \varrho /2} (\log  \frac{1}{\sqrt{t}\varepsilon'} )^{D+1}.
    $$
The constant $C$ depends on $\varrho$, $D$, $C_\calX$ 
and constants 
$\delta_h$, $c_h$ and $a_1$ as in Assumption \ref{assump:general-h}.
\end{lemma}

The framework of the proof was outlined in Lemma 4.6 of \cite{van2009adaptive}. 
We follow techniques from Lemma 3 of \cite{castillo2025deep} to fill in the details and derive the constants explicitly.

\begin{proof}[Proof of Lemma \ref{lemma:small_ball}]
As shown in the beginning of the proof of Lemma 4.6 in \cite{van2009adaptive},
by Theorem 2 in \cite{kuelbs1993metric},
 for any $\varepsilon'>0$ and $t>0$, 
$ \phi_0^{t}(2\varepsilon')+\log(\frac{1}{2}) \leq 
\log \calN(\varepsilon'/\sqrt{2\phi_0^{t}(\varepsilon')},\mathbb{H}^1_{t}(\mathcal{X}), ||\cdot||_{\infty}).$
Hence,
\begin{align}\label{thm:small_ball_manifold 01}
\phi_0^{t}(\varepsilon') 
\leq  \log\calN(\varepsilon'/\Big(2\sqrt{2\phi_0^{t}(\varepsilon'/2)}\Big),\mathbb{H}^1_{t}(\mathcal{X}), ||\cdot||_{\infty})
	+ \log 2.
\end{align}
Since $\log \calN(\varepsilon',\mathbb{H}^1_{t}(\mathcal{X}), ||\cdot||_{\infty})$ is estimated in Lemma \ref{lemma:unit_ball_covering_general}, we need to find a crude upper bound for $2\sqrt{2\phi_0^{t}(\varepsilon'/2)}$. 

Suppose $v$ is a compact linear operator from a separable Hilbert space space $(E, (\cdot, \cdot)_{E})$ with a unit ball $B_E$ to a Banach space $(F, \|\cdot\|_{F})$. Recall the following definition of a functional $e_\ell$ from  \cite{li1999approximation} for $\ell \geq 1$: 
$$e_\ell(v)=\inf\{\eta>0: \calN(\eta, v(B_E),  \|\cdot\|_{F})\leq 2^{\ell-1}\}$$
Specifically, in this proof, we consider $u_t:\mathbb{H}_{t}(\mathcal{X}) \rightarrow C(\mathcal{X}, ||\cdot||_{\infty})$ which is the inclusion map. Then, for $\ell \geq 1$,
$$e_\ell(u_t)=\inf\{\eta>0: \log \calN(\eta, \mathbb{H}^1_{t}(\mathcal{X}), ||\cdot||_{\infty})\leq (\ell-1)\log 2\}.$$
By the part when $\varepsilon' \ge 1$ in Lemma \ref{lemma:unit_ball_covering_general}, $e_\ell(u_t) \leq 1$ for all $\ell$, and in particular when  $\ell=1$.
For $\ell \geq 2$ and $t<r_0^2$, $e_\ell(u_t)$ should be bounded above by the solution $\eta^*$ of 
$$K_1 t^{-\varrho/2}\log(\frac{c^D}{\eta^*})^{D+1}= (\ell-1)\log 2,$$
where, according to the proof of Lemma \ref{lemma:unit_ball_covering_general}, 
$K_1= (C_{\calX} \vee 1) (\frac{24\sqrt{D}}{  \delta_h \wedge 1})^D>1,$
and $c>4$ depends on $c_h$ in Assumption \ref{assump:general-h}(ii). 
Hence, for any $\ell \geq 2$ and $t<r_0^2$,
$$e_\ell(u_t) \leq c^D \exp\Big(-(K_1 t^{-\varrho/2})^{-1/(D+1)} ((\ell-1)\log 2)^{1/(D+1)}\Big).$$
When $\ell \geq 2$, we have $\ell \leq 2\ell-2$. Therefore,
\begin{align*}
& \ell e_\ell(u_t) \leq (2\ell-2)  c^D \exp\Big(-(K_1 t^{-\varrho/2})^{-1/(D+1)} ((\ell-1)\log 2)^{1/(D+1)}\Big) \\
& = \frac{2 c^D}{\log2} (K_1 t^{-\varrho/2})  \frac{(\ell-1)\log 2}{K_1 t^{-\varrho/2}} \exp\Big(-(K_1 t^{-\varrho/2})^{-1/(D+1)} ((\ell-1)\log 2)^{1/(D+1)}\Big)\\
& \leq 3 c^D K_1 t^{-\varrho/2} D^{D+1},
\end{align*}
where we use the fact $ye^{-y^{1/(D+1)}}$ has a maximum $(\frac{1+D}{e})^{1+D}\leq D^{D+1}$ over $y \geq 0$ and $y=\frac{(\ell-1)\log 2}{K_1 t^{-\varrho/2}}$.
In conclusion, for $\ell \geq 1$ and $t <\min\{ r_0^2, 1 \}$,
\begin{align}\label{lemma: e ell (u t) upper bound}
\ell e_\ell(u_t) \leq 3 c^D K_1 D^{D+1} t^{-\varrho/2}.
\end{align}

Let $\{\tilde{f}^t_i\}_{i=1}^\infty$ be an orthonormal basis of ${H}_{t}(\mathcal{X})$. The $n$ th approximation number of $u_t$ is defined as 
$$\ell_n(u_t)=\inf\Big\{\big(\mathbb{E}\|
	\sum_{j=n}^\infty a_j u_t(\tilde{f}^t_j )\|^2_\infty\big)^{1/2}: 
  	a_j\stackrel{i.i.d}{\sim}\mathcal{N}(0,1) \Big\},$$
where the infimum is taken over all orthonormal basis $\{\tilde{f}^t_i\}_{i=1}^\infty$. Moreover, we have the following definition of $n$ th approximation number of $f^t$:
$$\ell_n( f^t)=\inf\Big\{\big(\mathbb{E}\|\sum_{j=n}^\infty a_jg_j\|^2_\infty\big)^{1/2}:  f^t\stackrel{d}{=}\sum_{j=1}^\infty a_jg_j, \quad a_j\stackrel{i.i.d}{\sim}\mathcal{N}(0,1), \quad g_j \in C(\mathcal{X}, \| \cdot \|_\infty) \Big\}.$$
By Lemma 2.3 in \cite{li1999approximation}, 
$\ell_n(u_t)= \ell_n( f^t) :=\ell_n.$
By Lemma 2.1 in \cite{li1999approximation}, there are universal constants $\tilde{c}_1$ and $\tilde{c}_2$ s.t.
$$\ell_n \leq \tilde{c}_1\sum_{m \geq \tilde{c}_2n } e_m(u_t^*) m^{-1/2}(1+\log m),$$
where $u_t^*$ is the dual of $u_t$. By \cite{tomczak1987dualite}, for any $m \geq 1$,
$$m e_m (u_t^*) \leq \sup_{\ell \leq m} \ell e_\ell (u_t^*) \leq 32 \sup_{\ell \leq m} \ell e_\ell(u_t).$$
Therefore, by \eqref{lemma: e ell (u t) upper bound}, 
\begin{align*}
4\ell_n \leq & 384 \tilde{c}_1 c^D K_1 D^{D+1} t^{-\varrho/2} \sum_{m \geq \tilde{c}_2n }m^{-3/2}(1+\log m) \\
\leq& c^*  c^D K_1 D^{D+1} t^{-\varrho/2} n^{-1/2}(\log n+1)  ,
\end{align*}
where $c^*>1$ is a universal constant depending on $\tilde c_1, \tilde c_2$. 
Define
$$n(\varepsilon')=\max\{n, 4\ell_n \geq \varepsilon'\}.$$
By \cite{li1999approximation}, $\ell_n$ is a decreasing function of $n$. Let 
$$B=c^*  c^D K_1 D^{D+1},$$ 
then, $n(\varepsilon')$ can be bounded above by the solution of $B  t^{-\varrho/2} n_*^{-1/2}(\log n_*+1)=\varepsilon'$.
  Note that since $n_* \geq 1$,
$$n_* = B^2  t^{-\varrho}(\varepsilon')^{-2} (\log n_*+1)^2 \leq 4 B^2  t^{-\varrho}(\varepsilon')^{-2} (n_*)^{1/2}.$$
Hence, we have a crude upper bound $n_* \leq 16 B^4 t^{-2\varrho}(\varepsilon')^{-4}$. If we use $(\log n_*+1)^2 \leq 12 n_*^{1/4}$  for $n_* \geq 1$ and the crude upper bound, then we have a refined upper bound $n_* \leq 24 B^3  t^{-3\varrho/2}(\varepsilon')^{-3}$. When $t <\min(r_0^2, 1)$ and $\varepsilon' < 1$, 
$$1 \leq n(\varepsilon') \leq n_* \leq 24 B^3  t^{-3\varrho/2}(\varepsilon')^{-3}.$$ 

Define $s_t = \mathbb{E} [ \| f^t\|^2_\infty | t ]^{1/2}$. 
When $t < 1$, 
by Lemma \ref{Appendix: lemma bounds on s t}, 
$s_t  \leq c_3 D t^{-1/2}$,
where $c_3 > 1$ is a constant depending $a_1$ in Assumption \ref{assump:general-h}(i). 
 By Proposition 2.3 in \cite{li1999approximation},
 ${{\P}} [ ||f^{t}||_{\infty} \le \varepsilon' | \, t]
 \geq \frac{3}{4}(\frac{\varepsilon'}{6s_t n(\varepsilon') })^{n(\varepsilon')}
 \geq (\frac{\varepsilon'}{8s_t n(\varepsilon') })^{n(\varepsilon')}$.
 Therefore, we substitute $s_t  \leq c_3 D t^{-1/2}$ and obtain
$$\phi_0^{t}(\varepsilon') \leq n(\varepsilon') \log(\frac{8s_t n(\varepsilon') }{\varepsilon'}) \leq 8c_3 D t^{-1/2} n(\varepsilon')^2/\varepsilon',
$$
where we use that $\log x < x $ for $ x >0$. 
Substituting the upper bound of $n(\varepsilon')$, we conclude that 
when $0<t < \min\{ r_0^2, 1 \}$ and $0<\varepsilon' < 1$,
\begin{align}\label{thm:small_ball_manifold 2}
2\sqrt{2\phi_0^{t}(\varepsilon'/2)} \leq 2 \sqrt{9216 c_3 D B^6 t^{-3\varrho-1/2}(\varepsilon'/2)^{-7}}
\leq B_1 t^{-3\varrho/2-1/4}(\varepsilon')^{-7/2} =: \bar \lambda,
\end{align}
where $B_1=c_4 (c_3D)^{1/2} (c^*  c^D K_1 D^{D+1})^3$ and $c_4>1$ is a universal constant,
and also $B_1 >1$.

We are ready to revisit \eqref{thm:small_ball_manifold 01}. By \eqref{thm:small_ball_manifold 2},
\[
\log\calN(\varepsilon'/\Big(2\sqrt{2\phi_0^{t}(\varepsilon'/2)}\Big),\mathbb{H}^1_{t}(\mathcal{X}), ||\cdot||_{\infty})
\le \log\calN(\varepsilon'/ \bar \lambda,\mathbb{H}^1_{t}(\mathcal{X}), ||\cdot||_{\infty}).
\]
Since $B_1>1$, $t<1$, $\varepsilon'<1$,  we have $\bar \lambda > 1$, and then 
$\varepsilon'/\bar \lambda <1$. 
By Lemma \ref{lemma:unit_ball_covering_general}, and define $B_2: =c^DB_1$, 
\[
\log\calN(\varepsilon'/ \bar \lambda,\mathbb{H}^1_{t}(\mathcal{X}), ||\cdot||_{\infty})
\le  K_1 t^{-\varrho/2}\log(\frac{B_2}{ t^{3\varrho/2+1/4}(\varepsilon')^{9/2}})^{D+1}.
\]
Putting together, we have
\begin{align*}
\phi_0^{t}(\varepsilon') \leq & K_1 t^{-\varrho/2}\log(\frac{B_2}{ t^{3\varrho/2+1/4}(\varepsilon')^{9/2}})^{D+1}+\log 2\\
\leq &  K_1 (3 \varrho+9)^{D+1} t^{-\varrho/2} \log(\frac{B_2^{1/(3\varrho+9)}}{ \sqrt{t}\varepsilon'})^{D+1}+\log 2.
\end{align*}
Note that $B_2=c^D c_4 (c_3D)^{1/2}\big (c^* c^D  (C_{\calX} \vee 1) (\frac{24\sqrt{D}}{  \delta_h \wedge 1})^D D^{D+1}\big)^3 \leq (K' D)^{9D/2+4}$, where $K'>4$ depends on $C_{\calX}$ in Assumption \ref{assump:A1-A2-prime}(A1) about $\calX$  
and $\delta_h$, $c_h$ and $a_1$ in Assumption \ref{assump:general-h} about $h$. 
\rev{To make this choice explicit, we may take
\[
K'=\max\left\{5,\ c_4\sqrt{c_3}(c^*)^3(C_{\calX}\vee1)^3,
\ c^4\left(\frac{24}{\delta_h\wedge1}\right)^3\right\}.
\]
The displayed expression for $B_2$ then gives
\[
B_2\le (K')^{D+1}D^{9D/2+7/2}
\le (K'D)^{9D/2+4}.
\]
Here $c=32c_h$, $c_3$ depends only on $a_1$ by
Lemma \ref{Appendix: lemma bounds on s t}, and $c_4,c^*$ are universal.
Thus this choice of $K'$ has no domain dependence beyond $C_{\calX}$
and is common to all domains with the same covering and kernel constants.}
Therefore, $B_2^{1/(3\varrho+9)} \leq (K'D)^{D}$. Hence, when $t<\min \{ r^2_0, 1 \}$ and $\varepsilon'<1$,  
$$\phi_0^{t}(\varepsilon') \leq K_1 (3\varrho+9)^{D+1} t^{-\varrho/2} \log(\frac{(K'D)^{D}}{ \sqrt{t}\varepsilon'})^{D+1}+\log2.$$

Observe that $ K_1 (3\varrho+9)^{D+1} t^{-\varrho/2} \log(\frac{(K'D)^{D}}{ \sqrt{t}\varepsilon'})^{D+1} \geq 1>\log 2$. Therefore,
 $$\phi_0^{t}(\varepsilon') \leq 2K_1 (3\varrho+9)^{D+1} t^{-\varrho/2} \log(\frac{(K'D)^{D}}{ \sqrt{t}\varepsilon'})^{D+1}.
 $$
Finally, similar to the proof of Lemma \ref{lemma:unit_ball_covering_general},  when in addition $\varepsilon< {1}/{2}$, since $(K'D)^{D}>4$,
\begin{align*}
\log(\frac{(K'D)^{D}}{\sqrt{t}\varepsilon'})=\log((K'D)^{D})+\log(\frac{1}{ \sqrt{t}\varepsilon'}) \leq& 3\log((K'D)^{D})\log(\frac{1}{ \sqrt{t}\varepsilon'}) \\
=& 3D\log(K'D)\log(\frac{1}{ \sqrt{t}\varepsilon'}).
\end{align*}
In conclusion, when $t<\min \{ r^2_0, 1 \}$ and $\varepsilon< {1}/{2}$,
\begin{align*}
\phi_0^{t}(\varepsilon') \leq & 2K_1 (3\varrho+9)^{D+1}(3D)^{D+1}\log(K'D)^{D+1}  t^{-\varrho/2}  \log(\frac{1}{ \sqrt{t}\varepsilon'})^{D+1} \\
=&2  (C_{\calX} \vee 1) (\frac{24\sqrt{D}}{  \delta_h \wedge 1})^D  (3\varrho+9)^{D+1}(3D)^{D+1}\log(K'D)^{D+1}  t^{-\varrho/2}  \log(\frac{1}{ \sqrt{t}\varepsilon'})^{D+1}.
\end{align*}
This proves the lemma with $C=2 (C_{\calX} \vee 1) (\frac{24\sqrt{D}}{  \delta_h \wedge 1})^D  (3\varrho+9)^{D+1}(3D)^{D+1}\log(K'D)^{D+1}$ and $ C>1$. 
\end{proof}

\begin{lemma}\label{Appendix: lemma bounds on s t}
Suppose  $\mathcal{X} \subset [0,1]^D$ satisfies Assumption \ref{assump:A1-A2-prime}(A1). For $t \leq 1$, under Assumption \ref{assump:general-h} about the kernel $h$, consider the Gaussian process $f^{t}$ on $ \mathcal{X} $ associated with $h_t$.  Let $s_t = \mathbb{E} [ \| f^t\|^2_\infty | t ]^{1/2}$. Then, $s_t \leq c D t^{-1/2}$, where $c > 1$  is a constant depending $a_1$ Assumption \ref{assump:general-h}(i) on $h$.
\end{lemma}

The proof of Lemma \ref{Appendix: lemma bounds on s t}
uses similar techniques as Lemma 5 in \cite{castillo2025deep}, and we include the detailed proof for completeness.
Given a general  topological space $S$, 
denote by $W_y$, $y \in S$, a GP on $S$. 
Let $\| W \|_\infty := \sup_{y \in S} |W_y |$. Recall the Borell-TIS and Dudley's inequality for GP, 
see e.g. Chapter 2 of \cite{gine2021mathematical}.
\begin{lemma}[Borell–TIS, Theorem 2.5.8 \cite{gine2021mathematical}]\label{Borell's inequality}
Suppose $W_y$, $y\in S$, 
is a centered GP where $\Pr[ \|W\|_\infty<\infty ] > 0$. 
Let $\sigma^2= \sup_{y \in S} \mathbb{E} W_y^2 $. Then, $\forall u \ge 0$,
\[
\Pr [  \|W\|_\infty - \E \|W\|_\infty \ge u ] \le e^{-u^2/2\sigma^2}, \quad
\Pr [  \|W\|_\infty - \E \|W\|_\infty \le -u ] \le e^{-u^2/2\sigma^2}.
\]
\end{lemma}

\begin{lemma}[Dudley, Theorem 2.3.7 \cite{gine2021mathematical}]\label{Dudley's inequality}
Define a metric $d(y,y')^2 :=\mathbb{E} | W_y -W_{y'} |^2 $ on $S$,
and let $2\sigma_0 = \sup_{y,y'}d(y,y')$.
Then, for any $y_0 \in S$,
$$
\mathbb{E}\|W\|_\infty \leq  \mathbb{E} |W_{y_0} |
+4 \sqrt{2} \int_{0}^{\sigma_0} \sqrt{2\log \calN(\varepsilon, S, d) } d\varepsilon.
$$
\end{lemma}

\begin{proof}[Proof of Lemma \ref{Appendix: lemma bounds on s t}]
Denote by $\mathbb{V}$ the variance of a random variable.
By that 
$$
s_t^2=\mathbb{E} [ \| f^t\|_\infty | t ]^2+ \mathbb{V} [ \| f^t\|_\infty | t ],
$$
we bound each term in the above expression. 
Since $\mathbb{E}[f^t(x)^2|t]=h_t(x,x) = h(0)=1$ for any $x \in \calX$,  we can apply
Lemma \ref{Borell's inequality} with $\sigma^2 =1$ to obtain
$$
\Pr\{|\|f^t\|_\infty-\mathbb{E}\|f^t\|_\infty| \ge u|t\} \leq 2e^{-u^2/2}.
$$
Hence,
$$
\mathbb{V} [ \| f^t\|_\infty | t ] 
= \int_{0}^\infty 2u \Pr\{|\|f^t\|_\infty-\mathbb{E}\|f^t\|_\infty| \ge u|t\} du 
\leq \int_{0}^\infty 4u e^{-u^2/2}du =4.
$$

To bound $\E  [ \| f^t\|_\infty | t ]$, we will apply Lemma \ref{Dudley's inequality}
by that $W = f^t$ is a centered GP on $\calX$.
Let $y_0 $ be an arbitrary point in $\calX$, $\E |W_{y_0}| \le (\E W_{y_0}^2 )^{1/2} = h_t(y_0, y_0)^{1/2} = \sqrt{h(0)} = 1$.
Consider the metric $d$ on $\mathcal{X}$ defined as 
$$d(x, x')^2
: =\mathbb{E}[|f^t(x)-f^t(x')|^2|t]
= 2 h(0)-2h( {\|x-x'\|^2_{\mathbb{R}^D}}/{t}).
$$
By Mean Value Theorem and $|h'(r)| \le a_1 e^{-ar}$ as in Assumption \ref{assump:general-h}(i),
$d(x, x')^2 \leq 2 a_1\|x-x'\|^2_{\mathbb{R}^D} / t$. Then,
\begin{align*}
\calN(\varepsilon, \calX, d) 
& \leq \calN(\sqrt{ \frac{t}{2a_1}}\varepsilon, \calX, \|\cdot\|_{\mathbb{R}^D}) 
\leq \calN(\sqrt{ \frac{t}{2a_1} }\varepsilon, [0,1]^D, \|\cdot\|_{\mathbb{R}^D}) \\
& \leq \calN(\sqrt{ \frac{t}{2 a_1 D}}\varepsilon, [0,1]^D, \|\cdot\|_{\infty})
\leq (\sqrt{\frac{ 2a_1D}{t}}\varepsilon^{-1})^{D}.
\end{align*}
Moreover, for  $x, x' \in \mathcal{X} \subset [0,1]^D$, $d(x, x') \leq \sqrt{\frac{2 a_1 D}{t}} = : \alpha$. 
Hence, $ 2 \sigma_0 \le \alpha$ and we have
\begin{align*}
\int_{0}^{\sigma_0} \sqrt{2\log \calN(\varepsilon, \calX, d)  } d\varepsilon
\le &\int_{0}^{ \alpha } \sqrt{2\log \calN(\varepsilon, \calX, d) } d\varepsilon 
\leq \sqrt{2D}  \int_{0}^{ \alpha } \sqrt{\log( \alpha \varepsilon^{-1})} d\varepsilon \\
=& \sqrt{2D} \frac{\sqrt{\pi}}{2} \alpha
 =   D \sqrt{\pi} \sqrt{\frac{a_1}{t}}, 
\end{align*}
where we used that $\int_0^1 \sqrt{ \log \frac{1}{u}} du = \sqrt{\pi}/2$.
Lemma \ref{Dudley's inequality} then gives that 
$$
\mathbb{E}[ \|f^t\|_\infty| t]^2
 \leq \left( 1 +4 \sqrt{2} D \sqrt{\pi} \sqrt{\frac{a_1}{t}} \right)^2 
 \le 2(1+ 128 a_1 \frac{ D^2}{t}).
$$

Putting together, we have 
$
s_t^2 \le 2(1+ 128 a_1 \frac{ D^2}{t}) + 4 \le (6 + 256 a_1) \frac{D^2}{t}$,
because $ t \le 1$ and $D \ge 1$. 
\end{proof}

\subsection{Differential geometry lemmas}\label{app:diff-geo-lemmas}

\subsubsection{Riemannian geometry concepts}

Given the metric tensor $g$, 
the manifold (Riemannian) distance between two points $x$ and $y$ on $\calM$ is the infimum of the lengths of all piece-wise regular curves on $\calM$ connecting $x$ and $y$.
Because $\calM$ is connected and compact, for any two points $x,y$ on $\calM$ there exists a length-minimizing geodesic joining from $x$ to $y$, and the length of the geodesic is equal to the manifold distance. 
We call this distance the {\it geodesic distance} and denote it by $d_{\calM}(x,y)$.
In this case, $(\calM, d_{\calM})$ is a complete metric space.
Meanwhile,  $g$ induces a (local) Riemannian volume form on $\mathcal{M}$, denoted by $dV$, and $(\calM, dV)$ is a measure space. 
We call $Vol(\calM) := \int_\calM dV$ the volume of $\calM$.

We consider the geodesic ball centered at a point $x \in \calM$ that is diffeomorphic to a Euclidean ball in $\R^d$. 
This is characterized by the {\it exponential map} at $x$ when the radius of the ball is less than the {\it injectivity radius} of $\mathcal{M}$,
denoted as $\xi > 0$. 
 $\xi = \min_{x \in \calM} inj(x)$, where $inj(x)$ is the injectivity radius at point $x$. 
   For any $x \in \mathcal{M}$, let $T_x \mathcal{M} \cong \mathbb{R}^d$ denote the tangent space of $\mathcal{M}$ at $x$. 
Let $B_r(x) \subset \calM$ denote the open geodesic ball of radius $r$ centered at $x$, and $B^{\R^d}_r(u)$ the open ball in $\R^d$ of radius $r$ and center $u$. 
Then 
$$
\exp_x: B^{\mathbb{R}^d}_{\xi}(0) \subset  \R^d \cong T_x \mathcal{M} \rightarrow B_{\xi}(x) \subset \calM
$$
is a diffeomorphism, and the corresponding coordinates are called {\it normal coordinates} at $x$. 

The normal coordinate is closely related to the geodesic curve.
For any $y \in B_\xi(x)$, there is a unique (constant speed) minimizing geodesic $\gamma$ from $x$ to $y$, satisfying $\gamma(0) = x$ and $\gamma(t) = y$ where $t = d_\calM(x,y) < \xi$. This curve can be extended to be defined on $t \in (-\xi, \xi)$, and along the curve, $\dot \gamma(t) \in T_{\gamma(t)} \calM$ and $\| \dot \gamma(t) \| \equiv 1$ where the norm is induced by $g$ at $\gamma(t)$.
At $t=0$, $\dot \gamma(0) = v \in S^{d-1} \subset T_x\calM$, and $\gamma$ can be expressed as $\gamma(t) = \exp_x( tv)$.

The $\ell$-th covariant derivative of $f$ is an order-$\ell$ tensor field on $\calM$.
We introduce the notations for vector fields and tensor fields on $\calM$.
We say that $U$ is a vector field on $\calM$ if $U(x) \in T_x \calM$ for every $x\in \calM$,
 and usually we consider $U$ having the same order of differentiability as $\calM$, e.g., when $\calM$ is smooth then $U$ is also smooth.
We can also consider a vector field on a neighborhood on $\calM$.
 For an order-$r$ tensor field $T^{(r)}$ on $\calM$ (or a neighborhood of $\calM$),
 the evaluation $T^{(r)}$ at a point $x$ gives a tensor $T^{(r)}(x): \underbrace{T_x \calM \times \cdots T_x \calM}_{\text{$r$ many}} \to \R$.
For vector fields $U_1, \cdots, U_r$ on $\calM$, 
we use the notation $T^{(r)}(U_1 ,\cdots, U_r)|_x = T^{(r)}(x)( U_1(x), \cdots, U_r(x))$.

The first covariant derivative $\nabla f$ is an order-1 tensor field and can be specified by ``directional derivative'':
at any $x\in \calM$ and  $\forall v \in T_x \calM$, 
we denote $\nabla f (x)(v)$ also as $ \nabla_v f (x) := \frac{d}{dt} f( \gamma (t))|_{t=0}$, 
where $\gamma(t)$ is a differentiable curve on $\calM $ s.t. $\gamma (0)= x$ and $\gamma'(0) = v$. 
By definition, for a vector field $U$, $\nabla f(U)|_x = \nabla f(x) (U(x))= \nabla_{U(x)} f (x)$.
We also write $\nabla f(U)$ as $\nabla_U f$.
Given $\ell$-th derivative of $f$, the $\ell+1$-derivative is defined as 
\begin{align*}
& \nabla^{\ell+1} f(V, U_1,\dots, U_\ell) 
 = (\nabla_V \nabla^\ell f) (U_1, \dots, U_\ell)  \\
&~~~~~~ 
= \nabla_V [  \nabla^\ell f (U_1, \dots, U_\ell) ]
 	 - \nabla^\ell f ( \nabla_V U_1, \cdots, U_\ell) 
	 - \cdots
	 - \nabla^\ell f (U_1, \cdots,  \nabla_V U_\ell), 
\end{align*}
where $V$ and  $U_1,\dots, U_\ell$  are arbitrary vector fields on $\calM$.
For a general differentiable tensor field $T^{(r)}$ and vector fields $V$, $U_1,\dots, U_r$, we have
\begin{align*}
& \nabla T^{(r)}( V, U_1, \cdots, U_r) 
=  ( \nabla_V T^{(r)} )(U_1, \cdots, U_r)  \\   
&~~~~~~ 
 = \nabla_V [ T^{(r)} (U_1, \cdots, U_r)  ] 
	- T^{(r)} (\nabla_V U_1, \cdots, U_r) 
	- \cdots 
	-  T^{(r)} ( U_1, \cdots, \nabla_V U_r).
\end{align*}
The definition of $\nabla^{\ell+1} f$ is same as the covariant derivative of the tensor field $\nabla^\ell f$.

\subsubsection{Local expansions of volume form and Euclidean distance }

We recall the notations as in Section \ref{sec:rem_manifold}.
In particular, $\calM $ is a $d$-dimensional connected close (compact and without boundary) Riemannian manifold
 isometrically embedded in $\R^D$ through $\iota: \calM \to \R^D$.
We say $\calM$ is $C^r$ with $r \ge 1$ integer when both the Riemannian metric $g$ and the embedding map $\iota$ are at least $C^r$.
The injectivity radius of $\calM$ is $\xi $, and the reach of $\iota(\calM)$ is $\tau $.

We first introduce a lemma on the comparison of Euclidean distance and manifold geodesic distance.

\begin{lemma}[Proposition 6.3 in \cite{niyogi2008finding}]\label{lemma:manifold-reach}
    Suppose $\calM $ is a $C^2$ manifold isometrically embedded in $\R^D$ with reach $\tau > 0$, 
    then for any two points $x, y \in \calM$ with $\| \iota(x) - \iota(y)\|< \tau/2$,
    we have $ d_\calM(x,y) \le 2 \| \iota(x) - \iota(y)\|$.
\end{lemma}
\begin{proof}[Proof of Lemma \ref{lemma:manifold-reach}]
It was shown in \cite[Proposition 6.3]{niyogi2008finding} that for any two points $x, y \in \calM$ with $ \| \iota(x) - \iota(y)\|< \tau/2$, 
we always have $d_\calM(x,y) \le \tau - \tau \sqrt{1 - 2 \| \iota(x) - \iota(y)\| /\tau}$.
The statement of the lemma then follows  from 
that $1- \sqrt{1-x} \le x$ for all $x \in [0,1]$.
\end{proof}

The local expansions of volume form and Euclidean distance at a point $x \in \calM$ have been derived in the literature, see e.g. \cite{gray1974volume,wu2018think}. 
We recall these expansions in the first few terms: below, $(t,\theta)$ is the polar coordinate in $T_x \calM \cong \R^d$.
We have
\begin{align*}
dV(\exp_x(t \theta ))=\Big(&\,1-\frac{1}{6}\texttt{Ric}_{x}(\theta,\theta)t^2-\frac{1}{12}\nabla_\theta \texttt{Ric}_{x}(\theta,\theta) t^3 \\
&\,-\big(\frac{1}{40}\nabla^2_{\theta}\texttt{Ric}_{x}(\theta,\theta)
+\frac{1}{180}\sum_{a,b=1}^d
\texttt{R}_x(\theta, E_a,\theta, E_b)^2
- \frac{1}{72}\texttt{Ric}_{x}(\theta,\theta)^2\big) t^{4}\nonumber\\
&\,+ O(t^5) \Big) t^{d-1}dt d\theta,\nonumber
\end{align*}
where  $\texttt{R}_x$ and $\texttt{Ric}_{x}$ are the curvature tensor and the Ricci curvature tensor of $\mathcal{M}$ at $x$ respectively, and $\{E_a\}_{a=1}^d$ is an orthonormal basis of $T_x \calM$; 
{{The
squared Euclidean-distance expansion is}}
\begin{align*}
& \|\iota \circ \exp_x( t \theta )-\iota(x)\|^2_{\mathbb{R}^D} 
\,=  t^2-\frac{1}{12} \|\Second_x(\theta,\theta)\|^2 t^4
-\frac{1}{12}
\rev{\bar\nabla_{\theta}\Second_x(\theta,\theta)}
\cdot  \Second_x(\theta,\theta) t^5 \\
&~~~~~~~~~~~~~~~~~~~~~~~~
- \big(\frac{1}{40}
\rev{\bar\nabla^2_{\theta}\Second_x(\theta,\theta)}
\cdot \Second_x(\theta,\theta)
+ \frac{1}{45}
\rev{  \| \bar\nabla_{\theta}\Second_x(\theta,\theta) \|_{\R^D}^2  } 
\big)t^6+O(t^7),  
\end{align*}
where $\Second_x$ is the second fundamental form of $\iota(\mathcal{M})$ at $\iota(x)$. 
\rev{Here $\bar\nabla$ denotes the Euclidean connection on the trivial
bundle $\calM\times\mathbb R^D$ and, by the same notation, its induced
connection on $\mathbb R^D$-valued covariant tensor fields, using the
Levi--Civita connection in the tangent input slots; see, e.g.,
\cite{do1992riemannian}.}

These expansions, however, are not enough for our purpose because in this work we will need to expand to arbitrary high order
 depending on $k$ the target function differentiability order.
Meanwhile, we do not use the specific expressions of the expansion and can treat each $O(t^\ell)$ term as abstract. 
An important observation is that in the above expansions, 
the terms that involve $\theta$ and multiplied to $t^\ell$ are tensor fields evaluated at $x$ and $(\theta, \cdots, \theta) \in T_x \calM \times \cdots \times T_x \calM$, 
and that is why the $t^\ell$ factor can be extracted thanks to the linearity of tensor fields.
For our analysis, we will need to show that this pattern holds for expansion to arbitrarily high orders;
we also need to show that the constant in big-O is uniform for $x$ and reveal the  constant dependence. 
These needed results are summarized in the following lemma, for which we include a proof for completeness.

\begin{lemma}\label{expansion-of-volume-of-form-and-Euclidean-distance}
For any $x \in \calM$, we consider the normal coordinates at $x$ provided by $\exp_x:  T_x\calM  \cong  \R^d  \to \calM$, and on $\R^d$ we use the polar coordinates $(t,\theta)$. 

\begin{itemize}

\item[(i)] Local expansion of volume form.

\begin{itemize}
\item[a)] Suppose $\calM$ is $C^2$, then $\forall  0 \le t < \xi$, $\theta \in S^{d-1} \subset T_x \calM$, 
\[
dV( \exp_x (t \theta)) = (1 + R_V(t) ) t^{d-1} dt d\theta, \quad |R_V(t)| \le C_{V,1} t^2,
\]
where the constant  $C_{V,1}$ depends on $d$ and the uniform bounds of up to the 2nd intrinsic derivative of the Riemannian metric $g$, and $C_{V,1}$ is uniform for $x \in \calM$. 

\item[b)]
 Given $\calK \ge 2$, suppose $\calM$ is $C^{ \calK+1}$. 
For each $\ell = 2, \cdots, \calK$,  there exist an order-$\ell $ tensor field $\bar V_\ell$ on $\calM$ 
such that, after defining $V_\ell( x, v) := \bar V_\ell (x)( v, \cdots, v)$ for any $v \in T_x \calM$,
we have
$\forall  0 \le t < \xi$, $\theta \in S^{d-1} \subset T_x \calM$, 
\begin{align*}
dV(\exp_x( t \theta ))= (1 + \sum_{\ell=2}^{\mathcal{K}} V_\ell (x, \theta) t^\ell+R_{V, \mathcal{K}}(t))t^{d-1}dtd\theta,  
\quad | R_{V, \mathcal{K}}(t)| \leq C_{V, \mathcal{K}} t^{\mathcal{K}+1},
\end{align*}
where the constant $C_{V, \mathcal{K}}$ depends on $d$ and the uniform bounds of the intrinsic derivatives of $g$ up to $(\mathcal{K}+1)$-th order,
and $C_{V, \mathcal{K}}$ is uniform for $x \in \calM$. 
In addition, 
$\bar V_\ell$ can be expressed through the products and sums of the curvature tensor $\texttt{R}$ of $\calM$ and its covariant derivatives up to ($\ell-2$)-th order, including a contraction of the tensors,
with coefficients depending on $d$ and $\ell$.
\end{itemize}

\item[(ii)] Local expansion of squared Euclidean distance.

\begin{itemize}
\item[a)] Suppose $\calM$ is $C^3$, then
$\forall  0 \le t < \min \{ 1, \xi \}$, $\theta \in S^{d-1} \subset T_x \calM$, 
\[
\|\iota \circ \exp_x(t  \theta  )-\iota(x)\|^2_{\mathbb{R}^D}
=t^2 + R_q(t), 
\quad |R_q(t)| \le c_{q,3} t^4,
\]
where the constant $c_{q,3}$ depends on 
the $\|\cdot\|_\infty$ norm
of  the second fundamental form $\Second$ of $\iota(\calM)$ and its first covariant derivative, 
and $c_{q,3}$ is uniform for $x \in \calM$.

\item[b)] Given $\calJ \ge 4$, suppose $\calM$ is $C^{ \calJ}$.
 For each $j = 4, \cdots, \calJ$,  there exist an order-$j$ tensor field $\bar q_j$ on $\calM$  
 such that, after defining $q_j( x, v) := \bar q_j (x)( v, \cdots, v)$ for any $v \in T_x \calM$, 
 we have $\forall  0 \le t <  \min\{ 1, \xi \}$, $\theta \in S^{d-1} \subset T_x \calM$, 
 \begin{align*}
\|\iota \circ \exp_x( t  \theta )-\iota(x)\|^2_{\mathbb{R}^D}=t^2+ \sum_{j =4}^{\mathcal{J}} q_j (x, \theta)t^j+R_{q,\mathcal{J}}(t),
\quad 
|R_{q,\mathcal{J}}(t)| \leq c_{q, \mathcal{J}} t^{\mathcal{J}+1},
\end{align*}
where the constant $c_{q, \mathcal{J}}$ depends on 
the  $\|\cdot\|_\infty$ norm of the up to ($\calJ-2$)-th covariant derivatives of $\Second$,
and $c_{q, \mathcal{J}}$ is uniform for $x \in \calM$. 
In addition, $\bar q_j$ can be expressed through
\rev{sums, tensor products, and contractions}
of the second fundamental form  $\Second$
of $\iota(\mathcal{M})$ and its covariant derivatives up to ($j-4$)-th order with coefficients depending on $j$.
\end{itemize}
\end{itemize}
\end{lemma}

\begin{proof}[Proof of Lemma \ref{expansion-of-volume-of-form-and-Euclidean-distance}]

(i) To analyze the volume form, we will need to consider the determinant of the Riemannian metric tensor $g$ represented as a $d$-by-$d$ matrix $[g] = [ g_{ij} ]_{ij}$ under a local coordinates. 
We use normal coordinates at $x \in \calM$ provided by $\exp_x$. 
Let $\{E_i\}_{i=1}^d$ be an orthonormal basis of $T_x \calM$. 
We construct $g_{ij}$ for each $i,j$ as a function on $B_\xi(x)$ by
\[
g_{ij} := g( X_i, X_j ), \quad g_{ij}: B_\xi(x) \to \R, \quad i,j = 1,\cdots, d,
\]
where $\{ X_i \}_{i=1}^d$ is a frame on $B_\xi(x)$ induced by $d \exp_x$ (from the frame on $T_x \calM \cong \R^d$)
satisfying that $X_i( x) = E_i$.

Note that because $\{ X_i\}_i$ is orthonormal at $x$,  we have $g_{ij}(x)=\delta_{ij}$ and  $\nabla g_{ij}(x)=0$. 
We defne $\tilde g_{ij} : = g_{ij} \circ \exp_x$ 
and identify vectors in $\R^d$ with those in $T_x \calM$ (using the basis $\{ E_ i \}_i$).
\rev{For the finite-regularity assumptions in parts (i), we only
use the one-dimensional radial expansion of the Riemannian volume
density $\sqrt{\det[\tilde g(t\theta)]}$.  The required expansion and
remainder bounds follow from the standard Jacobi-field expansion of
the volume density (equivalently, the normal-coordinate volume
expansion).
Since the determinant is invariant under orthogonal
transformation, $\det[\tilde g]$, viewed as a scalar function on
$T_x\calM$, is independent of the choice of the orthonormal basis
$\{E_i\}_{i=1}^d$ at $x$, even though the matrix
$[\tilde g_{ij}]_{ij}$ depends on this choice.}

Recall that $dV( \exp_x( v )) = \sqrt{  \det [ g( \exp_x( v ))  ] } dv$ for any $v \in B_\xi^{\R^d}(0)$.
Then, for any 
$\theta \in S^{d-1} \subset T_x \calM$
and $ 0 \le t  <\xi$, we have
\[
dV( \exp_x( t \theta ))
= \sqrt{\det [ \tilde{g} ( {\theta} t) ]} t^{d-1}dtd\theta.  
\]
When $\calM$ is $C^{\calK+1}$, 
we apply a one-dimensional Taylor expansion of $\sqrt{\det [ \tilde{g} ( {\theta} t) ]}$ at $t=0$.
Recall that 
$\sqrt{\det  [\tilde g   (0) ]}=1$ and 
$\frac{d}{ds}\sqrt{\det [\tilde{g}  ({\theta}  s)] } \big |_{s=0}=0$. Then,
\begin{equation}\label{eq:vol-form-expand-lemma-proof-1}
\sqrt{\det [ \tilde{g} ( {\theta} t) ]}
 =1+\sum_{\ell=2}^{\mathcal{K}}
 	\frac{1}{\ell!}\frac{d^\ell}{ds^\ell} \sqrt{\det [\tilde{g}({\theta} s)] } \big |_{s=0} 
 	 t^\ell+R_{V, \mathcal{K}}(t).
\end{equation}

\

a) When $\calK = 1$, the expansion \eqref{eq:vol-form-expand-lemma-proof-1} is reduced to 
$\sqrt{\det [\tilde{g}(\theta t)] }=1+R_V(t)$, 
where $ |R_V(t)| \le  \frac{1}{2} \Big|\frac{d^2}{ds^2} \sqrt{ \det [\tilde{g} (\theta s) ] } |_{s=t'}\Big| t^{2}$ for $0 \leq t' \leq t<\xi$. Thus, by the definition of determinant and the Product Rule, $|R_V(t)| \leq C_{V,1}t^2$ with $C_{V,1}$ depending on $\frac{d^\ell}{ds^\ell}\tilde{g}_{ij}(\theta s)|_{s=t'}$ for $\ell =0,1, 2$ and all $i,j$.
 Hence, $C_{V,1}$ depends on $d$ and the uniform bounds of up to the 2nd intrinsic derivative of the Riemannian metric $g$.

\

b) The remainder $|R_{V, \mathcal{K}}(t)| \leq \frac{1}{(\mathcal{K}+1)!} \Big|\frac{d^{\mathcal{K}+1}}{ds^{\mathcal{K}+1}}
	\sqrt{ \det [\tilde{g} (\theta s) ] } 
	|_{s=t'}\Big| t^{\mathcal{K}+1}$ for $0 \leq t' \leq t<\xi$,
and then, similarly as in a),
 $|R_{V, \mathcal{K}}(t)| \leq C_{V, \mathcal{K}} t^{\mathcal{K}+1}$
with $C_{V, \mathcal{K}}$ depending on $d$ and the intrinsic derivatives of $g$ up to ($\mathcal{K}+1$)-th order. 
It remains to show that
\begin{equation}\label{eq:compute-derivative-detg-proof-goal}
\frac{1}{\ell!}\frac{d^\ell}{ds^\ell} \sqrt{\det [\tilde{g}(\theta s)] } \big |_{s=0} 
= \bar{V}_\ell(x)(\theta, \cdots, \theta),
\quad \ell = 2, \cdots, \calK,
\end{equation}
 where $\bar V_\ell$ is a tensor field on $\calM$ as described in the statement of b).
 
To compute the l.f.s. of \eqref{eq:compute-derivative-detg-proof-goal}, let $f(t)=\sqrt{t}$, then the higher-order Chain Rule gives that
\begin{align*}
\frac{d^\ell}{ds^\ell}\sqrt{ \det[ \tilde{g} (\theta s) ] }|_{s=0}
&= \sum_{i_1+2i_2+\cdots+\ell i_\ell=\ell}
{{\frac{\ell !}{ i_1 ! \cdots i_\ell !}}}
f^{(i_1+i_2+\cdots+i_\ell)} ( \det [\tilde{g} (0)] )\\
&~~~~~~~~~~~~~~~~~~	
	\prod_{m=1}^\ell \Big(
        \frac{1}{m!}\frac{d^m}{ds^m}\det[ \tilde{g} (\theta s) ]|_{s=0}
    \Big)^{i_m},
\end{align*}
and we always use the convention $u^0 = 1$ even when $u=0$.
Since $\det [\tilde{g} (0)] =1$, we define 
$b(i_1, \cdots, i_\ell):=f^{(i_1+i_2+\cdots+i_\ell)}( \det [\tilde{g} (0)]  )=f^{(i_1+i_2+\cdots+i_\ell)}(1)$. Then, 
\begin{align}
\frac{1}{\ell!}\frac{d^\ell}{ds^\ell}\sqrt{\det[ \tilde{g} (\theta s) ]} & |_{s=0} 
 = \sum_{i_1+2i_2+\cdots+\ell i_\ell=\ell}\frac{b(i_1, \cdots, i_\ell)}{i_1! (1!)^{i_1} \cdots i_\ell! (\ell!)^{i_{\ell}}}
	\prod_{m=1}^\ell \Big( \frac{d^m}{ds^m}\det[ \tilde{g} (\theta s) ]|_{s=0}\Big)^{i_m} \nonumber \\
& = \sum_{i_1+2i_2+\cdots+\ell i_\ell=\ell}\frac{b(i_1, \cdots, i_\ell)}{i_1! (1!)^{i_1} \cdots i_\ell! (\ell!)^{i_{\ell}}}
	\prod_{m=1}^\ell \Big( 
			D^m \det[ \tilde{g}](0)( \theta, \cdots, \theta)
				 \Big)^{i_m}, \label{eq:dlds-detg-proof-2}
\end{align}
where in the second equality we used that 
$
\frac{d^m}{ds^m} \det[ \tilde{g}  ] (\theta s) |_{s=0}
= D^m \det[ \tilde{g}](0)( \theta, \cdots, \theta).
$

\rev{By \cite[Corollary 2.11]{gray1974volume}, for
$m\ge2$ there is an order-$m$ tensor field $H_m$ which can be
expressed through sums, tensor products, and contractions of the
curvature tensor $\texttt{R}$ and its covariant derivatives up to
order $m-2$, such that
\[
H_m(x)=\frac{1}{m!}D^m\det[\tilde g](0).
\]
For example, $H_2(x)=-\frac13\texttt{Ric}(x)$.  We also
define $H_1=0$.  Therefore
\[
D^m\det[\tilde g](0)(\theta,\ldots,\theta)
=m!\,H_m(x)(\theta,\ldots,\theta).
\]
Substituting this identity into
\eqref{eq:dlds-detg-proof-2}, the factors $(m!)^{i_m}$ cancel and,
for $\ell\ge2$,
\begin{align*}
\frac{1}{\ell!}\frac{d^\ell}{ds^\ell}
\sqrt{\det[\tilde g(\theta s)]}\Big|_{s=0}
=
\sum_{i_1+2i_2+\cdots+\ell i_\ell=\ell}
\frac{b(i_1,\ldots,i_\ell)}{i_1!\cdots i_\ell!}
\prod_{m=1}^{\ell}
\Big(H_m(x)(\theta,\ldots,\theta)\Big)^{i_m}.
\end{align*}
Thus we define the order-$\ell$ tensor field
\[
\bar V_\ell
=
\sum_{i_1+2i_2+\cdots+\ell i_\ell=\ell}
\frac{b(i_1,\ldots,i_\ell)}{i_1!\cdots i_\ell!}
\bigotimes_{m=1}^{\ell}H_m^{\otimes i_m},
\qquad \ell\ge2.
\]
Since we only use diagonal evaluations, no symmetrization is
needed: the diagonal evaluation of
$\bigotimes_{m=1}^{\ell}H_m^{\otimes i_m}$ equals
$\prod_{m=1}^{\ell}(H_m(x)(\theta,\ldots,\theta))^{i_m}$.
Consequently,
\[
\frac{1}{\ell!}\frac{d^\ell}{ds^\ell}
\sqrt{\det[\tilde g(\theta s)]}\Big|_{s=0}
=
\bar V_\ell(x)(\theta,\ldots,\theta),
\]
which proves \eqref{eq:compute-derivative-detg-proof-goal}.  The
claimed curvature dependence of $\bar V_\ell$ follows directly from
that of the tensor fields $H_m$.}

\

(ii) 
In this part of proof we use $\cdot$ to denote vector inner-product in $\R^D$.
At any $x \in \calM$, for any $ \theta \in S^{d-1} \subset T_x \calM$ and $0 \leq t <\xi$, let $\gamma(t)=\exp_x(\theta t)$ be the unit speed geodesic on $\calM$, we have $\gamma(0) =x$ and $\gamma'(0) = \theta$.  
Let $\varphi(t)=\iota(\gamma(t))$ and $\varphi^{(i)}(t)$ denote its $i$-th derivative. 
Note that $\varphi^{(2)}(s)=\Second_{\gamma(s)} (\dot\gamma(s),\dot\gamma(s))$.
We claim that when $\calM$ is $C^\mathcal{J}$,
\begin{align}\label{expansion-of-Euclidean-distance eqn3}
\varphi^{(i)}(s)=\rev{\bar\nabla^{i-2}}\Second_{\gamma(s)} (\dot\gamma(s),\cdots,\dot\gamma(s)),
\quad \forall 0 \leq s<\xi,
\quad i = 2, \cdots, \calJ,
\end{align}
which can be proved by induction:
the equation in  \eqref{expansion-of-Euclidean-distance eqn3} holds at $i=2$; 
suppose it holds at $i$, at $i+1$, we have 
$
\varphi^{(i+1)}(s)=\frac{d}{ds}(\varphi^{(i)}(s))
=\rev{\bar\nabla_{\dot\gamma(s)}\bar\nabla^{i-2}}\Second_{\gamma(s)} (\dot\gamma(s),\cdots, \dot\gamma(s))
=\rev{\bar\nabla^{i-1}}\Second_{\gamma(s)} (\dot\gamma(s),\cdots, \dot\gamma(s)),
$
where we use $\nabla_{\dot\gamma} \dot\gamma=0$ along $\gamma$ in the last step.

\

a) If $\calM$ is $C^3$, then
$$
\iota \circ \exp_x( \theta t )-\iota(x)=\varphi^{(1)}(0)t+\frac{1}{2}\varphi^{(2)}(0)t^2+r_q(t),$$
where 
$\|r_q(t)\|_{\mathbb{R}^D} \leq \frac{1}{6}\max_{0 \leq s \leq t}\|\varphi^{(3)}(s)\|_{\mathbb{R}^D} t^{3}$. By \eqref{expansion-of-Euclidean-distance eqn3}, 
we have  $\|r_q(t)\|_{\mathbb{R}^D} \leq c_0 t^3$, where
\[
c_0:=
\rev{\frac{1}{6}
\sup_{\substack{x\in\calM\\v\in S_x^{d-1}}}
\|(\bar\nabla\Second)_x(v,v,v)\|_{\mathbb R^D}}.
\]
Since $\varphi^{(1)}(t) \cdot \varphi^{(1)}(t)=1$, we have $\varphi^{(1)}(t) \cdot \varphi^{(2)}(t)=0$. Therefore
\begin{align*}
\|\iota \circ \exp_x( \theta t )-\iota(x)\|^2_{\mathbb{R}^D}=t^2+\frac{1}{4}\varphi^{(2)}(0)\cdot \varphi^{(2)}(0) t^4+2 \varphi^{(1)}(0) \cdot r_q(t) t + \varphi^{(2)}(0) \cdot r_q(t) t^2
\rev{+\|r_q(t)\|_{\mathbb R^D}^2},
\end{align*}
and then
$$R_q(t)=\frac{1}{4}\varphi^{(2)}(0)\cdot \varphi^{(2)}(0) t^4+2 \varphi^{(1)}(0) \cdot r_q(t) t + \varphi^{(2)}(0)\cdot   r_q(t) t^2
\rev{+\|r_q(t)\|_{\mathbb R^D}^2}.$$
Since $t <1$ and $\|\varphi^{(1)}(0)\|_{\mathbb{R}^D}=1$, $|R_q(t)|\leq c_{q,3}t^4$ where $c_{q,3}$ depends on$\|\varphi^{(2)}(0)\|_{\mathbb{R}^D}$ and $c_0$. 
\rev{Here the Gauss-Weingarten formulas give, for every $x\in\calM$
and $v\in T_x\calM$,
$(\bar\nabla\Second)_x(v,v,v)=(\nabla_v\Second)(v,v)
-S_{\Second(v,v)}v$, where the shape operator is defined by
$S_N(X)\cdot Y=\Second(X,Y)\cdot N$.  Hence $c_{q,3}$ depends on uniform
bounds for $\Second$ and its first covariant derivative, as stated in the
lemma.}

\

b) Suppose $\calM$ is $C^{\mathcal{J}}$ with $\mathcal{J} \geq 4$, then 
$$\iota \circ \exp_x( \theta t )-\iota(x)=\sum_{i=1}^{\mathcal{J}-1}\frac{1}{i!}\varphi^{(i)}(0)t^i+r_{q,\mathcal{J} }(t),$$
where $r_{q,\mathcal{J} }(t) \in \mathbb{R}^D$ and $\|r_{q,\mathcal{J} }(t)\|_{\mathbb{R}^D} \leq \frac{1}{\mathcal{J}!} \max_{0 \leq s \leq t}\|\varphi^{(\mathcal{J})}({{s}})\|_{\mathbb{R}^D} t^{\mathcal{J}}$.
By \eqref{expansion-of-Euclidean-distance eqn3}, we have  $\|r_{q,\mathcal{J} }(t)\|_{\mathbb{R}^D} \leq c_{\mathcal{J}} t^{\mathcal{J}}$ with
\[
c_{\mathcal J}:=
\rev{\frac{1}{\mathcal J!}
\sup_{\substack{x\in\calM\\v\in S_x^{d-1}}}
\|(\bar\nabla^{\mathcal J-2}\Second)_x
(v,\ldots,v)\|_{\mathbb R^D}}.
\]
Therefore, we have
\begin{align}\label{expansion-of-Euclidean-distance eqn1}
\|\iota \circ \exp_x( \theta t )-\iota(x)\|^2_{\mathbb{R}^D}=\sum_{j=2}^{\mathcal{J}}\Big(\sum_{i+\ell=j, \hspace{1.5mm} 1\leq i, \ell}\frac{1}{i! \ell!}\varphi^{(i)}(0) \cdot \varphi^{(\ell)}(0)\Big) t^j +R_{q,\mathcal{J}}(t),
\end{align}
where
$$R_{q,\mathcal{J}}(t)=\sum_{j=\mathcal{J}+1}^{2\mathcal{J}-2}\Big(\sum_{i+\ell=j, \hspace{1.5mm} 2 \leq i, \ell \leq \mathcal{J}-1}\frac{1}{i! \ell!}\varphi^{(i)}(0) \cdot \varphi^{(\ell)}(0)\Big) t^j+ \sum_{j=1}^{\mathcal{J}-1}\frac{2}{j!} \Big( \varphi^{(j)}(0) \cdot r_{q,\mathcal{J} }(t) \Big) t^j
\rev{+\|r_{q,\mathcal J}(t)\|_{\mathbb R^D}^2}.$$
Since $t <1$ and $\|\varphi^{(1)}(0)\|_{\mathbb{R}^D}=1$, $|R_{q,\mathcal{J}}(t)| \leq c_{q,\mathcal{J}}t^{\mathcal{J}+1}$ where $c_{q,\mathcal{J}}$ depends on$\|\varphi^{(j)}(0)\|_{\mathbb{R}^D}$ for $j=2, \cdots, \mathcal{J}-1$ and $c_{\mathcal{J}}$.
\rev{Iterating the Gauss-Weingarten formulas shows that
$\bar\nabla^r\Second$ is a universal sum of tensor products and
contractions of $\Second$ and its covariant derivatives through order $r$.
Therefore, the constant $c_{q,\mathcal J}$ depends on their uniform bounds
through order $\mathcal J-2$, as stated in the lemma.}

Since $\varphi^{(1)}(t) \cdot \varphi^{(1)}(t)=1$, by applying the high order product rule, we have $\varphi^{(1)}(t) \cdot \varphi^{(2)}(t)=0$ and for $j>2$,
$$
\varphi^{(1)}(t) \cdot \varphi^{(j)}(t)=-\frac{1}{2}\sum_{i=1}^{j-2} \binom{j-1}{i}\varphi^{(1+i)}(t) \cdot \varphi^{(j-i)}(t).
$$
Hence, \eqref{expansion-of-Euclidean-distance eqn1} can be simplified to
\begin{align}\label{expansion-of-Euclidean-distance eqn2}
\|\iota \circ \exp_x( \theta t )-\iota(x)\|^2_{\mathbb{R}^D}
& =t^2+\sum_{j=4}^{\mathcal{J}}\Big(
	\sum_{i+\ell=j, \hspace{1.5mm} 2\leq i, \ell \leq j-2}\frac{1}{i! \ell!}\varphi^{(i)}(0) \cdot \varphi^{(\ell)}(0) \nonumber\\
&~~~
	-\frac{1}{(j-1)!}\sum_{i=1}^{j-3} \binom{j-2}{i}\varphi^{(1+i)}(0) \cdot \varphi^{(j-i-1)}(0)\Big) t^j +R_{q,\mathcal{J}}(t),
\end{align}
For $4 \leq j \leq \mathcal{J}$, we define the  order-$j$  tensor field as follows: for any vector fields $U_1 ,\cdots, U_j$,
\begin{align*}
\bar{q}_j(U_1, \cdots, U_j) 
&: = \sum_{i+\ell=j, \hspace{1.5mm} 2\leq i, \ell \leq j-2}\frac{1}{i! \ell!}
\rev{\bar\nabla^{i-2}}\Second(U_1,\cdots, U_i)
\cdot
\rev{\bar\nabla^{\ell-2}}\Second(U_{i+1},\cdots, U_{i+\ell}) \\
&~~~
-\frac{1}{(j-1)!}\sum_{i=1}^{j-3} \binom{j-2}{i}
\rev{\bar\nabla^{i-1}}\Second(U_1,\cdots, U_{i+1})
\cdot
\rev{\bar\nabla^{j-i-3}}\Second(U_{i+2},\cdots, U_{j}).
\end{align*}
By comparing to \eqref{expansion-of-Euclidean-distance eqn2} and using  \eqref{expansion-of-Euclidean-distance eqn3}, 
this proves the expansion in b) after we define $q_j( x, v) := \bar q_j (x)( v, \cdots, v)$ for any $v \in T_x \calM$.
\rev{Indeed, substituting
$\varphi^{(i)}(0)=(\bar\nabla^{i-2}\Second)_x
(\theta,\ldots,\theta)$ into
\eqref{expansion-of-Euclidean-distance eqn2} gives precisely
$q_j(x,\theta)=\bar q_j(x)(\theta,\ldots,\theta)$.}
\rev{Finally, the Gauss-Weingarten formulas show that $\bar q_j$ is a
universal expression formed from tensor products and contractions of
$\Second$ and its covariant derivatives through order $j-4$, with
coefficients depending on $j$.}
\end{proof}

\subsubsection{Taylor expansion of manifold H\"older function and a lemma on tensor field}

\begin{lemma}\label{lemma:taylor-f-intrinsic-Holder}
Given $f \in C^{k, \beta}(\calM)$ for non-negative integer $k$ and $0 <\beta \le 1$,
for any $x \in \calM$, any $\theta \in S^{d-1}$ and
$\rev{0\le t<\xi}$, we have
\[
f( \exp_x(t \theta)) = \sum_{i=0}^k \frac{t^i}{i!} \nabla_\theta^i f(x)  + r_f(t),
\quad |r_f(t)| \le \frac{1}{k!} L_{k, \beta}(f, x) t^{k+\beta}.
\]
\end{lemma}
\begin{remark}
In the bound of the remainder term,
the factor $L_{k, \beta}(f, x)  \le L_{k, \beta}(f)  \le \|f\|_{k,\beta}$. 
The lemma derives a Taylor expansion of $f$ 
but it differs from the Taylor expansion in normal coordinates of $f$, namely that of $ f \circ \exp_x$ as a function on $\R^d$. 
Specifically, the proof uses the fact that  along the geodesic $\gamma(t)=\exp_x(t\theta)$ from $x$,
the $t$-derivatives of the function $f \circ \gamma$
 can always be interpreted as covariant derivatives of $f$ at $\gamma(t)$ in the direction of $\dot \gamma$, 
due to that $\dot \gamma$ is parallel along $\gamma$. 
\end{remark}

\begin{proof}[Proof of Lemma \ref{lemma:taylor-f-intrinsic-Holder}]
\rev{When $k=0$, the result follows directly from the definition of
$L_{0,\beta}(f,x)$. Suppose henceforth that $k\ge1$.}

For fixed $x$, $\theta$, 
we use the geodesic $\gamma(t)$ s.t. $\gamma(0) = x$, $\dot \gamma(0)=\theta$,
and then we consider 
$F(t): = f(\exp_x( t \theta) ) = f( \gamma(t))$
as a one-dimensional function of $t$.
For any $l = 0,\cdots, k$  and $ |t | < \xi$,  we have
\[
F^{(l)}(t) = \nabla^l f( \gamma(t) ) ( \dot \gamma(t), \cdots, \dot \gamma(t) ),
\quad \dot \gamma(t) = P_{x, \gamma(t) } \theta,
\]
where $P_{x,y}: T_x \calM \to T_y \calM$ is the parallel transport.
In particular, $F^{(l)}(0) = \nabla^l f( x ) ( \theta, \cdots, \theta )$.

By Taylor expansion of $F(t)$ at $t=0$ up to $(k-1)$-th derivative, we have
\[
f( \exp_x(t \theta)) 
= F(t)
= \sum_{i=0}^{k-1} 
\frac{t^i}{i!} \nabla_\theta^i f(x)  + 
\frac{t^k}{k!}  \nabla_{ P_{x, \gamma(s) } \theta }^k f( \gamma( s) ),
\quad \text{for some $s \in [0,t]$.}
\]
By the definition of $L_{k,\beta}(f,x)$, we have that 
\[
| \nabla^k_{P_{x, \gamma(s) } \theta } f( \gamma(s) ) - \nabla^k_\theta f(x) |
 \le L_{k,\beta}(f,x) s^\beta 
 \le L_{k,\beta}(f,x) t^\beta.
\]
This gives that 
\[
f( \exp_x(t \theta)) = \sum_{i=0}^{k-1} 
\frac{t^i}{i!} \nabla_\theta^i f(x)  + 
\frac{t^k}{k!}  ( \nabla^k_\theta f(x) +  \tilde r(t)  ),
\quad |\tilde r(t)| \le L_{k,\beta}(f,x) t^\beta,
\]
which proves the lemma.
\end{proof}
\begin{lemma}\label{lemma:bound-Lip-tensor-field-derivative}
\rev{
Suppose $p,r $ are nonnegative integers
and $T^{(r)}$ is a $C^{p+1}$ order-$r$ tensor field on $\calM$.
Define 
\[
M_{p+1}:=
\sup_{z\in\calM}
\sup_{\substack{w_0,\ldots,w_p\in S_z^{d-1}\\
                  u_1,\ldots,u_r\in S_z^{d-1}}}
\left|
\nabla^{p+1}T^{(r)}(z)
(w_0,\ldots,w_p,u_1,\ldots,u_r)
\right|,
\]
which depends on $T^{(r)}$.
Then, for every $x\in\calM$, every $y\in B_\xi(x)$, and every
\[
v_1,\ldots,v_p,u_1,\ldots,u_r\in S_x^{d-1},
\]
where either list is omitted when its length is zero, one has
\begin{align*}
&\big|
\nabla^pT^{(r)}(x)(v_1,\ldots,v_p,u_1,\ldots,u_r) 
-\nabla^pT^{(r)}(y)
(P_{x,y}v_1,\ldots,P_{x,y}v_p,
 P_{x,y}u_1,\ldots,P_{x,y}u_r)
\big |
\le M_{p+1}d_\calM(x,y).
\end{align*}
In particular, taking $v_1=\cdots=v_p=v$ and
$u_1=\cdots=u_r= \theta$ gives the condensed repeated-slot estimate: 
recall the condensed notation \eqref{eq:condense-notation-tensor-Ti},
 for any $x \in \calM$, every $y\in B_\xi(x)$,
\begin{align*}
\sup_{v, \, \theta \in S_x^{d-1}} 
|\nabla_v^p T^{(r)}(x)(\theta) - \nabla_{P_{x,y} v}^p T^{(r)}(y)( P_{x,y} \theta)   |
\le M_{p+1} d_\calM(x,y).
 \end{align*}
}
\end{lemma}
\begin{proof}[Proof of Lemma \ref{lemma:bound-Lip-tensor-field-derivative}]
\rev{
We prove the bound for general $v_1,\ldots,v_p,u_1,\ldots,u_r\in S_x^{d-1}$,
which directly implies the condensed repeated-slot estimate. 
For any $x\in\calM$ and any $y\in B_\xi(x)$, let $s=d_\calM(x,y)$, and consider the unit-speed radial
geodesic
\[
\gamma(t)=\exp_x(tw),\qquad 0\le t\le s,
\]
where $w\in S_x^{d-1}$, $\gamma(0)=x$, and $\gamma(s)=y$.
Parallel transport each $v_j$ and $u_\ell$ along $\gamma$, and write
\[
V_j(t)=P_{x,\gamma(t)}v_j,
\qquad U_\ell(t)=P_{x,\gamma(t)}u_\ell.
\]
These vector fields are parallel along $\gamma$.  Define
\[
H(t):=\nabla^pT^{(r)}(\gamma(t))
(V_1(t),\ldots,V_p(t),U_1(t),\ldots,U_r(t)).
\]
Then
\begin{align*}
H(0)&=\nabla^pT^{(r)}(x)(v_1,\ldots,v_p,u_1,\ldots,u_r),\\
H(s)&=\nabla^pT^{(r)}(y)
(P_{x,y}v_1,\ldots,P_{x,y}v_p,
 P_{x,y}u_1,\ldots,P_{x,y}u_r).
\end{align*}
Because all the inserted vector fields are parallel,
\[
H'(t)=\nabla^{p+1}T^{(r)}(\gamma(t))
(\dot\gamma(t),V_1(t),\ldots,V_p(t),U_1(t),\ldots,U_r(t)).
\]
Here $\dot\gamma(t)$ is the additional derivative-slot vector.  Since
$\gamma$ has unit speed and parallel transport preserves lengths,
$\dot\gamma(t)$, $V_j(t)$, and $U_\ell(t)$ are all unit vectors.  Thus, taking
$w_0=\dot\gamma(t)$, $w_j=V_j(t)$, and $u_\ell=U_\ell(t)$ in the definition of
$M_{p+1}$ gives $|H'(t)|\le M_{p+1}$.  Therefore,
\begin{align*}
|H(s)-H(0)|
&=\left|\int_0^s H'(t)\,dt\right|
\le \int_0^s |H'(t)|\,dt 
\le sM_{p+1}=M_{p+1}d_\calM(x,y),
\end{align*}
which is the claimed bound.
}
\end{proof}

\subsection{Auxiliary lemmas}

The following lemma is used in the proof of Lemma  \ref{lemma:22_holder}.
\begin{lemma}\label{upper incomplete gamma bound}
Suppose $ d, k \in \mathbb{Z}$, $d\geq 1$, $k \geq 0$.
If  $0<\epsilon<{1}/{e}$, then there exists a constant $c(k,d)$ only depending on $k$ and $d$ such that 
\[
\int_{2\sqrt{(d+k+1)\log(\frac{1}{\epsilon})}}^\infty e^{-t^2/2} t^i dt \leq c(k,d) \epsilon^{d+k+1},
\quad \forall  0 \leq i \leq 2k+d-1.
\] 
\end{lemma}

\begin{proof}[Proof of Lemma \ref{upper incomplete gamma bound}]
Let $t_0 = 2\sqrt{(d+k+1)\log(\frac{1}{\epsilon})}$. 
Since $\epsilon< {1}/{e}$, 
then $t_0 >2$ always. 
For any $t>2$ and any $0 \leq i \leq 2k+d-1$,
$ e^{-t^2/2} t^i \leq  e^{-t^2/2} t^{2k+d-1}$.
In addition, there exists a constant $c(k,d)$ such that
\[
e^{-t^2/2} t^{2k+d-1} \le c(k,d) e^{-t^2/4}\frac{t}{2}, \quad \forall t \in (2,\infty).
\]
Putting together, we have
$$
\int_{t_0}^\infty e^{-t^2/2} t^i dt 
\le \int_{t_0}^\infty e^{-t^2/2} t^{2k+d-1} dt  
\leq \frac{c(k,d)}{2} \int_{t_0}^\infty e^{-t^2/4} t dt  
=  c(k,d) \epsilon^{d+k+1},
$$
where in the last inequality we used that 
$\int_{t_0}^\infty e^{-t^2/4} t dt  = 2 e^{-t_0^2/4} = 2 \epsilon^{ d+k+1}$.
\end{proof}

We also recall the classical Bernstein inequality used in the proof of Lemma \ref{lemma:concentrate_v_hat}.
\begin{lemma}[Bernstein inequality]
    \label{lemma:berstein}
   Let $\xi_j$ be i.i.d bounded random variables, 
   $j = 1, \cdots, N$. 
   $\E(\xi_j) = 0$,
   $|\xi_j| \le L $ and $\E \xi_j^2 \le \nu$ for positive constants $L$ and $\nu$.
   Then,  $\forall \tau > 0$, 
   $$ 
   \Pr 
   \Big[ \frac{1}{N}\sum_{j=1}^N \xi_j > \tau \Big], 
   \Pr 
   \Big[ \frac{1}{N}\sum_{j=1}^N \xi_j < -\tau \Big] \le \exp 
   \{- \frac{\tau^2N}{2(\nu+\frac{\tau L}{3})} \}. 
   $$
   In particular, when $\tau L< 3 \nu$, both the tail probabilities are bounded by $\exp\{-\frac{1}{4}\frac{N\tau^2}{\nu}\}$.
\end{lemma}

\section{Experimental details}
\label{ap:numerical}

\subsection{Algorithm}
\label{ap:algo}

Denote the observation data as  $X = \{X_i \}_{i=1}^{n}$, $Y= \{Y_i \}_{i=1}^{n}$, and we are also given a stand-alone test set $X^{te} = \{ X^{\textit{te}}_j \}_{j=1}^{n_{te}}$. 
We also denote by $X$ the $n$-by-$D$ matrix, where each row is a sample $X_i \in \R^D$.
Similarly, $Y$ is a length-$n$ vector,
and $X^{te}$ is an  $n_{te}$-by-$D$ matrix.
Given a kernel bandwidth $t$, we denote by $h_t(X,X)$ the $n$-by-$n$ kernel matrix, the $i,j$-th entry of which equals $h_t(X_i,X_j)$.
Similarly, $h_{t}(X, X_j^{\textit{te}})$ is a length-$n$ vector whose $i$-th entry equals $h_t(X_i,X_j^{te} )$.

Following the proposed method in Section \ref{sec:knn_prior},
we use the prior $p(t)$ as defined in \eqref{eq:EBprior}\eqref{eq:def-Tn-knn}\eqref{eq:def_vn}, 
where 
$k = \lceil 0.25 \log^2(n) \rceil $ and when $n < 200$ we set $k=2$.
The statistic $T_n$ is computed by averaging on a random subset $S$, where
we choose $|S| = \lceil ( \log n)^3 \rceil$
and  $|S| = n$ when $n < 200$.
We sample $t$ from the marginal posterior which can be written as
\begin{equation}\label{marginal_pos}
    p(t|X,Y) \propto L(Y|X, t)p(t),
\end{equation}
where $L(Y|X, t)$ is the marginal likelihood, and
\begin{align}\label{eq:log-likelihood-algo}
   & \log( L(Y|X, t)) 
    = \log(P(Y|X ,t ))  \\
   & ~~~
   = -\frac{1}{2} Y^{T}(h_{t}(X,X)+\sigma^2I)^{-1}Y - \frac{1}{2}\log(|h_{t}(X,X)+\sigma^2I|)  - \frac{n}{2} \log(2 \pi).
   \nonumber
\end{align}
In \eqref{marginal_pos}, 
we do not need to obtain the normalizing constant in $p(t)$ 
because the Metropolis-Hasting MCMC only needs the ratio of the marginal posterior.

Given a bandwidth $t$, we can compute 
\begin{equation}
\label{condition_pos}
   \hat f(X_j^{te} | t): = 
   \E [ f^t(X_j^{\textit{te}})|X, Y] 
    = Y^T (h_{t}(X,X)+\sigma^2 I)^{-1} h_{t}(X, X_j^{\textit{te}}).
\end{equation}
By an MCMC sampling of $t$ from \eqref{marginal_pos},
one can compete \eqref{condition_pos} in each iteration.
Taking average of $\hat f(X_j^{te} | t)$ over the iterations
provides an estimate of the posterior mean of $f$ on the test samples. 
The procedure is summarized in  Algorithm \ref{alg:gibbs}.
In all reported experiments in Section \ref{sec:num}, we conduct 3000 iterations with the first 1000 iterations discarded as burn-in.

\begin{algorithm}[t] 
\caption{Bayesian posterior-mean estimator (output on a test set)}

\begin{flushleft}

\textit{Input}: hyperparameters  $\sigma^2 ,a_0, b_0$, 
observed data $X = \{X_i \}_{i=1}^{n}$, $Y= \{Y_i \}_{i=1}^{n}$, 
 test set $X^{te} = \{ X^{\textit{te}}_j \}_{j=1}^{n_{te}}$, number of iterations $B$.

 \vspace{5pt}
 
 \textit{Output}: $\{\hat f (X^{\textit{te}}_j)\}_{j=1}^{n_{te}}$ on the test set
\end{flushleft}

\begin{algorithmic}[1]
        \State Set initial value $t_0$.
	\For {$b = 1, \cdots, B$}
		\State 
	Sample $t_b \sim t | X, Y$ by Metropolis-Hasting MCMC 
    from the marginal posterior \eqref{marginal_pos},
    where $L(Y|X,t)$ is as in \eqref{eq:log-likelihood-algo},
    and $p(t)$ is the EB prior as in \eqref{eq:EBprior}\eqref{eq:def-Tn-knn}\eqref{eq:def_vn}.
    
        \State 
		 Compute the conditional posterior mean 
   $\hat f_b (X_j^{te} | t_b) = \E [  f^{t_b }(X^{\textit{te}}_j) |  X, Y ]$ by  \eqref{condition_pos} for $ j=1, \cdots, n_{te}$.
    \EndFor	
    \\
    Compute $\hat{f}(X^{\textit{te}}_j) = \frac{1}{B} \sum_{b=1}^{B} \hat f_b (X_j^{te} | t_b) $ for $ j=1, \cdots, n_{te}$.
\end{algorithmic} 
\label{alg:gibbs}
\end{algorithm}

\begin{figure}[t!]
\centering
 \begin{minipage}{0.45\textwidth}
 \centering
\includegraphics[height=0.75\linewidth]{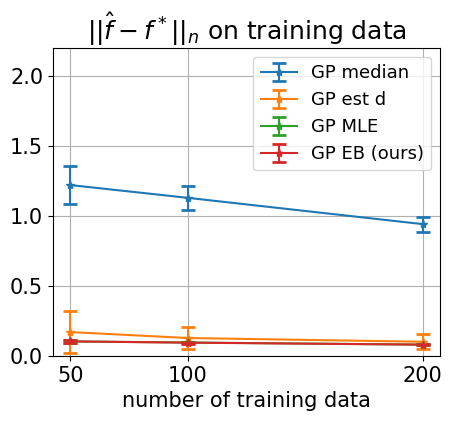}\subcaption{}
\end{minipage}
 \begin{minipage}{0.45\textwidth}
  \centering
\includegraphics[height=0.75\linewidth]{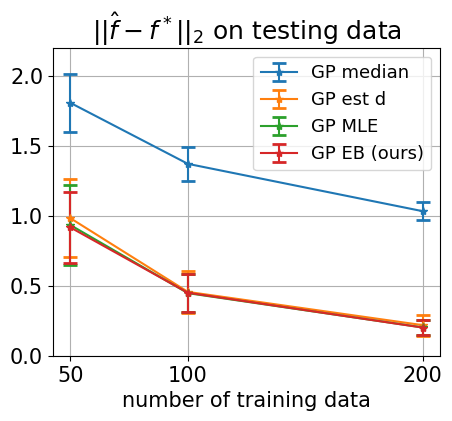}\subcaption{}
\end{minipage}
\caption{
Additional GP baselines on Swiss Roll data at training size 50, 100, 200.
The in-sample error (a) 
and out-of-sample error (b) 
are averaged from 200 repeated runs.}
\label{fig:additional-EB-swiss-roll}
\end{figure}

\subsection{Additional details of numerical experiments}
 \label{app:exp_detail}

\subsubsection{Swiss Roll data}

We generate $n$ samples $X_i$, $i=1,\ldots,n$,
by parametrizing $X \in \R^3$ in two variables $(u,v)$ as
\[
X(u,v) = [ u \cos(u), v, u \sin(u) ],
\]
and sample $u$ and $v$ i.i.d. from the distribution
\[
u \sim  \mbox{Unif}\Big(\frac{2\pi}{2},\frac{9\pi}{2}\Big),
\quad  v \sim \mbox{Unif}(0, 15).
\]
The true regression function $f^*$, written as a function of $(u,v)$, is set to be
\[
f_{\rm Swiss}(u,v)= 4 \left( \frac{ u-7\pi/2 }{ 3 \pi / 2} \right)^2 + \frac{\pi}{45} v.
\]
Finally, the predictor $X$ is rescaled by 
\[
X_1 \leftarrow (X_1 + 15)/30,
\quad  X_2 \leftarrow  X_2/15,
\quad X_3 \leftarrow (X_3 + 15)/30,
\]
so that $X$ is lying inside $[0,1]^3$.

\subsubsection{Mixed dimension data}

The data samples $\{ X_i, Y_i\}$ are generated according to the following procedure: 
With $1/2$ probability, we randomly pick $X_i$ from the Swiss roll as defined above, 
and $Y_i$ generated from the same $f^*$ on the 2D manifold therein.
With $1/2$ probability, 
we draw $X_i$ from the 1D curve 
\[
X(t)=  \frac{7\pi}{2} [ \cos(\pi  t) \cos(4 \pi t), \, 
	   			1 +  \cos(\pi t) \sin(4 \pi t),\,
				 \sin(\pi t)], 
\]
where $t$ is sampled  i.i.d. from $t \sim {\rm Unif}(-1,1)$.
The true regression function $f^*$  on the 1D curve is set as $f^*(X(t)) = \bar f_{\rm Swiss}(X(t))$, where
\begin{equation*}
   \bar f_{\rm Swiss}(x_1, x_2, x_3 ) = f_{\rm Swiss}( \sqrt{ x_1^2+x_3^2}, x_2).
\end{equation*}
This design ensures that
 the function $f^*$ takes 
 the same value on the intersection of the surface and the curve, thus preserving the continuity of the function. 
 We rescale $X$ to be inside $[0,1]^3$ similarly as for the Swiss Roll data.

\begin{figure}[t!]
  \begin{minipage}{0.33\textwidth}
\hspace{15pt}
\includegraphics[height=0.86\linewidth]{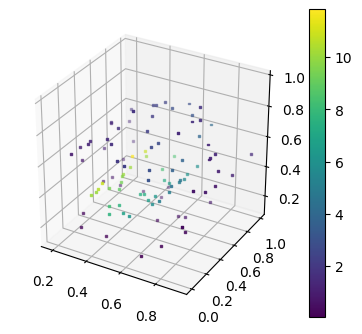}\subcaption{}
\end{minipage}
 \begin{minipage}{0.66\textwidth}
\includegraphics[height=0.43\linewidth]{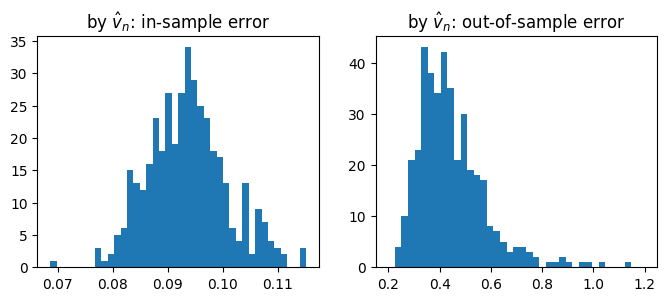}\subcaption{}
\end{minipage}
  \begin{minipage}{0.33\textwidth}
\includegraphics[height=0.86\linewidth]{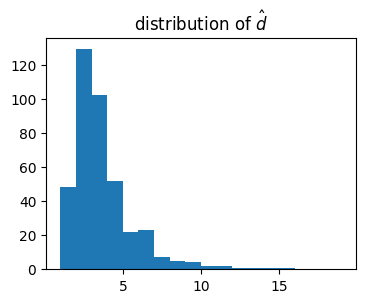}\subcaption{}
\end{minipage}
 \begin{minipage}{0.66\textwidth}
\includegraphics[height=0.43\linewidth]{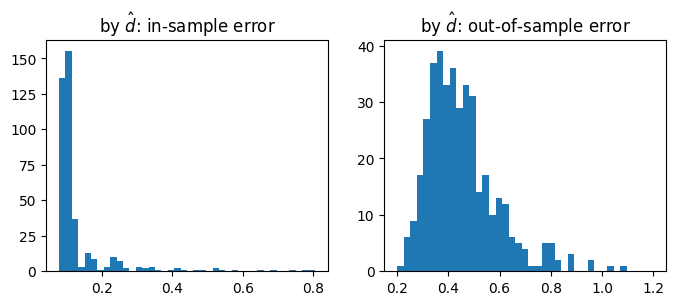}\subcaption{}
\end{minipage}
\caption{
Comparison of the proposed EB GP and GP with estimated dimension $\hat d$.
The dataset is Swiss Roll  with training size 100 and testing size 2000.
(a) One realization of the training set, the predictors $X_i$ are colored by the observed values $Y_i$.
(b)  Histograms of the in-sample and out-of-sample errors of the proposed EB GP.
(c) Histogram of estimated $\hat d$.
(d) Error histograms of the rescaled Gamma GP with estimated dimension $\hat d$.
The experiments are repeated for 400 runs to produce the histograms.}
\label{fig:outlier-swiss-roll}
\end{figure}

\subsection{Additional comparison of empirical Bayes approaches}
\label{app:subsec-additional-EB}

We conducted additional experiments to compare our empirical Bayes (EB) prior with alternatives on the Swiss roll data.
The baselines are:  

(ii') GP estimated $d$: 
the prior is the rescaled Gamma distribution, where the manifold dimension is estimated from data as proposed in \cite{yang2016bayesian}.
Specifically, following \cite{yang2016bayesian}, we adopt the estimator of the manifold dimension as
\begin{equation}\label{eq:hat-d-estimator-YD16}
\hat d = \text{the closest integer to }
 \frac{\log 2}{ \log \hat R_k(X_1) - \log \hat R_{\lceil k/2 \rceil}(X_1)},
\end{equation}
where $\hat R_k(x)$ is the distance from $x$ to its $k$-nearest neighbor in the training set $\{X_i\}_{i=1}^n$,
 $X_1$ is a random member of the training set,
and $k = \lceil \sqrt{n} \rceil$.

(iv) GP max-likelihood (GP MLE): selecting the kernel bandwidth $t$ based on maximizing the marginal likelihood. 
Note that this method differs from our theoretical setting because it corresponds to using a uniform prior.

(v) GP median heuristic (GP median): setting $t$ to be the median of the distances between samples. 

We use training size up to 200 since all methods give comparable performance on larger sample sizes. 
The experiments are repeated for 200 runs. 
The results are shown in Figure \ref{fig:additional-EB-swiss-roll}.
It can be seen that the median heuristic gave much larger errors, both the in-sample and out-of-sample ones.
The GP MLE perform similarly as the proposed EB approach on this example;
the GP with estimated $d$ gave comparable performance on the test data, and larger in-sample error with larger variance, especially at the small training size.

To further investigate the effect of estimating dimension $d$, we compare the errors of (ii') and our EB via their distributions. 
We choose training size 100 and increase the number of runs to 400.
The distribution of the errors are plotted as histograms in Figure \ref{fig:outlier-swiss-roll}(b) and (d) for our EB  (called ``by $\hat v_n$'')
and (ii') (called ``by $\hat d$'') respectively. 
It can be seen that the out-of-sample errors of the two methods are comparable, 
yet the in-sample error of (ii') with estimated $d$ has a longer tail distribution, which are ``outlier'' errors that can be as large as 0.8.
In comparison, the proposed EB produce a more concentrated in-sample error around its average and up to 0.12,
showing a more stable performance. 

The reason of these outlier errors by (ii') is likely due to the outlier errors in estimating the dimension, that is, with a small chance the estimated $\hat d$ can be far from the true $d$ (which is 2 here).
This is verified by the histogram of the estimated $\hat d$ shown in Figure \ref{fig:outlier-swiss-roll}(c),
which has a long tail up to 15. 
The training size is 100, and the Swiss Roll data at this sample size barely reveal the underlying manifold if one only considers $k$NN distance at a random point, see Figure \ref{fig:outlier-swiss-roll}(a).
As a result, the dimension estimator \eqref{eq:hat-d-estimator-YD16} gives unstable performance at such low sample size.
On this example, The error in $\hat d$ affects the in-sample error more visibly, 
possibly due to that both models are already getting large testing errors at this low training sample size.  
In summary, this result suggests that our EB prior based on the averaged kernel affinity statistic $\hat v_n(t)$ can be more robust than EB based on manifold dimension estimation in practice, especially at relatively small sample size.

\end{appendix}

\end{document}